\documentclass[reqno]{amsart}

\usepackage{amssymb}
\usepackage{amscd}
\usepackage{amsfonts}
\usepackage{version}
\usepackage{color}

\numberwithin{equation}{section}

\theoremstyle{plain}

\newtheorem{theorem}[subsection]{Theorem}
\newtheorem{proposition}[subsection]{Proposition}
\newtheorem{lemma}[subsection]{Lemma}
\newtheorem{corollary}[subsection]{Corollary}
\newtheorem{conjecture}[subsection]{Conjecture}

\newtheorem{question}[subsection]{Question}

\newtheorem*{theorem74-repeat}{Theorem \ref{aderiv}}

\theoremstyle{definition}
\newtheorem{definition}[subsection]{Definition}

\newtheorem{remark}[subsection]{Remark}

\newtheorem{example}[subsection]{Example}

\renewcommand{\leq}{\leqslant}
\renewcommand{\geq}{\geqslant}

\newsavebox{\proofbox}
\savebox{\proofbox}{\begin{picture}(7,7)%
  \put(0,0){\framebox(7,7){}}\end{picture}}

\renewcommand{\mod}{{\ \operatorname{mod}\ }}
\newcommand\E{\mathbb{E}}
\newcommand\Z{\mathbb{Z}}
\newcommand\R{\mathbb{R}}
\newcommand\T{\mathbb{T}}
\newcommand\C{\mathbb{C}}
\newcommand\N{\mathbb{N}}

\newcommand\Q{\mathbb{Q}}
\newcommand\Lip{\operatorname{Lip}}

\newcommand\Symb{\operatorname{Symb}}

\newcommand\n{{\bf n}}

\newcommand\eps{\varepsilon}

\newcommand\id{\operatorname{id}}
\renewcommand\th{{\operatorname{th}}}

\newcommand\poly{{\operatorname{poly}}}
\newcommand\GI{{\operatorname{GI}}}
\newcommand\DR{{\operatorname{DR}}}
\newcommand\MD{{\operatorname{Multi}}}
\newcommand\Nil{{\operatorname{Nil}}}
\newcommand\ind{{\operatorname{ind}}}
\newcommand\rat{{\operatorname{rat}}}
\newcommand\sml{{\operatorname{sml}}}
\newcommand\lin{{\operatorname{lin}}}

\newcommand\Taylor{{\operatorname{Taylor}}}
\newcommand\petal{{\operatorname{petal}}}
\newcommand\Horiz{{\operatorname{Horiz}}}
\newcommand\dist{{\operatorname{dist}}}
\newcommand\orbit{{\mathcal{O}}}
\newcommand\F{{\mathcal{F}}}
\newcommand\HK{\operatorname{HK}}
\newcommand\ultra{{{}^*}}

\begin{document}
\title{An inverse theorem for the Gowers $U^{s+1}[N]$-norm}
\author{Ben Green}
\address{Centre for Mathematical Sciences\\
Wilberforce Road\\
Cambridge CB3 0WA\\
England }
\email{b.j.green@dpmms.cam.ac.uk}
\author{Terence Tao}
\address{Department of Mathematics\\
UCLA\\
Los Angeles, CA 90095\\
USA}
\email{tao@math.ucla.edu}
\author{Tamar Ziegler}
\address{Department of Mathematics \\
Technion - Israel Institute of Technology\\
Haifa, Israel 32000}
\email{tamarzr@tx.technion.ac.il}

\begin{abstract} We prove the \emph{inverse conjecture for the Gowers $U^{s+1}[N]$-norm} for all $s \geq 1$; this is new for $s \geq 4$. More precisely, we establish that if $f : [N] \rightarrow [-1,1]$ is a function with $\Vert f \Vert_{U^{s+1}[N]} \geq \delta$ then there is a bounded-complexity $s$-step nilsequence $F(g(n)\Gamma)$ which correlates with $f$, where the bounds on the complexity and correlation depend only on $s$ and $\delta$. From previous results, this conjecture implies the Hardy-Littlewood prime tuples conjecture for any linear system of finite complexity. 
\end{abstract}

\maketitle
\setcounter{tocdepth}{1}	

\tableofcontents

\section{Introduction}

The purpose of this paper is to establish the general case of a conjecture named the \emph{Inverse Conjecture for the Gowers norms} by  the first two authors in \cite[Conjecture 8.3]{green-tao-linearprimes}. If $N$ is a (typically large) positive integer then we write $[N] := \{1,\dots,N\}$. For each integer $s \geq 1$ the inverse conjecture $\GI(s)$, whose statement we recall shortly, describes the structure of $1$-bounded functions $f : [N] \rightarrow \C$ whose $(s+1)^{\operatorname{st}}$ Gowers norm $\Vert f \Vert_{U^{s+1}[N]}$ is large. These conjectures together with a good deal of motivation and background to them are discussed in \cite{green-icm,green-tao-u3inverse,green-tao-linearprimes}. The conjectures $\GI(1)$ and $\GI(2)$ have been known for some time, the former being a straightforward application of Fourier analysis, and the latter being the main result of \cite{green-tao-u3inverse} (see also \cite{sam} for the characteristic $2$ analogue).  The case $\GI(3)$ was also recently established by the authors in \cite{u4-inverse}. The aim of the present paper is to establish the remaining cases $\GI(s)$ for $s \geq 3$, in particular reestablishing the results in \cite{u4-inverse}.
 
We begin by recalling the definition of the Gowers norms. If $G$ is a finite abelian group, $d \geq 1$ is an integer, and $f : G \rightarrow \C$ is a function then we define
\begin{equation}\label{ukdef}
 \Vert f \Vert_{U^{d}(G)} := \left(  \E_{x,h_1,\dots,h_d \in G} \Delta_{h_1} \ldots \Delta_{h_d} f(x)\right)^{1/2^d},
\end{equation}
where $\Delta_h f$ is the multiplicative derivative
$$ \Delta_h f(x) := f(x+h) \overline{f(x)}$$
and $\E_{x \in X} f(x) := \frac{1}{|X|} \sum_{x \in X} f(x)$ denotes the average of a function $f: X \to \C$ on a finite set $X$. Thus for instance we have
\[ \Vert f \Vert_{U^2(G)} := \left( \E_{x,h_1,h_2 \in G} f(x) \overline{f(x+h_1) f(x+h_2)} f(x+h_1 + h_2)\right)^{1/4}.\]
One can show that $U^d(G)$ is indeed a norm on the functions $f: G \to \C$ for any $d \geq 2$, though we will not need this fact here.

In this paper we will be concerned with functions on $[N]$, which is not quite a group. To define the Gowers norms of a function $f : [N] \rightarrow \C$, set $G := \Z/\tilde N\Z$ for some integer $\tilde N \geq 2^d N$, define a function $\tilde f : G \rightarrow \C$ by $\tilde f(x) = f(x)$ for $x = 1,\dots,N$ and $\tilde f(x) = 0$ otherwise, and set 
\[ \Vert f \Vert_{U^d[N]} := \Vert \tilde f \Vert_{U^d(G)} / \Vert 1_{[N]} \Vert_{U^d(G)},\] where $1_{[N]}$ is the indicator function of $[N]$.  It is easy to see that this definition is independent of the choice of $\tilde N$. One could take $\tilde N := 2^d N$ for definiteness if desired. 

The \emph{Inverse conjecture for the Gowers $U^{s+1}[N]$-norm}, abbreviated as $\GI(s)$, posits an answer to the following question.

\begin{question}
Suppose that $f : [N] \rightarrow \C$ is a function bounded in magnitude by $1$, and let $\delta > 0$ be a positive real number. What can be said if $\Vert f \Vert_{U^{s+1}[N]} \geq \delta$?
\end{question}

Note that in the extreme case $\delta = 1$ one can easily show that $f$ is a phase polynomial, namely  $f(n)=e(P(n))$ for some polynomial $P$ of degree at most $s$.  Furthermore, if $f$ correlates with a phase polynomial, that is to say if $|\E_{n \in [N]} f(n) \overline{e( P(n))}| \geq \delta$, then it is easy to show that
$\Vert f \Vert_{U^{s+1}[N]} \geq c(\delta)$. It is natural to ask whether the converse is also true - does a large Gowers norm imply correlation with a polynomial phase function?
Surprisingly, the answer is no, as was observed by Gowers \cite{gowers-4aps} and, in the related context of \emph{multiple recurrence}, somewhat earlier by Furstenberg and Weiss \cite{furst, fw-char}. The work of Furstenberg-Weiss and Conze-Lesigne \cite{conze} draws attention to the role of homogeneous spaces $G/\Gamma$ of nilpotent Lie groups, and subsequent work of Host and Kra \cite{host-kra} provides a link, in an ergodic-theoretic context, between these spaces and certain seminorms with a formal similarity to the Gowers norms under discussion here. Later work of Bergelson, Host and Kra \cite{bhk} highlights the role of a class of functions arising from these spaces $G/\Gamma$ called \emph{nilsequences}. The inverse conjecture for the Gowers norms, first formulated precisely in \cite[\S 8]{green-tao-linearprimes}, postulates that this class of functions (which contains the polynomial phases) represents the full set of obstructions to having large Gowers norm. 

We now recall that precise formulation. Recall that an \emph{$s$-step nilmanifold} is a manifold of the form $G/\Gamma$, where $G$ is a connected, simply-connected nilpotent Lie group of step at most $s$ (i.e. all $s+1$-fold commutators of $G$ are trivial), and $\Gamma$ is a discrete, cocompact\footnote{A subgroup $\Gamma$ of a topological group $G$ is \emph{cocompact} if the quotient space $G/\Gamma$ is compact.} subgroup of $G$.

\begin{conjecture}[$\GI(s)$]\label{gis-conj}  Let $s \geq 0$ be an integer, and let $0 < \delta \leq 1$.   Then there exists a finite collection ${\mathcal M}_{s,\delta}$ of $s$-step nilmanifolds $G/\Gamma$, each equipped with some smooth Riemannian metric $d_{G/\Gamma}$ as well as constants $C(s,\delta), c(s,\delta) > 0$ with the following property. Whenever $N \geq 1$ and $f : [N] \rightarrow \C$ is a function bounded in magnitude by $1$ such that $\Vert f \Vert_{U^{s+1}[N]} \geq \delta$, there exists a nilmanifold $G/\Gamma \in {\mathcal M}_{s,\delta}$, some $g \in G$ and a function $F: G/\Gamma \to \C$ bounded in magnitude by $1$ and with Lipschitz constant at most $C(s,\delta)$ with respect to the metric $d_{G/\Gamma}$ such that
$$ |\E_{n \in [N]} f(n) \overline{F(g^n x)}| \geq c(s,\delta).$$
\end{conjecture}

We remark that there are many equivalent ways to reformulate this conjecture.  For instance, instead of working with a finite family ${\mathcal M}_{s,\delta}$ of nilmanifolds, one could work with a single nilmanifold $G/\Gamma = G_{s,\delta}/\Gamma_{s,\delta}$, by taking the Cartesian product of all the nilmanifolds in the family.  Other reformulations include an equivalent formulation using polynomial nilsequences rather than linear ones (see Conjecture \ref{gis-poly}) and an ultralimit formulation (see Conjecture \ref{gis-conj-nonst}).  One can also formulate the conjecture using bracket polynomials, or local polynomials; see \cite{green-tao-u3inverse} for a discussion of these equivalences in the $s=2$ case.

Let us briefly review the known partial results on this conjecture:
\begin{enumerate}
\item $\GI(0)$ is trivial.
\item $\GI(1)$ follows from a short Fourier-analytic computation.
\item $\GI(2)$ was established about five years ago in \cite{green-tao-u3inverse}, building on work of Gowers \cite{gowers-4aps}.
\item $\GI(3)$ was established, quite recently, in \cite{u4-inverse}.
\item In the extreme case $\delta = 1$ one can easily show that $f(n)=e(P(n))$ for some polynomial $P$ of degree at most $s$, and every such function \emph{is} an $s$-step nilsequence by a direct construction. See, for example, \cite{green-tao-u3inverse} for the case $s = 2$. 
\item In the almost extremal case $\delta \geq 1- \eps_s$, for some $\eps_s > 0$, one may see that $f$ correlates with a phase $e(P(n))$ by adapting arguments first used in the theoretical computer-science literature \cite{akklr}.
\item The analogue of $\GI(s)$ in ergodic theory (which, roughly speaking, corresponds to the asymptotic limit $N \to \infty$ of the theory here; see \cite{host-kra-uniformity} for further discussion) was formulated and established in \cite{host-kra}, work done independently of the work of Gowers (see also the earlier paper \cite{hk1}). This work was the first place in the literature to link objects of Gowers-norm type (associated to functions on a measure-preserving system $(X, T,\mu)$) with flows on nilmanifolds, and the subsequent paper \cite{bhk} was  the first work to underline the importance of \emph{nilsequences}. The formulation of $\GI(s)$ by the first two authors in \cite{green-tao-linearprimes} was very strongly influenced by these works. For the closely related problem of analysing multiple ergodic averages, the relevance of flows on nilmanifolds was earlier pointed out in  \cite{furst, fw-char,lesigne-nil}, building upon earlier work in  \cite{conze}.  See also \cite{hk0,ziegler} for related work on multiple averages and nilmanifolds in ergodic theory.
\item The analogue of $\GI(s)$ in finite fields of large characteristic was established by ergodic-theoretic methods in \cite{bergelson-tao-ziegler,tao-ziegler}.
\item A weaker ``local'' version of the inverse theorem (in which correlation takes place on a subprogression of $[N]$ of size $\sim N^{c_s}$) was established by Gowers \cite{gowers-longaps}. This paper provided a good deal of inspiration for our work here.
\item The converse statement to $\GI(s)$, namely that correlation with a function of the form $n \mapsto F(g^n x)$ implies that $f$ has large $U^{s+1}[N]$-norm, is also known. This was first established in \cite[Proposition 12.6]{green-tao-u3inverse}, following arguments of Host and Kra \cite{host-kra} rather closely. A rather simple proof of this result is given in \cite[Appendix G]{u4-inverse}. 
\end{enumerate}

The main result of this paper is a proof of Conjecture \ref{gis-conj}:

\begin{theorem}\label{mainthm}  For any $s \geq 3$,
the inverse conjecture for the $U^{s+1}[N]$-norm, $\GI(s)$, is true.
\end{theorem}

By combining this result with the previous results in \cite{green-tao-linearprimes,green-tao-mobiusnilsequences} we obtain a quantitative Hardy-Littlewood prime tuples conjecture for all linear systems of finite complexity; in particular, we now have the expected asymptotic for the number of primes $p_1 < \ldots < p_k \leq X$ in arithmetic progression, for every fixed positive integer $k$.  We refer to \cite{green-tao-linearprimes} for further discussion, as we have nothing new to add here regarding these applications.  Several further applications of the $\GI(s)$ conjectures are given in \cite{fhk,green-tao-arithmetic-regularity}. \vspace{11pt}

\section{Strategy of the proof}\label{strategy-sec}

The proof of Theorem \ref{mainthm} is long and complicated, but broadly speaking it follows the strategy laid out in previous works \cite{gowers-4aps,gowers-longaps,green-tao-u3inverse,u4-inverse,sam}.  We induct on $s$, assuming that $\GI(s-1)$ has already been established and using this to prove $\GI(s)$. To explain the argument, let us first summarise the main steps taken in \cite{u4-inverse} in order to deduce $\GI(3)$, the inverse theorem for the $U^4$-norm, from $\GI(2)$, the inverse theorem for the $U^3$ norm (established in \cite{green-tao-u3inverse}). Once this is done we will explain some of the extra difficulties involved in handling the general case.  For a more extensive (but informal) discussion of the proof strategy, see \cite{gtz-announce}.  Once we set up some technical machinery, we will also be able to give a more detailed description of the strategy in \S \ref{overview-sec}.

Here, then, is an overview of the argument in \cite{u4-inverse}.

\begin{enumerate}
\item (Apply induction) If $\Vert f \Vert_{U^4[N]} \gg 1$ then, for many $h$, $\Vert \Delta_h f \Vert_{U^3[N]} \gg 1$ and so $\Delta_h f$ correlates with a $2$-step nilsequence $\chi_h$.
\item (Nilcharacter decomposition) $\chi_h$ may be decomposed as a sum of a special type of nilsequence called a \emph{nilcharacter}, essentially by a Fourier decomposition. For the sake of illustration, these $2$-step nilcharacters may be supposed to have the form
\[ \chi_h(n) = e(\{\alpha_h n\} \beta_h n),\]
although these are not quite nilcharacters due to the discontinuous nature of the fractional part function $x \mapsto \{x\}$, and in any event a general $2$-step nilcharacter will be modeled by a linear combination of such ``bracket quadratic monomials'', rather than by a single such monomial (see \cite{green-tao-u3inverse} for further discussion).
\item (Rough linearity) The fact that $\Delta_h f$ correlates with $\chi_h$ forces $\chi_h$ to behave weakly linearly in $h$. To get a feel for why this is so, suppose that $|f| \equiv 1$; then we have the cocycle identity
\[ \Delta_{h+k} f(n) = \Delta_h f(n+k) \Delta_k f(n).\] To capture something like the same behaviour in the much weaker setting where $\Delta_h f$ correlates with $\chi_h$, we use an extraordinary argument of Gowers \cite{gowers-4aps} relying on the Cauchy-Schwarz inequality. Roughly speaking, the information obtained is of the form
\begin{equation}\label{linear-eq} \chi_{h_1} \chi_{h_2} \sim \chi_{h_3} \chi_{h_4} \quad \mbox{modulo lower order terms} \end{equation} for many $h_1, h_2, h_3, h_4$ with $h_1 + h_2 = h_3 + h_4$.
\item (Furstenberg-Weiss) An argument of Furstenberg and Weiss \cite{fw-char} is adapted in order to study \eqref{linear-eq}. The quantitative distribution theory of nilsequences developed in \cite{green-tao-nilratner} is a major input here. It is concluded that we may assume that the frequency $\beta_h$ does not actually depend on $h$. Note that this step appeared for the first time in the proof of $\GI(3)$; it did not feature in the proof of $\GI(2)$ in \cite{green-tao-u3inverse}.
\item (Linearisation) A similar argument allows one to then assert that 
\begin{equation}\label{additive-eq} \alpha_{h_1} + \alpha_{h_2} \approx \alpha_{h_3} + \alpha_{h_4} \pmod{1} 
\end{equation}
for many $h_1,h_2,h_3,h_4$ with $h_1 + h_2 = h_3 + h_4$.
\item (Additive Combinatorics) By arguments from additive combinatorics related to the Balog-Szemer\'edi-Gowers theorem \cite{balog,gowers-4aps} and Freiman's theorem, as well as some geometry of numbers, we may then assume that $\alpha_h$ varies ``bracket-linearly'' in $h$, thus 
\begin{equation}\label{bracket-lin}  \alpha_h = \gamma_1 \{ \eta_1 h\} + \dots + \gamma_d \{\eta_d h\}. \end{equation}
Up to top order, then, the nilcharacter $\chi_h(n)$ can now be assumed to take the form $e(\psi(h,n,n))$, where $\psi$ is ``bracket-multilinear''; it is a sum of terms such as $\{\gamma \{\eta h\} n\} \beta n$.
\item (Symmetry argument) The bracket multilinear form $\psi$ obeys an additional symmetry property. This is a reflection of the identity $\Delta_h \Delta_k f = \Delta_k \Delta_h f$, but transferring this to the much weaker setting in which we merely have correlation of $\Delta_h f$ with $\chi_h$ requires another appeal to Gowers' Cauchy-Schwarz argument from (iii). In fact, the key point is to look at the second order terms in \eqref{linear-eq}.
\item (Integration) Assuming this symmetry, one is able to express
\[  \chi_h(n)  \sim \Theta(n+h) \overline{\Theta'(n)}\]
for some bracket cubic functions $\Theta, \Theta'$, which morally take the form \[ \Theta(n), \Theta'(n) \sim e(\psi(n,n,n)/3)\] (for much the same reason that $x^3/3$ is an antiderivative of $x^{2}$).  Thus we morally have
\[ \Delta_h f(n) \sim \Theta(n+h) \overline{\Theta'(n)}\]
\item (Construction of a nilsequence) Any bracket cubic form like $e(\psi(n,n,n))$ ``comes from'' a 3-step nilmanifold; this construction is accomplished in \cite{u4-inverse} in a rather \emph{ad hoc} manner.
\item From here, one can analyse lower order terms by the induction hypothesis $\GI(2)$. This is a relatively easy matter.
\end{enumerate}

Let us now discuss the argument of this paper in the light of each point of this outline. 
A more detailed outline is given in \S \ref{overview-sec}.
Assume that $\GI(s-1)$ has been established.

\begin{enumerate}
\item (Apply induction) If $\Vert f \Vert_{U^{s+1}[N]} \gg 1$ then, for many $h$, $\Vert \Delta_h f \Vert_{U^s[N]} \gg 1$ and so $\Delta_h f$ correlates with an $(s-1)$-step nilsequence $\chi_h$. This is straightforward (see \S \ref{overview-sec}).
\item (Nilcharacter decomposition) $\chi_h$ may be decomposed into nilcharacters; this is fairly straightforward as well. It is somewhat reassuring to think of $\chi_h(n)$ as having the form $e(\psi_h(n))$, where $\psi_h(n)$ is a bracket polynomial ``of degree $s-1$'', but we will not be working explicitly with bracket polynomials much in this paper, except as motivation and as a source of examples. One of the main challenges one is faced with during an attempt to prove $\GI(4)$ by a direct generalisation of our arguments from \cite{u4-inverse} is the fact that already bracket cubic polynomials are rather complicated to deal with and can take different forms such as $\{\alpha n\}\{\beta n\}\gamma n$ and $\{ \{\alpha n\} \beta n\} \gamma n$.

Instead of objects such as $e(\alpha n\{\beta n\})$, then, we will work with the rather more abstract notion of a \emph{symbol}. This notion, which is fairly central to our paper, is defined and discussed in \S \ref{nilcharacters}. One additional technical point is worth mentioning here. This is the fact that $e(\alpha n\{\beta n\})$ (say) cannot be realised as a nilsequence $F(g^n \Gamma)$ with $F$ \emph{continuous}, and therefore the distributional results of \cite{green-tao-nilratner} do not directly apply. In \cite{u4-inverse} these discontinuities could be understood quite explicitly, but here we take a different approach: we decompose $G/\Gamma$ into $D$ pieces using a smooth partition of unity for some $D=O(1)$, and then work instead with the (smooth) $\C^D$-valued nilsequence consisting of these pieces. 

We discuss this device more fully in \S \ref{nilcharacters}, but we emphasise that this is a technical device and the reader is advised not to give this particular aspect of the proof too much attention.
\item (Rough linearity) $\chi_h$ varies roughly linearly in $h$;  this is another fairly straightforward modification of the arguments of Gowers, already employed in \cite{u4-inverse}, which is performed in \S \ref{cs-sec}.
\item (Furstenberg-Weiss) This proceeds along similar lines to the corresponding argument in \cite{u4-inverse} but is, in a sense, rather easier once one has developed the device of $\C^D$-valued nilsequences, which allow one to remain in the smooth category; this is accomplised in \S \ref{linear-sec}, after a substantial amount of preparatory material in \S \ref{freq-sec}, \S \ref{reg-sec} and Appendix \ref{equiapp}.
\item (Linearisation) This is also quite similar to the corresponding argument in \cite{u4-inverse}, and is performed in \S \ref{linear-sec}. In both of parts (iv) and (v), the ``bracket calculus'' from \cite{u4-inverse} is replaced by the more conceptual ``symbol calculus'' developed in Appendix \ref{basic-sec}. 
\item (Additive Combinatorics) The additive combinatorial input is much the same as in \cite{u4-inverse}. For the convenience of the reader we sketch it in Appendix \ref{app-f}.
\item (Construction of a nilsequence) Our argument differs quite substantially from that in \cite{u4-inverse} at this point. The $s$-step nilobject, which is now a two-variable object $\chi(h,n)$, is constructed \emph{before} the symmetry argument and in a more conceptual manner. This may be compared with the rather \emph{ad hoc} approach taken in \cite{green-tao-u3inverse, u4-inverse}, where various bracket polynomials were merely exhibited as arising from nilsequences.  We perform this construction in \S \ref{multi-sec}.
\item (Symmetry argument) We replace  $\chi(h,n)$ with an equivalent nilcharacter $\tilde \chi(h,n,\ldots,n)$ where $\tilde \chi$ is a nilcharacter in $s$ variables, that is symmetric in the last $s-1$ variables.  The symmetry argument given in \S \ref{symsec} shows that $\tilde \chi(h,n,\ldots,n)$ is equivalent to $\tilde \chi(n,h,\ldots,n)$. Again the key idea in the analysis is to look at the second order terms in \eqref{linear-eq}. 
\item (Integration) With the symmetry in hand, we can use the calculus of multilinear nilcharacters essentially express $\tilde \chi(h,n,\ldots,n)$ as the derivative of an expression which is roughly of the form $\tilde \chi(n,\ldots,n)/s$; see \S \ref{symsec} for details.
\item The final step of the argument is relatively straightforward, as before; see \S \ref{overview-sec}.
\end{enumerate}

In our previous paper \cite{u4-inverse} it was already rather painful to keep proper track of such notions as ``many'' and ``correlates with''. Here matters are even worse, and so to organise the above tasks it turns out to be quite convenient to first take an ultralimit of all objects being studied, effectively placing one in the setting of \emph{nonstandard analysis}.  This allows one to easily import results from infinitary mathematics, notably the theory of Lie groups and basic linear algebra, into the finitary setting of functions on $[N]$.  In \S \ref{nsa-sec} and Appendix \ref{nsa-app} we review the basic machinery of ultralimits that we will need here; we will not be exploiting any particularly advanced aspects of this framework.  The reader does not really need to understand the ultrafilter language in order to comprehend the basic structure of the paper, provided that he/she is happy to deal with concepts like ``dense'' and ``correlates with'' in a somewhat informal way, resembling the way in which analysts actually talk about ideas with one another (and, in fact, analogous to the way we wrote this paper). It is possible to go through the paper and properly quantify all of these notions using appropriate parameters $\delta$ and (many) growth functions $\mathcal{F}$. This would have the advantage of making the paper on some level comprehensible to the reader with an absolute distrust of ultrafilters, and it would also remove the dependence on the axiom of choice and in principle provide explicit but very poor bounds. However it would cause the argument to be significantly longer, and the notation would be much bulkier.

Our exposition will be as follows. We will begin by spending some time introducing the ultrafilter language and then, motivated by examples, the notions of nilsequence, nilcharacter and symbol. Once that is done we will, in \S \ref{overview-sec}, give the high-level argument for Theorem \ref{mainthm}; this consist of detailing points (i), (ii) and (x) of the outline above and giving proper statements of the other main points. 

The discussion above concerning points (iii), (iv), (v) and (vi) has been simplified for the sake of exposition. In actual fact, these points are dealt with together by a kind of iterative loop, in which more and more bracket-linear structure is placed on the nilcharacters $\chi_h(n)$ by cycling from (iii) to (vi) repeatedly. 

We remark that a quite different approach using ultrafilters to the structural theory of the Gowers norms is in the process of being carried out in \cite{szeg-1,szeg-2,szeg-3}; this seems related to the work of Host and Kra, whereas our work ultimately derives from the work of Gowers.

We also make the minor remark that our proof of $\GI(s)$ is restricted to the case $s \geq 3$ case for minor technical reasons. In particular, we take advantage of the non-trivial nature of the degree $s-2$ ``lower order terms'' in the Gowers Cauchy-Schwarz argument (Proposition \ref{gcs-prop}) in the symmetry argument step; and we will also observe that the various ``smooth'' and ``periodic'' error terms arising from the equidistribution theory in Appendix \ref{equiapp} are of degree $1$ and thus negligible compared with the main terms in the analysis, which are of degree $s-1$.  The arguments can be modified to give a proof of $\GI(2)$, although this proof would basically be a notationally intensive repackaging of the arguments in \cite{green-tao-u3inverse}.

\emph{Acknowledgements.} BG was, for some of the period during which this work was carried out, a fellow of the Radcliffe Institute at Harvard. He is very grateful to the Radcliffe Institute for providing excellent working conditions. TT is supported by NSF Research Award DMS-0649473, the NSF Waterman award and a grant from the MacArthur Foundation. TZ is supported by ISF grant  557/08, an Alon fellowship and a Landau fellowship of the Taub foundation. All three authors are very grateful to the University of Verona for allowing them to use classrooms at Canazei during a week in July 2009. This work was largely completed during that week.

\section{Basic notation}\label{notation-sec}

We write $\N := \{0,1,2,\ldots\}$ for the natural numbers, and $\N^+ := \{1,2,\ldots\}$ for the positive natural numbers.  Given two integers $N,M$, we write $[N,M]$ for the discrete interval $[N,M] := \{ n: N \leq n \leq M\}$. We also make the  abbreviations $[N] :=[1,N]$, and , and $[[N]]:=[-N,N]$.  If $x$ is a real number, we write $x \mod 1$ for the associated residue class in the unit circle $\T := \R/\Z$, and write $x=y \mod 1$ if $x$ and $y$ differ by an integer.

We will rely frequently on the following two elementary functions: the \emph{fundamental character} $e: \R \to \C$ (or $e: \T \to \C$) defined by
$$ e(x) := e^{2\pi i x},$$
and the \emph{signed fractional part function}\footnote{The signed fractional part will be slightly more convenient to work with than the unsigned fractional part, as it is equal to the identity near the origin.} $\{\}: \R \to I_0$, where $I_0$ is the \emph{fundamental domain}
$$ I_0 := \{ x \in \R: -1/2 < x \leq 1/2\}$$
and $\{x\}$ is the unique real number in $I_0$ such that $x = \{x\} \mod 1$.  We will often rely on the identity
$$ e(x) = e(\{x\}) = e( x \mod 1 )$$
without further comment.

For technical reasons, we will need to manipulate vector-valued complex quantities in a manner analogous to scalar complex quantities.
If $v = (v_i)_{i=1}^D$ and $w = (w_i)_{i = 1}^{D'}$ are vectors in $\C^D$ and $\C^{D'}$ respectively then we form the \emph{tensor product} $v \otimes w \in \C^{DD'}$ by the formula
\[ v \otimes w := (v_1 w_1,\dots, v_{D} w_{D'})\]
and the \emph{complex conjugate} $\overline{v}\in \C^D$ by the formula
\[ \overline{v} := (\overline{v_1},\dots,\overline{v_D}).\] 
Similarly, if $X$ is some set and $f : X \rightarrow \C^D$ and $g : X \rightarrow \C^{D'}$ are functions then we write $f \otimes g: X \rightarrow \C^{DD'}$ for the function defined by $(f\otimes g)(x) := f(x) \otimes g(x)$, and similarly define $\overline{f}: X \to \C^D$.

If $G = (G,+)$ is an additive group, $k \in \N$, $\vec g = (g_1,\ldots,g_k) \in G^k$, and $\vec a = (a_1,\ldots,a_k) \in \Z^k$, we define the dot product
$$ \vec a \cdot \vec g := a_1 g_1 + \ldots + a_k g_k.$$

Given a set $H$ in an additive group, define an \emph{additive quadruple} in $H$ to be a quadruple $(h_1,h_2,h_3,h_4) \in H$ with $h_1+h_2=h_3+h_4$.  The number of additive quadruples in $H$ is known as the \emph{additive energy} of $H$ and is denoted $E(H)$.  

A map $\phi: H \to G$ from $H$ to another additive group $G$ is said to be a \emph{Freiman homomorphism} if it preserves additive quadruples, i.e. if $\phi(h_1)+\phi(h_2)=\phi(h_3)+\phi(h_4)$ for all additive quadruples $(h_1,h_2,h_3,h_4)$ in $H$.

Given a multi-index $\vec d = (d_1,\ldots,d_k) \in \N^k$, we write $|\vec d| := d_1+\ldots+d_k$.

We now briefly review and clarify some standard notation from group theory.   

When we do not assume a group $G$ to be abelian, we will always write $G$ multiplicatively: $G = (G,\cdot)$.  However, when dealing with abelian groups, we reserve the right to use additive notation instead. 

%Thus, for instance, a general Lie group $G$ will be written multiplicatively, but we will often write abelian Lie groups additively, e.g. $\R^k = (\R^k,+)$.  Of course, any notation or concept that was defined for multiplicative groups can be trivially specialised to additive groups, and we will do so freely in this paper.  For example, we will view the notion of a polynomial map $\phi: H \to G$ from an additive group $H = (H,+)$ to a multiplicative group $G = (G,\cdot)$ as a special case of the notion of a polynomial map between two multiplicative groups, and also as a generalisation of the notion of a polynomial map between two additive groups. This perspective of viewing additive structure as a special case of multiplicative structure may be somewhat confusing at first, but it will save us having to tediously duplicate a lot of notation, and it should always be clear from context whether a given group is using multiplicative or additive notation.

We view an $n$-tuple $(a_1,\ldots,a_n)$ of labels as a finite ordered set with the ordering $a_1 < \ldots < a_n$.  If $A = (a_1,\ldots,a_n)$ is a finite ordered set and $(g_a)_{a \in A}$ are a collection of group elements in a multiplicative group $G$, we define the ordered products
$$ \prod_{a \in A} g_a := g_{a_1} \ldots g_{a_n}, \; \; \prod_{i=1}^n g_i := g_1 \ldots g_n \; \;  \mbox{and} \; \; \prod_{i=n}^1 g_i := g_n \ldots g_1$$
for any $n \geq 0$, with the convention that the empty product is the identity.  We extend this notation to infinite products under the assumption that all but finitely many of the factors are equal to the identity.

Given a subset $A$ of a group $G$, we let $\langle A \rangle$ denote the subgroup of $G$ generated by $A$.  Given a family $(H_i)_{i \in I}$ of subgroups of $G$, we write $\bigvee_{i \in I} H_i$ for the smallest subgroup of $G$ that contains all of the $H_i$.

Given two elements $g, h$ of a multiplicative group $G$, we define the \emph{commutator}
$$ [g,h] := g^{-1}h^{-1}gh.$$
We write $H \leq G$ to denote the statement that $H$ is a subgroup of $G$.  If $H, K \leq G$, we let $[H,K]$ be the subgroup generated by the commutators $[h,k]$ with $h \in H$ and $k \in K$, thus $[H,K] = \left\langle \{ [h,k]: h \in H, k \in K \} \right\rangle$. 

If $r \geq 1$ is an integer and $g_1,\ldots,g_r \in G$, we define an $(r-1)$-\emph{fold iterated commutator} of $g_1,\ldots,g_r$ inductively by declaring $g_1$ to be the only $0$-fold iterated commutator of $g_1$, and for $r>1$ defining an $(r-1)$-fold iterated commutator to be any expression of the form $[w,w']$, where $w$ and $w'$ are $(s-1)$-fold and $(s'-1)$-fold commutators of $g_{i_1},\ldots,g_{i_s}$ and $g_{i'_1},\ldots,g_{i'_{s'}}$ respectively, where $s, s' \geq 1$ are such that $s+s'=r$, and $\{i_1,\ldots,i_s\} \cup \{ i'_1,\ldots,i'_{s'} \} = \{1,\ldots,r\}$ is a partition of $\{1,\ldots,r\}$ into two classes.  Thus for instance $[[g_3,g_1],[g_2,g_4]]$ and $[g_2,[g_1,[g_3,g_4]]]$ are $3$-fold iterated commutators of $g_1,\ldots,g_4$.

The following lemma will be useful for computing commutator groups.

\begin{lemma}\label{normal}  Let $H = \langle A \rangle, K = \langle B \rangle$ be normal subgroups of a nilpotent group $G$ that are generated by sets $A \subset H$, $B \subset K$ respectively.  Then $[H,K]$ is normal, and is also the subgroup generated by the $i+j-1$-fold iterated commutators of $a_1,\ldots,a_i,b_1,\ldots,b_j$ with $a_1,\ldots,a_i \in A$, $b_1,\ldots,b_j \in B$ and $i,j \geq 1$.
\end{lemma}

\begin{proof}  The normality of $[H,K]$ is follows from the identity 
\[
g[H,K]g^{-1} = [gHg^{-1},gKg^{-1}].
\]
It is then clear that $[H,K]$ contains the group generated by the iterated commutators of elements in $A,B$ that involve at least one element from each. The converse follows inductively using the identities
\begin{equation}\label{com-ident}
[x,y]=[y,x]^{-1}, \; \; [xy,z]=[x,z][[x,z],y][y,z] \; \; \mbox{and} \; \;  [x,y^{-1}]=[y,x][[y,x],y^{-1}].
\end{equation}
This concludes the proof.
\end{proof}

As a corollary of the above lemma, we have the distributive law
$$ \left[ \bigvee_{i \in I} H_i, \bigvee_{j \in J} K_j \right] = \bigvee_{i \in I, j \in J} [H_i, K_j]$$
whenever $(H_i)_{i \in I}, (K_j)_{j \in J}$ are families of normal subgroups of a nilpotent group $G$.

If $H \lhd G$ is a normal subgroup of $G$, and $g \in G$, we use $g \mod H$ to denote the coset representative $gH$ of $g$ in $G/H$. For instance, $g = g' \mod H$ if $gH = g' H$.\vspace{11pt}

At various stages in the paper we will need the (discrete) \emph{Baker-Campbell-Hausdorff formula} in the following weak form:
\begin{equation}\label{bch}
g_1^{n_1} g_2^{n_2} = g_2^{n_2} g_1^{n_1} \prod_a g_a^{P_a(n_1,n_2)}
\end{equation}
for all $g_1,g_2$ in a nilpotent group $G$ and all integers $n_1,n_2$, where $g_a$ ranges over all iterated commutators of $g_1, g_2$ that involve at least one copy of each (note from nilpotency that there are only finitely many non-trivial $g_a$), with the $a$ ordered in some arbitrary fashion, and $P_a: \Z \times \Z \to \Z$ are polynomials.  Furthermore, if $g_a$ involves $d_1$ copies of $g_1$ and $d_2$ copies of $g_2$, then $P_a$ has degree at most $d_1$ in the $n_1$ variable and $d_2$ in the $n_2$ variable.

Let $G$ be a connected, simply connected, nilpotent Lie group (or \emph{nilpotent Lie group} for short).  Then we denote the Lie algebra of $G$ as $\log G$. As is well known (see e.g. \cite{bourbaki}), the exponential map $\exp: \log G \to G$ is a homeomorphism, inverted by the logarithm map $\log: G \to \log G$, and we can then define the exponentiation operation $g^t$ for any $g \in G$ and $t \in \R$ by the formula
$$ g^t := \exp( t \log g ).$$
There is a continuous version of the Baker-Campbell-Hausdorff formula:
\begin{equation}\label{bch-cont}
g_1^{t_1} g_2^{t_2} = g_2^{t_2} g_1^{t_1} \prod_a g_a^{P_a(t_1,t_2)}
\end{equation}
for all $t_1,t_2 \in \R$ and $g_1, g_2 \in G$, where $P_a$ are the polynomials occurring in \eqref{bch}.  We also observe the variant formulae
$$
(g_1 g_2)^{t} = g_1^t g_2^t \prod_a g_a^{Q_a(t)}$$
for some polynomials $Q_a$ and all $t \in \R$, $g_1, g_2 \in G$, and
$$
\exp( t_1 \log g_1 + t_2 \log g_2 ) = g_1^{t_1} g_2^{t_2} \prod_a g_a^{R_a(t_1,t_2)}$$
for some further polynomials $R_a$ and all $t_1, t_2 \in \R$, $g_1, g_2 \in G$.  We refer to all of these formul{\ae} collectively as \emph{the Baker-Campbell-Hausdorff formula}.

If $A$ is a subset of a nilpotent Lie group $G$, we let $\langle A \rangle_\R$ be the smallest connected Lie subgroup of $G$ containing $A$, or more explicitly 
$$ \langle A \rangle_\R := \langle \{ a^t: a \in A; t \in \R \} \rangle.$$
Equivalently, $\log \langle A \rangle_\R$ is the Lie algebra generated by $\log A$.

A \emph{lattice} of a nilpotent Lie group $G$ is a discrete cocompact subgroup $\Gamma$ of $G$.  Thus for instance, we see from \eqref{bch} that for any finite set $A$ in $G$, $\langle A \rangle$ will be a cocompact subgroup of $\langle A \rangle_\R$, and will thus be a lattice if $\langle A \rangle$ is discrete.

A connected Lie subgroup $H$ of $G$ is said to be \emph{rational} with respect to $\Gamma$ if $\Gamma \cap H$ is cocompact in $H$.  For instance, if $G = \R^2$, $\Gamma$ is the standard lattice $\Z^2$, and $\alpha \in \R$, then the connected Lie subgroup $H := \{ (x,\alpha x): x \in \R \}$ is rational if and only if $\alpha$ is rational.\vspace{11pt}

\textsc{Further notation.} Here is a list of further notation used in the paper for reference, together with the place in the paper where each piece is defined and discussed.

\noindent\begin{tabular}{lll}
$\poly(H_\N \to G_\N)$ & polynomial maps from one filtered group $H_\N$ to $G_\N$ & \ref{poly-map-def}\\
$\poly(\Z_\N \to G_\N)$ & polynomial maps with the degree filtration &  \ref{poly-map-def} \\
$\poly(\Z^k_{\N^k} \to G_{\N^k})$ & polynomial maps with the multidegree filtration &  \ref{poly-map-def}  \\
$\poly(\Z_{\DR} \to G_{\DR})$ & polynomial maps with the degree-rank filtration &  \ref{poly-map-def} \\
$L^\infty(\Omega \to \overline{\C}^D)$ & bounded limit functions to $\ultra \C^d$&  \eqref{sigma-bounded}\\ 
$L^\infty(\Omega \to \overline{\C}^w)$ & bounded limit functions (also $L^{\infty}(\Omega)$) &  \eqref{sigma-bounded}\\ 
$\Lip(\ultra(G/\Gamma) \to \overline{\C}^D)$ & bd'd limit functions with bounded Lipschitz constant & \ref{lip-def} \\
$\Nil^{d}([N])$ & nilsequences of degree $\le d$ on $[N]$ & \ref{nilseq} \\
$\Nil^{\subset J}(\Omega)$ & nilsequences of degree $\subset J$ & \ref{nilch-def-gen} \\
$\Xi^d([N])$ & space of degree $d$ nilcharacters on $[N]$ & \ref{nilch-def} \\
$\Xi^{(d_1,\ldots,d_k)}_\MD(\Omega)$ & multidegree nilcharacters & \ref{nilch-def-gen} \\
$\Xi^{(d,r)}_\DR(\Omega)$ & degree-rank nilcharacters & \ref{nilch-def-gen} \\
$\Symb^d([N])$ & equiv. classes of degree $d$  nicharacters in $\Xi^d([N])$  & \ref{symbol-def} \\
$\Symb^{(d_1,\ldots,d_k)}_{\MD}(\Omega)$ & equiv. classes of multidegree nicharacters & \ref{equiv-def} \\
$\Symb^{(d,r)}_\DR(\Omega)$ & equiv. classes of degree-rank nicharacters & \ref{equiv-def} \\
$G^{\vec D},G^{\vec D, \leq (s-1,r_*)}$ & universal nilpotent Lie group of degree-rank $(s-1,r_*)$ & \ref{universal-nil}\\
$\Horiz_i(G)$ & $i$'th horizontal space of $G$  & \ref{horton} \\
$\Taylor_i(g)$ & $i'$th horizontal Taylor coefficient of a polynomial map & \ref{horton} \\
$(\vec D, \eta, \F)$ & total frequency representation of a nilcharacter & \ref{representation-def}
\end{tabular}

\section{The polynomial formulation of $\GI(s)$}\label{polysec}

The inverse conjecture $\GI(s)$, Conjecture \ref{gis-conj}, has been formulated using \emph{linear} nilsequences $F(g^n x\Gamma)$. This is largely for compatibility with the earlier paper \cite{green-tao-linearprimes} of the first two authors on linear equations in primes, where this form of the conjecture was stated in precisely this form as Conjecture 8.3. Subsequently, however, it was discovered that it is more natural to deal with a somewhat more general class of object called a \emph{polynomial nilsequence} $F(g(n)\Gamma)$. This is particularly so when it comes to discussing the distributional properties of nilsequences, as was done in \cite{green-tao-nilratner}.  Thus, we shall now recast the inverse conjecture in terms of polynomial nilsequences, which is the formulation we will work with throughout the rest of the paper.

Let us first recall the definition of a polynomial nilsequence of degree $d$.

\begin{definition}[Polynomial nilsequence]
Let $G$ be a (connected, simply-connected) nilpotent Lie group. By a \emph{filtration} $G_\N = (G_i)_{i \in \N}$ of degree $\leq d$ we mean a nested sequence $G \supseteq G_{0} \supseteq G_{1} \supseteq G_{2} \supseteq \dots \supseteq G_{d+1} = \{\id\}$ with the property that $[G_{i}, G_{j}] \subseteq G_{i+j}$ for all $i, j \geq 0$, adopting the convention that $G_{i}=\{\id\}$ for all $i>d$. By a \emph{polynomial sequence} adapted to $G_\N$ we mean a map $g : \Z \rightarrow G$ such that $\partial_{h_i} \dots \partial_{h_1} g \in G_i$ for all $i \geq 0$ and $h_1,\dots, h_i \in \Z$, where $\partial_h g(n) := g(n+h) g(n)^{-1}$. Write $\poly(\Z_\N \to G_{\N})$ for the collection of all such polynomial sequences.

Let $\Gamma \leq G$ be a lattice in $G$ (i.e. a discrete and cocompact subgroup), so that the quotient $G/\Gamma$ is a nilmanifold, and assume that each of the $G_i$ are \emph{rational} subgroups (i.e. $\Gamma_i := \Gamma \cap G_i$ is a cocompact subgroup of $G_i$). We refer to the pair $G/\Gamma = (G/\Gamma,G_\N)$ as a \emph{filtered nilmanifold}.  A \emph{polynomial orbit} $\orbit: \Z \to G/\Gamma$ is a sequence of the form $\orbit(n) := g(n) \Gamma$, where $g \in \poly(\Z_\N \to G_\N)$; we let $\poly(\Z_\N \to (G/\Gamma)_\N)$ denote the space of all such polynomial orbits.
If $F : G/\Gamma \rightarrow \C$ is a $1$-bounded, Lipschitz function then the sequence $F \circ \orbit = (F(g(n)\Gamma))_{n \in \Z}$ is called a \emph{polynomial nilsequence} of degree $d$.  
\end{definition}

The subscripts $\N$ will become more relevant later in this paper, when we start filtering nilpotent groups and nilmanifolds by other index sets $I$ than the natural numbers $\N$.  Note that we do not require $G_0$ or $G_1$ to equal $G$; this freedom will be convenient for some minor technical reasons, although ultimately it will not enlarge the space of polynomial nilsequences.

Let us give the basic examples of nilsequences and polynomials:

\begin{example}[Linear nilsequences are polynomial nilsequences]\label{polylin} Let $G$ be a $d$-step nilpotent Lie group, and let $\Gamma$ be a lattice of $G$.  Then, as is well known (see e.g. \cite{bourbaki}), the \emph{lower central series filtration} defined by $G_{0} = G_1 := G$, $G_{2} := [G, G_{1}]$, $G_{3} := [G, G_{2}], \dots, G_{d+1} := [G, G_{d}] = \{\id\}$ is a filtration on $G$.  Using the Baker-Campbell-Hausdorff formula \eqref{bch-cont} it is not difficult to show that the lower central series filtration is rational with respect to $\Gamma$, so the nilmanifold $G/\Gamma$ becomes a filtered nilmanifold.  If $g(n) := g_1^n g_0$ for some $g_0, g_1 \in G$, then $\partial_{h_1} g(n) = g_1^{h_1}$ and $\partial_{h_i} \dots \partial_{h_1} g(n) = \id$ for $i \geq 2$: therefore $g$ is a polynomial sequence, and so every linear orbit $n \mapsto g^n x$ with $g \in G$ and $x \in G/\Gamma$ is a polynomial orbit also.  As a consequence we see that every $d$-step linear nilsequence $n \mapsto F(g^n x)$ is automatically a polynomial nilsequence of degree $\leq d$.
\end{example}

\begin{example}[Polynomial phases are polynomial nilsequences]\label{polyphase}  Let $d \geq 0$ be an integer. Then we can give the unit circle $\T$ the structure of a degree $\leq d$ filtered nilmanifold by setting $G := \R$ and $\Gamma := \Z$, with $G_i := \R$ for $i \leq d$ and $G_i := \{0\}$ for $i>d$.  This is clearly a filtered nilmanifold.  If $\alpha_0,\ldots,\alpha_d$ are real numbers, then the polynomial $P(n) := \alpha_0 + \ldots + \alpha_d n^d$ is then
polynomial with respect to this filtration, with $n \mapsto P(n) \mod 1$ being a polynomial orbit in $\T$.  Thus, for any Lipschitz function $F: \T \to \C$, the sequence $n \mapsto F(P(n))$ is a polynomial nilsequence of degree $\leq d$; in particular, the polynomial phase $n \mapsto e(P(n))$ is a polynomial nilsequence.
\end{example}

\begin{example}[Combinations of monomials are polynomials]\label{lazard-ex}  By Corollary \ref{laz}, we see that if $G = (G,(G_i)_{i \in\N})$ is a filtered group of degree $\leq d$, then any sequence of the form
$$ n \mapsto \prod_{j=1}^k g_j^{P_j(n)},$$
in which $g_j \in G_{d_j}$ for some $d_j \in \N$, and $P_j: \Z \to \R$ is a polynomial of degree $\leq d_j$, will be a polynomial map.  Thus for instance
$$ n \mapsto g_d^{\binom{n}{d}} \ldots g_2^{\binom{n}{2}} g_1^n g_0$$
is a polynomial map whenever $g_j \in G_j$ for $j=0,\ldots,d$. In fact, all polynomial maps can be expressed in such a fashion via a \emph{Taylor expansion}; see Lemma \ref{taylo}.
\end{example}

We will give several further examples and properties of polynomial maps and polynomial nilsequences in \S \ref{nilcharacters}. 

As a consequence of Example \ref{polylin}, the following variant of the inverse conjecture $\GI(s)$ is ostensibly weaker than that stated in the introduction.

\begin{conjecture}[$\GI(s)$, polynomial formulation]\label{gis-poly} Let $s \geq 0$ be an integer, and let $0 < \delta \leq 1$.   Then there exists a finite collection ${\mathcal M}_{s,\delta}$ of filtered nilmanifolds $G/\Gamma = (G/\Gamma,G_\N)$, each equipped with some smooth Riemannian metric $d_{G/\Gamma}$ as well as constants $C(s,\delta), c(s,\delta) > 0$ with the following property. Whenever $N \geq 1$ and $f : [N] \rightarrow \C$ is a function bounded in magnitude by $1$ such that $\Vert f \Vert_{U^{s+1}[N]} \geq \delta$, there exists a filtered nilmanifold $G/\Gamma \in {\mathcal M}_{s,\delta}$, some $g \in \poly(\Z_\N \to G_{\N})$ and a function $F: G/\Gamma \to \C$ bounded in magnitude by $1$ and with Lipschitz constant at most $C(s,\delta)$ with respect to the metric $d_{G/\Gamma}$ such that
$$ |\E_{n \in [N]} f(n) \overline{F(g(n)\Gamma)}| \geq c(s,\delta).$$
\end{conjecture}

It turns out that this conjecture is actually \emph{equivalent} to Conjecture \ref{gis-conj}; we shall prove this equivalence in Appendix \ref{lift-app}. We remark that, though it might seem odd to put a non-trivial part of the proof of our main theorem in an appendix, we would rather encourage the reader to regard the proof of Conjecture \ref{gis-poly} as our main theorem. The rationale behind this is that everything that is done with linear nilsequences $F(g^nx \Gamma)$ in \cite{green-tao-linearprimes} could have been done equally well, and perhaps more naturally, with polynomial nilsequences $F(g(n)\Gamma)$. Further remarks along these lines were made in the introduction to our earlier paper \cite{u4-inverse}, where the polynomial formulation was emphasised from the outset. Here, however, we have felt a sense of duty to formally complete the programme outlined in \cite{green-tao-linearprimes}.

Henceforth we shall refer simply to a \emph{nilsequence}, rather than a polynomial nilsequence.

In \S \ref{nilcharacters} we will need to generalise the notion of a (polynomial) nilsequence by allowing more exotic filtrations $G_I$ on the group $G$, indexed by more complicated index sets $I$ than the natural numbers $\N$. In particular, we shall introduce the \emph{multidegree filtration}, which allows us to define nilsequences of several variables, as well as the \emph{degree-rank} filtration which provides a finer classification of polynomial sequences than merely the degree. We will discuss these using examples, and then develop a more unified theory that contains all three.

\section{Taking ultralimits}\label{nsa-sec}

The inverse conjecture, Conjecture \ref{gis-poly}, is a purely finitary statement, involving functions on a finite set $[N] = \{1,\ldots,N\}$ of integers.  As such, it is natural to look for proofs of this conjecture which are also purely finitary, and much of the previous literature on these types of problems is indeed of this nature.

However there is a very notable exception, namely the portion of the literature that exploits the \emph{Furstenberg correspondence principle} between combinatorial problems and ergodic theory. See \cite{furstenberg} for the original application to Szemer\'edi's theorem, or \cite{tao-ziegler} for a more recent application to Gowers norms over finite fields.  Here we use a somewhat different type of limit object, namely an \emph{ultralimit}. We are certainly not the first to employ ultralimits (a.k.a. \emph{nonstandard analysis}) in additive number theory; see for example \cite{jin}.  

The ultralimit formalism allows us to convert a ``finitary'' or ``standard'' statement such as Conjecture \ref{gis-poly} into an equivalent statement concerning \emph{limit objects}, constructed as ultralimits of standard objects.  This procedure is closely related to the use of the \emph{transfer principle} in nonstandard analysis, but we have elected to eschew the language of nonstandard analysis in order to reduce confusion, instead focusing on the machinery of ultralimits.

Here is a brief and somewhat vague list of the advantages of using the ultralimit approach. 
\begin{itemize}
\item Pigeonholing arguments are straightforward (due to the fact that a limit function taking finitely many values is constant);
\item Book-keeping of constants: one can talk rigorously about such concepts as ``bounded'' functions without a need to quantify the bounds;
\item One may make rigorous sense of such statements as ``the function $f: [N] \to \C$ and the function $g: [N] \to \C$ are equivalent modulo degree $s$ nilsequences''.  
\item In the infinitary context one may easily perform \emph{rank reduction} arguments in which one seeks to find the ``minimal bounded-complexity'' representation of a given system.
\end{itemize}

There are also some drawbacks of the approach:

\begin{itemize}
\item It becomes quite difficult to extract any quantitative bounds from our results, in particular we do not give explicit bounds on the constant $c(s,\delta)$ or on the complexity of the nilsequence in Conjecture \ref{gis-conj} or Conjecture \ref{gis-poly}.  It is in principle possible to expand the ultralimit proof into a standard proof, but the bounds are quite poor (of Ackermann type) due to the repeated use of ``rank reduction arguments'' and other highly iterative schemes that arise in the conversion of ultralimit arguments to standard ones.  For further discussion of the relation of ultralimit analysis to finitary analysis see \cite[\S 1.3, \S 1.5]{structure}.
\item The language of ultrafilters adds one more layer of notational complexity to an already notationally-intensive paper; however, there are gains to be made elsewhere, most notably in eliminating many quantitative constants (e.g. $\eps$, $N$) and growth functions (e.g. ${\mathcal F}$).
\end{itemize}

\textsc{Limit formulation of $\GI(s)$.} The basic notation and theory of ultralimits are reviewed in Appendix \ref{nsa-app}.  We now use this formalism to convert the inverse conjecture, $\GI(s)$, into an equivalent statement formulated in the framework of ultralimits.  We first consider a limit version of the concept of a Lipschitz function on a nilmanifold.  For technical reasons we will need to consider vector-valued functions, taking values in $\C^D$ or $\overline \C^D$ rather than $\C$ or $\overline\C$.  

\begin{definition}[Lipschitz functions]\label{lip-def}  Let $G/\Gamma$ be a standard nilmanifold, and let $D \in \N^+$ be standard.  
\begin{itemize}
\item We let $\Lip(G/\Gamma \to \C^D)$ be the space of standard Lipschitz functions $F: G/\Gamma \to \C^D$.  (Here we endow the compact manifold $G/\Gamma$ with a smooth metric in an arbitrary fashion; the exact choice of metric is not relevant.)
\item We let $\Lip(\ultra(G/\Gamma) \to \overline{\C}^D)$ be the space of bounded limit functions $F: \ultra(G/\Gamma) \to \overline{\C}^D$ whose Lipschitz constant is bounded (or equivalently, $F$ is an ultralimit of uniformly bounded functions $F_\n: G/\Gamma \to \C^D$ with uniformly bounded Lipschitz constant).
\item We let $\Lip(\ultra(G/\Gamma) \to \overline{S^{2D-1}})$ be the functions in $\Lip(\ultra(G/\Gamma) \to \overline{\C}^D)$ that take values in the (limit) complex sphere
$$  \overline{S^{2D-1}} := \{ z \in \overline{\C}^D: |z| = 1\}.$$
\item We write \[ \Lip(\ultra(G/\Gamma) \to \overline{\C}^\omega) := \bigcup_{D \in \N^+} \Lip(\ultra(G/\Gamma) \to \overline{\C}^D)\] and \[\Lip(\ultra(G/\Gamma) \to \overline{S^\omega}) := \bigcup_{D \in \N^+} \Lip(\ultra(G/\Gamma) \to \overline{S^{2D-1}}).\]
\end{itemize}
We will often abbreviate these spaces as $\Lip(G/\Gamma)$ or $\Lip(\ultra(G/\Gamma))$ when the range of the functions involved is not relevant to the discussion.
\end{definition}

\emph{Remark.} As $G/\Gamma$ is compact, we see from the Arzel\`a-Ascoli theorem that $\Lip(G/\Gamma \to \C^D)$ is locally compact in the $L^\infty(G/\Gamma \to \C^D)$ topology.  As a consequence, if we embed $\Lip(G/\Gamma \to \C^D)$ into $\Lip(\ultra(G/\Gamma) \to \overline{\C}^D)$ in the obvious manner, then the former is a dense subspace of the latter in the (standard) uniform topology, in the sense that for every $F \in \Lip(\ultra(G/\Gamma) \to \overline{\C}^D)$ and every standard $\eps > 0$ there exists $F' \in \Lip(G/\Gamma \to \C^D)$ such that $|F(x)-F'(x)| \leq \eps$ for all $x \in \ultra(G/\Gamma)$.

\emph{Remark.} Observe that the spaces $\Lip(\ultra(G/\Gamma) \to \overline{\C}^D)$ and $\Lip(\ultra(G/\Gamma) \to \overline{\C}^\omega)$ are vector spaces over $\overline{\C}$.  The spaces $\Lip(\ultra(G/\Gamma) \to \overline{\C}^\omega)$ and $\Lip(\ultra(G/\Gamma) \to \overline{S^\omega})$ are also closed under tensor product (as defined in \S \ref{notation-sec}).  All the spaces defined in Definition \ref{lip-def} are closed under complex conjugation.

Using the above notion, we can define the limit version of a (polynomial) nilsequence.

\begin{definition}[Nilsequence]\label{nilseq}  Let $s \geq 0$ be standard.  A \emph{nilsequence} of degree $\leq s$ is any limit function $\psi: \ultra \Z \to \ultra \C$ of the form $\psi(n) := F(g(n) \Gamma)$, where $G/\Gamma = (G/\Gamma,G_\N)$ is a standard filtered nilmanifold of degree $\leq s$, $g: \ultra \Z \to \ultra G$ is a limit polynomial sequence (i.e. an ultralimit of polynomial sequences $g_\n: \Z \to G$), and $F \in \Lip(\ultra(G/\Gamma) \to \overline{\C})$.
\end{definition}

Given any limit subset $\Omega$ of $\ultra \Z$, we denote the space of degree $d$ nilsequences, restricted to $\Omega$, as $\Nil^{d}(\Omega) = \Nil^{d}(\Omega \to \overline{\C}^\omega)$; this is a subset of $L^\infty(\Omega \to \overline{\C}^\omega)$.  We write $\Nil^{d}(\Omega \to \overline{\C}^D)$ for the nilsequences that take values in $\overline{\C}^D$; this is a subspace (over $\overline{\C}$) of $L^\infty(\Omega \to \overline{\C}^D)$. We make the technical remark that $\Nil^{d}(\Omega)$ is a $\sigma$-limit set, since one can express this space as the union, over all standard $M$ and dimensions $D$, of the nilsequences taking values in $\overline{\C}^D$ arising from a nilmanifold of ``complexity'' $M$ and a Lipschitz function of constant at most $M$, where one defines the complexity of a nilmanifold in some suitable fashion.  In particular, the limit selection lemma in Corollary \ref{mes-select} can be applied to this set.

We also define the Gowers uniformity norm $\Vert f\Vert_{U^{s+1}[N]}$ of an ultralimit $f= \lim_{\n \to p} f_\n$ of standard functions $f_\n: [N_\n] \to \C$ in the usual limit fashion 
$$ \|f\|_{U^{s+1}[N]} := \lim_{\n \to p} \|f_\n\|_{U^{s+1}[N_\n]}.$$
If $f$ is vector-valued instead of scalar valued, say $f = (f_1,\ldots,f_d)$, then we define the uniformity norm by the formula
$$ \|f\|_{U^{s+1}[N]} := (\sum_{i=1}^d \|f_i\|_{U^{s+1}[N]}^{2^{s+1}})^{1/2^{s+1}}.$$
(The exponent $2^{s+1}$ is not important here, but has some very slight aesthetic advantages over other equivalent formulations of the vector-valued norm.)

The ultralimit formulation of $\GI(s)$ can then be given as follows:

\begin{conjecture}[Ultralimit formulation of $\GI(s)$]\label{gis-conj-nonst}  Let $s \geq 0$ be standard and $N \geq 1$ be a limit natural number.  Suppose that $f \in L^\infty([N] \to \overline{\C})$ is such that $\Vert f \Vert_{U^{s+1}[N]}$ $\gg 1$. Then $f$ correlates with a degree $\leq s$ nilsequence on $[N]$.  
\end{conjecture}

See Definition \ref{linfty} for the definition of \emph{correlation} in this context.

We now show why, for any fixed standard $s$, Conjecture \ref{gis-conj-nonst} is equivalent to its more traditional counterpart, Conjecture \ref{gis-poly}.

\begin{proof}[Proof of Conjecture \ref{gis-conj-nonst} assuming Conjecture \ref{gis-poly}]  Let $f$ be as in Conjecture \ref{gis-conj-nonst}.  We may normalise the bounded function $f$ to be bounded by $1$ in magnitude throughout.  By hypothesis, there exists a standard $\delta > 0$ such that $\Vert f \Vert_{U^{s+1}[N]} \geq \delta$.  Writing $N$ and $f$ as the ultralimits of $N_\n$, $f_\n$ respectively for some $f_\n: [N_\n] \to \C$ bounded in magnitude by $1$, and applying Conjecture \ref{gis-poly}, we conclude that for $\n$ sufficiently close to $p$, we have the correlation bound
$$ |\E_{n_\n \in [N_\n]} f_\n(n_\n) \overline{F_\n(g_\n(n_\n) \Gamma_\n)}| \geq c(s,\delta)> 0$$
where $G_\n/\Gamma_\n, g_\n, x_\n, F_\n$ are as in Conjecture \ref{gis-conj}.  Writing $G/\Gamma, g, x, F$ for the ultralimits of $G_\n/\Gamma_\n, g_\n, x_\n, F_\n$ respectively, we thus have
$$ |\E_{n \in [N]} f(n) \overline{F(g(n) \ultra \Gamma)}| \gg 1.$$
By the pigeonhole principle (cf. Appendix \ref{nsa-app}), we see that $G/\Gamma$ is a standard degree $\leq s$ nilmanifold, while $g: \ultra \Z \to \ultra G$ and $x \in G/\Gamma$ remain limit objects.  The limit function $F$ lies in $\Lip(\ultra(G/\Gamma) \to \overline{\C})$ by construction, and the claim follows.
\end{proof}

\emph{Proof of Conjecture \ref{gis-poly} assuming Conjecture \ref{gis-conj-nonst}.}  Observe (from the theory of Mal'cev bases \cite{malcev}) that there are only countably many degree $\leq s$ nilmanifolds $G/\Gamma$ up to isomorphism, which we may enumerate as $G_\n/\Gamma_\n$.  We endow each of these nilmanifolds arbitrarily with some smooth Riemannian metric $d_{G_\n/\Gamma_\n}$.

Suppose for contradiction that Conjecture \ref{gis-poly} failed.  
Carefully negating all the quantifiers, we may thus find a $\delta > 0$, a sequence $N_\n$ of standard integers, and a function $f_\n: [N_\n] \to \C$ bounded in magnitude by $1$ with 
$\|f_\n\|_{U^{s+1}[N]} \geq \delta$, such that 
\begin{equation}\label{george}
|\E_{n_\n \in [N_\n]} f_\n(n_\n) \overline{F(g(n_\n) \Gamma_{\n'}))}| \leq 1/\n
\end{equation}
whenever $\n' \leq \n$, $g \in \poly(\Z_\N \to (G_{\n'})_\N)$, and $F: G_{\n'}/\Gamma_{\n'} \to \C$ is bounded in magnitude by $1$ and has a Lipschitz constant of at most $\n$ with respect to $d_{G_\n/\Gamma_\n}$.

On the other hand, viewing $f$ as a bounded limit function, we can apply Conjecture \ref{gis-conj-nonst} and conclude that there exists a standard filtered nilmanifold $G/\Gamma$ with some smooth Riemannian metric $d_{G/\Gamma}$, a limit polynomial $g: \ultra \Z \to \ultra G$, and some ultralimit $F \in \Lip(\ultra(G/\Gamma) \to \overline{\C})$ of functions $F_\n: G/\Gamma \to \C$ with uniformly bounded Lipschitz norm, such that
$$ |\E_{n \in [N]} f(n) \overline{F(g(n) \ultra \Gamma)}| \geq \eps$$
for some standard $\eps > 0$.  

By construction, $G/\Gamma$ is isomorphic to $G_{\n_0}/\Gamma_{\n_0}$ for some $\n_0$, so we may assume without loss of generality that $G/\Gamma = G_{\n_0}/\Gamma_{\n_0}$; since all smooth Riemannian metrics on a compact manifold are equivalent, we can also assume that $d_{G/\Gamma} = d_{G_{\n_0}/\Gamma_{\n_0}}$.  We may also normalise $F$ to be bounded in magnitude by $1$.  But this contradicts \eqref{george} for $\n$ sufficiently large, and the claim follows.\endproof

Thus, to establish Theorem \ref{mainthm}, it will suffice to establish Conjecture \ref{gis-conj-nonst} for $s \geq 3$.  This is the objective of the remainder of the paper.

\emph{Remark.}  We transformed the finitary linear inverse conjecture, Conjecture \ref{gis-conj}, into a nonstandard polynomial formulation, Conjecture \ref{gis-conj-nonst}, via the finitary polynomial inverse conjecture, Conjecture \ref{gis-poly}.  One can also swap the order of these equivalences, transforming the finitary linear inverse conjecture into a nonstandard linear formulation by arguing as above, and then transforming the latter into a nonstandard polynomial formulation by using Proposition \ref{lift}.  Of course the two arguments are essentially equivalent.

Conjecture \ref{gis-conj-nonst} is trivial when $N$ is bounded, since every function in $L^\infty[N]$ is then a nilsequence of degree at most $s$.  For the remainder of the paper we shall thus adopt the convention that $N$ denotes a fixed \emph{unbounded} limit integer.

To conclude this section we reformulate Conjecture \ref{gis-poly} by introducing the important notion of \emph{bias}. 

\begin{definition}[Bias and correlation] Let $\Omega$ be a limit finite subset of $\Z$, and let $d \in \N$.
We say that $f, g \in L^\infty(\Omega \to \overline{\C}^\omega)$ \emph{$d$-correlate} if we have
$$|\E_{n \in \Omega} f(n) \otimes \overline{g(n)} \otimes \psi(n)| \gg 1$$ 
for some degree $d$ nilsequence $\psi \in \Nil^{d}(\Omega \to \overline{\C}^\omega)$.  We say that $f$ is \emph{$d$-biased} if $f$ $d$-correlates with the constant function $1$, and \emph{$d$-unbiased} otherwise.
\end{definition}

With this definition, Conjecture \ref{gis-conj-nonst} can be reformulated in the following manner.

\begin{conjecture}[Limit formulation of $\GI(s)$, II]\label{gis-conj-nonst-2}  Let $s \geq 0$ be standard.  Suppose that $f \in L^\infty([N] \to \overline{\C})$ is such that $\Vert f \Vert_{U^{s+1}[N]} \gg 1$. Then $f$ is $s$-biased.
\end{conjecture}

From previous literature, we see that Conjecture \ref{gis-conj-nonst-2} has already been proven for $s \leq 2$; we need to establish it for all $s \geq 3$.  We also make the basic remark that while the conjecture is only phrased for scalar-valued functions $f \in L^\infty([N] \to \overline \C)$, it automatically generalises to vector-valued functions $f \in L^\infty([N] \to \overline \C^\omega)$, since if a vector-valued function $f$ has large $U^{s+1}[N]$ norm, then so does one of its components.

Finally we remark that the converse implication is known.

\begin{proposition}[Converse $\GI(s)$, ultralimit formulation]\label{inv-nec-nonst}  Let $s \geq 0$ be standard. Suppose that $f \in L^\infty([N] \to \overline{\C})$ is $\leq s$-biased.  Then $\Vert f \Vert_{U^{s+1}[N]} \gg 1$.  
\end{proposition}

\begin{proof}  This follows from \cite[Proposition 12.6]{green-tao-u3inverse}, \cite[\S 11]{green-tao-linearprimes}, or \cite[Proposition 1.4]{u4-inverse}, transferred to the ultralimit setting in the usual fashion.
\end{proof}

\section{Nilcharacters and symbols in one and several variables}\label{nilcharacters}

Conjecture \ref{gis-conj-nonst} asserts that a function in $L^\infty([N] \to \overline{\C})$ on an unbounded interval $[N]$ correlates with a degree $\leq s$ nilsequence.  For inductive reasons, it is useful to observe that this conclusion implies a strengthened version if itself, in which $f$ correlates with a special type of degree $\leq s$ nilsequence, namely a degree $s$ \emph{nilcharacter}. A nilcharacter is a special type of nilsequence and should be thought of, very roughly speaking, as a generalisation of characters $e(\alpha n)$ in the degree $1$ setting, or objects such as $e(\alpha n \{\beta n\})$ in the degree $2$ setting; these were crucial in our paper on $\GI(3)$ \cite{u4-inverse}, although the notation there was slightly different in some minor ways.  See \cite{gtz-announce} for further informal discussion of nilcharacters.

In the $s=1$ case, a nilcharacter is essentially (ignoring constants) the same thing as a linear phase function $n \mapsto e(\xi n)$, and the frequency $\xi$ can be viewed as living in the Pontryagin dual of $\ultra \Z$ (or, in some sense, of $[N]$, even though the latter set is not quite a locally compact abelian group).  It will turn out that more generally, a degree $s$ nilcharacter will have a ``symbol'' (analogous to the frequency $\xi$) that takes values in a ``higher order Pontryagin dual'' $\Symb^s([N])$ of $[N]$; this symbol can be interpreted as the ``top order term'' of a nilcharacter, for instance the symbol of the degree $3$ nilcharacter $n \mapsto e(\alpha n^3 + \beta n^2 + \gamma n + \delta)$ is basically\footnote{This is an oversimplification; it would be more accurate to say that the symbol is given by $\alpha$ modulo $\ultra \Z + \Q + O(N^{-3})$.} $\alpha$.  This higher order dual obeys a number of pleasant algebraic properties, and the primary purpose of this section is to develop those properties.

There are various additional complications to be taken into account:

\begin{itemize}
\item We will require multidimensional generalisations of these concepts (think of the two-dimensional sequence $(n_1,n_2) \mapsto e(\alpha n_1 \{\beta n_2\})$) together with appropriate notions of \emph{multidegree} in order to make sense of ``top-order'' and ``lower-order terms'';
\item We will be dealing with $\C^D$-valued (or, rather, $S^{2D-1}$-valued) nilsequences rather than merely scalar ones. This is so that we may continue to work in the smooth category, as discussed in the introduction;
\item The language of ultrafilters will be used.
\end{itemize}

Our main focus here will be on the first of these points. The second is largely a technicality, whilst the third is actually helpful in that the notion of symbol (for example) is rather clean and does not require discussion of complexity bounds. \vspace{11pt}

\textsc{Motivation and one-dimensional definitions.} We now give the definitions of a (one-dimensional) nilcharacter and its symbol, and give a few examples.  However, we will hold off for now on actually proving too much about these concepts, because we will shortly need to generalise these notions to a more abstract setting in which one also allows multidimensional nilcharacters, and nilcharacters that are atuned not just to a specific degree, but also to a specific ``rank'' inside that degree.

\begin{definition}[Nilcharacter]\label{nilch-def}  Let $d \geq 0$ be a standard integer.  A \emph{nilcharacter} $\chi$ of degree $d$ on $[N]$ is a nilsequence $\chi(n) = F(\orbit(n)) = F(g(n) \ultra \Gamma)$ on $[N]$ of degree $\leq d$, where the function $F \in \Lip(\ultra(G/\Gamma) \to \overline{\C}^\omega)$ obeys two additional properties:
\begin{itemize}
\item $F \in \Lip(\ultra(G/\Gamma) \to \overline{S^{\omega}})$ (thus $|F|=1$ pointwise, and hence $|\chi|=1$ pointwise also); and
\item $F( g_d x ) = e( \eta(g_d) ) F(x)$ for all $x \in G/\Gamma$ and $g_d \in G_{d}$, where $\eta: G_{d} \to \R$ is a continuous standard homomorphism which maps $\Gamma_{d}$ to the integers (or equivalently, $\eta$ is an element of the Pontryagin dual of the torus $G_d/\Gamma_d$).  We call $\eta$ the \emph{vertical frequency} of $F$.
\end{itemize}
The space of all nilcharacters of degree $d$ on $[N]$ is denoted $\Xi^d([N])$.
\end{definition}

\begin{example}  When $d=1$, the only examples of nilcharacters are the linear phases $n \mapsto e( \alpha n + \beta )$ for $\alpha, \beta \in \ultra \R$.
\end{example}

\begin{example}  For any $\alpha_0,\ldots,\alpha_d \in \ultra \R$, the function $n \mapsto e(\alpha_0 + \ldots + \alpha_d n^d)$ is a nilcharacter of degree $\leq d$.  To see this, we set $G/\Gamma$ to be the unit circle $\T$ with the filtration $G_i := \R$ for $i \leq d$ and $G_i := \{0\}$ for $i>d$ (thus $G/\Gamma$ is of degree $d$), let $g(n) := \alpha_0 + \ldots + \alpha_d n^d$, and let $F(x) := e(x)$.  The vertical frequency $\eta: \R \to \R$ is then just the identity function.
\end{example}

Now we give an instructive \emph{near}-example of a nilcharacter.  Let $G$ be the free $2$-step nilpotent Lie group on two generators $e_1,e_2$, thus 
\begin{equation}\label{heisen}
G := \langle e_1,e_2\rangle_\R = \{ e_1^{t_1} e_2^{t_2} [e_1,e_2]^{t_{12}}: t_1,t_2,t_{12} \in \R\}
\end{equation}
with the element $[e_1,e_2]$ being central, but with no other relations between $e_1, e_2$ and $[e_1,e_2]$.  This is a degree $\leq 2$ nilpotent group if we set $G_0, G_1 := G$ and 
$$G_2 := \langle [e_1,e_2] \rangle_\R = \{ [e_1,e_2]^{t_{12}}: t_{12} \in \R \}.$$   
We let 
$$\Gamma := \langle e_1,e_2 \rangle = \{ e_1^{n_1} e_2^{n_2} [e_1,e_2]^{n_{12}}: n_1,n_2,n_{12} \in \Z\}$$
be the discrete subgroup of $G$ generated by $e_1,e_2$, then $G/\Gamma$ is a degree $\leq 2$ filtered nilmanifold, known as the \emph{Heisenberg nilmanifold}, and elements of $G/\Gamma$ can be uniquely expressed using the fundamental domain
$$ G/\Gamma = \{ e_1^{t_1} e_2^{t_2} [e_1,e_2]^{t_{12}} \Gamma:  t_1,t_2,t_{12} \in I_0 := (-1/2,1/2]\}.$$
If we then set $g: \ultra \Z \to \ultra G$ to be the limit polynomial sequence $g(n) := e_2^{\beta n} e_1^{\alpha n}$ for some fixed $\alpha,\beta \in \ultra \R$, and let $F: G/\Gamma \to \C$ be the function defined on the fundamental domain by the formula
\begin{equation}\label{fdef}
 F( e_1^{t_1} e_2^{t_2} [e_1,e_2]^{t_{12}} \Gamma ) := e( -t_{12} )
\end{equation}
for $t_1,t_2,t_{12} \in I_0$, then one easily computes that
$$ F( g(n) \ultra \Gamma ) = e( \{\alpha n\} \beta n )$$
where $\{\}: \R \to I_0$ is the signed fractional part function.  The function $n \mapsto e( \{\alpha n\} \beta n )$ is then \emph{almost} a nilcharacter of degree $2$, with vertical frequency given by the function $\eta: [e_1,e_2]^{t_{12}} \mapsto -t_{12}$.  All the properties required to give a nilcharacter in Definition \ref{nilch-def} are satisfied, save for one: the function $F$ is not Lipschitz on all of $G/\Gamma$, but is instead merely \emph{piecewise} Lipschitz, being discontinuous at some portions of the boundary of the fundamental domain.  To put it another way, one can view $n \mapsto e(\{ \alpha n \} \beta n)$ as a \emph{piecewise} nilcharacter of degree $2$.

Indeed, a topological obstruction prevents one from constructing \emph{any} scalar function $F \in \Lip(\ultra(G/\Gamma) \to \overline{S^1})$ of unit magnitude on the Heisenberg nilmanifold with the above vertical frequency.  By taking standard parts, we may assume that $F$ comes from a standard Lipschitz function $F: G/\Gamma \to S^1$ with the same vertical frequency.  For any standard $t \in [-1/2,1/2]$, consider the loop $\gamma_t := \{ e_1^t e_2^s \Gamma: s \in I_0\}$.  The image $F(\gamma_t)$ of this loop lives on the unit circle and thus has a well-defined winding number (or degree).  As this degree must vary continuously in $t$ while remaining an integer, it is constant in $t$; in particular, $F(\gamma_{-1/2})$ and $F(\gamma_{1/2})$ must have the same winding number.  On the other hand, from the Baker-Campbell-Hausdorff formula \eqref{bch} we see that 
$$F( e_1^{1/2} e_2^s \Gamma ) = F( e_1^{-1/2} e_2^s e_1 [e_1,e_2]^s \Gamma ) = e(s) F( e_1^{-1/2} e_2^s \Gamma )$$
and so the winding number of $F(\gamma_{1/2})$ is one larger than the winding number of $F(\gamma_{-1/2})$, a contradiction.

If however we allow ourselves to work with higher dimensions $D$, then this topological obstruction disappears.  Indeed, let us take a smooth partition of unity $1 = \sum_{k=1}^D \varphi_k^2(t,s)$ on $\T^2$, where $D \in \N^+$ and each $\varphi_k$ is supported in $B_k \mod \Z^2$, where $B_k$ is a ball of radius $1/100$ (say) in $\R^2$.  Then if we define $F := (F_1,F_2,\ldots,F_D)$, where 
\begin{equation}\label{fkts}
 F_k( e_1^t e_2^s [e_1,e_2]^u \ultra \Gamma) := \varphi_k(t,s) e(u) 
\end{equation}
whenever $(t,s) \in \ultra B_k$ and $u \in \ultra \R$, with $F_k = 0$ if no such representation of the above form exists, then one easily verifies that $F$ lies in $\Lip(\ultra(G/\Gamma) \to \overline{S^{2D-1}})$ with the vertical frequency $\eta$, and so the vector-valued sequence $\chi: n \mapsto F( g(n) \ultra \Gamma)$ is a nilcharacter of degree $2$.  A computation shows that each component $\chi_k$ of this nilcharacter $\chi = (\chi_1,\ldots,\chi_D)$ takes the form 
$$ \chi_k(n) = e( \{ \alpha n - \theta_k \} \beta n ) \psi_k(n)$$
for some offset $\theta_k \in \ultra \R$ and some degree $1$ nilsequence $\psi_k$.  Thus we see that $\chi$ is in some sense ``equivalent modulo lower order terms'' with the bracket polynomial phase $n \mapsto e( \{ \alpha n \} \beta n)$.  We refer to the vector-valued nilsequence $\chi$ as a \emph{vector-valued smoothing} of the piecewise nilsequence $n \mapsto e(\{\alpha n \} \beta n)$; we will informally refer to this smoothing operation several times in the sequel when discussing further examples of nilsequences that are associated with bracket polynomials.

Similar computations can be made in higher degree.  For instance, bracket cubic phases such as $n \mapsto e( \{ \{ \alpha n \} \beta n \} \gamma n )$ or $n \mapsto e( \{ \alpha n^2 \} \beta n )$ with $\alpha,\beta,\gamma \in \ultra \R$ can be viewed as near-examples of degree $3$ nilcharacters (with the problem again being that $F$ is discontinuous on the boundary of the fundamental domain), but there exist vector-valued smoothings of these phases which are genuine degree $3$ nilcharacters.  We will not detail these computations here, but they can essentially be found in \cite[Appendix E]{u4-inverse}.  More generally, one can view bracket polynomial phases of degree $d$ as near-examples of nilcharacters of degree $d$ that can be converted to genuine examples using vector-valued smoothings; this fact can be made precise using the machinery from \cite{leibman}, but we will not need this machinery here.

\emph{Remark.} The above topological obstruction is quite annoying; it is the sole reason that we are forced to work with vector-valued functions. There are two other approaches to avoid this topological obstruction that we know of.  One is to work with \emph{piecewise} Lipschitz functions rather than Lipschitz functions.  This allows one in particular to build (piecewise) nilcharacters out of \emph{bracket polynomials}.  This is the approach taken in \cite{u4-inverse}; however, it requires one to develop a certain amount of ``bracket calculus'' to manipulate these polynomials, and some additional arguments are also needed to deal with the discontinuities at the edges of the piecewise components of the nilmanifold.  Another approach is to work with randomly selected fundamental domains of the nilmanifold (cf. \cite{green-tao-longaps}) which eliminates topological obstructions, with the randomness being used to ``average out'' the effects of the boundary of the domain.  While all three methods will eventually work for the purposes of establishing the inverse conjecture, we believe that the vector-valued approach introduces the least amount of artificial technicality.

By definition, every nilcharacter of degree $d$ is a nilsequence of degree $\leq d$.  The converse is far from being true; however, one can approximate nilsequences of degree $\leq d$ as bounded linear combinations of nilcharacters of degree $d$.  More precisely, we have the following lemma.

\begin{lemma}  Let $\psi \in \Nil^d([N] \to \overline{\C})$ be a scalar nilsequence of degree $d$, and let $\eps > 0$ be standard.  Then one can approximate $\psi$ uniformly to error $\eps$ by a bounded linear combination \textup{(}over $\overline{\C}$\textup{)} of the components of nilcharacters in $\Xi^d([N])$.
\end{lemma}

\begin{proof}  Unpacking the definitions, it suffices to show that for every degree $d$ filtered nilmanifold $G/\Gamma$, every $F \in \Lip(\ultra(G/\Gamma) \to \overline{\C})$, and every standard $\eps>0$, one can approximate $F$ uniformly to error $\eps$ by a bounded linear combination of functions in the class ${\mathcal F}(G/\Gamma)$ of components of standard Lipschitz functions $F' \in \Lip( G/\Gamma \to S^\omega )$ that have a vertical frequency in the sense of Definition \ref{nilch-def}.

By taking standard parts, we may assume that $F$ is a standard Lipschitz function.  Observe that ${\mathcal F}(G/\Gamma)$ is closed under multiplication and complex conjugation.  By the Stone-Weierstrass theorem, it thus suffices to show that ${\mathcal F}(G/\Gamma)$ separates any two distinct points $x, y \in G/\Gamma$.   If $x, y$ do not lie in the same orbit of the $G_d$, then this is clear from a partition of unity (taking $\eta = 0$).  If instead $x = g_d y$ for some $g_d \in G_d$, then the distinctness of $x,y$ forces $g_d \not \in \Gamma_d$, and hence by Pontryagin duality there exists a vertical frequency $\eta$ with $\eta(g_d) \neq 0$.  If one then builds a nilcharacter with this frequency (by adapting the vector-valued smoothing construction \eqref{fkts}) we obtain the claim.
\end{proof}

We remark that this lemma can also be proven, with better quantitative bounds, by Fourier-analytic methods: see \cite[Lemma 3.7]{green-tao-nilratner}.  As a corollary of the lemma, we have the following.

\begin{corollary}\label{nilch-cor} Suppose that $f \in L^\infty([N] \to \overline{\C}^\omega)$.  Then $f$ is $d$-biased if and only if $f$ correlates with a nilcharacter $\chi \in \Xi^d([N])$.
\end{corollary}

It is easy to see that if $\chi, \chi'$ are two nilcharacters of degree $d$, then the tensor product $\chi \otimes \chi'$ and complex conjugate $\overline{\chi}$ are also nilcharacters.   If all nilcharacters were scalar, this would mean that the space $\Xi^d([N])$ of degree $d$ nilcharacters form a multiplicative abelian group.  Unfortunately, nilcharacters can be vector-valued, and so this statement is not quite true. However, it becomes true if one only focuses on the ``top order'' behaviour of a nilcharacter.  To isolate this behaviour, we adopt the following key definition.

\begin{definition}[Symbol]\label{symbol-def}  Let $d \geq 0$.  Two nilcharacters $\chi, \chi' \in \Xi^d([N])$ of degree $d$ are \emph{equivalent} if $\chi \otimes \overline{\chi'}$ is equal on $[N]$ to a nilsequence of degree $\leq d-1$.  This can be shown to be an equivalence relation (see Lemma \ref{equiv-lemma}); the equivalence class of a nilcharacter $\chi$ will be called the \emph{symbol} of $\chi$ and is denoted $[\chi]_{\Symb^d([N])}$.  The space of all such symbols will be denoted $\Symb^d([N])$; we will show later (see Lemma \ref{symbolic}) that this is an abelian multiplicative group.  
\end{definition}

When $d=1$, two nilcharacters $n \mapsto e(\alpha n + \beta)$ and $n \mapsto e( \alpha' n + \beta')$ are equivalent if and only if $\alpha-\alpha'$ is a limit integer, and $\Symb^1([N])$ is just $\ultra\T$ in this case.  However, the situation is more complicated in higher degree.
To get some feel for this, consider two polynomial phases
$$ \chi:  n \mapsto e(\alpha_0 + \ldots + \alpha_d n^d)$$
and
$$ \chi':  n \mapsto e(\alpha'_0 + \ldots + \alpha'_d n^d)$$
with $\alpha_0,\ldots,\alpha_d,\alpha'_0,\alpha'_d \in \ultra \R$, and consider the problem of determining when $\chi$ and $\chi'$ are equivalent nilcharacters of degree $d$.  Certainly this is the case if $\alpha_d$ and $\alpha'_d$ are equal, or differ by a limit integer.  When $d \geq 2$, there are two further important cases in which equivalence occurs.  The first is when $\alpha'_d = \alpha_d + O(N^{-d})$, because in this case the top degree component $e( (\alpha_d - \alpha'_d) n^d)$ of $\chi \overline{\chi'}$ can be viewed as a Lipschitz function of $n/2N \mod 1$ (say) on $[N]$ and is thus a $1$-step nilsequence.  The second is when $\alpha'_d = \alpha_d + a/q$ for some standard rational $q$, since in this case the top degree component
$e( (\alpha_d - \alpha'_d) n^d)$ of $\chi \overline{\chi'}$ is periodic with period $q$ and can thus be viewed as a Lipschitz function of $n/q \mod 1$ and is therefore again a $1$-step nilsequence.  We can combine all these cases together, and observe that $\chi$ and $\chi'$ are equivalent when $\alpha'_d = \alpha_d + a/q + O(N^{-d}) \mod 1$ for some standard rational $a/q$.  It is possible to use the quantitative equidistribution theory of nilmanifolds (see \cite{green-tao-nilratner}) to show that these are in fact the \emph{only} cases in which $\chi$ and $\chi'$ are equivalent; this is a variant of the classical theorem of Weyl that a polynomial sequence is (totally) equidistributed modulo $1$ if and only if at least one non-constant coefficients is irrational.  In view of this, we see that $\Symb^d([N])$ contains $\ultra \R / (\ultra \Z + \Q + N^{-d} \overline{\R})$ as a subgroup, and the symbol of $n \mapsto e(\alpha_0 + \ldots + \alpha_d n^d)$ can be identified with 
$$\alpha_d \hbox{ mod } 1, \Q, O(N^{-d}) := \alpha + \ultra \Z + \Q + N^{-d} \overline{\R}.$$

However, the presence of bracket polynomials (suitably modified to avoid the topological obstruction mentioned earlier) means that when $d \geq 2$, that $\Symb^d([N])$ is somewhat larger than the above mentioned subgroup.  We illustrate this with the following (non-rigorous) discussion.  Take $d=2$ and consider two degree $2$ nilcharacters $\chi, \chi'$ of the form
$$ \chi(n) \approx e( \{ \alpha n \} \beta n + \gamma n^2 )$$
and
$$ \chi'(n) \approx e( \{ \alpha' n \} \beta' n + \gamma' n^2 )$$
for some $\alpha, \beta, \gamma, \alpha', \beta', \gamma' \in \ultra \R$, where we interpet the symbol $\approx$ loosely to mean that $\chi, \chi'$ are suitable vector-valued smoothings of the indicated bracket phases, of the type discussed earlier in this section. These may also involve some lower order nilsequences of degree $1$.

As before, we consider the question of determining those values of $\alpha,\beta,\gamma,\alpha',\beta',\gamma'$ for which $\chi$ and $\chi'$ are equivalent.  There are a number of fairly obvious ways in which equivalence can occur.  For instance, by modifying the previous arguments, one can show that equivalence holds when $\alpha=\alpha', \beta=\beta'$, and $\gamma-\gamma'$ is equal to a limit integer, a standard rational, or is equal to $O(N^{-2})$.  Similarly, equivalence occurs when $\beta=\beta'$, $\gamma=\gamma'$, and $\alpha-\alpha'$ is equal to a limit integer, a standard rational, or is equal to $O(N^{-1})$.

However, there are also some slightly less obvious ways in which equivalence can occur.  Observe that the expression $e( \{ \alpha n\} \{\beta n\} )$ is a Lipschitz function of the fractional parts of $\alpha n$ and $\beta n$ and is thus a (piecewise) nilsequence of degree $1$ (and will become a genuine nilsequence after one performs an appropriate vector-valued smoothing).  On the other hand, we have the obvious identity
$$ e( (\alpha n - \{ \alpha n \}) (\beta n - \{\beta n\}) )  = 1$$
since the exponent is the product of two (limit) integers.  Expanding this out and rearranging, we obtain the (slightly imprecise) relation
\begin{equation}\label{brackalg}
 e( \{ \alpha n \} \beta n ) \approx e( - \{ \beta n \} \alpha n + \alpha \beta n^2 )
\end{equation}
where we again interpret $\approx$ loosely to mean ``after a suitable vector-valued smoothing, and ignoring lower order factors''.  This gives an additional route for $\chi$ and $\chi'$ to be equivalent.  A similar argument also gives the variant
$$ e( \{ \alpha n \} \beta n ) \approx e( \frac{1}{2} \alpha \beta n^2 )$$
whenever $\alpha,\beta$ are \emph{commensurate} in the sense that $\alpha/\beta$ is a standard rational.  We thus see that the notion of equivalence is in fact already somewhat complicated in degree $2$, and the situation only becomes worse in higher degree.  One can describe equivalence of bracket polynomials explicitly using \emph{bracket calculus}, as developed in \cite{leibman} (see also the earlier works \cite{bl,ha1,ha2,ha3}), but this requires a fair amount of notation and machinery.  Fortunately, in this paper we will be able to treat the notion of a symbol \emph{abstractly}, without requiring an explicit description of the space $\Symb^d([N])$.\vspace{11pt}

\textsc{More general types of filtration.}
The notion of a one-dimensional polynomial $n \mapsto \alpha_0  + \ldots + \alpha_d n^d$ of degree $\leq d$ can of course be generalised to higher dimensions.  For instance, we have the notion of a multidimensional polynomial
$$ (n_1,\ldots,n_k) \mapsto \sum_{i_1,\ldots,i_k \geq 0: i_1+\ldots+i_k \leq d} \alpha_{i_1,\ldots,i_k} n_1^{i_1} \ldots n_k^{i_d}$$
of degree $\leq d$.  We also have the slightly different notion of a multidimensional polynomial
$$ (n_1,\ldots,n_k) \mapsto \sum_{i_1,\ldots,i_k \geq 0: i_j \leq d_j \hbox{ for } 1 \leq j \leq k} \alpha_{i_1,\ldots,i_k} n_1^{i_1} \ldots n_k^{i_d}$$
of \emph{multidegree} $\leq (d_1,\ldots,d_k)$ for some integers $d_1,\ldots,d_k\geq 0$.  We can unify these two concepts into the notion of a multi-dimensional polynomial
\begin{equation}\label{multipoly}
 (n_1,\ldots,n_k) \mapsto \sum_{(i_1,\ldots,i_k) \in J} \alpha_{i_1,\ldots,i_k} n_1^{i_1} \ldots n_k^{i_d}
\end{equation}
of \emph{multidegree} $\subset J$ for some finite \emph{downset} $J \subset \N^k$, i.e. a finite set of tuples with the property that $(i_1,\ldots,i_k) \in J$ whenever  $(i_1,\ldots,i_k) \in \N^k$ and  $i_j \leq i'_j$ for all $j=1,\ldots,k$ for some $(i'_1,\ldots,i'_k) \in J$.  Thus for instance the two-dimensional polynomial
$$ (h,n) \mapsto \alpha h n + \beta h n^2 + \gamma n^3$$
for $\alpha,\beta,\gamma \in \ultra \R$ is of multidegree $\subset J$ for 
\[
J := \{ (0,0), (0,1), (0,2), (0,3), (1,0), (1,1), (1,2) \},\]
 and is also of multidegree $\leq (1,3)$ and of degree $\leq 3$.  (One can view the downset $J$ as a variant of the \emph{Newton polytope} of the polynomial.)

In our subsequent arguments, we will need to similarly generalise the notion of a one-dimensional nilcharacter $n \mapsto \chi(n)$ of degree $\leq d$ to a multidimensional nilcharacter $(n_1,\ldots,n_k) \mapsto \chi(n_1,\ldots,n_k)$ of degree $\leq d$, of multidegree $\leq (d_1,\ldots,d_k)$, or of multidegree $\subset J$.  We will define these concepts precisely in a short while, but we mention for now that the polynomial phase
$$ (h,n) \mapsto e( \alpha h n + \beta h n^2 + \gamma n^3 )$$
will be a two-dimensional nilcharacter of multidegree $\subset J$, multi-degree $\leq (1,3)$, and degree $\leq 3$ where $J$ is as above.  Moreover, variants of this phase, such as (a suitable vector-valued smoothing of)
$$ (h,n) \mapsto e( \{ \alpha_1 h\} \alpha_2 n + \{ \{ \beta_1 n \} \beta_2 h \} \beta_3 n + \{ \gamma_1 n^2 \} \gamma_2 n ),$$
will also have the same multidegree and degree as the preceding example.

The multidegree of a nilcharacter $\chi(n_1,\ldots,n_k)$ is a more precise measurement of the complexity of $\chi$ than the degree, because it separates the behaviour of the different variables $n_1,\ldots,n_k$.  We will also need a different refinement of the notion of degree, this time for a one-dimensional nilcharacter $n \mapsto \chi(n)$, which now separates the behaviour of different top degree components of $\chi$, according to their ``rank''.  Heuristically, the rank of such a component is the number of fractional part operations $x \mapsto \{ x \}$ that are needed to construct that component, plus one; thus for instance
$$ n \mapsto e( \alpha n^3 ) $$
has degree $3$ and rank $1$,
$$ n \mapsto e( \{ \alpha n^2 \} \beta n ) $$
has degree $3$ and rank $2$ (after vector-valued smoothing),
$$ n \mapsto e( \{ \{ \alpha n \} \beta n \} \gamma n) $$
has degree $3$ and rank $3$ (after vector-valued smoothing), and so forth.  We will then need a notion of a nilcharacter $\chi$ \emph{of degree-rank $\leq (d,r)$}, which roughly speaking means that all the components used to build $\chi$ either are of degree $<d$, or else are of degree exactly $d$ but rank at most $r$.  Thus for instance,
$$ n \mapsto e( \{ \alpha n \} \beta n + \gamma n^3 )$$
has degree-rank $\leq (3,1)$ (after vector-valued smoothing), while
$$ n \mapsto e( \{ \alpha n \} \beta n + \gamma n^3 + \{ \delta n^2 \} \epsilon n )$$
has degree-rank $\leq (3,2)$ (after vector-valued smoothing), and
$$ n \mapsto e( \{ \alpha n \} \beta n + \gamma n^3 + \{ \delta n^2 \} \epsilon n+ \{ \{ \mu n \} \nu n \} \rho n)$$
has degree-rank $\leq (3,3)$ (after vector-valued smoothing).

In order to make precise the notions of multidegree and degree-rank for nilcharacters, it is convenient to adopt an abstract formalism that unifies degree, multidegree, and degree-rank into a single theory.  We need the following abstract definition.

\begin{definition}[Ordering]\label{order-def}  An \emph{ordering} $I = (I, \prec, +, 0)$ is a set $I$ equipped with a partial ordering $\prec$, a binary operation $+: I \times I \to I$, and a distinguished element $0 \in I$ with the following properties:
\begin{enumerate}
\item The operation $+$ is commutative and associative, and has $0$ as the identity element.
\item The partial ordering $\prec$ has $0$ as the minimal element.
\item If $i, j \in I$ are such that $i \prec j$, then $i + k \prec j+k$ for all $k \in I$.
\item For every $d \in I$, the initial segment $\{ i \in I: i \prec d \}$ is finite.
\end{enumerate}
A \emph{finite downset} in $I$ is a finite subset $J$ of $I$ with the property that $j \in J$ whenever $j \in I$ and $j \prec i$ for some $i \in J$.
\end{definition}

In this paper, we will only need the following three specific orderings (with $k$ a standard positive integer):

\begin{enumerate}
\item The \emph{degree ordering}, in which $I = \N$ with the usual ordering, addition, and zero element.
\item The \emph{multidegree ordering}, in which $I = \N^k$ with the usual addition and zero element, and with the product ordering, thus $(i'_1,\ldots,i'_k) \preceq (i_1,\ldots,i_k)$ if $i'_j \leq i_j$ for all $1 \leq j \leq k$.
\item The \emph{degree-rank ordering}, in which $I$ is the sector $\DR := \{ (d,r) \in \N^2: 0 \leq r \leq d \}$ with the usual addition and zero element, and the lexicographical ordering, that is to say $(d',r') \prec (d,r)$ if $d' < d$ or if $d'=d$ and $r'<r$.
\end{enumerate}

It is easy to verify that each of these three explicit orderings obeys the abstract axioms in Definition \ref{order-def}.  In the case of the degree or degree-rank orderings, $I$ is totally ordered (for instance, the first few degree-ranks are $(0,0), (1,0), (1,1), (2,0)$, $(2,1), (2,2), (3,0), \ldots$), and so the only finite downsets are the initial segments. For the multidegree ordering, however, the initial segments are not the only finite downsets that can occur. 

The one-dimensional notions of a filtration, nilsequence, nilcharacter, and symbol can be easily generalised to arbitrary orderings.  We give the bare definitions here, and defer the more thorough treatment of these concepts to Appendix \ref{poly-app} and Appendix \ref{basic-sec}.  We will however remark that when $I$ is the degree ordering, then all of the notions defined below simplify to the one-dimensional counterparts defined earlier.

\begin{definition}[Filtered group]\label{filtered-group}  Let $I$ be an ordering and let $G$ be a group. By an \emph{$I$-filtration} on $G$ we mean a collection $G_{I} = (G_{ i})_{i \in I}$ of subgroups indexed by $I$, with the following properties:
\begin{enumerate}
\item (Nesting) If $i,j \in I$ are such that $i \prec j$, then $G_i \supseteq G_j$.
\item (Commutators) For every $i,j \in I$, we have $[G_{i}, G_{ j}] \subseteq G_{i+j}$.
\end{enumerate}
If $d \in I$, we say that $G$ has \emph{degree} $\leq d$ if $G_i$ is trivial whenever $i \not \preceq d$.  More generally, if $J$ is a downset in $I$, we say that $G$ has \emph{degree} $\subseteq J$ if $G_i$ is trivial whenever $i \not \in J$.
\end{definition}

Let us explicitly adapt the above abstract definitions to the three specific orderings mentioned earlier.

\begin{definition}  If $(d_1,\ldots,d_k) \in \N^k$, we define a \emph{nilpotent Lie group of multi-degree $\leq (d_1,\ldots,d_k)$} to be a nilpotent $I$-filtered Lie group of degree $\leq (d_1,\ldots,d_k)$, where $I = \N^k$ is the multidegree ordering.  Similarly, if $J$ is a downset, define the notion of a nilpotent Lie group of multidegree $\subseteq J$.

If $(d,r) \in \DR$, define a \emph{nilpotent Lie group of degree-rank $\leq (d,r)$} to be a nilpotent $\DR$-filtered Lie group $G$ of degree $\leq (d,r)$, with the additional axioms $G_{(0,0)}=G$ and $G_{(d,0)} = G_{(d,1)}$ for all $d \geq 1$.  

We define the notion of a filtered nilmanifold of multidegree $\leq (d_1,\ldots,d_k)$, multidegree $\subseteq J$, or degree-rank $\leq (d,r)$ similarly.
\end{definition}

Note that the degree-rank filtration needs to obey some additional axioms, which are needed in order for the rank $r$ to play a non-trivial role.  As such, the unification here of degree, multidegree, and degree-rank, is not quite perfect; however this wrinkle is only of minor technical importance and should be largely ignored on a first reading.

\begin{example}  If $G$ is a filtered nilpotent group of multidegree $\leq (1,1)$, then the groups $G_{(1,0)}$ and $G_{(0,1)}$ must be abelian normal subgroups of $G_{(0,0)}$, and their commutator $[G_{(1,0)}, G_{(0,1)}]$ must lie inside the group $G_{(1,1)}$, which is a central subgroup of $G_{(0,0)}$.  

If $G$ is a filtered nilpotent group of degree-rank $\leq (d,d)$, then $(G_{(i,0)})_{i \geq 0}$ is a $\N$-filtration of degree $\leq d$.  But if we reduce the rank $r$ to be strictly less than $d$, then we obtain some additional relations between the $G_{(i,0)}$ that do not come from the filtration property.  For instance, if $G$ has degree-rank $\leq (3,2)$, then the group $[G_{(1,0)},[G_{(1,0)},G_{(1,0)}]]$ must now be trivial; if $G$ has degree-rank $\leq (3,1)$, then the group $[G_{(1,0)}, G_{(2,0)}]$ must also be trivial.  More generally, if $G$ has degree-rank $\leq (d,r)$, then any iterated commutator of $g_{i_1},\ldots,g_{i_m}$ with $g_j \in G_{(i_j,0)}$ for $j=1,\ldots,m$ will be trivial whenever $i_1+\ldots+i_m > d$, or if $i_1+\ldots+i_m=d$ and $m>r$.
\end{example}

\begin{example}\label{inclusions}  If $(G_i)_{i \in \N}$ is an $\N$-filtration of $G$ of degree $\leq d$, then $(G_{|\vec i|})_{\vec i \in \N^k}$ is an $\N^k$-filtration of $G$ of multidegree $\subset \{\vec i \in \N^k: |\vec i| \leq d \}$, where we recall the notational convention $|(i_1,\ldots,i_k)| = i_1 + \ldots + i_k$.  Conversely, if $J$ is a finite downset of $\N^k$ and $(G_{\vec i})_{\vec i \in \N^k}$ is a $\N^k$-filtration of $G$ of multidegree $\subset J$, then
$$ \left( \bigvee_{\vec i: |\vec i| \leq i} G_{\vec i} \right)_{i \in \N}$$
is easily verified (using Lemma \ref{normal}) to be an $\N$-filtration of degree $\leq \max_{\vec i \in J} |\vec i|$, where $\bigvee_{a \in A} G_a$ is the group generated by $\bigcup_{a \in A} G_a$.  In particular, any multidegree $\leq (d_1,\ldots,d_k)$ filtration induces a degree $\leq d_1+\ldots+d_k$ filtration.

In a similar spirit, every degree-rank $\leq (d,r)$ filtration $(G_{(d',r')})_{(d',r') \in \DR}$ of a group $G$ induces a degree $\leq d$ filtration $(G_{(i,0)})_{i \in \N}$.  In the converse direction, if $(G_i)_{i \in \N}$ is a degree $\leq d$ filtration of $G$ with $G=G_0$, then we can create a degree-rank $\leq (d,d)$ filtration $(G_{(d',r')})_{(d',r') \in \DR}$ by setting $G_{(d',r')}$ to be the space generated by all the iterated commutators of $g_{i_1},\ldots,g_{i_m}$ with $g_j \in G_{(i_j,0)}$ for $j=1,\ldots,m$ for which either $i_1+\ldots+i_m > d'$, or $i_1+\ldots+i_m=d$ and $m \geq \max(r',1)$; this can easily be verified to indeed be a filtration, thanks to Lemma \ref{normal}.
\end{example}

\begin{example}\label{dr-f}  Let $d \geq 1$ be a standard integer.  We can give the unit circle $\T$ the structure of a degree-rank filtered nilmanifold of degree-rank $\leq (d,1)$ by setting $G=\R$ and $\Gamma=\Z$ with $G_{(d',r')} := \R$ for $(d',r') \leq (d,1)$ and $G_{(d',r')} := \{0\}$ otherwise.  This is also the filtration obtained from the degree $\leq d$ filtration (see Example \ref{polyphase}) using the construction in Example \ref{inclusions}.
\end{example}

\begin{example}[Products]\label{prodeq}  If $G_{I}$ and $G'_I$ are $I$-filtrations on groups $G, G'$ then we can give the product $G \times G'$ an $I$-filtration in an obvious way by setting $(G \times G')_i := G_i \times G'_i$.  The degree of $G \times G'$ is the union of the degrees of $G$ and $G'$.  Similarly the product $G_1/\Gamma_1 \times G_2/\Gamma_2$ of two $I$-filtered nilmanifolds is an $I$-filtered nilmanifold.
\end{example}

\begin{example}[Pushforward and pullback]\label{pushpull}  Let $\phi: G \to H$ be a homomorphism of groups.  Then  any any $I$-filtration $H_I = (H_{ i})_{i \in I}$ of $H$ induces a \emph{pullback $I$-filtration} $\phi^* H_I := (\phi^{-1}(H_{i}))_{i \in I}$.  Similarly, any $I$-filtration $G_{I} = (G_{i})_{i \in I}$ on $G$ induces a \emph{pushforward $I$-filtration} $\phi_* G_{I} := (\phi(G_{i}))_{i \in I}$ on $H$.   In particular, if $\Gamma$ is a subgroup of $G$, then we can pullback a filtration $G_{I} = (G_{ i})_{i \in I}$ of $G$ by the inclusion map $\iota : \Gamma \hookrightarrow G$ to create the \emph{restriction} $\Gamma_{I} := (\Gamma_{i})_{i \in I}$ of that filtration. It is a trivial matter to check that the subgroups of this filtration are given by $\Gamma_{ i} := \Gamma \cap G_{i}$.
\end{example}

\begin{definition}[Filtered quotient space]\label{quot} A \emph{$I$-filtered quotient space} is a quotient $G/\Gamma$, where $G$ is an $I$-filtered group and $\Gamma$ is a subgroup of $G$ (with the induced filtration, see Example \ref{pushpull}).

A \emph{$I$-filtered homomorphism} $\phi: G/\Gamma \to G'/\Gamma'$ between $I$-filtered quotient spaces is a group homomorphism $\phi: G \to G'$ which maps $\Gamma$ to $\Gamma'$, and also maps $G_i$ to $G'_i$ for all $i \in I$.  Note that such a homomorphism descends to a map from $G/\Gamma$ to $G'/\Gamma'$.  

If $G$ is a nilpotent $I$-filtered Lie group, and $\Gamma$ is a discrete cocompact subgroup of $G$ which is rational with respect to $G_I$ (thus $\Gamma_i := \Gamma \cap G_i$ is cocompact in $G_i$ for each $i \in I$), we call $G/\Gamma = (G/\Gamma, G_I)$ an \emph{$I$-filtered nilmanifold}.
We say that $G/\Gamma$ has degree $\leq d$ or $\subseteq J$ of $G$ has degree $\leq d$ or $\subseteq J$.
\end{definition}

\begin{example}[Subnilmanifolds]  Let $G/\Gamma$ be an $I$-filtered nilmanifold of degree $\subset J$.  If $H$ is a rational subgroup of $G$, then $H/(H \cap \Gamma)$ is also a filtered nilmanifold degree $\subset J$ (using Example \ref{pushpull}), with an inclusion homomorphism from $H/(H \cap \Gamma)$ to $G/\Gamma$; we refer to $H/(H \cap \Gamma)$ as a \emph{subnilmanifold} of $G/\Gamma$.
\end{example}

We isolate three important examples of a filtered group, in which $G$ is the additive group $\Z$ or $\Z^k$.
\begin{definition}[Basic filtrations]\label{basic-filter} We define the following filtrations:
\begin{itemize} 
\item The \emph{degree filtration} $\Z^k_\N$ on $G = \Z^k$, in which $I = \N$ is the degree ordering and $G_i = G$ for $i \leq 1$ and $G_i = \{0\}$ otherwise.  In many cases $k$ will equal $1$ or $2$.
\item The \emph{multidegree filtration} $\Z^k_{\N^k}$ on $G = \Z^k$, in which $I=\N^k$ is the multidegree ordering and $G_{\vec{0}} = \Z^k$, $G_{\vec{e}_i} = \langle \vec{e}_i\rangle$, $i = 1,\dots,k$, and $G_{\vec{v}} = \{ 0\}$ otherwise, with $e_1,\ldots,e_k$ being the standard basis for $\Z^k$;
\item The \emph{degree-rank filtration} $\Z_\DR$ on $G = \Z$, in which $I=\DR$ is the degree-rank ordering and $G_{(0,0)} = G_{(1,0)} = \Z$ and $G_{(d,r)} = \{0\}$ otherwise. 
\end{itemize}
\end{definition}

\begin{definition}[Polynomial map]\label{poly-map-def}
Suppose that $H$ and $G$ are $I$-filtered groups with $H = (H,+)$ abelian\footnote{This is not actually a necessary assumption; see Appendix \ref{poly-app}.  However, in the main body of the paper we will only be concerned with polynomial maps on additive domains.}. Then for any map $g : H \rightarrow G$ we define the derivative 
\begin{equation}\label{partial-def}  \partial_h g(n) := g(n+h) g(n)^{-1}.\end{equation} 
We say that $g : H \rightarrow G$ is \emph{polynomial} if
\begin{equation}\label{polynomial-sequence-def} \partial_{h_1} \dots \partial_{h_m} g (n) \in G_{i_1 + \dots + i_m}\end{equation}
for all $m \geq 0$, all $i_1,\dots, i_m \in I$ and all $h_j \in H_{i_j}$ for $j = 1,\dots, m$, and for all $n \in H_0$. 

We denote by $\poly(H_I \to G_I)$ the space of all polynomial maps from $H_I$ to $G_I$.  As usual, we use $\ultra \poly(H_I \to G_I)$ to denote the space of all limit polynomial maps from $\ultra H_I$ to $\ultra G_I$ (i.e. ultralimits of polynomial maps in $\poly(H_I \to G_I)$).
\end{definition}

Many facts about these spaces (in some generality) are established in Appendix \ref{poly-app} where, in particular, a remarkable result essentially due to Lazard and Leibman \cite{lazard,leibman-group-1,leibman-group-2} is established: $\poly(H_I \to G_I)$ is a group.  The material in Appendix \ref{poly-app} is formulated in the general setting of abstract orderings $I$ and for arbitrary (and possibly non-abelian) groups $H_I$, but for our applications we are only interested in the special case when $H_I$ is $\Z$ or $\Z^k$ with the degree, multidegree, or degree-rank filtration as defined above.

Before moving on let us be quite explicit about what the notion of a polynomial map is in each of the three cases, since  the definitions take a certain amount of unravelling.

\begin{itemize}
\item (Degree filtration) If $H = \Z^k$ with the degree filtration $\Z^k_\N$, then $\poly(\Z^k_\N \to G_\N)$ consists of maps $g : \Z^k \rightarrow G$ with the property that 
\[ \partial_{h_1} \dots \partial_{h_m} g(n) \in G_m\] for all $m \geq 0$, $h_1,\dots,h_m \in \Z^k$ and all $n \in \Z^k$. This space is precisely the same space as the one considered in \cite[\S 6]{green-tao-nilratner}.  The space $\ultra \poly(\Z^k \to G_\N)$ is defined similarly, except that $g: \ultra \Z^k \to \ultra G$ is now a limit map, and all spaces such as $\Z$ and $G_m$ need to be replaced by their ultrapowers.  (Similarly for the other two examples in this list.)
\item (Multidegree  filtration) If $H = \Z^k$ with the multidegree filtration $\Z^k_{\N^k}$, then $\poly(\Z^k_{\N^k} \to G_{\N^k})$ consists of maps $g : \Z^k \rightarrow G$ with the property that 
\[ \partial_{\vec{e}_{i_1}} \dots \partial_{\vec{e}_{i_m}} g(\vec{n}) \in G_{\vec{e}_{i_1} + \dots + \vec{e}_{i_m}}\] 
for all $k \ge 0$, all $i_1,\dots, i_m$ and all $\vec{n} \in \Z^k$. To relate this space to the analogous spaces for the degree ordering, observe (using Example \ref{inclusions}) that
$$ \poly(\Z^k_\N \to (G_i)_{i \in \N} ) = \poly(\Z^k_{\N^k} \to (G_{|\vec i|})_{\vec i \in \N^k} )$$
for any $\N$-filtration $(G_i)_{i \in \N}$, and conversely one has
$$ \poly(\Z^k_{\N^k} \to (G_{\vec i})_{\vec i \in \N^k} ) \subset \poly\left(\Z^k_\N \to ( \bigvee_{|\vec i| = i} G_{\vec i} )_{i \in \N} \right)$$
for any $\N^k$-filtration $(G_{\vec i})_{\vec i \in \N^k}$.  This is of course related to the obvious fact that a polynomial of multidegree $\leq (d_1,\ldots,d_k)$ is automatically of degree $\leq d_1+\ldots+d_k$.
\item (Degree-rank  filtration) If $H = \Z$ with the degree-rank filtration $\Z_\DR$, $\poly(\Z_{\DR} \to G_{\DR})$ consists of maps $g : \Z \rightarrow G$ with the property that
\[ \partial_{h_1} \dots \partial_{h_m} g(n) \in G_{(m,0)}\] 
whenever $m \geq 0$, $h_1,\dots,h_m \in \Z$ and $n \in G_0$.  We observe (using Example \ref{inclusions}) the obvious equality
\begin{equation}\label{dreq}
\poly(\Z_\DR \to (G_{(d,r)})_{(d,r) \in \DR} ) = \poly(\Z_\N \to (G_{(i,0)})_{i \in \N} )
\end{equation}
for any $\DR$-filtration $(G_{(d,r)})_{(d,r) \in \DR}$.  Thus, a degree-rank filtration $G_\DR$ on $G$ does not change the notion of a polynomial sequence, but instead gives some finer information on the group $G$ (and in particular, it indicates that certain iterated commutators of the $G_{(d,r)}$ vanish, which is information that cannot be discerned just from the knowledge that $(G_{(i,0)})_{i \in \N}$ is a $\N$-filtration).
\end{itemize}

\begin{definition}[Nilsequences and nilcharacters]\label{nilch-def-gen} Let $I$ be an ordering, and let $J$ be a finite downset in $I$.  Let $H$ be an abelian $I$-filtered group. A (polynomial) nilsequence of degree $\subset J$ is any function of the form 
\[\chi(n) = F(g(n) \ultra \Gamma),\] where
\begin{itemize}
\item $G/\Gamma = (G/\Gamma,G_I)$ is an $I$-filtered nilpotent manifold of degree $\subset J$;
\item $g \in \ultra\poly(H_{I} \to G_{I})$ is a limit polynomial map from $\ultra H_I$ to $\ultra G_I$; and
\item $F \in \Lip(\ultra (G/\Gamma) \rightarrow \overline{C}^{\omega})$.
\end{itemize}
The space of all such nilsequences will be denoted $\Nil^{\subset J}(\ultra H)$.  We define the notion of a nilsequence of degree $\leq d$ for some $d \in I$, and the space $\Nil^{\leq d}(\ultra H)$, similarly.  If $\Omega$ is a limit subset of $\ultra H$, the restriction of the nilsequences in $\Nil^{\subset J}(\ultra H)$ to $\Omega$ will be denoted $\Nil^{\subset J}(\Omega)$, and we define $\Nil^{\leq d}(\Omega)$ similarly. 

We refer to the map $n \mapsto g(n) \ultra \Gamma$ as a \emph{limit polynomial orbit} in $G/\Gamma$, and denote the space of such orbits as $\ultra \poly(H_I \to (G/\Gamma)_I)$.

Suppose that $d \in I$. Then $\chi$ is said to be a \emph{degree $d$ nilcharacter} if $\chi$ is a degree $\leq d$ nilsequence with the following additional properties:
\begin{itemize}
\item $F \in \Lip(\ultra(G/\Gamma) \to \overline{S^{\omega}})$ (thus $|F|=1$) and
\item $F( g_d x ) = e( \eta(g_d) ) F(x)$ for all $x \in G/\Gamma$ and $g_d \in G_{d}$, where $\eta: G_{d} \to \R$ is a continuous standard homomorphism which maps $\Gamma_{d}$ to the integers.  We call $\eta$ the \emph{vertical frequency} of $F$.
\end{itemize}
The space of all degree $d$ nilcharacters on $\ultra H$ will be denoted $\Xi^d(\ultra H)$.  If $\Omega$ is a limit subset of $\ultra H$, the restriction of the nilcharacters in $\Xi^d(\ultra H)$ to $\Omega$ will be denoted $\Xi^d(\Omega)$. 

With the multidegree ordering, a degree $(d_1,\ldots,d_k)$ nilcharacter will be referred to as a multidegree $(d_1,\ldots,d_k)$ nilcharacter, and the space of such characters on $\Omega$ denoted $\Xi^{(d_1,\ldots,d_k)}_\MD(\Omega)$; we similarly write $\Nil^{\subset J}(\Omega)$ or $\Nil^{\leq (d_1,\ldots,d_k)}(\Omega)$ as $\Nil^{\subset J}_\MD(\Omega)$ or $\Nil^{\leq (d_1,\ldots,d_k)}(\Omega)$ for emphasis.  

Similarly, with the degree-rank ordering, and assuming $G/\Gamma$ is a filtered nilmanifold of degree-rank $\leq (d,r)$ (so in particular, we enforce the axioms $G_{(0,0)} = G$ and $G_{(d,0)}=G_{(d,1)}$), a degree $(d,r)$ nilcharacter will be referred to as a degree-rank $(d,r)$ nilcharacter.  The space of nilcharacters on $\Omega$ of degree-rank $(d,r)$ will be denoted $\Xi^{(d,r)}_\DR(\Omega)$ (note that this is distinct from the space $\Xi^{(d_1,d_2)}_\MD(\Omega)$ of two-dimensional nilcharacters of multidegree $(d_1,d_2)$), and the nilsequences on $\Omega$ of degree-rank $\leq (d,r)$ will similarly be denoted $\Nil^{\leq (d,r)}_\DR(\Omega)$.
\end{definition}

\begin{example}  Let $J \subset \N^k$ be a finite downset.  Then any sequence of the form
$$ (n_1,\ldots,n_k) \mapsto F\left( \sum_{(i_1,\ldots,i_k) \in J} \alpha_{i_1,\ldots,i_k} n_1^{i_1} \ldots n_k^{i_k} \mod 1\right),$$
where $\alpha_{i_1,\ldots,i_k} \in \ultra \R$ and $F \in \Lip( \ultra \T \to \overline{\C}^\omega )$, is a nilsequence on $\Z^k$ of multidegree $\subseteq J$, as can easily be seen by giving $G := \R$ the $\Z^k$-filtration $G_i := \R$ for $i \in J$ and $G_i := \{0\}$ otherwise, and setting $\Gamma := \Z$ and $g \in \ultra \poly( \Z^k \to \R )$ to be the limit polynomial $n \mapsto \sum_{(i_1,\ldots,i_k) \in J} \alpha_{i_1,\ldots,i_k} n_1^{i_1} \ldots n_k^{i_k}$.

For similar reasons, any sequence of the form
$$ (n_1,\ldots,n_k) \mapsto e\left( \sum_{(i_1,\ldots,i_k) \in \N^k: i_1+\ldots+i_k \leq d} \alpha_{i_1,\ldots,i_k} n_1^{i_1} \ldots n_k^{i_k} \mod 1\right),$$
is a degree $d$ nilcharacter on $\Z^k$ of degree $d$, and any sequence of the form
$$ (n_1,\ldots,n_k) \mapsto e\left( \sum_{(i_1,\ldots,i_k) \in \N^k: i_j \leq d_j \hbox{ for } j=1,\ldots,k} \alpha_{i_1,\ldots,i_k} n_1^{i_1} \ldots n_k^{i_k} \mod 1\right),$$
is a multidegree $(d_1,\ldots,d_k)$ nilcharacter on $\Z^k$.
\end{example}

\begin{example}\label{abn}  Any degree $2$ nilsequence of magnitude $1$ is automatically a degree-rank $\leq (3,0)$ nilcharacter, since every degree $\leq 2$ nilmanifold is automatically a degree-rank $\leq (2,2)$ nilmanifold, which can then converted trivially to a degree-rank $\leq (3,0)$ nilmanifold (with a trivial group $G_{(3,0)}$).  Thus for instance for $\alpha,\beta \in \R$,
$$ n \mapsto e( \{ \alpha n \} \beta n )$$
is nearly a degree-rank $(3,0)$ nilcharacter, and becomes a genuine degree-rank $(3,0)$ nilcharacter after vector-valued smoothing.

If $\alpha \in \ultra \R$, then the sequence
$$ n \mapsto e( \alpha n^3 )$$
is a degree-rank $(3,1)$ nilcharacter.  Indeed, we can give $G=\R$ a degree-rank $\leq (3,1)$ filtration $G_\DR$ by setting $G_{(d,r)} := \R$ for $(d,r) \leq (3,1)$, and $G_{(d,r)} := \{0\}$ otherwise.  

Next, if $\alpha, \beta \in \ultra \R$, then the sequence
\begin{equation}\label{abn-eq}
n \mapsto e( \{ \alpha n^2 \} \beta n )
\end{equation}
is \emph{nearly} a degree-rank $(3,2)$ nilcharacter (and becomes a genuinely so after vector-valued smoothing).  To see this, let $G$ be the Heisenberg nilpotent group \eqref{heisen}, which we give the following degree-rank filtration:
\begin{align*}
G_{(0,0)} = G_{(1,0)} = G_{(1,1)} &:= G \\
G_{(2,0)} = G_{(2,1)} &:= \langle e_1, [e_1,e_2] \rangle_\R =  \{ e_1^{t_1} [e_1,e_2]^{t_{12}}: t_1,t_{12} \in \R \} \\
G_{(2,2)} = G_{(3,0)} = G_{(3,1)} = G_{(3,2)} &:= \langle [e_1,e_2] \rangle_\R = \{ [e_1,e_2]^{t_{12}}: t_{12} \in \R \} \\
G_{(d,r)} &:= \{\id\} \hbox{ for all other } (d,r) \in \DR.
\end{align*}
One easily verifies that this is a degree-rank $\leq (3,2)$ filtration.  If we then set $g: \ultra \Z \to \ultra G$ to be the limit sequence $g(n) := e_2^{\beta n} e_1^{\alpha n^2}$, one easily verifies that $g$ is a limit polynomial with respect to this degree-rank filtration.  If one then lets $F$ be the piecewise Lipschitz function \eqref{fdef}, then we see that
$$ F( g(n) \ultra \Gamma ) = e( \{ \alpha n^2 \} \beta n )$$
and so we see that $n \mapsto e( \{ \alpha n^2 \} \beta n )$ is a indeed piecewise degree-rank $(3,2)$ nilcharacter.

A similar argument (using the free $3$-step nilpotent manifold on three generators, which has degree $\leq 3$ and hence degree-rank $\leq (3,3)$) shows that
$$ n \mapsto e( \{ \{ \alpha n \} \beta n \} \gamma n )$$
is nearly a degree-rank $(3,3)$ nilcharacter, and becomes a genuine degree-rank $(3,3)$ nilcharacter after applying vector-valued smoothing; see \cite[Appendix E]{u4-inverse} for the relevant calculations.

These examples should help illustrate the heuristic that a degree-rank $(d,r)$ nilcharacter is built up using (suitable vector-valued smoothings of) bracket monomials which either have degree less than $d$, or have degree exactly $d$ and involve at most $r-1$ applications of the fractional part operation.
\end{example}

We observe (using Example \ref{inclusions}) the following obvious inclusions:
\begin{enumerate}
\item A multidegree $\leq (d_1,\ldots,d_k)$ nilsequence on $\Z^k$ is automatically a degree $\leq d_1+\ldots+d_k$ nilsequence.
\item A multidegree $(d_1,\ldots,d_k)$ nilcharacter on $\Z^k$ is automatically a degree $d_1+\ldots+d_k$ nilcharacter.
\item A multidegree $(d_1,\ldots,d_{k-1},0)$ nilsequence on $\Z^k$ is constant in the $n_k$ variable, and descends to a multidegree $(d_1,\ldots,d_{k-1})$ nilsequence on $\Z^{k-1}$.
\item A degree-rank $\leq (d,r)$ nilsequence on $\Z$ is automatically a degree $\leq d$ nilsequence.  
\item A degree $\leq d$ nilsequence on $\Z$ is automatically a degree-rank $\leq (d,d)$ nilsequence.
\item A degree $d$ nilcharacter on $\Z$ is automatically a degree-rank $\leq (d,d)$ nilcharacter.
\end{enumerate}
It is not quite true, though, that a degree-rank $(d,r)$ nilcharacter is a degree $d$ nilcharacter if $r>1$, because the former need not exhibit vertical frequency behaviour for degree-ranks $(d,r')$ with $r'<r$.

\begin{definition}[Equivalence and symbols]\label{equiv-def}
Let $H$ be an $I$-filtered group, let $d \in I$, and let $\Omega$ be a limit subset of $\ultra H$. Two nilcharacters $\chi, \chi' \in \Xi^d(\Omega)$ are said to be \emph{equivalent} if $\chi\otimes\overline{\chi'}$ is a nilsequence of degree strictly less than $d$.
Write $[\chi]_{\Symb^d(\Omega)}$ for the equivalence class of $\chi$ with respect to this relation; this we shall refer to as the \emph{symbol} of $\chi$. Write $\Symb^d(\Omega)$ for the space of all such equivalence classes. 
\end{definition}

We write $\Symb^{(d_1,\ldots,d_k)}_{\MD}(\Omega)$ for the symbols of nilcharacters $\chi \in \Xi^{(d_1,\ldots,d_k)}_\MD(\Omega)$ of multidegree $(d_1,\ldots,d_k)$, and $\Symb^{(d,r)}_\DR(\Omega)$ for the symbols of nilcharacters $\chi \in \Xi^{(d,r)}_\DR(\Omega)$ of degree-rank $(d,r)$.  The basic properties of such symbols are set out in Appendix \ref{basic-sec}.

\section{A more detailed outline of the argument}\label{overview-sec}

Now that we have set up the notation to describe nilcharacters and their symbols, we are ready to give a high-level proof of Conjecture \ref{gis-conj-nonst-2} (and hence Theorem \ref{mainthm}), contingent on some key sub-theorems which will be proven in later sections. This corresponds to the realisation of points (i), (ii) and (ix) from the overview in \S \ref{strategy-sec}.

As the cases $s=1,2$ of this conjecture are already known, we assume that $s \geq 3$.  We also assume inductively that the claim has already been proven for smaller values of $s$.  Henceforth $s$ is fixed.

Let $f \in L^\infty[N]$ be such that 
\begin{equation}\label{fus}
\|f\|_{U^{s+1}[N]} \gg 1.  
\end{equation}
Define $f$ to be zero outside of $[N]$. Raising \eqref{fus} to the power $2^{s+1}$, we see that
$$ \E_{h \in [[N]]} \| \Delta_h f \|_{U^{s}[N]}^{2^s} \gg 1$$
and thus
$$ \| \Delta_h f \|_{U^{s}[N]} \gg 1$$
for all $h$ in a dense subset $H$ of $[[N]]$.  Applying the inductive hypothesis, we thus see that $\Delta_h f$ is $(s-1)$-biased for all $h \in H$.

By definition, we now know that $\Delta_h f$ correlates with a nilsequence of degree $(s-1)$.  By Lemma \ref{nilch-cor}, we see that for each $h \in H$, $\Delta_h f$ correlates with a nilcharacter $\chi_h \in \Xi^{s-1}([N])$.  It is not hard to see that the space of such nilcharacters is a $\sigma$-limit set (see Definition \ref{separ}), so by Lemma \ref{mes-select} we can ensure that $\chi_h$ depends in a limit fashion on $h$.

The aim at this point is to obtain, in several stages, information about the dependence of $\chi_h$ on $h$.   A key milestone in this analysis is a \emph{linearisation} of $\chi_h$ on $h$.  In the case $s = 2$, treated in \cite{gowers-4aps,green-tao-u3inverse}, the $\chi_h(n)$ were essentially just linear phases $e(\xi_h n)$, and the outcome of the linearisation analysis was that the frequencies $\xi_h$ may be assumed to vary in a bracket-linear fashion with $h$. In the case $s = 3$ (treated in \cite{u4-inverse} but also dealt with in our present work), a model special case occurs when $\chi_h(n) \approx e(\{\alpha_h n\} \beta_h n)$ (interpreting $\approx$ loosely). The outcome of the linearisation analysis in that case was that at most one of $\alpha_h, \beta_h$ really depends on $h$, and furthermore that this dependence on $h$ is bracket-linear in nature.

Now we formally set out the general case of this linearisation process.

\begin{theorem}[Linearisation]\label{linear-thm}  Let $f \in L^\infty[N]$, let $H$ be a dense subset of $[[N]]$, and let $(\chi_h)_{h \in H}$ be a family of nilcharacters in $\Xi^{s-1}([N])$ depending in a limit fashion on $h$, such that $\Delta_h f$ correlates with $\chi_h$ for all $h \in H$.  Then there exists a multidegree $(1, s-1)$-nilcharacter $\chi \in \Xi^{(1,s-1)}_\MD(\ultra \Z^2)$ such that $\Delta_h f$ $(s-2)$-correlates with $\chi(h,\cdot)$ for many $h \in H$.
\end{theorem}

This statement represents the outcome of points (iii) to (vii) of the outline in \S \ref{strategy-sec} and must therefore address the following points:
\begin{itemize}
\item For some suitable notion of ``frequency'', the symbol of $\chi_h(n)$ contains only one frequency that genuinely depends on $h$;
\item That frequency depends on $h$ in a bracket-linear manner;
\item Once this is known, it follows that, for many $h$, $\Delta_h f$ $(s-2)$-correlates with $\chi(h, n)$, where $\chi$ is a certain $2$-variable nilsequence.
\end{itemize}

These three tasks are, in fact, established together and in an incremental fashion. The nilcharacter $\chi_h(n)$ is gradually replaced by objects of the form $\chi'(h,n)\otimes \chi'_h(n)$ where $\chi'(h,n)$ is a $2$-dimensional nilcharacter of multidegree $(1, s-1)$ and, at each stage, the nilcharacter $\chi'_h(n)$ (which has so far not been shown to vary in any nice way with $h$) is ``simpler'' than $\chi_h(n)$. The notion of \emph{simpler} in this context is measured by the degree-rank filtration, a concept that was introduced in the previous section. Thus the result of a single pass over the three points listed above is the following subclaim.

\begin{theorem}[Linearisation, inductive step]\label{linear-induct}  Let $1 \leq r_* \leq s-1$, let $f \in L^\infty[N]$, let $H$ be a dense subset of $[[N]]$, let $\chi \in \Xi^{(1,s-1)}_\MD(\ultra \Z^2)$, let $(\chi_h)_{h \in H}$ be a family of nilcharacters of degree-rank $(s-1,r_*)$ depending in a limit fashion on $h$, such that $\Delta_h f$ $(s-2)$-correlates with $\chi(h,\cdot) \otimes \chi_h$ for all $h \in H$.  Then there exists a dense subset $H'$ of $H$, a multidegree $(1, s-1)$-nilcharacter $\chi' \in \Xi^{(1,s-1)}_\MD(\ultra \Z^2)$ and a family $(\chi'_h)_{h \in H}$ of nilcharacters of degree-rank $(s-1,r_*-1)$ depending in a limit fashion on $h$, such that $\Delta_h f$ $(s-2)$-correlates with $\chi'(h,\cdot) \otimes \chi'_h$ for all $h \in H'$.
\end{theorem}

Theorem \ref{linear-thm} follows easily by inductive use of this statement, starting with $r_*$ equal to $s-1$ and using Theorem \ref{linear-induct} iteratively to decrease $r_*$ all the way to zero. 

To prove Theorem \ref{linear-induct}, we follow steps (iii) to (vii) in the outline quite closely. The first step, which is the realisation of (iii), is a Gowers-style Cauchy-Schwarz inequality to eliminate the function $f$ as well as the $2$-dimensional nilcharacter $\chi(h,n)$ and therefore obtain a statement concerning only the (so far) unstructured-in-$h$ object $\chi_h(n)$. Here is a precise statement of the outcome of this procedure; the proof of this proposition is the main business of \S \ref{cs-sec}. 

\begin{proposition}[Gowers Cauchy-Schwarz argument]\label{gcs-prop}  Let $f,H,\chi,(\chi_h)_{h \in H}$ be as in Theorem \ref{linear-induct}.  Then the sequence
\begin{equation}\label{gowers-cs-arg}
 n \mapsto \chi_{h_1}(n) \otimes \chi_{h_2} (n + h_1 - h_4) \otimes \overline{\chi_{h_3}(n)} \otimes \overline{\chi_{h_4}(n + h_1 - h_4)}
\end{equation}
is $(s-2)$-biased for many additive quadruples $(h_1,h_2,h_3,h_4)$ in $H$.
\end{proposition}

With this in hand, we reach the most complicated part of the argument. This is the use of Proposition \ref{gcs-prop} to study the ``frequencies'' of the nilcharacters $\chi_h$ and the way they depend on $h$. Roughly speaking, the aim is to interpret the tensor product \eqref{gowers-cs-arg}
as a nilsequence itself (depending on $h_1, h_2, h_3, h_4$) and use results from \cite{green-tao-nilratner} to analyse its equidistribution and bias properties.

To make proper sense of this one must first find a suitable ``representation'' of the $\chi_h(n)$ in which the frequencies are either independent of $h$, depend in a bracket-linear fashion on $h$, or are appropriately \emph{dissociated} in $h$, in the sense that the frequencies associated to \eqref{gowers-cs-arg} are ``linearly independent'' for most additive quadruples $h_1+h_2=h_3+h_4$.  This task is one of the more technical part of the papers and is performed in in \S \ref{reg-sec}; it incorporates the additive combinatorial step (vi) of the outline from \S \ref{strategy-sec}.  The precise statement of what we prove is Lemma \ref{sunflower}, the ``sunflower decomposition''.

The representation of the $\chi_h$ (and hence of \eqref{gowers-cs-arg}) involves constructing a suitable polynomial orbit on something resembling a free nilpotent Lie group $\tilde G$; this device also featured in \cite[\S 5]{u4-inverse}.   Once this is done, one applies the results from \cite{green-tao-nilratner} to examine the orbit of this polynomial sequence on the corresponding nilmanifold $\tilde G/\tilde \Gamma$. The results of \cite{green-tao-nilratner} assert (roughly speaking) that this orbit is close to the uniform measure on a subnilmanifold $H\tilde\Gamma/\tilde\Gamma$, where $H \leq \tilde G$ is some closed subgroup. In \S \ref{linear-sec}, we then crucially apply a commutator argument of Furstenberg and Weiss that exploits some equidistribution information on projections of $H$ to say something about this group $H$. The upshot of this critical phase of the argument is that the $h$-dependence of the frequencies of $\chi_h$ cannot be dissociated in nature, and must instead be completely bracket-linear; the precise statement here is Theorem \ref{slang-petal}. 

At this point in the argument, we have basically shown that the top-order behaviour (in the degree-rank order) of the nilcharacters $\chi_h(n)$ is bracket-linear in $h$. To complete the proof of Theorem \ref{linear-induct} (and hence of Theorem \ref{linear-thm}) it remains to carry out part (vii) of the outline, that is to say to interpret this bracket-linear part of $\chi_h(n)$ as a multidegree $(1,s-1)$ nilcharacter $\chi'(h,n)$. This is the first part of the argument where some sort of ``degree $s$ nil-object'' is actually constructed, and is thus a key milestone in the inductive derivation of $\GI(s)$ from $\GI(s-1)$.  As remarked previously, our construction here is a little more conceptual (and abstractly algebraic) than in previous works, which have been somewhat \emph{ad hoc}. The construction is given in \S \ref{multi-sec}. At the end of that section we wrap up the proof of Theorem \ref{linear-thm}: by this point, all the hard work has been done.

With Theorem \ref{linear-thm} in hand, we have completed the first seven steps of the outline. The only remaining substantial step is step (viii), the symmetry argument.  Here is a formal statement of it:

\begin{theorem}[Symmetrisation]\label{aderiv}   Let $f \in L^\infty[N]$, let $H$ be a dense subset of $[[N]]$, and let $\chi \in \Xi^{(1,s-1)}_\MD(\ultra \Z^2)$ be such that $\Delta_h f$ $<s-2$-correlates with $\chi(h,\cdot)$ for all $h \in H$.  Then there exists a nilcharacter $\Theta \in \Xi^{s}(\ultra \Z)$ \textup{(}with the degree filtration\textup{)} and a nilsequence $\Psi \in \Nil^{\subset J}_\MD(\ultra \Z^2)$, with $J \subset \N^2$ given by the downset
\begin{equation}\label{lower}
J := \{ (i,j) \in \N^2: i+j \leq s-1 \} \cup \{ (i,s-i): 2 \leq i \leq s \},
\end{equation}
such that $\chi(h,n)$ is a bounded linear combination of $\Theta(n+h) \otimes \overline{\Theta(n)} \otimes \Psi(h,n)$.
\end{theorem}

The proof is given in \S \ref{symsec}.  Informally, this theorem asserts that the multidimensional degree $(1,s-1)$ nilcharacter $\chi(h,n)$ can be expressed as a derivative $\Theta(n+h) \otimes \overline{\Theta(n)}$ of a degree $s$ nilcharacter $\Theta$, modulo ``lower order terms'', which in this context means multidimensional nilsequences $\Psi(h,n)$ that either have total degree $\leq s-1$, or are of degree at most $s-2$ in the $n$ variable.

The remaining task for this section is to show how to complete the proof of Conjecture \ref{gis-conj-nonst} (and Theorem \ref{mainthm}) from this point.
From the discussion at the beginning of this section, we have already arrived at a situation in which the given function $f \in L^\infty[N]$ has the property that $\Delta_h f$ correlates with $\chi_h$ for all $h$ in a dense subset $H$ of $[[N]]$, where $(\chi_h)_{h \in H}$ be a family of nilcharacters in $\Xi^{s-1}([N])$ depending in a limit fashion on $h$.

From Theorem \ref{linear-thm} and Theorem \ref{aderiv} we see that for many $h \in [[N]]$, $\Delta_h f$ $\leq s-2$-correlates with the sequence
$$ n \mapsto \Theta(n+h) \otimes \overline{\Theta(n)} \otimes \Psi(h,n).$$
The next step is to break up $J$ and $\Psi$ into simpler components, and our tool for this purpose shall be Lemma \ref{approx}.
Applying this lemma for $\eps$ sufficiently small, followed by the pigeonhole principle, one can thus find scalar-valued nilsequences $\psi, \psi'$ on $\ultra \Z^2$ (with the multidegree filtration) of multidegree
$$ \subset \{ (i,0) \in \N^2: i \leq s-1 \}$$
and
$$ \subset \{ (i,j) \in \N^2: i \leq s-2; i+j \leq s \}$$
respectively, such that for many $h \in [[N]]$, $\Delta_h f$ $\leq (s-2)$-correlates with
$$ n \mapsto \Theta(n+h) \otimes \overline{\Theta(n)} \psi(h,n) \psi'(h,n).$$
For fixed $h$, the nilsequence $\psi'(h,n)$ has degree $\leq s-2$ and can thus be ignored.  Also, $\psi(h,n) = \psi(n)$ is of multidegree $\leq (s-1,0)$ and is thus independent of $h$, with $n \mapsto \psi(n)$ being a degree $\leq s-1$ nilsequence.  Thus, for many $h \in [[N]]$, $\Delta_h f$ $\leq s-2$-correlates with
$$ n \mapsto \Theta(n+h) \otimes \overline{\Theta(n)} \psi(n).$$
Applying the pigeonhole principle again, we can thus find scalar nilsequences $\theta, \theta' \in \Nil^{\leq s}(\ultra \Z)$ such that for many $h \in [[N]]$, $\Delta_h f$ $\leq (s-2)$-correlates with
$$ n \mapsto \theta(n+h) \theta'(n)$$
(indeed one takes $\theta, \theta'$ to be coefficients of $\Theta$ and $\overline{\Theta} \psi$ respectively).
Applying the converse to $\GI(s)$ (Proposition \ref{inv-nec-nonst}), we conclude 
$$ \| f\overline{\theta}(\cdot+h) \overline{f\theta'}(\cdot) \|_{U^{s-1}[N]} \gg 1$$
for many $h \in H$.  Averaging over $h$ (using Corollary \ref{auton-2} to obtain the required uniformity), we conclude that
$$ \E_{h \in [[N]]} \| f\overline{\theta}(\cdot+h) \overline{f\theta'}(\cdot) \|_{U^{s-1}[N]}^{2^{s-1}} \gg 1.$$
Applying the Cauchy-Schwarz-Gowers inequality (see e.g. \cite[Equation (11.6)]{tao-vu}) we conclude that
$$ \| f\overline{\theta} \|_{U^s[N]} \gg 1$$
and hence by the inductive hypothesis (Conjecture \ref{gis-conj-nonst-2} for $s-1$), $f\overline{\theta}$ is $\leq (s-1)$-biased.  Since $\theta$ is a degree $\leq s$ nilsequence, we conclude that $f$ is $\leq s$-biased, as required.
This concludes the proof of Conjecture \ref{gis-conj-nonst-2}, Conjecture \ref{gis-conj-nonst}, and hence Theorem \ref{mainthm}, contingent on Theorem \ref{linear-thm} and Theorem \ref{aderiv}.

\section{A variant of Gowers's Cauchy-Schwarz argument}\label{cs-sec}

The aim of this section is prove Proposition \ref{gcs-prop}.  Thus, we have standard integers $1 \leq r_* \leq s-1$, a function $f \in L^\infty[N]$, a dense subset $H$ of $[[N]]$, a two-dimensional nilcharacter $\chi \in \Xi^{(1,s-1)}_\MD(\ultra \Z^2)$ of multidegree $(1,s-1)$, and a family $(\chi_h)_{h \in H}$ of nilcharacters of degree-rank $(s-1,r_*)$ depending in a limit fashion on $h$.  We are given that $\Delta_h f$ $(s-2)$-correlates with $\chi(h,\cdot) \otimes \chi_h$ for all $h \in H$.  Our objective is to show that, for many additive quadruples $(h_1,h_2,h_3,h_4)$ in $H$, the expression
\begin{equation}\label{biasing}
 n \mapsto \chi_{h_1}(n) \otimes \chi_{h_2} (n + h_1 - h_4) \otimes \overline{\chi_{h_3}(n)} \otimes \overline{\chi_{h_4}(n + h_1 - h_4)}
\end{equation}
(where we extend the $\chi_h$ by zero outside of $[N]$) is $(s-2)$-biased.

The strategy, following the work of Gowers \cite{gowers-4aps}, is to start with the $\leq s-2$-correlation between $\Delta_h f$ and $\chi(h,\cdot) \chi_h$ and then apply the Cauchy-Schwarz inequality repeatedly to eliminate all terms involving $f$, $\chi(h,\cdot)$, finally arriving at a correlation statement that only involves $\chi_h$ (and lower order terms).

Unfortunately, there is a technical issue that prevents one from doing this directly, namely that the behaviour of $\chi(h,\cdot)$ in $h$ is not quite linear enough to ensure that these terms are completely eliminated by a Cauchy-Schwarz procedure.  In order to overcome this issue, one must first prepare $\chi$ into a better form, as follows.  We need the following technical notion (which will not be used outside of this section):

\begin{definition}\label{lindef}  A \emph{linearised $(1,s-1)$-function} is a limit function $\chi: (h,n) \to \overline{\C}^\omega$ which has a factorisation
\begin{equation}\label{chan}
 \chi(h,n) = c(n)^h \psi(n)
\end{equation}
where $\psi \in L^\infty(\Z \to \overline{\C}^\omega)$ and $c \in L^\infty(\Z \to S^1)$ are such that, for every $h,l \in \Z$, the sequence
$$ n \mapsto c(n-l)^h \overline{c(n)}^h$$
is a degree $\leq s-2$ nilsequence.
\end{definition}

\begin{remark} Heuristically, one should think of a linearised $(1,s-1)$-function as (a vector-valued smoothing of) a function of the form
$$ (h,n) \mapsto e( P(n) + h Q(n) )$$
where $P, Q$ are bracket polynomials of degree $s-1$; for instance,
$$ (h,n) \mapsto e( \{ \alpha n \} \beta n + \{ \gamma n \} \delta n h )$$
is morally a linearised $(1,2)$ function.  This should be compared with more general multidegree $(1,2)$ nilcharacters, such as
$$ (h,n) \mapsto e( \{ \{ \alpha h \} \beta n \} \gamma n )$$
which are not quite linear in $h$ because the dependence on $h$ is buried inside one or more fractional part operations.  Intuitively, the point is that one can use the laws of bracket algebra (such as \eqref{brackalg}) to move the $h$ outside of all the fractional part expressions (modulo lower order terms).  While one can indeed develop enough of the machinery of bracket calculus to realise this intuition concretely, we will instead proceed by the more abstract machinery of nilmanifolds in order to avoid having to set up the bracket calculus.
\end{remark}

The key preparation for this is the following.

\begin{proposition}\label{prepare} Let $\chi \in \Xi^{(1,s-1)}_\MD(\ultra \Z^2)$ be a two-dimensional nilcharacter of multidegree $(1,s-1)$, and let $\eps > 0$ be standard.  Then one can approximate $\chi$ to within $\eps$ in the uniform norm by a bounded linear combination of linearised $(1,s-1)$-functions.
\end{proposition}

\begin{proof}
From Definition \ref{nilch-def}, we can express
$$ \chi(h,n) = F(g(h,n) \ultra \Gamma)$$
where $G/\Gamma$ is a $\N^2$-filtered nilmanifold of multidegree $\leq (1,s-1)$, $g \in \ultra \poly(\Z^2_{\N^2} \to G_{\N^2})$ (with $\Z^2$ being given the multidegree filtration $\Z^2_{\N^2}$), and $F \in \Lip(\ultra(G/\Gamma) \to \overline{S^\omega})$ has a vertical frequency $\eta: G_{(1,s-1)} \to \R$.

We consider the quotient map $\pi: G/\Gamma \to G/(G_{(1,0)}\Gamma)$ from $G/\Gamma$ onto the nilmanifold $G/(G_{(1,0)}\Gamma)$, which can be viewed as an $\N$-filtered nilmanifold of degree $\leq s-1$ (where we $\N$-filter $G/G_{(1,0)}$ using the subgroups $G_{(0,i)} G_{(1,0)} / G_{(1,0)}$).  The fibers of this map are isomorphic to $T := G_{(1,0)} / \Gamma_{(1,0)}$.  Observe that $G_{(1,0)}$ is abelian, and so $T$ is a torus; thus $G/\Gamma$ is a torus bundle over $G/(G_{(1,0)}\Gamma)$ with structure group $T$.  The idea is to perform Fourier analysis on this large torus $T$, as opposed to the smaller torus $G_{(1,s-1)}/\Gamma_{(1,s-1)}$, to improve the behaviour of the nilcharacter $\chi$.

We pick a metric on the base nilmanifold $G/(G_{(1,0)}\Gamma)$ and a small standard radius $\delta>0$, and form a smooth partition of unity $1 = \sum_{k=1}^K \varphi_k$ on $G/(G_{(1,0)}\Gamma)$, where each $\varphi_k \in \Lip(G/(G_{(1,0)}\Gamma) \to \C)$ is supported on an open ball $B_k$ of radius $r$.  This induces a partition $\chi = \sum_{k=1}^K \tilde \chi_k$, where
$$ \tilde \chi_k(h,n) = F(g(h,n) \ultra \Gamma) \varphi_k(\pi(g(h,n) \ultra \Gamma)).$$
Now fix one of the $k$.  Then we have
$$ \tilde \chi_k(h,n) = \tilde F_k(g(h,n) \ultra \Gamma)$$
where $\tilde F_k$ is compactly supported in the cylinder $\pi^{-1}(B_k)$.

If $r$ is small enough, we have a smooth section $\iota: B_k \to G$ that partially inverts the projection from $G$ to $G/(G_{(1,0)}\Gamma)$, and so we can parameterise any element $x$ of $\pi^{-1}(B_k)$ uniquely as $\iota(x_0) t \Gamma$ for some $x_0 \in B_k$ and $t \in T$ (noting that $t\Gamma$ is well-defined as an element of $G/\Gamma$).  Similarly, we can parameterise any element of $\ultra \pi^{-1}(B_k)$ uniquely as $\iota(x_0) t \Gamma$ for $x_0 \in \ultra B_k$ and $t\in \ultra T$.

We can now view the Lipschitz function $F_k \in \Lip(\ultra(G/\Gamma))$ as a compactly supported Lipschitz function in $\Lip(\ultra(B_k \times T))$.  Applying a Fourier (or Stone-Weierstrass) decomposition in the $T$ directions (cf. Lemma \ref{limone}), we thus see that for any standard $\eps > 0$ we can approximate $\tilde F_k$ uniformly to error $\eps/K$ by a sum $\sum_{k'=1}^{K'} \tilde F_{k,k'}$, where $K'$ is standard and each $F_{k,k'} \in \Lip(\ultra(B_k \times T))$ is compactly supported and has a character $\xi_{k'}: T \to \T$ such that
\begin{equation}\label{fan}
 \tilde F_{k,k'}(\iota(x_0) t\Gamma) = e(\xi_{k'}(t)) \tilde F_{k,k'}(\iota(x_0) \Gamma)
\end{equation}
for all $x_0 \in \ultra(2B_k)$ and $t \in \ultra T$.  It thus suffices to show that for each $k, k'$, the sequence
$$ \tilde \chi_{k,k'}: (h,n) \mapsto \tilde F_{k,k'}( g(h,n) \ultra \Gamma )$$
is a linearised $(1,s-1)$-function.

Fix $k,k'$.  Performing a Taylor expansion (Lemma \ref{taylo}) of the polynomial sequence $g \in \ultra\poly(\Z^2_{\N^2} \to G_{\N^2})$, we may write
$$ g(h,n) = g_0(n) g_1(n)^h$$
where $g_0 \in \ultra \poly(\Z_\N \to G_\N)$ is a one-dimensional polynomial map (giving $G$ the $\N$-filtration $G_\N := (G_{(i,0)})_{i \in \N}$), and $g_1 \in \ultra \poly(\Z \to (G_{(1,0)})_\N)$ is another one-dimensional polynomial map (giving the abelian group $G_{(1,0)}$ the $\N$-filtration $(G_{(1,0)})_\N := (G_{(1,i)})_{i \in \N}$).  In particular, we see that $\tilde \chi_{k,k'}(h,n)$ is only non-vanishing when $\pi( g_0(n) \ultra \Gamma ) \in B$.  Furthermore, in that case we see from \eqref{fan} that
\begin{equation}\label{chimn}
 \tilde \chi_{k,k'}(h,n) = 
 e( h \xi( g_1(n) \mod \Gamma_{(1,0)} ) ) \tilde F_{k,k'}(g_0(n) \ultra \Gamma),
 \end{equation}
which gives the required factorisation \eqref{chan} with $c(n) := e( \xi( g_1(n) \mod \Gamma_{(1,0)} ) )$ and $\psi(n) := \tilde F_{k,k'}(g_0(n) \ultra \Gamma)$.  

The only remaining task is to establish that for any given $h, l$, the sequence $n \mapsto c(n-l)^h \overline{c(n)}^h$ is a degree $\leq s-2$ nilsequence.  We expand this sequence as
$$n \mapsto e( h ( \xi( g_1(n-l) \mod \Gamma_{(1,0)} ) - \xi( g_1(n) \mod \Gamma_{(1,0)} ) ) )$$
But from the abelian nature of $G_{(1,0)}$, the map $n \mapsto \xi(g_1(n) \mod \Gamma_{(1,0)})$ is a polynomial map from $\ultra \Z$ to $\ultra \T$ of degree at most $s-1$, and the claim follows.
\end{proof}

We now return to the proof of Theorem \ref{gcs-prop}.  
With this multiplicative structure, we can now begin the Cauchy-Schwarz argument.  By hypothesis, for each $h \in H$ we can find a scalar nilsequence $\psi_h$ of degree $\leq s-2$ such that
$$ |\E_{n \in [N]} \Delta_h f(n) \overline{\chi(h,n)} \otimes \overline{\chi_h(n)} \overline{\psi_h(n)}| \gg 1.$$
By Corollary \ref{mes-select}, we may ensure that $\psi_h$ varies in a limit fashion on $h$.  Applying Corollary \ref{auton-2}, this lower bound is uniform in $h$.

Applying Proposition \ref{prepare} (with a sufficiently small $\eps$) and using the pigeonhole principle, we may then find a linearised $(1,s-1)$-function $(h,n) \mapsto c(n)^h \psi(n)$ such that
$$ |\E_{n \in [N]} \Delta_h f(n) c(n)^{-h} \overline{\psi(n)} \otimes \overline{\chi_h(n)} \overline{\psi_h(n)}| \gg 1.$$
By Corollary \ref{auton-2} again, the lower bound is still uniform in $h$.  We may then average in $h$ (extending $\psi_h, \chi_h$ by zero for $h$ outside of $H$) and conclude that
$$ \E_{h \in [[N]]} |\E_{n \in [N]} \Delta_h f(n) c(n)^{-h} \overline{\psi(n)} \otimes \overline{\chi_h(n)} \overline{\psi_h(n)}| \gg 1,$$
thus there exists a scalar function $b \in L^\infty[[N]]$ such that
$$ |\E_{h \in [[N]]} \E_{n \in [N]} b(h) f(n+h) \overline{f}(n) c(n)^{-h} \overline{\psi(n)} \otimes \overline{\chi_h(n)} \overline{\psi_h(n)}| \gg 1.$$
By absorbing $b(h)$ into the $\psi_h$ factor, we may now drop the $b(h)$ factor.  We write $n+h = m$ and obtain
$$|\E_{m \in [N]} f(m) \E_{h \in [[N]]} c(m-h)^{-h} f'(m-h)  \otimes \overline{\chi_h(m-h)} \overline{\psi_h(m-h)}| \gg 1$$
where $f' := \overline{f} \overline{\psi}$ (recall that $f$ is extended by zero outside of $[N]$), which by Cauchy-Schwarz implies that
\begin{align*}
|\E_{m \in [N]} \E_{h,h' \in [[N]]} c(m-h)^{-h} c(m-h')^{h'} &f'(m-h) \otimes \overline{f(m-h')}   \\
\otimes \overline{\chi_h(m-h)} \otimes \chi_{h'}(m-h') &\overline{\psi_h(m-h)} \psi_{h'}(m-h')| \gg 1.
\end{align*}
Making the change of variables $h' = h+l$, $n = m-h$, we obtain
\begin{align*}
|\E_{h,l \in [[2N]]; n \in [N]} c(n)^{-h} c(n-l)^{h+l} &f'(n) \otimes \overline{f'}(n-l) \\
\otimes \overline{\chi_h(n)}\otimes \chi_{h+l}(n-l) &\overline{\psi_h(n)} \psi_{h+l}(n-l)| \gg 1.
\end{align*}
We then simplify this as
\begin{equation}\label{hank}
 |\E_{h,l \in [[2N]];n \in [N]} c_2(l,n) \otimes \overline{\chi_h(n)} \otimes \chi_{h+l}(n-l) \psi_{h,l}(n)| \gg 1
\end{equation}
where
\begin{align*}
c_2(l,n) &:= c(n-l)^l f'(n) \otimes \overline{f'(n-l)} \\
\psi_{h,l}(n) &= c(n-l)^h c(n)^{-h} \overline{\psi_h(n)} \psi_{h+l}(n-l) 
\end{align*}
Clearly $c_2$ is bounded.   As for $\psi_{h,l}$, we see from Definition \ref{lindef} and Corollary \ref{alg} that $\psi_{h,l}$ is a nilsequence of degree $\leq s-2$ for each $h,l$.

Returning to \eqref{hank}, we use the pigeonhole principle to conclude that for many $k\in [[2N]]$, we have
$$ |\E_{h \in [[2N]]; n \in [N]} c_2(k,n) \otimes \overline{\chi_h(n)} \otimes \chi_{h+k}(n-k) \psi_{h,k}(n)| \gg 1.$$
Let $k$ be such that the above estimate holds.  Applying Cauchy-Schwarz in the $n$ variable to eliminate the $c_2(k,n)$ term, we have
$$ 
|\E_{h,h' \in [[2N]]; n \in [N]} \overline{\chi_h(n)} \otimes \chi_{h+k}(n-k) \otimes
\overline{\chi_{h'}(n)} \otimes \chi_{h'+k}(n-k) \psi_{h,k}(n)| \gg 1$$
and thus for many $k,h,h' \in [[2N]]$, we have
$$ 
|\E_{n \in [N]} \overline{\chi_h(n)} \otimes \chi_{h+k}(n-k) \otimes
\overline{\chi_{h'}(n)} \otimes \chi_{h'+k}(n-k) \psi_{h,k}(n)| \gg 1,$$
which implies that
$$ n \mapsto \overline{\chi_h(n)} \otimes \chi_{h+k}(n-k) \otimes
\overline{\chi_{h'}(n)} \otimes \chi_{h'+k}(n-k)$$
is $(s-2)$-biased on $[N]$.  Note that this forces $h,h+k, h',h'+k$ to be an additive quadruple in $H$, as otherwise the expression vanishes.  Applying a change of variables, we obtain Proposition \ref{gcs-prop}.

For future reference we observe that a simpler version of the same argument (in which the $\chi$ and $\psi_h$ factors are not present) gives

\begin{proposition}[Cauchy-Schwarz]\label{cs}
Let $f \in L^\infty[N]$, let $H$ be a dense subset of $[[N]]$, and suppose that one has a family of functions $\chi_h \in L^\infty(\ultra \Z)$ depending in a limit fashion on $h$, such that $\Delta_h f$ correlates with $\chi_h$ on $[N]$ for all $h \in H$.  Then for many \textup{(}i.e. for $\gg N^3$\textup{)} additive quadruples $(h_1,h_2,h_3,h_4)$ in $H$, the sequence
\begin{equation}\label{slam}
  n \mapsto \chi_{h_1}(n) \otimes \chi_{h_2} (n + h_1 - h_4) \otimes \overline{\chi_{h_3}(n)} \otimes \overline{\chi_{h_4}(n + h_1 - h_4)}
\end{equation}
is biased.
\end{proposition}

This proposition in fact has quite a simple proof; see \cite{gtz-announce}.  Note how we can conclude \eqref{slam} to be biased and not merely $(s-2)$-biased.  As such, Proposition \ref{cs} saves some ``lower order'' information that was not present in Proposition \ref{gcs-prop}; this lower order information will be crucial later in the argument, when we establish the symmetry property in Theorem \ref{aderiv}.

\section{Frequencies and representations}\label{freq-sec}

We will use Proposition \ref{gcs-prop} to analyse the ``frequency'' of the nilcharacters $(\chi_h)_{h \in H}$ appearing in Theorem \ref{linear-induct}.  To motivate the discussion, let us first suppose that we are in the (significantly simpler) $s=2$ case, rather than the actual case $s \geq 3$ of interest.  When $s=2$, we can represent $\chi_h$ as a linear phase $\chi_h(n) = e(\xi_h n + \theta_h)$ for some $\xi_h, \theta_h \in \ultra\T$; one can then interpret $\xi_h$ as the \emph{frequency} of $h$.

In order to describe how this frequency $\xi_h$ behaves in $h$, it will be convenient to \emph{represent} $\xi_h$ as a linear combination
\begin{equation}\label{xih}
 \xi_h = a_{1,h} \xi_{1,h} + \ldots + a_{D,h} \xi_{D,h}
\end{equation}
of other frequencies $\xi_{1,h},\ldots,\xi_{D,h} \in \ultra\T$, where the $a_{i,h} \in \Z$ are (standard) integer coefficients, and the $(\xi_{i,h})_{h \in H}$ are families of frequencies which have better properties with regards to their dependence on $h$; for instance, they might be ``core frequencies'' $\xi_{i,h} = \xi_{*,i}$ that are independent of $h$, or they might be ``bracket-linear petal'' frequencies that depend in a bracket-linear fashion on $h$, or they might be ``regular petal'' frequencies which behave in a suitably ``dissociated'' manner in $h$.

We can schematically depict the relationship \eqref{xih} as
$$
[\chi_h] \approx \eta_h(\F_h)
$$
where $[\chi_h]$ is some sort of ``symbol'' of $\chi_h$ (which, in the linear case $s=2$, is just $\xi_h \mod 1$), $\F_h \in \ultra \T^D$ is the \emph{frequency vector} $\F_h = (\xi_{1,h},\ldots,\xi_{D,h})$, and $\eta_h: \ultra \T^D \to \ultra \T$ is the \emph{vertical frequency}
\begin{equation}\label{etaxd}
\eta_h(x_1,\ldots,x_D) := a_{1,h} x_1 + \ldots + a_{D,h} x_D.
\end{equation}

We will need to find analogues of the above type of representation in higher degree $s \geq 3$.  Heuristically, we will wish to represent the symbol $[\chi]_{\Xi^{(s-1,r_*)}_\DR([N])}$ of a nilcharacter $\chi$ on $[N]$ of degree-rank $(s-1,r_*)$ (which will ultimately depend on a parameter $h$, though we will not need this parameter in the current discussion) heuristically as
\begin{equation}\label{chih-abstract}
[\chi]_{\Xi^{(s-1,r_*)}_\DR([N])} \approx \eta(\F)
\end{equation}
where $\F = (\xi_{i,j})_{1 \leq i \leq s-1; 1 \leq j \leq D_i}$ is a \emph{horizontal frequency vector} of frequencies $\xi_{i,j} \in \ultra \T$ associated to a \emph{dimension vector} $\vec D = (D_1,\ldots,D_{s-1})$, and $\eta$ is a \emph{vertical frequency} that generalises \eqref{etaxd}, but whose precise form we are not yet ready to describe precisely.  We then say that the triple $(\vec D, \eta, \F)$ forms a \emph{total frequency representation} of $\chi$.

In the previous paper \cite{u4-inverse} that treated the $s=3$ case, such a representation was implicitly used via the description of degree-rank $(2,2)$ nilcharacters $\chi_h$ as essentially being bracket quadratic phases $e(\sum_{j=1}^J \{ \alpha_{h,j} n \} \beta_{h,j} n)$ modulo lower order terms (and ignoring the issue of vector-valued smoothing for now).  In our current language, this would correspond to a dimension vector $\vec D = (2J,0)$ and a horizontal frequency vector of the form $(\alpha_{h,1},\ldots,\alpha_{h,J},\beta_{h,1},\ldots,\beta_{h,J})$, and  a certain vertical frequency $\eta$ depending only on $J$ that we are not yet ready to describe explicitly here.  Bracket-calculus identities such as \eqref{brackalg} could then be used to manipulate such a universal frequency representation into  a suitably ``regularised'' form.

In principle, one could also use bracket calculus to extract the symbol of $\chi_h$ in terms of frequencies such as $\alpha_{h,j}$ and $\beta_{h,j}$ for higher values of $s$.  However, as we are avoiding the use of bracket calculus machinery here, we will proceed instead using the language of nilmanifolds, and in particular by lifting the nilmanifold $G_h/\Gamma_h$ up to a \emph{universal nilmanifold} in order to obtain a suitable space (independent of $h$) in which to detect relationships between frequencies such as $\alpha_{h,j}, \beta_{h,j}$.  In some sense, this universal nilmanifold will play the role that the unit circle $\T$ plays in Fourier analysis.

We first define the notion of universal nilmanifold that we need.  

\begin{definition}[Universal nilmanifold]\label{universal-nil}  A \emph{dimension vector} is a tuple \[ \vec D = (D_1,\ldots,D_{s-1}) \in \N^{s-1} \] of standard natural numbers.  Given a dimension vector, we define the \emph{universal nilpotent group} $G^{\vec D} = G^{\vec D, \leq (s-1,r_*)}$ of degree-rank $(s-1,r_*)$ to be the Lie group generated by formal generators $e_{i,j}$ for $1 \leq i \leq s-1$ and $1 \leq j \leq D_i$, subject to the following constraints:
\begin{itemize}
\item Any $(m-1)$-fold iterated commutator of $e_{i_1,j_1},\ldots,e_{i_m,j_m}$ with $i_1+\ldots+i_m \geq s$ is trivial.
\item Any $(m-1)$-fold iterated commutator of $e_{i_1,j_1},\ldots,e_{i_m,j_m}$ with $i_1+\ldots+i_m = s-1$ and $m \geq r+1$ is trivial.
\end{itemize}
We give this group a degree-rank filtration $(G^{\vec D}_{(d,r)})_{(d,r) \in \DR}$ by defining $G^{\vec D}_{(d,r)}$ to be the Lie group generated by $(m-1)$-fold iterated commutators of $e_{i_1,j_1},\ldots,e_{i_m,j_m}$ with $1 \leq i_l \leq s-1$ and $1 \leq j_l \leq D_{i_l}$ for all $1 \leq l \leq n$ for which either $i_1+\ldots+i_m > d$, or $i_1+\ldots+i_m=d$ and $m \geq r$.  It is not hard to verify that this is indeed a filtration of degree-rank $\leq (s-1,r_*)$.  We then let $\Gamma^{\vec D}$ be the discrete group generated by the $e_{i,j}$ with $1 \leq i \leq s-1$ and $1 \leq j \leq D_i$, and refer to $G^{\vec D}/\Gamma^{\vec D}$ as the \emph{universal nilmanifold} with dimension vector $\vec D$.  

A \emph{universal vertical frequency} at dimension vector $\vec D$ is a continuous homomorphism $\eta: G^{\vec D}_{(s-1,r_*)} \to \R$ which sends $\Gamma^{\vec D}_{(s-1,r_*)}$ to the integers (i.e. a filtered homomorphism from $G^{\vec D}_{(s-1,r_*)} / \Gamma^{\vec D}_{(s-1,r_*)}$ to $\T$).
\end{definition}

\emph{Remark.}  One can give an explicit basis for this nilmanifold in terms of certain iterated commutators of the $e_{i,j}$, following \cite{leibman,mks}.  This can then be used to relate nilcharacters to bracket polynomials, as in \cite{leibman}, and it is then possible to develop enough of a ``bracket calculus'' to substitute for some of the nilpotent algebra performed in this paper.  However, we will not proceed by such a route here (as it would make the paper even longer than it currently is), and in fact will not need an explicit basis for universal nilmanifolds at all.

\begin{example} The unit circle with the degree $\leq d$ filtration (see Example \ref{polyphase}) is isomorphic to the universal nilmanifold $G^{(0,\ldots,0,1),\leq (d,1)}$, thus for instance the unit circle with the lower central series filtration is isomorphic to $G^{(1),\leq (1,1)}$.  A universal vertical frequency for any of these nilmanifolds is essentially just a map of the form $\eta: x \mapsto nx$ for some integer $n$.
\end{example}

\begin{example}  The Heisenberg group \eqref{heisen} (with the lower central series filtration) is the universal nilpotent group $G^{(2,0)} = G^{(2,0), \leq (2,2)}$ of degree-rank $(2,2)$ (after identifying $e_1,e_2$ with $e_{1,1}$ and $e_{1,2}$ respectively), and the Heisenberg nilmanifold $G/\Gamma$ is the corresponding universal nilmanifold $G^{(2,0)}/\Gamma^{(2,0)}$.  If we reduce the degree-rank from $(2,2)$ to $(2,1)$, then the commutator $[e_1,e_2]$ now trivialises, and $G^{(2,0), \leq (2,1)}$ collapses to the abelian Lie group $\R^2 \equiv G^{2, \leq (1,1)}$, with universal nilmanifold $\T^2$.  

If, instead of the lower central series filtration, one gives the Heisenberg group \eqref{heisen} the filtration used in 	Example \ref{abn} to model the sequence \eqref{abn-eq}, then this group is isomorphic to the universal nilpotent group $G^{(1,1), \leq (3,2)}$, with the two generators $e_1, e_2$ of the Heisenberg group now being interpreted as $e_{1,1}$ and $e_{2,1}$ respectively.
\end{example}

\begin{example}
Consider the universal nilpotent group $G^{(D_1,D_2,D_3),\leq (3,3)}$.  This group is generated by ``degree $1$'' generators $e_{1,1},\ldots,e_{1,D_1}$, ``degree $2$'' generators $e_{2,1},\ldots,e_{2,D_2}$, and ``degree $3$'' generators $e_{3,1},\ldots,e_{3,D_3}$, with any iterated commutator of total degree exceeding three vanishing (thus for instance the degree $3$ generators are central, and the degree $2$ generators commute with each other).  If one drops the degree-rank from $(3,3)$ to $(3,2)$, then all triple commutators of degree $1$-generators, such as $[[e_{1,i}, e_{1,j}],e_{1,k}]$ now vanish, reducing the dimension of the nilpotent group.  Dropping the degree-rank further to $(3,1)$ also eliminates the commutators of degree $1$ and degree $2$ generators (thus making the degree $2$ generators central).  Finally, dropping the degree-rank to $(3,0)$ eliminates the degree $3$ generators completely, and indeed $G^{(D_1,D_2,D_3), \leq (3,0)}$ is isomorphic to $G^{(D_1,D_2), \leq (2,2)}$.
\end{example}

\begin{example}
The free $s$-step nilpotent group on $D$ generators, in our notation, becomes $G^{(D,0,\ldots,0), \leq (s,s)}$.  We may thus view the universal nilpotent groups $G^{\vec D, \leq (d,r)}$ as generalisations of the free nilpotent groups, in which some of the generators are allowed to be weighted to have degrees greater than $1$, and there is an additional rank parameter to cut down some of the top-order behaviour.
\end{example}

It will be an easy matter to lift a nilcharacter $\chi$ from a general degree-rank $\leq (s-1,r_*)$ nilmanifold $G/\Gamma$ to a universal nilmanifold $G^{\vec D}/\Gamma^{\vec D}$ for some sufficiently large dimension vector $\vec D$ (see Lemma \ref{existence} below).  Once one does so, we will need to extract the various ``top order frequencies'' present in that nilcharacter.  For instance, if $s=4$ and $\chi$ is (some vector-valued smoothing of) the degree $3$ phase
$$ n \mapsto e( \{ \alpha n \} \beta n^2 + \gamma n^3 + \delta n^2  + \{ \epsilon n \} \mu n + \nu n + \theta )$$
then we will need to extract out the ``degree $3$'' frequency $\gamma$, the ``degree $2$'' frequency $\beta$, and the ``degree $1$'' frequency $\alpha$.  (The remaining parameters $\delta,\epsilon,\mu,\nu,\theta$ only contribute to terms of degree strictly less than $3$, and will not need to be extracted.)  

As it turns out, the degree $i$ frequencies will most naturally live in the \emph{$i^\th$ horizontal torus} of the relevant universal nilmanifold; we now pause to define these torii precisely.  (These torii also implicitly appeared in \cite[Appendix A]{green-tao-arithmetic-regularity}.)

\begin{definition}[Horizontal Taylor coefficients]\label{horton}  Let $G = (G, (G_{(d,r)})_{(d,r) \in \DR})$ be a degree-rank-filtered nilpotent group.  For every $i \geq 0$, define the \emph{$i^{\th}$ horizontal space} $\Horiz_i(G)$ to be the abelian group 
$$ \Horiz_i(G) := G_{(i,1)} / G_{(i,2)},$$
with the convention that $G_{(d,r)} := G_{(d+1,0)}$ if $r>d$ (so in particular, $G_{(1,2)} = G_{(2,0)}$).

For any polynomial map $g \in \poly(\Z_\N \to G_\N)$, we define the \emph{$i^{th}$ horizontal Taylor coefficient} $\Taylor_i(g) \in \Horiz_i(G)$ to be the quantity
$$ \Taylor_i(g) := \partial_{1} \ldots \partial_{1} g(n) \mod G_{(i,2)}$$
for any $n \in \Z$.  Note that this map is well-defined since $\partial_{1} \ldots \partial_{1} g$ takes values in $G_{(i,1)}$ and has first derivatives in $G_{(i+1,1)}$ and hence in $G_{(i,2)}$.

If $\Gamma$ is a subgroup of $G$, we define
$$ \Horiz_i(G/\Gamma) := \Horiz_i(G) / \Horiz_i(\Gamma)$$
and for a polynomial orbit $\orbit \in \poly(\Z_\N \to (G/\Gamma)_\N) := \poly(\Z_\N \to G_\N) / \poly(\Z_\N \to \Gamma_\N)$, we define the \emph{$i^{\th}$ horizontal Taylor coefficient} $\Taylor_i(\orbit) \in \Horiz_i(G/\Gamma)$ to be the quantity defined by
$$ \Taylor_i( g \Gamma ) := \Taylor_i(g) \mod \Horiz_i(\Gamma)$$
for any $g \in \poly(\Z_\N \to G_\N)$; it is easy to see that this quantity is well-defined.

These concepts extend to the ultralimit setting in the obvious manner; thus for instance, if $\orbit \in \ultra \poly(H_\N \to (G/\Gamma)_\N)$, then $\Taylor_i(\orbit)$ is an element to $\ultra \Horiz_i(G/\Gamma)$.
\end{definition}

If $G/\Gamma$ is a degree-rank filtered nilmanifold, it is easy to see that the horizontal spaces $\Horiz_i(G)$ are abelian Lie groups, and that $\Horiz_i(\Gamma)$ is a sublattice of $\Horiz_i(G)$, so $\Horiz_i(G/\Gamma)$ is a torus, which we call the \emph{$i^{\th}$ horizontal torus} of $G/\Gamma$.\vspace{11pt}

\emph{Remark.}  The above definition can be generalised by replacing the domain $\Z$ with an arbitrary additive group $H = (H,+)$.  In that case, the Taylor coefficient $\Taylor_i(g)$ is not a single element of  $\Horiz_i(G)$, but is instead a map $\Taylor_i(g): H^i \to \Horiz_i(G)$ 
defined by the formula
$$ \Taylor_i(g)(h_1,\ldots,h_k) := \partial_{h_1} \ldots \partial_{h_k} g(n) \mod G_{(i,2)}$$
for $h_1,\ldots,h_k \in H$. Using Corollary \ref{collox} we easily see that this map is symmetric and multilinear; thus for instance when $H=\Z$ we have
$$ \Taylor_i(g)(h_1,\ldots,h_k) = h_1 \ldots h_k \Taylor_i(g).$$
However, we will not need this generalisation here.

A further application of Corollary \ref{collox} shows that the map $g \mapsto \Taylor_i(g)$ is a homomorphism.  As a corollary, we see that any translate $g(\cdot+h) = (\partial_h g) g$ of $g$ will have the same Taylor coefficients as $g$: $\Taylor_i(g(\cdot+h)) = \Taylor_i(g)$.

\begin{example}  Consider the unit circle $G/\Gamma = \T$ with the degree $\leq d$ filtration (see Example \ref{polyphase}).  Then the $d^\th$ horizontal torus is $\T$, and all other horizontal tori are trivial.  If $\alpha_0,\ldots,\alpha_d \in \ultra \R$, then the map $\orbit: n \mapsto \alpha_0 + \ldots + \alpha_d n^d \mod 1$ is a polynomial orbit in $\ultra \poly(\Z_\N \to \T_\N)$, and the $d^{th}$ horizontal Taylor coefficient is the quantity $d! \alpha_d \mod 1$ from $\ultra \Z^d$ to $\ultra\T$.  (All other horizontal Taylor coefficients are of course trivial.)  Thus we see that the horizontal coefficient captures most of the top order coefficient $\alpha_d$, but totally ignores all lower order terms.
\end{example}

\begin{example}
Let $G=G^{(2,1)}=G^{(2,1),\leq (2,2)}$ be the universal nilpotent group of degree-rank $(2,2)$. Thus $G$ is generated
by $e_{1,1},e_{1,2},e_{2,1}$, with relations 
\[
[[e_{1,1},e_{1,2}], e_{1,i}]=[e_{1,i},e_{2,1}]=1 \quad \text{ for $i=1,2$}.
\]
and with the degree-rank filtration
\begin{align*}
G_{(0,0)}=G_{(1,0)}=G_{(1,1)}&=G \\
G_{(2,0)}=G_{(2,1)}&= \langle [e_{1,1},e_{1,2}], e_{2,1}\rangle_\R \\
G_{(2,2)}&=\langle [e_{1,1},e_{1,2}] \rangle_\R
\end{align*}
and the lattice
$$ \Gamma = \Gamma^{(2,2)} = \Gamma^{(2,2), \leq (2,1)} := \langle e_{1,1},e_{1,2},e_{2,1} \rangle.$$
Let $\alpha,\beta,\gamma \in \ultra \R$, and consider the orbit $\orbit \in \ultra\poly(\Z_\N \to (G/\Gamma)_\N)$ defined by the formula
\[
\orbit(n):=e^{n \alpha}_{1,1} e_{1,2}^{n \beta} e_{2,1}^{n^2 \gamma};
\]
this is polynomial by Example \ref{lazard-ex}.
Then 
\[
\Taylor_1(g) = \partial_{1} g(n) \mod \ultra G_{(2,0)}  = e^{\alpha}_{1,1}e_{1,2}^{\beta}  \mod \ultra G_{(2,0)},
\]
and
\[
\Taylor_2(g) =  e_{2,1}^{2\gamma} \mod \ultra G_{(2,2)}.
\]
Then $\Taylor_0(g(n)\ultra \Gamma) = g(n)\ultra \Gamma$,
\[
\Taylor_1(g\ultra \Gamma)  = e^{\alpha}_{1,1}e_{1,2}^{\beta}  \mod G_{(2,0)}\ultra \Gamma
\]
and 
\[
\Taylor_2(g\ultra \Gamma) =  e_{2,1}^{2\gamma} \mod \ultra G_{(2,2)}\Gamma_{(2,0)}.
\]
\end{example}

\begin{example}\label{heist}  Let $G/\Gamma$ be the Heisenberg nilmanifold \eqref{heisen} with the lower central series filtration. Thus $G/\Gamma$ is a degree $\leq 2$ nilmanifold, which can then be viewed as a degree-rank $\leq (2,2)$ nilmanifold by Example \ref{inclusions}.  The first horizontal torus $\Horiz_1(G/\Gamma)$ is isomorphic to the $2$-torus $\T^2$, with generators given by $e_1, e_2 \mod G_2 \Gamma$.  The second horizontal torus $\Horiz_2(G/\Gamma)$ is trivial, since $G_{(2,1)} = [G,G]$ is equal to $G_{(2,0)} = G_2$.  If $\orbit \in \ultra \poly(\Z_\N \to (G/\Gamma)_\N)$ is the polynomial orbit $\orbit: n \mapsto e_2^{\beta n} e_1^{\alpha n} \ultra \Gamma$, then the first Taylor coefficient is the quantity $(\alpha, \beta)$.  Note also that if one modified the polynomial orbit by a further factor of $[e_1,e_2]^{\gamma n^2 + \delta n + \epsilon}$, this would not impact the Taylor coefficients at all.  Thus we see that the Taylor coefficients only capture the frequencies associated to raw generators such as $e_1$ and $e_2$, and not to commutators such as $[e_1,e_2]$.
\end{example}

\begin{example}  Now consider the Heisenberg group \eqref{heisen} with the filtration used in Example \ref{abn} to model the sequence \eqref{abn-eq}.  This is now a degree $\leq 3$ nilmanifold, whose first horizontal torus $\Horiz_1(G/\Gamma)$ is isomorphic to the one-torus $\T$ with generator $e_2 \mod G_{(2,0)} \Gamma$, whose second horizontal torus $\Horiz_2(G/\Gamma)$ is isomorphic to the one-torus $\T$ with generator $e_1 \mod G_{(2,2)} \Gamma_{(2,1)}$, and whose third horizontal torus $\Horiz_3(G/\Gamma)$ is trivial.  If $\orbit \in \ultra \poly(\Z_\N \to (G/\Gamma)_\N)$ is the polynomial orbit $\orbit: n \mapsto e_2^{\beta n} e_1^{\alpha n^2} \ultra \Gamma$, then the first Taylor coefficient is the linear limit map $n \mapsto \beta n \mod 1$, and the second Taylor coefficient is the quantity  $2! \alpha \mod 1$.
\end{example}

We now have enough notation to be able to formally assign frequencies to a nilcharacter, by means of a package of data which we shall call a \emph{representation}.

\begin{definition}[Representation]\label{representation-def}  Let $\chi \in L^\infty[N]$ be a nilcharacter of degree-rank $\leq (s-1,r_*)$.  A \emph{representation} of $\chi$ is a collection of the following data:
\begin{enumerate}
\item A filtered nilmanifold $G/\Gamma$ of degree-rank $\leq (s-1,r_*)$;
\item A filtered nilmanifold $G_0/\Gamma_0$ of degree-rank $\leq (s-1,r_*-1)$;
\item A function $F \in \Lip(\ultra(G/\Gamma \times G_0/\Gamma_0) \to \overline{S^\omega})$;
\item Polynomial orbits $\orbit \in \ultra \poly(\Z_\N \to (G/\Gamma)_\N)$ and $\orbit_0 \in \ultra \poly(\Z_\N \to (G_0/\Gamma_0)_\N)$;
\item A dimension vector $\vec D = (D_1,\ldots,D_{s-1}) \in \N^{s-1}$;
\item A universal vertical frequency $\eta: G^{\vec D}_{(s-1,r_*)} \to \R$ at dimension $\vec D$ on the universal nilmanifold $G^{\vec D}/\Gamma^{\vec D}$ of degree-rank $(s-1,r_*)$;
\item A filtered homomorphism $\phi: G^{\vec D}/\Gamma^{\vec D} \to G/\Gamma$ (see Definition \ref{quot});
\item A \emph{horizontal frequency vector} $\F = (\xi_{i,j})_{1 \leq i \leq s-1; 1 \leq j \leq D_i}$ of frequencies $\xi_{i,j} \in \ultra\T$.
\end{enumerate}
which obeys the following properties:
\begin{enumerate}
\item For all $n \in [N]$, one has
\begin{equation}\label{chin}
\chi(n) = F( \orbit(n), \orbit_0(n)).
\end{equation}
\item For every $t \in G^{\vec D}_{(s-1,r_*)}$, all $x \in G/\Gamma$, and $x_0 \in G_0/\Gamma_0$, one has
\begin{equation}\label{vert}
F( \phi(t) x, x_0 ) = e( \eta(t) ) F(x,x_0).
\end{equation}
\item For every $1 \leq i \leq s-1$, one has
\begin{equation}\label{taylor}
 \Taylor_i(\orbit) =  \pi_{\Horiz_i(G/\Gamma)}\left(\phi( \prod_{j=1}^{D_i} e_{i,j}^{\xi_{i,j}} )\right),
\end{equation}
where $\pi_{\Horiz_i(G/\Gamma)}: G_{i} \to \Horiz_i(G/\Gamma)$ is the projection map; observe that the right-hand side is well-defined even though $\xi_{i,j}$ is only defined modulo $1$.
\end{enumerate}
We call the triplet $(\vec D, \F, \eta)$ a \emph{total frequency representation} of the nilcharacter $\chi$.
\end{definition}

This is a rather complicated definition, and we now illustrate it with a number of examples.  We begin with the $s=2$, $r_*=1$ case, taking $\chi$ to be the degree-rank $(1,1)$ nilcharacter
$$ \chi(n) := e( \xi n + \theta )$$
for some $\xi, \theta \in \ultra \R$. Let $D_1 \geq 1$ be an integer, let $\F = (\xi_{1,1},\ldots,\xi_{1,D_1}) \in \ultra \T^{D_1}$ be a collection of frequencies, and let $\eta: \R^{D_1} \to \R$ be the universal vertical frequency $\eta(x_1,\ldots,x_{D_1}) := a_1 x_1 + \ldots + a_{D_1} x_{D_1}$ for some integers $a_1,\ldots,a_{D_1} \in \Z$.  Then $((D_1),\F,\eta)$ will be a total frequency representation of $\chi$ if $\xi = a_1 \xi_{1,1} + \ldots + a_{D_1} \xi_{1,D_1}$.  Indeed, in that case, one can take 
$G/\Gamma = \T$ (with the degree-rank $\leq (1,1)$ filtration, see Example \ref{dr-f}), $G_0/\Gamma_0$ to be trivial, $F$ equal to the exponential function $(x,()) \mapsto e(x)$, $\phi: \T^{D_1} \to \T$ to be the filtered homomorphism
$$ \phi(x_1,\ldots,x_{D_1}) := a_1 x_1 + \ldots + a_{D_1} x_{D_1},$$
and $\orbit \in \ultra \poly(\Z_\N \to \T_\N)$ to be the orbit $n \mapsto \xi n + \theta \mod 1$.  This should be compared with \eqref{chih-abstract} and the discussion at the start of the section.

For a slightly more complicated example, we take $s=3, r_* = 1$, and let $\chi$ be the degree-rank $(2,1)$ nilcharacter
$$ \chi(n) := e( \alpha n^2 + \beta n + \gamma ).$$
We let $D_2 \geq 1$ be an integer, set $D_1 := 0$, let $\F = ((),(\xi_{2,1},\ldots,\xi_{2,D_2})) \in \ultra \T^{0} \times \ultra \T^{D_2}$ be a collection of frequencies, and let $\eta: \R^{D_2} \to \R$ be the universal vertical frequency $\eta(x_1,\ldots,x_{D_2}) := a_1 x_1 + \ldots + a_{D_2} x_{D_2}$ for some integers $a_1,\ldots,a_{D_2} \in \Z$.  Then $((0,D_2), \F, \eta)$ will be a total frequency representation of $\chi$ if $\xi = a_1 \xi_{2,1} + \ldots + a_{D_2} \xi_{2,D_2}$ (cf. \eqref{chih-abstract}).  Indeed, we can take $G/\Gamma = \T$ with the degree-rank $\leq (2,1)$ filtration (see Example \ref{dr-f}), $G_0/\Gamma_0 = \T$ with the degree-rank $\leq (1,1)$ filtration, the orbit
$$ \orbit(n) := ( \alpha n^2 \mod 1, \beta n + \gamma \mod 1 )$$
and $F: G/\Gamma \times G_0/\Gamma_0 \to S^1$ to be the function
$$ F(x, y) := e(x) e(y),$$
and $\phi: \T^{D_2} \to \T$ to be the filtered homomorphism
$$ \phi(x_1,\ldots,x_{D_1}) := a_1 x_1 + \ldots + a_{D_1} x_{D_1}.$$
Note how the lower order terms $\beta_n + \gamma$ in the phase of $\chi$ are shunted off to the lower degree-rank nilmanifold $G_0/\Gamma_0$ and thus do not interact at all with the data $\F, \eta$.  In this particular case, this shunting off was unnecessary, and one could have easily folded these lower order terms into the dynamics of the primary nilmanifold $G/\Gamma$; but in the next example we give, the lower order behaviour does genuinely need to be separated from the top order behaviour by placing it in a separate nilmanifold.

We now turn to a genuinely non-abelian example of a universal representation.  For this, we take $s=3$, $r_*=2$, and let $\chi$ be a degree-rank $(2,2)$ nilcharacter that is a suitable vector-valued smoothing of the bracket polynomial phase
$$ n \mapsto e( \{ \alpha n \} \beta n + \gamma n^2 ).$$
We can express this nilcharacter as
$$ \chi(n) = F( \orbit(n), \orbit_0(n) ),$$
where $\orbit \in \ultra \poly(\Z_\N \to (G/\Gamma)_\N)$ is the orbit
$$ \orbit(n) := e_2^{\beta n} e_1^{\alpha n} \Gamma$$
into the Heisenberg nilmanifold \eqref{heisen} (which we give the degree-rank $\leq (2,2)$ filtration), $\orbit_0 \in \ultra \poly(\Z_\N \to (G/\Gamma)_\N)$ is the orbit
$$ \orbit_0(n) := \gamma n^2 \mod 1$$
into the unit circle $G_0/\Gamma_0 = \T$ (which we give the degree-rank $\leq (2,1)$ filtration, see Example \ref{dr-f}), and $F$ is a suitable vector-valued smoothing of the map
$$ ( e_1^{t_1} e_2^{t_2} [e_1,e_2]^{t_{12}} \Gamma, y ) \mapsto e( t_{12} ) e(y)
$$
for $t_1, t_2, t_{12} \in I_0$.  By Example \ref{heist}, we have $\Taylor_1(\orbit) = (\alpha \mod 1,\beta \mod 1)$ and $\Taylor_2(\orbit)$ is trivial. 

Now let $D_1 \geq 1$ be an integer, set $D_2 := 0$, let $\F = ((\xi_{1,1},\ldots,\xi_{1,D_1}),()) \in \ultra \T^{D_1} \times \ultra \T^{0}$ be a collection of frequencies.  The subgroup $G^{(D_1,0)}_{(2,2)}$ of the universal nilmanifold $G^{(D_1,0)} = G^{(D_1,0),\leq (2,2)}$ is then the abelian Lie group generated by the commutators $[e_{1,i},e_{1,j}]$ for $1 \leq i < j \leq D_1$.  We let $a_1,\ldots,a_{D_1},b_1,\ldots,b_{D_1} \in \Z$ be integers, and let $\phi: G^{(D_1,0)}/\Gamma^{(D_1,0)} \to G/\Gamma$ be the filtered homomorphism that maps $e_{1,i}$ to $e_1^{a_i} e_2^{b_i}$ for $i=1,\ldots,D_1$, thus
\begin{align*}
\phi( &\prod_{i=1}^{D_1} e_{1,i}^{t_i} \prod_{1 \leq i < j \leq D_1} [e_{1,i},e_{1,j}]^{t_{i,j}} \Gamma^{(D_1,0)} ) \\
&= \prod_{i=1}^{D_1} (e_1^{a_1} e_2^{b_i})^{t_i} \prod_{1 \leq i < j \leq D_1} [e_1^{a_i} e_2^{b_i}, e_1^{a_j} e_2^{b_j}]^{t_{i,j}} \Gamma \\
&= e_1^{\sum_{i=1}^{D_1} a_i t_i} e_2^{\sum_{i=1}^{D_1} b_i t_i} [e_1,e_2]^{-\sum_{i=1}^{D_1} a_i b_i \binom{t_i}{2} - \sum_{1 \leq i < j \leq d} b_i a_j t_i t_j + \sum_{1 \leq i < j \leq d} (a_i b_j - a_j b_i)t_{i,j}} \Gamma.
\end{align*}
Let us now see what conditions are required for $((D_1,0),\eta,\F)$ to be a total frequency representation of $\chi$.  The condition \eqref{taylor} becomes the constraints
\begin{align*}
\alpha &= \sum_{i=1}^{D_1} a_i \xi_{1,i} \\
\beta &= \sum_{i=1}^{D_1} b_i \xi_{1,i},
\end{align*}
while the condition \eqref{vert} becomes
\begin{equation}\label{etaij}
 \eta( [e_{1,i}, e_{1,j}] ) = a_i b_j - a_j b_i
\end{equation}
for all $1 \leq i < j \leq D_1$, or equivalently
$$ \eta( \prod_{1 \leq i < j \leq D_1} [e_{1,i}, e_{1,j}]^{t_{i,j}} ) = \sum_{1 \leq i < j \leq D_1} (a_i b_j - a_j b_i) t_{i,j}$$
Conversely, with these constraints we obtain a total frequency representation of $\chi$ by $((D_1,0),\eta,\F)$.  This should be compared with the heuristic \eqref{chih-abstract}.  (Note from \eqref{brackalg} that the top order component $\{\alpha n \} \beta n$ of $\chi$ is morally anti-symmetric in $\alpha,\beta$ modulo lower order terms, which is consistent with the anti-symmetry observed in \eqref{etaij}.)  Note also that the term $\gamma n^2$, which has lesser degree-rank than the top order term $\{ \alpha n \} \beta n$, plays no role, due to it being shunted off to the lower degree-rank nilmanifold $G_0/\Gamma_0$.  If instead we placed this term as part of the principal nilmanifold, then this would create a non-trivial second Taylor coefficient $\Taylor_2(\orbit)$ which would then require a non-zero value of $D_2$ in order to recover a total frequency representation.  Thus we see that in order to neglect terms of lesser degree-rank (but equal degree) it is necessary to create the secondary nilmanifold $G_0/\Gamma_0$ as a sort of ``junk nilmanifold'' to hold all such terms.

We make the easy remark that every nilcharacter $\chi$ of degree-rank $\leq (s-1,r_*)$ has at least one representation.

\begin{lemma}[Existence of representation]\label{existence}  Let $\chi$ be a nilcharacter of degree-rank $(s-1,r_*)$ on $[N]$.  Then there exists at least one total frequency representation $(\vec D, \F, \eta)$ of $\chi$.
\end{lemma}

\begin{proof} By definition, $\chi = F \circ \orbit$ for some degree-rank $\leq (s-1,r_*)$ nilmanifold $G/\Gamma$, some $\orbit \in \ultra \poly(\Z_\N \to (G/\Gamma)_\N)$, and some $F \in \Lip(\ultra(G/\Gamma))$ with a vertical frequency.  For each $1 \leq i \leq s-1$, let $f_{i,1},\ldots,f_{i,D_i}$ be a basis of generators for $\Gamma_{i}$, and let $\vec D := (D_1,\ldots,D_{s-1})$ be the associated dimension vector.  Then we have a filtered homomorphism $\phi: G^{\vec D} \to G$ which maps $e_{i,j}$ to $f_{i,j}$ for all $1 \leq i \leq s-1$ and $1 \leq j \leq D_i$.  It is easy to see that $\phi$ is surjective from $G^{\vec D}_{i}$ to $G_{i}$ for each $i$, and so the map $\pi_{\Horiz_i(G/\Gamma)} \circ \phi$ is surjective from $G^{\vec D}_{i}$ to $\Horiz_i(G/\Gamma)$.  It is now an easy matter to locate frequencies $\xi_{i,j}$ obeying \eqref{taylor}, and the vertical frequency property of $F$ can be pulled back via $\phi$ to give \eqref{vert}.  Setting $G_0/\Gamma_0$ to be trivial, we obtain the claim.
\end{proof}

To conclude this section, we now give some basic facts about total frequency representations.  These facts will not actually be used in this paper, but may serve to consolidate one's intuition about the nature of these representations.  We first observe some linearity in the vertical frequency $\eta$.

\begin{lemma}[Linearity] Suppose that $\chi, \chi'$ are two nilcharacters of degree-rank $(s-1,r_*)$ on $[N]$ that have total frequency representations $(\vec D, \F, \eta)$ and $(\vec D, \F, \eta')$ respectively.  Then $\overline{\chi}$ has a total frequency representation $(\vec D, \F, -\eta)$, and $\chi \otimes \chi'$ has a total frequency representation $(\vec D, \F, \eta+\eta')$.
\end{lemma}

\begin{proof} This is a routine matter of chasing down the definitions, and noting that nilmanifolds, polynomial orbits, etc. behave well with respect to direct sums.
\end{proof}

\begin{lemma}[Triviality]  Suppose that $\chi$ is a nilcharacter of degree-rank $(s-1,r_*)$ on $[N]$ that has a total frequency representation $(\vec D, \F, 0)$.  Then $\chi$ is a nilsequence of degree-rank $\leq (s-1,r_*-1)$ \textup{(}i.e. $[\chi]_{\Symb^{(s-1,r_*)}_\DR([N])} = 0$\textup{)}.
\end{lemma}

\begin{proof}  By construction, we have
$$ \chi(n) = F( \orbit(n), \orbit_0(n) )$$
for some limit polynomial orbits $\orbit \in \ultra \poly(\Z_\N \to (G/\Gamma)_\N)$, $\orbit_0 \in \ultra \poly(\Z_\N \to (G_0/\Gamma_0)_\N)$ into filtered nilmanifolds $G/\Gamma, G_0/\Gamma_0$ of degree-rank $\leq (s-1,r_*)$ and $\leq (s-1,r_*-1)$ respectively, where $F \in \Lip(\ultra(G/\Gamma \times G_0/\Gamma_0) \to \overline{S^\omega})$.  Furthermore, there exists a filtered homomorphism $\phi: G^{\vec D}/\Gamma^{\vec D} \to G/\Gamma$ such that
\eqref{taylor} holds, and such that
\begin{equation}\label{flat}
F( \phi(t) x, x_0 ) = F(x,x_0).
\end{equation}
for all $t \in G^{\vec D}_{(s-1,r_*)}$.

Let $T$ be the closure of the set $\{ \phi(t) \mod \Gamma_{(s-1,r_*)}: t \in G^{\vec D}_{(s-1,r_*)}\}$; this is a subtorus of the torus $G_{(s-1,r_*)}/\Gamma_{(s-1,r_*)}$, and thus acts on $G/\Gamma$.  As $F$ is continuous and obeys the invariance \eqref{flat}, we see that $F$ is $T$-invariant; we may thus quotient out by $T$ and assume that $T$ is trivial.  In particular, $\phi$ now annihilates $G^{\vec D}_{(s-1,r_*)}$.

We give $G$ a new degree-rank filtration $(G'_{(d,r)})_{(d,r) \in \DR}$ (smaller than the existing filtration $(G_{(d,r)})_{(d,r) \in \DR}$), by defining $G'_{(d,r)}$ to be the connected subgroup of $G$ generated by $G_{(d,r+1)}$ (recalling the convention $G_{(d,r)} := G_{(d+1,0)}$ when $r > d$) together with the image $\phi( G^{\vec D}_{(d,r)} )$ of $G^{\vec D}_{(d,r)}$.  It is easy to see that this is still a filtration, and that $G/\Gamma$ remains a filtered nilmanifold with this filtration, but now the degree-rank is $\leq (s-1,r_*-1)$ rather than $\leq (s-1,r_*)$. Furthermore, from \eqref{taylor} we see that $\orbit$ is still a polynomial orbit with respect to this new filtration.  As such, $\chi$ is a nilsequence of degree-rank $\leq (s-1,r_*-1)$ as required.  
\end{proof}

Combining the above two lemmas we obtain the following corollary.

\begin{corollary}[Representation determines symbol]  Suppose that $\chi, \chi'$ are two nilcharacters of degree-rank $(s-1,r_*)$ on $[N]$ that have a common total frequency representation $(\vec D, \F, \eta)$.  Then $\chi, \chi'$ are equivalent.  In other words, the symbol $[\chi]_{\Xi^{(s-1,r_*)([N])}}$ depends only on $(\vec D, \F, \eta)$.
\end{corollary}

Note that the above results are consistent with the heuristic \eqref{chih-abstract}.

\section{Linear independence and the sunflower lemma}\label{reg-sec}

A basic fact of linear algebra is that every finitely generated vector space is finite-dimensional.  In particular, if $v_1,\ldots,v_l$ are a finite collection of vectors in a vector space $V$ over a field $k$, then there exists a finite linearly independent set of vectors $v'_1,\ldots,v'_{l'}$ in $V$ such that each of the vectors $v_1,\ldots,v_l$ is a linear combination (over $k$) of the $v'_1,\ldots, v'_{l'}$.  Indeed, one can take $v'_1,\ldots,v'_{l'}$ to be a set of vectors generating $v_1,\ldots,v_l$ for which $l'$ is minimal, since any linear relation amongst the $v'_1,\ldots,v'_{l'}$ can be used to decrease\footnote{Indeed, one can recast this argument as a rank reduction argument instead of a minimal rank argument, for the same reason that the principle of infinite descent is logically equivalent to the well-ordering principle.  In this infinitary (ultralimit) setting, there is very little distinction between the two approaches, although the minimality approach allows for slightly more compact notation and proofs.  But in the finitary setting, it becomes significantly more difficult to implement the minimality approach, and the rank reduction approach becomes preferable.  See \cite{u4-inverse} for finitary ``rank reduction'' style arguments analogous to those given here.} the ``rank'' $l'$, contradicting minimality (cf. the proof of classical Steinitz exchange lemma in linear algebra).

We will need analogues of this type of fact for frequencies $\xi_1,\ldots,\xi_l$ in the limit unit circle $\ultra \T$. However, this space is not a vector space over a field, but is merely a module over a commutative ring $\Z$.  As such, the direct analogue of the above statement fails; indeed, any standard rational in $\ultra \T$, such as $\frac{1}{2} \mod 1$, clearly cannot be represented as a linear combination (over $\Z$) of a finite collection of frequencies in $\ultra \T$ that are linearly independent over $\Z$.  

However, the standard rationals are the \emph{only} obstruction to the above statement being true.  More precisely, we have

\begin{lemma}[Baby regularity lemma]\label{baby}  Let $l \in \N$, and let $\xi_1,\ldots,\xi_l \in \ultra \T$.  Then there exists $l',l'' \in \N$ and $\xi'_1,\ldots,\xi'_{l'}, \xi''_1,\ldots,\xi''_{l''} \in \ultra \T$ such that $\xi'_1,\ldots,\xi'_{l'}$ are linearly independent over $\Z$ \textup{(}i.e. there exist no standard integers $a_1,\ldots,a_{l'}$, not all zero, such that $a_1 \xi'_1+\ldots+a_{l'} \xi'_{l'} = 0$\textup{)}, each of the $\xi''_i$ are rational \textup{(}i.e. they live in $\Q \mod 1$\textup{)}, and each of the $\xi_1,\ldots,\xi_l$ are linear combinations \textup{(}over $\Z$\textup{)} of the $\xi'_1,\ldots,\xi'_{l'}, \xi''_1,\ldots,\xi''_{l''}$.
\end{lemma}

\begin{proof} Fix $l,\xi_1,\ldots,\xi_l$.  Define a \emph{partial solution} to be a collection of objects $l', l''$, $\xi'_1, \ldots,\xi'_{l'}$, $\xi''_1,\ldots,\xi''_{l''}$ satisfying all of the required properties, except possibly for the linear independence of the $\xi'_1,\ldots,\xi'_{l'}$.  Clearly at least one partial solution exists, since one can take $l' := l$, $l'' := 0$, and $\xi'_i := \xi_i$ for all $1 \leq i \leq l$.  Now let $l',l'',\xi'_1,\ldots,\xi'_{l'}, \xi''_1,\ldots,\xi''_{l''}$ be a partial solution for which $l'$ is minimal.  We claim that $\xi'_1,\ldots,\xi'_{l'}$ is linearly independent over $\Z$, which will give the lemma.  To see this, suppose for contradiction that there existed $a_1,\ldots,a_{l'} \in \Z$, not all zero, such that $a_1 \xi'_1 + \ldots + a_{l'} \xi'_{l'} = 0$.  Without loss of generality we may assume that $a_1$ is non-zero.  For each $2 \leq j \leq l'$, let $\tilde \xi'_j \in \ultra \T$ be such that $a_1 \tilde \xi'_j = \xi'_j$.  We then have
$$ \xi'_1 = - \sum_{j=2}^{l'} \frac{a_j}{a_1}  \xi'_j + q \mod 1$$
for some standard rational $q \in \Q$.  If we then replace $\xi'_1,\ldots,\xi'_{l'}$ by $\tilde \xi'_2,\ldots,\tilde \xi'_{l'}$ (decrementing $l'$ to $l'-1$) and append $q$ to $\xi''_1,\ldots,\xi''_{l''}$, then we obtain a new partial solution with a smaller value of $l'$, contradicting minimality.  The claim follows.
\end{proof}

This lemma is too simplistic for our applications, and we will need to modify it in a number of ways.  The first is to introduce an error term.

\begin{definition}[Linear independence]  Let $\eps > 0$ be a limit real, and let $l \in \N$.  A set of frequencies $\xi_1,\ldots,\xi_l \in \ultra \T$ is said to be \emph{independent modulo $O(\eps)$} if there do not exist any collection $a_1,\ldots,a_l \in \Z$ of standard integers, not all zero, for which
$$ a_1 \xi_1 + \ldots + a_l \xi_l = O(\eps) \mod 1$$
(Thus, for instance, the empty set (with $k=0$) is trivially independent modulo $O(\eps)$.)  Equivalently, $\xi_1,\ldots,\xi_l$ are linearly independent over $\Z$ after quotienting out by the subgroup $\eps \overline{\R} \mod 1$.
\end{definition}

This definition is only non-trivial when $\eps$ is an infinitesimal (i.e. $\eps=o(1)$).  In practice, $\eps$ will be a negative power of the unbounded integer $N$.

We have the following variant of Lemma \ref{baby}.

\begin{lemma}[Regularising one collection of frequencies]\label{toddler}  Let $l \in \N$, let $\xi_1,\ldots,\xi_l \in \ultra \T$, and let $\eps > 0$ be a limit real.  Then there exist $l',l'',l''' \in \N$ and 
\[
\xi'_1,\ldots,\xi'_{l'}, \xi''_1,\ldots,\xi''_{l''},\xi'''_1,\ldots,\xi'''_{l'''} \in \ultra \T
\]
 such that $\xi'_1,\ldots,\xi'_{l'}$ are linearly independent modulo $O(\eps)$, each of the $\xi''_i$ are rational, each of the $\xi'''_i$ are $O(\eps)$, and each of the $\xi_1,\ldots,\xi_l$ are linear combinations \textup{(}over $\Z$\textup{)} of the $\xi'_1,\ldots,\xi'_{l'}, \xi''_1,\ldots,\xi''_{l''}, \xi'''_1,\ldots,\xi'''_{l'''}$.
\end{lemma}

One can view Lemma \ref{baby} as the degenerate case $\eps=0$ of the above lemma.

\begin{proof} We repeat the proof of Lemma \ref{baby}.  Define a \emph{partial solution} to be a collection of objects $l',l'',l''$, $\xi'_1,\ldots,\xi'_{l'}, \xi''_1,\ldots,\xi''_{l''},\xi'''_1,\ldots,\xi'''_{l'''}$ obeying all the required properties except possibly for the linear independence property.  Again it is clear that at least one partial solution exists, so we may find a partial solution for which $l'$ is minimal.  We claim that this is a complete solution.  For if this is not the case, we have
$$ a_1 \xi'_1 + \ldots + a_{l'} \xi'_{l'} = O(\eps) \mod 1$$
for some $a_1,\ldots,a_{l'} \in \Z$, not all zero.  Again, we may assume that $a_1 \neq 0$.  We again select $\tilde \xi'_2,\ldots,\tilde \xi'_{l'} \in \ultra \T$ with $a_1 \tilde \xi'_j = \xi'_j$ for all $2 \leq j \leq l'$, and observe that
$$ \xi'_1 = - \sum_{j=2}^{l'} \frac{a_j}{a_1}  \xi'_j + q + s\mod 1$$
for some standard rational $q \in \Q$ and some $s = O(\eps)$.  If we then replace $\xi'_1,\ldots,\xi'_{l'}$ by $\tilde \xi'_2,\ldots,\tilde \xi'_{l'}$, and append $q$ and $s$ to $\xi''_1,\ldots,\xi''_{l''}$ and $\xi'''_1,\ldots,\xi'''_{l'''}$ respectively, we contradict minimality, and the claim follows.
\end{proof}

This lemma is still far too simplistic for our needs, because we will not be needing to regularise just one collection $\xi_1,\ldots,\xi_l$ of frequencies, but a whole \emph{family} $\xi_{h,1},\ldots,\xi_{h,l}$ of frequencies, where $h$ ranges over a parameter set $H$.  Such frequencies can exhibit a range of behaviour in $h$; at one extreme, they might be completely independent of $h$, while at the other extreme, the frequencies may vary substantially as $h$ does.  It turns out that in some sense, the general case is a combination of these extreme cases.

In this direction we have the following stronger version of Lemma \ref{toddler}.

\begin{lemma}[Regularising many collections of frequencies]\label{kid}  Let $l \in \N$, let $\eps > 0$ be a limit real, let $H$ be a limit finite set, and for each $h \in H$, let $\xi_{h,1},\ldots,\xi_{h,l}$ be frequencies in $\ultra \T$ that depend in a limit fashion on $h$.  Then there exists a dense subset $H'$ of $H$, standard natural numbers, $l_*, l',l''_*,l''' \in \N$, ``core'' frequencies $\xi_{*,1},\ldots,\xi_{*,l_*}, \xi''_{*,1},\ldots,\xi''_{l''_*} \in \ultra \T$, and ``petal'' frequencies \[ \xi'_{h,1},\ldots,\xi'_{h,l'}, \xi'''_{h,1},\ldots,\xi'''_{h,l'''} \in \ultra \T\] for each $h \in H'$ depending in a limit fashion on $h$, and obeying the following properties:
\begin{itemize}
\item[(i)] \textup{(Independence)} For almost all triples $(h_1,h_2,h_3) \in (H')^3$ \textup{(}i.e. for all but $o(|H'|^3)$ such triples\textup{)}, the frequencies 
\[
\xi_{*,1},\ldots,\xi_{*,l_*}, \xi'_{h_1,1},\ldots,\xi'_{h_1,l'}, \xi'_{h_2,1},\ldots,\xi'_{h_2,l'}, \xi'_{h_3,1},\ldots,\xi'_{h_3,l'}
\]
 are linearly independent modulo $O(\eps)$.
\item[(ii)] \textup{(Rationality)} For each $1 \leq j \leq l''$, $\xi''_{*,j}$ is a standard rational.
\item[(iii)] \textup{(Smallness)} For each $h \in H'$ and $1 \leq j \leq l'''$, $\xi'''_{h,j} = O(\eps)$.
\item[(iv)] \textup{(Representation)} For each $h \in H'$, the $\xi_{h,1},\ldots,\xi_{h,l}$ are linear combinations over $\Z$ of the frequencies \[ \xi_{*,1},\ldots,\xi_{*,l_*}, \xi'_{h,1},\ldots,\xi'_{h,l'}, \xi''_{*,1},\ldots,\xi''_{*,l''}, \xi'''_{h,1},\ldots,\xi'''_{h,l'''}.\]
\end{itemize}
\end{lemma}

Note that Lemma \ref{kid} collapses to Lemma \ref{toddler} if $H$ is a singleton set.

\begin{proof}  We again use the usual argument.  Define a \emph{partial solution} to be a collection of objects $H', l_*, l', l''_*, l''', \xi_{*,j}, \xi'_{h,j}, \xi''_{*,j}, \xi'''_{h,j}$ obeying all the required properties except possibly for the independence property.  Again, at least one partial solution exists, since we may take $H' := H$, $l_* := l'' := l''' := 0$, $l' := l$, and $\xi'_{j,h} := \xi_{j,h}$ for all $h \in H$ and $1 \leq j \leq l$.  We may thus select a partial solution for which $l'$ is minimal; and among all such partial solutions with $l'$ minimal, we choose a solution with $l_*$ minimal for fixed $l'$ (i.e. we minimise with respect to the lexicographical ordering on $l'$ and $l_*$).  We claim that this doubly minimal solution obeys the independence property, which would give the claim.

Suppose the independence property fails.  Carefully negating the quantifiers and using Lemma \ref{dense-dich}, we conclude that there exist standard integers $a_{*,j}$ for $1 \leq j \leq l_*$ and $a'_{i,j}$ for $i=1,2,3$ and $1 \leq j \leq l'$, not all zero, such that one has the relation
$$ a_{*,1} \xi_{*,1} + \ldots + a_{*,l_*} \xi_{*,l_*} + \sum_{i=1}^3 \sum_{j=1}^{l'} a'_{i,j} \xi'_{h_i,j} = O(\eps) \mod 1$$
for many triples $(h_1,h_2,h_3) \in (H')^3$.

Suppose first that all of the $a'_{i,j}$ vanish, so that we have a linear relation
$$ a_{*,1} \xi_{*,1} + \ldots + a_{*,l_*} \xi_{*,l_*} = O(\eps) \mod 1$$
that only involves core frequencies.  Then the situation is basically the same as that of Lemma \ref{toddler}; without loss of generality we may take $a_{*,1} \neq 0$, and if we then choose $\tilde \xi_{*,2},\ldots,\tilde \xi_{*,l_*}$ so that $a_{*,1} \tilde \xi_{*,j} = \xi_{*,j}$, then we can rewrite
$$ \xi_{*,1} = -\sum_{j=2}^{l'} a_{*,j} \tilde \xi_{*,j} + q + s \mod 1$$
for some $q \in \Q$ and $s = O(\eps)$, and one can then replace the $\xi_{*,1},\ldots,\xi_{*,l_*}$ with $\tilde \xi_{*,2},\ldots,\tilde \xi_{*,l_*}$ (decrementing $l_*$ by $1$) and append $q$ and $s$ to each of the collections $\xi''_{h,1},\ldots,\xi''_{h,l''}$ and $\xi'''_{h,1},\ldots,\xi'''_{h,l'''}$ respectively for each $h \in H$, contradicting minimality.

Now suppose that not all of the $a'_{i,j}$ vanish; without loss of generality we may assume that $a'_{1,1}$ is non-zero.  By the pigeonhole principple, we can find $h_2, h_3 \in H'$ such that
$$ a_{*,1} \xi_{*,1} + \ldots + a_{*,l_*} \xi_{*,l_*} + \sum_{i=1}^3 \sum_{j=1}^{l'} a'_{i,j} \xi'_{h_1,j} = O(\eps) \mod 1$$
for all $h_1$ in a dense subset $H''$ of $H'$.  Now let $\tilde \xi_{*,j} \in \ultra \T$ for $1 \leq j \leq l_*$ and $\tilde \xi'_{h,j} \in \ultra \T$ for $h_1 \in H'$ and $1 \leq j \leq l'$ be such that $a'_{1,1} \tilde \xi_{*,j} = \xi_{*,j}$ and $a'_{1,1} \tilde \xi'_{h,j} = \xi'_{h,j}$, then we have
$$ \xi'_{h_1,1} = - \sum_{j=2}^{l'} a'_{1,j} \tilde \xi'_{h_1,j} - \sum_{j=1}^{l_*} a_{*,j} \tilde \xi_{*,j} - \sum_{i=2}^3 \sum_{j=1}^{l'} a'_{i,j} \tilde \xi'_{i,j} + q_{h_1} + s_{h_1} \mod O(1)$$
for some standard rational $q_{h_1}$ and some $s_{h_1} = O(\eps)$.  Furthermore one can easily ensure that $q_{h_1}, s_{h_1}$ depend in a limit fashion on $h_1$.  By Lemma \ref{dense-dich} (and refining $H'$) we may assume that $q_{h_1} = q_*$ is independent of $h_1$.
We may thus replace $H'$ by $H''$ and replace $\xi'_{h,1},\ldots,\xi'_{h,l'}$ by $\tilde \xi'_{h,2},\ldots,\tilde \xi'_{h,l'}$ (decrementing $l'$ by $1$), while appending $q_*$ and $s_h$ to $\xi''_{*,1},\ldots,\xi''_{*,l''}$ and $\xi'''_{h,1},\ldots,\xi'''_{h,l'''}$ respectively, and replacing $\xi_{*,1},\ldots,\xi_{*,l_*}$ by $\tilde \xi_{*,1},\ldots,\tilde \xi_{*,l_*}, \tilde \xi'_{h_2,1},\ldots,\tilde \xi_{h_2,l'}, \tilde \xi'_{h_3,1},\ldots,\tilde \xi_{h_3,l'}$ (incrementing $l_*$ as necessary).  This contradicts the minimality of the partial solution, and the claim follows.
\end{proof}

This is still too simplistic for our applications, as the independence hypothesis on triples $(h_1,h_2,h_3)$ will not quite be strong enough to give everything we need.  Ideally, (in view of Proposition \ref{gcs-prop}) we would like to have independence of the
$\xi_{*,1},\ldots,\xi_{*,l_*}, \xi'_{h_1,1},\ldots,\xi'_{h_4,l'}$
for almost all additive quadruples $h_1+h_2=h_3+h_4$ in $H'$.  Unfortunately, this need not be the case; indeed, if the original $\xi_{h,i}$ are linear in $h$, say $\xi_{h,i} = \alpha_i h$ for some $\alpha_i \in \ultra \T$ and all $1 \leq i \leq l'$, then we have $\xi_{h_1,i} + \xi_{h_2,i} = \xi_{h_3,i} + \xi_{h_4,i}$ for all additive quadruples $h_1+h_2=h_3+h_4$ in $H'$ and all $1 \leq i \leq l'$, and as a consequence it is not possible to obtain a decomposition as in Lemma \ref{kid} with the stronger independence property mentioned above.  A similar obstruction occurs if the $\xi_{h,i}$ are \emph{bracket}-linear in $h$, for instance if $\xi_{h,i} = \{ \alpha_i h \} \beta_i \mod 1$ for some $\alpha_i \in \ultra \T$ and $\beta_i \in \ultra \R$.

By using tools from additive combinatorics, we can show that bracket-linear frequencies are the \emph{only} obstructions to independence on additive quadruples.   More precisely, we have

\begin{lemma}\label{teenager}  Let $l \in \N$, let $\eps > 0$ be a limit real, let $H$ be a dense limit subset of $[[N]]$, and for each $h \in H$, let $\xi_{h,1},\ldots,\xi_{h,l}$ be frequencies in $\ultra \T$ that depend in a limit fashion on $h$.  Then there exists a dense subset $H'$ of $H$, standard natural numbers, $l_*, l',l''_*,l''',l'''' \in \N$, ``core'' frequencies $\xi_{*,1},\ldots,\xi_{*,l_*}, \xi''_{*,1},\ldots,\xi''_{*,l''_*} \in \ultra \T$, and ``petal'' frequencies $\xi'_{h,1},\ldots,\xi'_{h,l'},\xi'''_{h,1},\ldots,\xi'''_{h,l'''} \xi''''_{h,1},\ldots,\xi''''_{h,l''''} \in \ultra \T$ for each $h \in H'$ depending in a limit fashion on $h$, obeying the following properties:
\begin{itemize}
\item[(i)] \textup{(Independence)} For almost all additive quadruples $h_1+h_2=h_3+h_4$ in $H'$ (i.e. for all but $o(|H'|^3)$ such quadruples), the frequencies $\xi_{*,j}$ for $1 \leq j \leq l_*$, $\xi'_{h_i,j}$ for $i=1,2,3,4$ and $1 \leq j \leq l'$, and $\xi''''_{h_i,j}$ for $i=1,2,3$ and $1 \leq j \leq l''''$ are jointly linearly independent modulo $O(\eps)$.
\item[(ii)] \textup{(Rationality)} For each $1 \leq j \leq l''_*$, $\xi''_{*,j}$ is a standard rational.
\item[(iii)] \textup{(Smallness)} For each $h \in H'$ and $1 \leq j \leq l'''$, $\xi'''_{h,j} = O(\eps)$.
\item[(iv)] \textup{(Bracket-linearity)} For each $1 \leq j \leq l''''$, there exist $\alpha_j \in \ultra \T$ and $\beta_j \in \ultra \R$ such that $\xi''''_{h,j} = \{ \alpha_j h \} \beta_j \mod 1$ for all $h \in H'$.  Furthermore, the map $h \mapsto \xi''''_{h,j}$ is a Freiman homomorphism on $H'$ \textup{(}see \S \ref{notation-sec} for the definition of a Freiman homomorphism\textup{)}.
\item[(v)] \textup{(Representation)} For each $h \in H'$, the $\xi_{h,1},\ldots,\xi_{h,l}$ are linear combinations over $\Z$ of the frequencies \[ \xi_{*,1},\ldots,\xi_{*,l_*}, \xi'_{h,1},\ldots,\xi'_{h,l'}, \xi''_{*,1},\ldots,\xi''_{*,l''}, \xi'''_{h,1},\ldots,\xi'''_{h,l'''}, \xi''''_{h,1},\ldots,\xi''''_{h,l''''}.\]
\end{itemize}
\end{lemma}

\begin{proof} As usual, we define a \emph{partial solution} to be a collection of objects $H'$, $l_*, l',l''_*,l''',l''''$, $\xi_{*,1},\ldots,\xi''''_{h,l''''}$, obeying all of the required properties except possibly for the independence property.  Again, there is clearly at least one partial solution, so we select a partial solution with a minimal value of $l'$, and then (for fixed $l'$) a minimal value of $l''''$, and then (for fixed $l',l''''$) a minimal value of $l_*$.  We claim that this partial solution obeys the independence property, which will give the lemma.

Suppose for contradiction that this were not the case; then by Lemma \ref{dense-dich}, there exist standard integers $a_{*,j}$ for $1 \leq j \leq l_*$, $a'_{i,j}$ for $1 \leq i \leq 4$ and $1 \leq j \leq l'$, and $a''_{i,j}$ for $1 \leq i \leq 3$ and $1 \leq j \leq l''''$, not all zero, such that
$$ \sum_{j=1}^{l_*} a_{*,j} \xi_{*,j} + \sum_{i=1}^4 \sum_{j=1}^{l'} a'_{i,j} \xi'_{h_i,j} + \sum_{i=1}^3 \sum_{j=1}^{l'''} a''''_{i,j} \xi''''_{h_i,j} = O(\eps) \mod 1$$
for many additive quadruples $h_1+h_2=h_3+h_4$ in $H'$.

Suppose first that all the $a'_{i,j}$ and $a''''_{i,j}$ vanished.  Then we have a relation 
$$ \sum_{j=1}^{l_*} a_{*,j} \xi_{*,j} = O(\eps) \mod 1$$
that only involves core frequencies; arguing as in Lemma \ref{kid} we can thus find another partial solution with a smaller value of $l_*$ (and the same value of $l'$, $l''''$), contradicting minimality.

Next, suppose that the $a'_{i,j}$ all vanished, but the $a''''_{i,j}$ did not all vanish.  Then we have a relation
\begin{equation}\label{triplicate}
 \sum_{j=1}^{l_*} a_{*,j} \xi_{*,j} + \sum_{i=1}^3 \sum_{j=1}^{l''''} a''''_{i,j} \xi''''_{h_i,j} = O(\eps) \mod 1
\end{equation}
for many triples $h_1,h_2,h_3$ in $H'$.

Without loss of generality let us suppose that $a''''_{1,1}$ is non-zero.  By the pigeonhole principle, we may find $h_2,h_3 \in H'$ such that \eqref{triplicate} holds for all $h_1$ in a dense subset $H''$ of $H'$.  As in previous arguments, we then find $\tilde \xi_{*,j} \in \ultra \T$ such that $a''''_{1,1} \tilde \xi_{*,j} = \xi_{*,j}$ for each $1 \leq j \leq l_*$, and also find $\tilde \beta_j \in \ultra \R$ such that $a''''_{1,1} \tilde \beta_j = \beta_j$ for all $1 \leq j \leq l''''$.  If we then set $\tilde \xi''''_{h,j} := \{ \alpha_j h \} \tilde \beta_j$ for each $h \in H'$ and $1 \leq j \leq l''''$, then $a''''_{1,1} \tilde \xi''''_{h,j} = \xi''''_{h,j}$, and so for any $h_1 \in H'$ we have
$$ \xi''''_{h_1,1} = - \sum_{j=1}^{l_*} a_{*,j} \tilde \xi_{*,j} - \sum_{j=2}^{l''''} a''''_{1,j} \tilde \xi''''_{h_1,j} - \sum_{i=2}^3 \sum_{j=1}^{l''''} a''''_{i,j} \tilde \xi''''_{h_i,j} + q_{h_1} + s_{h_1} \mod 1$$
for some standard rational $q_{h_1}$ and some $s_{h_1} = O(\eps)$, both depending on a limit fashion on $h_1$.  By refining $H'$ if necessary (and using the bracket-linear nature of the $\tilde \xi''''_{h,j}$) we may assume that the map $h \mapsto \tilde \xi''''_{h,j}$ is a Freiman homomorphism on $H'$, and by Lemma \ref{dense-dich} we may make $q_{h_1} = q_*$ independent of $h_1$.  If we then argue as in the proof of Lemma \ref{kid}, we may find a new partial solution with a smaller value of $l''''$ and the same value of $l'$, contradicting minimality.

Finally, suppose that the $a'_{i,j}$ did not all vanish.  Using the Freiman homomorphism property to permute the $i$ indices if necessary, we may assume that $a'_{4,1}$ does not vanish.  We then have
$$ \Xi_1(h_1) + \Xi_2(h_2) + \Xi_3(h_3) + \Xi_4(h_4) = O(\eps)$$
for many additive quadruples $h_1+h_2=h_3+h_4$ in $H'$, where the limit functions $\Xi_i: H \to \ultra \T$ are defined by
$$ \Xi_i(h) := \sum_{j=1}^{l'} a'_{i,j} \xi'_{h,j} + \sum_{j=1}^{l''''} a''''_{i,j} \xi''''_{h,j} \mod 1$$
for $i=1,2,3$ and $h \in H$, and
$$ \Xi_4(h) := \sum_{j=1}^{l_*} a_{*,j} \xi_{*,j} +
\sum_{j=1}^{l'} a'_{4,j} \xi'_{h,j} \mod 1.$$
We can use this additive structure to ``solve'' for $\Xi_4$, using a result from additive combinatorics which we present here as Lemma \ref{lin}.  Applying this lemma, we can then find a dense limit subset $H'$ of $H$, a standard integer $K$, and frequencies $\alpha'_1,\ldots,\alpha'_K, \delta \in \ultra \T$ and $\beta'_1,\ldots,\beta'_K \in \ultra \R$ such that
$$ \Xi_4(h) = \sum_{k=1}^K \{ \alpha'_k h \} \beta'_k + \delta + O(\eps) \mod 1$$
and thus
$$ a'_{4,1} \xi'_{h,1} =  \sum_{k=1}^K \{ \alpha'_k h \} \beta'_k + \delta -  \sum_{j=1}^{l_*} a_{*,j} \xi_{*,j} +
\sum_{j=2}^{l'} a'_{4,j} \xi'_{h,j} + O(\eps) \mod 1$$
for all $h \in H'$.

As usual, we now find $\tilde \beta_k \in \ultra \R$ for $1 \leq k \leq K$, $\tilde \beta_j \in \ultra \R$ for $1 \leq j \leq l''''$, $\tilde \delta \in \T$ and $\tilde \xi_{*,j}$ for $1 \leq j \leq l_*$ such that $a'_{4,1} \tilde \beta_k = \beta_k$, $a'_{4,1} \tilde \beta_j = \beta_j$, $a'_{4,1} \tilde \delta = \delta$, and $a'_{4,1} \tilde \xi_{*,j} = \xi_{*,j}$.  We then set $\tilde \xi'_{h,j} := \{ \alpha_j h \} \tilde \beta_j \mod 1$, and we conclude that
$$ \xi'_{h,1} =  \sum_{k=1}^K \{ \alpha'_k h \} \tilde \beta'_k + \tilde \delta -  \sum_{j=1}^{l_*} a_{*,j} \tilde \xi_{*,j} +
\sum_{j=2}^{l'} a'_{4,j} \tilde \xi'_{h,j} + q_h + s_h \mod 1$$
for all $h \in H'$, where $q_h \in \Q$ and $s_h = O(\eps)$ depend in a limit fashion on $h$.  By refining $H'$ we may take $q_h = q_*$ independent of $h$.

We can then use relation to build a new partial solution that decreases $l'$ by $1$, at the expense of enlarging the other dimensions $l_*, l'', l''', l''''$ (and also refining $H$ to $H'$), again contradicting minimality, and the claim follows.
\end{proof}

We now apply the above lemma to the language of horizontal frequency vectors introduced in the previous section.  We need some definitions:

\begin{definition}[Properties of horizontal frequency vectors]  Let 
\[ \F = (\xi_{i,j})_{1 \leq i \leq s-1; 1 \leq j \leq D_i}\; \;  \mbox{and} \; \; \F' = (\xi'_{i,j})_{1 \leq i \leq s-1; 1 \leq j \leq D'_i}\]
be horizontal frequency vectors.
\begin{itemize}
\item We say that $\F$ is \emph{independent} if, for each $1 \leq i \leq d$, the tuple $(\xi_{i,j})_{1 \leq j \leq D_i}$ is independent modulo $O(N^{-i})$.
\item We say that $\F$ is \emph{rational} if all the $\xi_{i,j}$ are standard rationals.
\item We say that $\F$ is \emph{small} if one has $\xi_{i,j} = O(N^{-i})$ for all $1 \leq i \leq s-1$ and $1 \leq j \leq D_i$.
\item We define the \emph{disjoint union} $\F \uplus \F' = (\xi''_{i,j})_{1 \leq i \leq s-1; 1 \leq j \leq D_i+D'_i}$ by declaring $\xi''_{i,j}$ to equal $\xi_{i,j}$ if $j \leq D_i$ and $\xi'_{i,j-D_i}$ if $D_i < j \leq D_i+D'_i$.  This is clearly a horizontal frequency vector with dimensions $(D_1+D'_1,\ldots,D_{s-1}+D'_{s-1})$.
\item We say that $\F$ is \emph{represented} by $\F'$ if for every $1 \leq i \leq s-1$ and $1 \leq j \leq D_i$, $\xi_{i,j}$ is a standard integer linear combination of the $\xi'_{i,j'}$ for $1 \leq j' \leq D'_i$.
\end{itemize}
\end{definition}

\begin{lemma}[Sunflower lemma]\label{sunflower-basic}  Let $H$ be a dense subset of $[[N]]$, and let $(\F_h)_{h \in H}$ be a family of horizontal frequency vectors depending in a limit fashion on $h$, whose dimension vector $\vec D = \vec D_h$ is independent of $h$.  Then we can find the following objects:
\begin{itemize}
\item A dense subset $H'$ of $H$;
\item Dimension vectors $\vec D_* = \vec D_{*,\ind} + \vec D_{*,\rat}$ and $\vec D' = \vec D'_\lin + \vec D'_\ind + \vec D'_\sml$, which we write as $\vec D_* = (D_{*,i})_{i=1}^{s-1}$, $\vec D_{*,\ind} = (D_{*,\ind,i})_{i=1}^{s-1}$, etc.;
\item A \emph{core horizontal frequency vector} $\F_* = (\xi_{*,i,j})_{1 \leq i \leq s-1; 1 \leq j \leq D_{*,i}}$, which is partitioned as $\F_* = \F_{*,\ind} \uplus \F_{*,\rat}$, with the indicated dimension vectors $\vec D'_\ind, \vec D'_\rat$;
\item A \emph{petal horizontal frequency vector} $\F'_h = (\xi'_{h,i,j})_{1 \leq i \leq s-1; 1 \leq j \leq D'_i}$, which is partitioned as $\F'_h = \F'_{h,\lin} \uplus \F'_{h,\ind} \uplus \F'_{h,\sml}$, which is a limit function of $h$ and with the indicated dimension vectors $\vec D'_\lin, \vec D'_\ind, \vec D'_\sml$
\end{itemize}
which obey the following properties:
\begin{itemize}
\item For all $h \in H'$, $\F'_{h,\sml}$ are small.
\item $\F_{*,\rat}$ is rational.
\item For every $1 \leq i \leq d$ and $1 \leq j \leq D'_{i,\lin}$, there exists $\alpha_{i,j} \in \ultra\T$ and $\beta_{i,j} \in \ultra \R$ such that \eqref{xih-def} holds for all $h \in H'$, and furthermore that the map $h \mapsto \xi'_{h,i,j}$ is a Freiman homomorphism on $H'$.
\item For all $h \in H$, $\F_h$ is represented by $\F_* \cup \F'_h$
\item \textup{(Independence property)} For almost all additive quadruples $(h_1,h_2,h_3,h_4)$ in $H$, 
$$\F_{*,\ind} \uplus \biguplus_{i=1}^4 \F'_{h_i,\ind} \uplus \biguplus_{i=1}^3 \F'_{h_i,\lin}$$ 
is independent.
\end{itemize}
\end{lemma}

\begin{proof}  Write $\F_h = (\xi_{h,i,j})_{1 \leq i \leq s-1; 1 \leq j \leq D_i}$. For each $1 \leq i \leq s-1$ in turn, apply Lemma \ref{teenager} to the collections $(\xi_{h,i,1},\ldots,\xi_{h,i,D_i})_{h \in H}$ and $\eps = O(N^{-i})$, refining $H$ once for each $i$.  The claim then follows by relabeling.
\end{proof}

To apply this lemma to families of nilcharacters, we will need two additional lemmas.

\begin{lemma}[Change of basis]\label{basis-change}  Suppose that $\chi \in \Xi^{(s-1,r_*)}_\DR([N])$ is a degree-rank $(s-1,r_*)$ nilcharacter with a total frequency representation $(\vec D, \F, \eta)$, and suppose that $\F$ is represented by another horizontal frequency vector $\F'$ with a dimension vector $\vec D'$.  Then there exists a vertical frequency $\eta': G^{\vec D'}_{s-1} \to \R$ such that $\chi$ has a total frequency representation $(\vec D', \F', \eta')$.
\end{lemma}

\begin{proof}  By hypothesis, each element $\xi_{i,j}$ of $\F$ can be expressed as a standard linear combination $\xi_{i,j} = \sum_{j'=1}^{D'_i} c_{i,j,j'} \xi'_{i,j'}$ of elements $\xi'_{i,j'}$ of $\F'$ of the same degree, where $c_{i,j,j'} \in \Z$.

Now let $\psi: G^{\vec D'} \to G^{\vec D}$ be the unique filtered homomorphism that maps $e'_{i,j'}$ to $\prod_{j=1}^{D_i} e_{i,j}^{c_{i,j,j'}}$ (this can be viewed as an ``adjoint'' of the representation of $\F$ by $\F'$). By hypothesis, $\chi$ has a representation $\chi(n) = F( \orbit(n), \orbit_0(n))$ of $\chi$ with
$$ 
 \Taylor_i(\orbit) =  \pi_{\Horiz_i(G/\Gamma)}\left(\phi( \prod_{j=1}^{D_i} e_{i,j}^{\xi_{i,j}} )\right)
$$
for some filtered homomorphism $\phi: G^{\vec D} \to G$.  A brief calculation shows that the right-hand side can also be expressed as
$$ \pi_{\Horiz_i(G/\Gamma)}\left(\phi \circ \psi( \prod_{j=1}^{D'_i} (e'_{i,j})^{\xi'_{i,j}} )\right).$$
As $\phi \circ \psi: G^{\vec D'} \to G$ is a filtered homomorphism, and $\eta \circ \psi: G^{\vec D'}_{(s-1,r_*)} \to \R$ is a vertical frequency, we obtain the claim.
\end{proof}

\begin{lemma}\label{discard}  Let $\F$ be a horizontal frequency vector of dimension $\vec D$ of the form
$$ \F = \F_{\rat} \uplus \F_{\sml} \uplus \F'$$
where $\F_\rat$ is rational and $\F_\sml$ is small, and $\F'$ has dimension $\vec D'$.  Suppose that $\chi \in \Xi^{(s-1,r_*)}_\DR([N])$ is a nilcharacter with a total frequency representation $(\vec D, \F, \eta)$.  Then there exists a vertical frequency $\eta': G^{\vec D'}_{s-1} \to \R$ such that $\chi$ has total frequency $(\vec D', \F'/M, \eta')$ for some standard integer $M \geq 1$.
\end{lemma}

\emph{Remark.}  This lemma crucially relies on the hypothesis $s \geq 3$, as it makes the (degree $1$) contributions of rational and small frequencies to be of lower order.  Because of this, the inverse conjecture for $s > 2$ is in a very slight way a little bit simpler than the $s \leq 2$ theory, though it is of course more complicated in many other ways.

\begin{proof}  By induction we may assume that $\F$ is formed from $\F'$ by adding a single frequency $\xi_{i_0,D_{i_0}}$, which is either rational or small.

Let us first suppose that we are adding a single frequency which is not just rational, but is in fact an integer.  Then if $\chi(n) = F(g(n)\ultra \Gamma, g_0(n) \ultra \Gamma_0)$ is a nilcharacter with a total frequency representation $(\vec D,\F,\eta)$, then we have a filtered homomorphism $\phi: G^{\vec D}/\Gamma^{\vec D} \to G/\Gamma$ such that
$$
g_i = \prod_{j=1}^{D_i} \phi(e_{i,j})^{\xi_{i,j}} \hbox{ mod } G_{(i,1)}
$$
for all $1 \leq i \leq s_*-1$, where $g_i$ are the Taylor coefficients of $g$.  Specialising to the degree $i_0$ and using the integer nature of $\xi_{i_0,D_{i_0}}$, we have 
$$ g_{i_0} = g'_{i_0} \gamma_{i_0}$$
where $\gamma_{i_0}$ is an element of $\Gamma_{i_0}$, and
$$g'_i = \prod_{j=1}^{D_i-1} \phi(e_{i,j})^{\xi_{i,j}} \hbox{ mod } G_{(i,1)}.$$
From this and the Baker-Campbell-Hausdorff formula \eqref{bch}, we can write $g(n) = g'(n) \gamma_{i_0}^{\binom{n}{i_0}}$, where $g'$ is a polynomial sequence with a horizontal frequency representation $(\vec D', \phi', \F')$, where $\vec D'$ is $\vec D$ with $D_{i_0}$ decremented by one, and $\phi'$ is the restriction of $\phi$ to the subnilmanifold $G^{\vec D'}/\Gamma^{\vec D'}$.  Since $g(n) \ultra \Gamma = g'(n) \ultra \Gamma$, we see that $\chi$ has a total frequency representation $(\vec D', \F', \eta')$, where $\eta'$ is the restriction of $\eta: G^{\vec D}_{(s-1,r_*)} \to \R$ to $G^{\vec D'}_{(s-1,r_*)}$.  This gives the claim in this case (with $M=1$).
 
Now suppose that $\xi_{i_0,D_{i_0}}$ is merely rational rather than integer.  Then we can argue as before, except that now $\gamma_{i_0}$ is a rational element of $G_{i_0}$, so that $\gamma_{i_0}^m \in \Gamma_{i_0}$ for some standard positive integer $m$.  As such, there exists a standard positive integer $q$ such that $\gamma_{i_0}^{\binom{n}{i_0}} \mod \ultra \Gamma$ is periodic with period $q$.  As a consequence, there exists a bounded index subgroup $\Gamma'$ of $\Gamma$ such that the point
$$ g'(n) \gamma_{i_0}^{\binom{n}{i_0}} \mod \ultra \Gamma$$
in $G/\Gamma$ can be expressed as a Lipschitz function of 
$$ g'(n) \mod \ultra \Gamma'$$
and of the quantity $n/q \mod 1$.  Repeating the previous arguments, we thus obtain a total frequency representation $(\vec D', \tilde \F', \eta')$ for some $\eta'$, and some $\tilde \F'$ whose coefficients are rational combinations of those of $\F'$; note that the $n/q$ dependence can be easily absorbed into the lower order term $G_0/\Gamma_0$ since $s \geq 3$.  The claim then follows from Lemma \ref{basis-change}.

Finally, suppose that $\xi_{i_0,D_{i_0}}$ is small rather than rational.  Then we can write
$$  g_{i_0} = c_{i_0} g'_{i_0}$$
where $g'_{i_0}$ is as before, and $c_{i_0} \in G_{i_0}$ is at a distance $O(N^{-i_0})$ from the origin.  We can thus write
$$ g(n) = c_{i_0}^{\binom{n}{i_0}} g'(n)$$
where $g'$ is a polynomial sequence with horizontal frequency representation \[ (\vec D', \phi',\F').\]  On $[N]$, the sequence $c_{i_0}^{\binom{n}{i_0}}$ is can be expressed as a bounded Lipschitz function of $n/2N \hbox{ mod } 1$.  As a consequence, we can thus write $\chi$ in the form
$$ \chi(n) = F'( g'(n) \ultra \Gamma, g_0(n) \ultra \Gamma_0, n/2N \hbox{ mod } 1 )$$
 for some $F' \in \Lip(\ultra( G/\Gamma \times G_0/\Gamma_0 \times \T ))$.  As $s \geq 3$, the final term $\T$ can be absorbed into the degree-rank $\leq (s-1,r_*-1)$ nilmanifold $G_0/\Gamma_0$, and the claim follows (with $M=1$).
\end{proof}

Finally, we can state the main result of this section.

\begin{lemma}[Sunflower lemma]\label{sunflower}  Let $H$ be a dense subset of $[[N]]$, and let $(\chi_h)_{h \in H}$ be a family of nilcharacters $\chi_h \in \Xi^{(s-1,r_*)}_\DR([N])$ depending in a limit fashion on $H$.  Then we can find
\begin{enumerate}
\item A dense subset $H'$ of $H$;
\item Dimension vectors $\vec D_*$ and $\vec D' = \vec D'_\lin + \vec D'_\ind$, which we write as $\vec D_* = (D_{*,i})_{i=1}^{s-1}$, $\vec D' = (D'_{i})_{i=1}^{s-1}$, $\vec D'_\lin = (D'_{\lin,i})_{i=1}^{s-1}$, $\vec D'_\ind = (D'_{\ind,i})_{i=1}^{s-1}$;
\item A \emph{core horizontal frequency vector} $\F_* = (\xi_{*,i,j})_{1 \leq i \leq d; 1 \leq j \leq D_{*,i}}$;
\item A \emph{petal horizontal frequency vector} $\F'_h = (\xi'_{h,i,j})_{1 \leq i \leq d; 1 \leq j \leq D'_i}$, which is partitioned as $\F'_h = \F'_{h,\lin} \uplus \F'_{h,\ind}$, which is a limit function of $h$, where $\F'_{h,\lin}$, $\F'_{h,\ind}$ have dimensions $\vec D'_\lin$, $\vec D'_\ind$ respectively;
\item A vertical frequency $\eta: G^{\vec D_* + \vec D'}_{(s-1,r_*)} \to \R$ with dimension vector $\vec D_* + \vec D'$
\end{enumerate}
which obey the following properties:
\begin{enumerate}
\item \textup{($\F'_{h,\lin}$ is bracket-linear)} For every $1 \leq i \leq d$ and $1 \leq j \leq D'_{i,\lin}$, there exists $\alpha_{i,j} \in \ultra\T$ and $\beta_{i,j} \in \ultra \R$ such that
\begin{equation}\label{xih-def}
\xi'_{h,i,j} = \{ \alpha_{i,j} h \} \beta_{i,j} \mod 1
\end{equation}
for all $h \in H'$, and furthermore that the map $h \mapsto \xi'_{h,i,j}$ is a Freiman homomorphism on $H'$.
\item \textup{(Independence)} For almost all additive quadruples $(h_1,h_2,h_3,h_4)$ in $H$, 
$$\F_{*,\ind} \uplus \biguplus_{i=1}^4 \F'_{h_i,\ind} \uplus \biguplus_{i=1}^3 \F'_{h_i,\lin}$$ 
is independent.
\item \textup{(Representation)} For all $h \in H'$, $\chi_h$ has a total frequency representation $( \vec D_* + \vec D', \F_* \cup \F'_h, \eta )$.
\end{enumerate}
\end{lemma}

\begin{proof}  Each $\chi_h$ thus has a total frequency representation $(\vec D_h,  \F_h, \eta_h)$.  The space of representations is a $\sigma$-limit set, so by Lemma \ref{int-select} we may assume that $(\vec D_h, \F_h, \eta_h)$ depends in a limit fashion on $h$.

The number of possible dimension vectors is countable. Applying Lemma \ref{dense-dich}, and passing from $H$ to a dense subset, we may assume that $\vec D = \vec D_h$ is independent of $h$.  

We then apply Lemma \ref{sunflower-basic} to the $(\F_h)_{h \in H}$, obtaining a dense subset $H'$ of $H$, dimension vectors $\vec D_* = \vec D_{*,\ind} + \vec D_{*,\rat}$ and $\vec D' = \vec D'_\lin + \vec D'_\ind + \vec D'_\sml$, a core horizontal frequency vector $\F_* = \F_{*,\ind} \uplus \F_{*,\rat}$, and petal horizontal frequency vectors $\F'_h = \F'_{h,\lin} \uplus \F'_{h,\ind} \uplus \F'_{h,\sml}$ for each $h \in H'$ with the stated properties.

Applying Lemma \ref{basis-change}, we see that for each $h \in H'$, $\chi_h$ has a total frequency representation
$$ (\vec D_* + \vec D', \F_* \uplus \F'_h, \eta'_h )$$
for some vertical frequency $\eta'_h$.  Applying Lemma \ref{discard}, we conclude that $\chi_h$ has a total frequency representation
$$ (\vec D_{*,\ind} + \vec D'_\lin + \vec D'_\ind, \F_{*,\ind} \uplus \F'_{h,\lin} \uplus \F'_{h,\ind}, \eta''_h )$$
for some vertical frequency $\eta'_h$.  The number of vertical frequencies $\eta''_h$ is countable, so by Lemma \ref{dense-dich} we may assume that $\eta = \eta''_h$ is also independent of $h$.  The claim then follows.
\end{proof}

\section{Obtaining bracket-linear behaviour}\label{linear-sec}

We return now to the task of proving Theorem \ref{linear-induct}.  To recall the situation thus far, we are given a two-dimensional nilcharacter $\chi \in \Xi^{(1,s-1)}_\MD(\ultra \Z^2)$ and a family of degree-rank $(s-1,r_*)$ nilcharacters $(\chi_h)_{h \in H}$ depending in a limit fashion on a parameter $h$ in a dense subset $H$ of $[[N]]$, with the property that there is a function $f \in L^\infty[N]$ such that $\chi(h,\cdot) \otimes \chi_h$ $(s-2)$-correlates with $f$ for all $h \in H$.  Using Proposition \ref{gcs-prop} to eliminate $f$ and $\chi$, and refining $H$ to a dense subset if necessary, we conclude that the nilcharacter \eqref{gowers-cs-arg} is $(s-2)$-biased for many additive quadruples $h_1+h_2=h_3+h_4$ in $H$.  We make the simple but important remark that this conclusion is ``hereditary'' in the sense that it continues to hold if we replace $H$ with an arbitrary dense subset $H'$ of $H$, since the hypothesis of Proposition \ref{gcs-prop} clearly restricts from $H$ to $H'$ in this fashion.

Next, we apply Lemma \ref{sunflower} to obtain a dense refinement $H'$ on $H$ for which the $\chi_h$ have a frequency representation involving various types of frequencies: a core set of frequencies $\F_*$, a bracket-linear family $(\F'_{h,\lin})_{h \in H'}$ of petal frequencies and an independent family $(\F'_{h,\ind})_{h \in H'}$ of petal frequencies.

<<<<<<< .mine
The main result of this section uses the bias of \eqref{gowers-cs-arg}, combined with the quantitative equidistribution theory on nilmanifolds (as reviewed in Appendix \ref{equiapp}) to obtain an important milestone towards establishing Theorem \ref{linear-induct}, namely that the independent petal frequencies $\F'_{h,\ind}$ do not actually have any influence on the top-order behaviour of the nilcharacters $\chi_h$, and that the bracket-linear frequencies only influence this top-order behaviour in a linear fashion.  For this, we use an argument of Furstenberg and Weiss \cite{fw-char}, also used in the predecessor \cite{u4-inverse} to this paper.  See also \cite{gtz-announce} for another exposition of this argument.
=======
The main result of this section uses the bias of \eqref{gowers-cs-arg}, combined with the quantitative equidistribution theory on nilmanifolds (as reviewed in Appendix \ref{equiapp}) to obtain an important milestone towards establishing Theorem \ref{linear-induct}, namely that the independent petal frequencies $\F'_{h,\ind}$ do not actually have any influence on the top-order behaviour of the nilcharacters $\chi_h$, and that the bracket-linear frequencies only influence this top-order behaviour in a linear fashion.  For this, we use an argument of Furstenberg and Weiss \cite{fw-char} that was also used in the predecessor \cite{u4-inverse} to this paper.  See \cite{gtz-announce} for another, somewhat simplified, exposition of this argument.
>>>>>>> .r207

We begin by formally stating the result we will prove in this section.

\begin{theorem}[No petal-petal or regular terms]\label{slang-petal}  Let $f,H,\chi,(\chi_h)_{h \in H}$ be as in Theorem \ref{linear-induct} and let $H', \vec D_*, \vec D', \vec D'_\lin, \vec D'_\ind, \F_*, \F'_h, \F'_{h,\lin}, \F'_{h,\ind}, \eta$ be as in Lemma \ref{sunflower}.  Let $w \in G^{\vec D_* + \vec D'}$ be an $r_*-1$-fold commutator of $e_{i_1,j_1},\ldots,e_{i_{r_*},j_{r_*}}$, where $1 \leq i_1,\ldots,i_{r_*} \leq s-1$, $i_1+\ldots+i_{r_*}=s-1$, and $1 \leq j_l \leq D_{*,i_l} + D'_{i_l}$ for all $l$ with $1 \leq l \leq r_*$.
\begin{enumerate}
\item \textup{(No petal-petal terms)} If $j_l > D_{*,i_l}$ for at least two values of $l$, then $\eta(w)=0$.
\item \textup{(No regular terms)} If $j_l > D_{*,i_l} + D'_{\lin,i_l}$ for at least one value of $l$, then $\eta(w)=0$.
\item \textup{(No petal-petal terms)}  If $j_l > D_{*,i_l}$ for at least two values of $l$ then $\eta(w)=0$.
\item \textup{(No regular terms)} If $j_l > D_{*,i_l} + D'_{\lin,i_l}$ for at least one value of $l$ then $\eta(w)=0$.
\end{enumerate}
\end{theorem}

The remainder of this section is devoted to the proof of Theorem \ref{slang-petal}.

Let the notation and assumptions be as in the above theorem.
From Proposition \ref{gcs-prop} we know that, for many additive quadruples $(h_1,h_2,h_3,h_4)$ in $H'$, the sequence \eqref{gowers-cs-arg} is $(s-2)$-biased.  Also, from Lemma \ref{sunflower}, we see that for almost all of these quadruples, the horizontal frequency vectors
\begin{equation}\label{jinnai-1}
\F_{*,\ind} \uplus \biguplus_{i=1}^4 \F_{h_i,\ind} \uplus \biguplus_{i=a,b,c} \F_{h_i,\lin}
\end{equation}
are independent for all distinct $a,b,c \in \{1,2,3,4\}$.   We may therefore find an additive quadruple $(h_1,h_2,h_3,h_4)$ for which \eqref{gowers-cs-arg} is $(s-2)$-biased, and for which \eqref{jinnai-1} is independent for all choices of distinct $a,b,c \in \{1,2,3,4\}$.

Fix $(h_1,h_2,h_3,h_4)$ with these properties.  We convert the above information to a non-equidistribution result concerning a polynomial orbit.  

For each $i=1,2,3,4$, we see from Lemma \ref{sunflower} that $\chi_{h_i}$ has a total frequency representation
$$ ( \vec D_* + \vec D', \F_* \uplus \F'_{h_i}, \eta ).$$
We write
$$ \F_* \uplus \F'_{h_i} = ( \xi_{h_i,j,k} )_{1 \leq j \leq s-1; 1 \leq k \leq D_j},$$
where
$$ D_j = D_{*,j} + D'_j;$$
thus the frequencies associated to $\F_{*}$, $\F'_{h_i,\ind}$, $\F'_{h_i,\lin}$ correspond to the ranges $1 \leq k \leq D_{*,j}$, $D_{*,j} < k \leq D_{*,j}+D'_{\ind,j}$, and $D_{*,j} + D'_{\ind,j} < k \leq D_j$ respectively.

As \eqref{gowers-cs-arg} is $(s-2)$-biased, we conclude that
\begin{equation}\label{expect}
 |\E_{n \in [N]} \chi_{h_1}(n) \otimes \chi_{h_2}(n+h_1-h_4) \otimes \overline{\chi_{h_3}}(n) \otimes \overline{\chi_{h_4}}(n+h_1-h_4) \psi_{h_1,h_2,h_3,h_4}(n)| \gg 1
 \end{equation}
for some degree $\leq (s-2)$ nilsequence $\psi_{h_1,h_2,h_3,h_4}$, where $\chi_h$ is defined to be zero outside of $[N]$.  As any cutoff to an interval can be approximated to arbitrary standard accuracy by a degree $1$ nilsequence, and $s \geq 3$, we see that the same claim holds if $\chi_h$ is instead extended to be a nilsequence on all of $\ultra \Z$.

From Definition \ref{nilch-def} and the total frequency representation of the $\chi_{h_i}$, we can rewrite the sequence inside the expectation of \eqref{expect} as a degree-rank $\leq (s-1,r_*)$ nilsequence $n \mapsto F(\orbit(n))$. Here $G/\Gamma$ is the product nilmanifold\footnote{Unfortunately, there will be several types of subscripts on nilpotent Lie groups $G$ in this argument.  Firstly one has the factor groups $G_{(i)}$.  Then one also has the degree filtration groups $G_d$ and the degree-rank filtration groups $G_{(d,r)}$ of $G$ (and also the analogous subgroups $(G_{(i)})_d$, $(G_{(i)})_{(d,r)}$ of the factor groups $G_{(i)}$), as well as the free nilpotent groups $G^{\vec D} = G^{\vec D}_{(s-1,r_*)}$.  Finally, a Ratner subgroup $G_P$ of $G$ will also make an appearance later.  We hope that these notations can be kept separate from each other.}
$$ G/\Gamma := \left(\prod_{i=1}^4 G_{(i)}/\Gamma_{(i)}\right) \times G_{(0)}/\Gamma_{(0)}$$
for some filtered nilmanifold $G_{(0)}/\Gamma_{(0)}$ of degree-rank $<(s-1,r_*-1)$ and filtered nilmanifolds $G_{(i)}/\Gamma_{(i)}$ of degree-rank $\leq(s-1,r_*)$ for $i=1,2,3,4$. The orbit $\orbit$ is defined by
$$\orbit = (\orbit_1,\orbit_2,\orbit_3,\orbit_4,\orbit_0) \in \ultra \poly(\Z_\N \to (G/\Gamma)_\N)$$
where, for each $i,j$ with $1 \leq i \leq 4$ and $1 \leq j \leq s-1$ we have
\begin{equation}\label{gij-spin}
\Taylor_j(\orbit_{(i)}) = \pi_{\Horiz_j(G_{(i)}/\Gamma_{(i)})}\left(\phi_{(i)}(\prod_{1 \leq k \leq D_j} e_{j,k}^{\xi_{h_i,j,k}})\right)
\end{equation}
where $\vec D := (D_1,\ldots,D_{s-1})$, $\phi_{(i)}: G^{\vec D}/\Gamma^{\vec D} \to G_{(i)}/\Gamma_{(i)}$ is a filtered homomorphism and $\pi_{\Horiz_j(G_{(i)}/\Gamma_{(i)})}: (G_{(i)})_j \to \Horiz_j(G_{(i)}/\Gamma_{(i)})$ is the projection to the $j^{\operatorname{th}}$ horizontal torus. Finally $F \in \Lip(\ultra(G/\Gamma))$ is defined by
\begin{align}\nonumber
 F( \phi_{(1)} & (t_{(1)}) x_{(1)}, \ldots, \phi_{(4)}(t_{(4)}) x_{(4)}, y ) = \\
& e( (\eta(t_{(1)})+\eta(t_{(2)})-\eta(t_{(3)})-\eta(t_{(4)})) ) F(x_{(1)},\ldots,x_{(4)},y)\label{fallow}
\end{align}
for all $(x_{(1)},\ldots,x_{(4)},y) \in G/\Gamma$ and $t_{(1)},\ldots,t_{(4)} \in G^{\vec D}_{(s-1,r_*)}$.  (Note that the shifts by $h_1-h_4$ in \eqref{expect} do not affect the Taylor coefficients of $\orbit_{(i)}$, thanks to the remarks following Definition \ref{horton}.)

By hypothesis, we have
$$ |\E_{n \in [N]} F( \orbit(n) )| \gg 1.$$
Applying Theorem \ref{ratt}, we conclude that
\begin{equation}\label{gapp}
 |\int_{G_P / \Gamma_P} F(\eps x)\ d\mu(x)| \gg 1
\end{equation}
for some bounded $\eps \in G$ and some rational subgroup $G_P$ of $G$ with the property that
\begin{equation}\label{soo}
 \pi_{\Horiz_j(G)}(G_P \cap G_{(i)}) \geq \Xi_j^\perp
\end{equation}
for all $1 \leq j \leq s-1$, where
$$ \Xi_j^\perp := \{ x \in \Horiz_j(G) : \xi_j(x) = 0 \hbox{ for all } \xi_j \in \Xi_j \}$$
and $\Xi_j \leq \widehat{\Horiz_j(G/\Gamma)}$ is the group of all (standard) continuous homomorphisms $\xi_j: \Horiz_j(G/\Gamma) \to \T$ such that 
$$ \xi_j( \Taylor_j(\orbit) ) = O( N^{-j} ).$$

From \eqref{fallow} and \eqref{gapp} we conclude the following lemma.

\begin{lemma}\label{gapp-vanish}  The group $G_P \cap ((G_{(1)})_{(s-1,r_*)} \times \{\id\} \times \{\id\} \times \{\id\} \times \{\id\})$ is annihilated by $\eta$.
\end{lemma}

\begin{proof}  Let $g = (g_{(1)},\id,\id,\id,\id)$ lie in the indicated group.  Then $g$ is central, and so from the invariance of Haar measure we have
$$
\int_{G_P / \Gamma_P} F(\eps x)\ d\mu(x) = \int_{G_P / \Gamma_P} F(g \eps x)\ d\mu(x).$$
On the other hand, from \eqref{fallow} we have
$$ \int_{G_P / \Gamma_P} F(g \eps x)\ d\mu(x) = e(\eta(g)) \int_{G_P / \Gamma_P} F(\eps x)\ d\mu(x).$$
Comparing these relationships with \eqref{gapp} we obtain the claim.
\end{proof}

We now analyse the group $G_P$ further.  For each $1 \leq j \leq s-1$, let $V_{123,j}$ denote the subgroup of $\Horiz_j(G_{(1)}) \times \Horiz_j(G_{(2)}) \times \Horiz_j(G_{(3)})$ generated by the diagonal elements
$$ (\phi_{(1)}(e_{j,k}), \phi_{(2)}(e_{j,k}), \phi_{(3)}(e_{j,k}))$$
for $1 \leq k \leq D_{*,j}$, and by the elements
$$ (\phi_{(1)}(e_{j,k}), \id, \id), (\id, \phi_{(2)}(e_{j,k}), \id), (\id, \id, \phi_{(3)}(e_{j,k}))$$
for $D_{*,j} < k \leq D_j$.  We define the subgroup $V_{124,j}$ of $\Horiz_j(G_{(1)}) \times \Horiz_j(G_{(2)}) \times \Horiz_j(G_{(4)})$ similarly by replacing $(3)$ with $(4)$ throughout.

\begin{lemma}[Components of $G_P$]\label{gp-comp}  Let $1 \leq j \leq s-1$.  Then the projection of $G_P \cap G_{j}$ to $\Horiz_j(G_{(1)}) \times \Horiz_j(G_{(2)}) \times \Horiz_j(G_{(3)})$ contains $V_{123,j}$.  Similarly, the projection to $\Horiz_j(G_{(1)}) \times \Horiz_j(G_{(2)}) \times \Horiz_j(G_{(4)})$ contains $V_{124,j}$.
\end{lemma}

\begin{proof} We shall just prove the first claim; the second claim is similar (but uses $\{a,b,c\} = \{1,2,4\}$ instead of $\{a,b,c\}=\{1,2,3\}$).

Suppose the claim failed for some $j$.  Using \eqref{soo} and duality, we conclude that there exists a $\xi_j \in \Xi_j$ which annihilates the kernel of the projection to $\Horiz_j(G_{(1)}) \times \Horiz_j(G_{(2)}) \times \Horiz_j(G_{(3)})$, and which is non-trivial on $V_{123,j}$.  As $\xi_j$ annihilates the kernel of the projection to $\Horiz_j(G_{(1)}) \times \Horiz_j(G_{(2)}) \times \Horiz_j(G_{(3)})$, we have a decomposition of the form
$$ \xi_j(x_{(1)},x_{(2)},x_{(3)},x_{(4)},x_{(0)}) = \xi_{(1),j}(x_{(1)}) + \xi_{(2),j}(x_{(2)}) + \xi_{(3),j}(x_{(3)})$$
for $x_{(i)} \in \Horiz_j(G_{(i)})$ for $i=1,2,3,4,0$, where $\xi_{(i),j}: \Horiz_j(G_{(i)}) \to \R$ for $i=1,2,3$ are characters.

By definition of $\Xi_j$, we conclude that
$$ \xi_{(1),j}( \Taylor_j(\orbit_{(1)}) ) + \xi_{(2),j}( \Taylor_j(\orbit_{(2)}) ) + \xi_{(3),j}( \Taylor_j(\orbit_{(3)}) ) = O(N^{-j}).$$
However, from \eqref{gij-spin} we have
\begin{equation}\label{star}
 \xi_{(i),j}(\Taylor_j(\orbit_{(i)})) = \sum_{k=1}^{D_j} c_{(i),j,k} \xi_{h_i,j,k}
 \end{equation}
where the $c_{(i),j,k}$ are standard integers, defined by the formula
\begin{equation}\label{cdef}
 c_{(i),j,k} := \xi_{(i),j}(\phi_{(i)}(e_{j,k})).
\end{equation}
From the independence of \eqref{jinnai-1} with $\{a,b,c\}=\{1,2,3\}$, we conclude that the $c_{(i),j,k}$ all vanish for $i=1,2,3$ and $D_{*,j} < k \leq D_j$, and that the sum $c_{(1),j,k}+c_{(2),j,k}+c_{(3),j,k}$ vanishes for $1 \leq k \leq D_{*,j}$.  But this forces $\xi_j$ to vanish on $V_{123,j}$, contradiction.
\end{proof}

We now take commutators in the spirit of an argument of Furstenberg and Weiss \cite{fw-char} (see also \cite{hrush,ribet} for similar arguments in completely different settings) to conclude the following result which roughly speaking asserts that all ``petal-petal interactions'' are trivial.

\begin{corollary}[Furstenberg-Weiss commutator argument]\label{fw}  Let $w$ be an $r_*-1$-fold iterated commutator of generators $e_{j_1,k_1},\ldots,e_{j_{r_*},k_{r_*}}$ with $1 \leq j_l \leq s-1$, $1 \leq k_l \leq D_l$ for $l=1,\ldots,r_*$ and $j_1+\ldots+j_{r_*} = s-1$ \textup{(}thus $w$ has ``degree-rank $(s-1,r_*)$'' in some sense\textup{)}.  Suppose that at least two of the generators, say $e_{j_1,k_1}, e_{j_2,k_2}$, are ``petal'' generators in the sense that $k_1 > D_{*,j_1}$ and $k_2 > D_{*,j_2}$.  Then $(\phi_{(1)}(w),\id,\id,\id,\id) \in G_P$.
\end{corollary}

\begin{proof}  
For $e_{j_1,k_1}$, we may invoke Lemma \ref{gp-comp} and find an element $g_{j_1,k_1}$ of $G_P \cap G_{j_1}$ for which the coordinates $1,2,3$ are equal (modulo projection to \[ \Horiz_{j_1}(G_{(1)}) \times \Horiz_{j_1}(G_{(2)}) \times \Horiz_{j_1}(G_{(3)}))\] to $(\phi_1(e_{j_1,k_1}),\id,\id)$.  Similarly, we may find an element $g'_{j_2,k_2}$ of $G_P \cap G_{j_2}$ for which the coordinates $1,2,4$ are equal (modulo projection to \[ \Horiz_{j_2}(G_{(1)}) \times \Horiz_{j_2}(G_{(2)}) \times \Horiz_{j_2}(G_{(4)}))\]  to $(\phi_1(e_{j_2,k_2}),\id,\id)$.  Finally, for all of the other $e_{j,k}$, we can find elements  $g''_{j,k}$ of $G_P \cap G_{j}$ for which the first coordinate is equal (modulo projection to $\Horiz_j(G_{(1)})$) to $\phi_{(1)}(e_{j,k})$.  If one then takes iterated commutators of the $g_{j_1,k_1}, g'_{j_2,k_2}, g''_{j,k}$ in the order indicated by $w$, we see (using the filtration property, the homomorphism property of $\phi_{(1)}$, and the fact that the $G_i/\Gamma_i$ have degree $\leq (s-1,r_*)$ for $i=1,2,3,4$ and degree $<(s-1,r_*-1)$ for $i=0$) that we obtain the element $(\phi_{(1)}(w),\id,\id,\id,\id)$.  Since the iterated commutator of elements in $G_P$ stays in $G_P$, the claim follows.
\end{proof}

From Lemma \ref{gapp-vanish} and Corollary \ref{fw} we immediately obtain the first part (i) of Theorem \ref{slang-petal}.  We now turn to the second part of the theorem.  For this, we need two further variants of Lemma \ref{gp-comp}.  For any $1 \leq j \leq s-1$, let $V_{\ind,j}$ be the subspace of $\Horiz_j(G_{(1)}) \times \Horiz_j(G_{(2)}) \times \Horiz_j(G_{(3)}) \times \Horiz_j(G_{(4)})$ generated by the elements
$$ (\phi_{(1)}(e_{j,k}), \phi_{(2)}(e_{j,k}), \phi_{(3)}(e_{j,k}),\phi_{(4)}(e_{j,k}))$$
for $1 \leq k \leq D_{*,j}$ and the elements
$$ (\phi_{(1)}(e_{j,k}), \id, \id,\id), (\id, \phi_{(2)}(e_{j,k}), \id,\id), (\id, \id, \phi_{(3)}(e_{j,k}),\id),
(\id, \id,\id, \phi_{(4)}(e_{j,k}))
$$
for $D_{*,j} < k \leq D_{*,j}+D'_{\ind,j}$.

\begin{lemma}[Components of $G_P$, II]\label{gp-comp2}  Let $1 \leq j \leq s-1$.  Then the projection of $G_P \cap G_{j}$ to $\Horiz_j(G_{(1)}) \times \Horiz_j(G_{(2)}) \times \Horiz_j(G_{(3)}) \times \Horiz_j(G_{(4)})$ contains $V_{\ind,j}$.
\end{lemma}

\begin{proof} 
Suppose the claim failed for some $j$.  Using \eqref{soo} and duality, we conclude that there exists a $\xi_j \in \Xi_j$ which annihilates the kernel of the projection to $\Horiz_j(G_{(1)}) \times \Horiz_j(G_{(2)}) \times \Horiz_j(G_{(3)}) \times \Horiz_j(G_{(4)})$, and which is non-trivial on $V_{\ind,j}$.  In particular, we have a decomposition of the form
\begin{equation}\label{xij}
 \xi_j(x_{(1)},x_{(2)},x_{(3)},x_{(4)},x_{(0)}) = \sum_{i=1}^4 \xi_{(i),j}(x_{(i)})
\end{equation}
for $x_{(i)} \in \Horiz_j(G_{(i)})$ for $i=1,2,3,4,0$, where $\xi_{(i),j}: \Horiz_j(G_{(i)}) \to \R$ for $i=1,2,3,4$ are characters.  

By definition of $\Xi_j$, we conclude that
$$\sum_{i=1}^4  \xi_{(i),j}( \Taylor_j(\orbit_{(i)}) )
= O(N^{-j}).$$
Inserting \eqref{star}, we conclude that
\begin{equation}\label{star2}
 \sum_{k=1}^{D_j} \sum_{i=1}^4 c_{(i),j,k} \xi_{h_i,j,k} = O(N^{-j}).
 \end{equation}
The left-hand side is an integer linear combination of the degree $j$ frequencies in
$$ \F_{*,\ind} \uplus \biguplus_{i=1}^4 \F_{h_i,\ind} \uplus \biguplus_{i=1}^4 \F_{h_i,\lin}.$$
Using the Freiman homomorphism property from Lemma \ref{sunflower} we can eliminate the role of $\F_{h_4,\lin}$, leaving only
$$ \F_{*,\ind} \uplus \biguplus_{i=1}^4 \F_{h_i,\ind} \uplus \biguplus_{i=1}^3 \F_{h_i,\lin}.$$
But this is just \eqref{jinnai-1} for $\{a,b,c\}=\{1,2,3\}$.  We conclude that the coefficients of the left-hand side of
\eqref{star2} in this basis vanish, which in terms of the original coefficients $c_{(i),j,k}$ means that
$$ \sum_{i=1}^4 c_{(i),j,k}=0$$
for $1 \leq k \leq D_{*,j}$, and
$$ c_{(i),j,k} = 0$$
for $D_{*,j} < k \leq D_{*,j} + D'_{\ind,j}$.  But this forces $\xi_j$ to vanish on $V_{\ind,j}$, a contradiction.
\end{proof}

We now apply the commutator argument to show that ``independent'' frequencies also ultimately have a trivial effect.

\begin{corollary}[Furstenberg-Weiss commutator argument, II]\label{fw2}  Let $w$ be an $(r_*-1)$-fold iterated commutator of generators $e_{j_1,k_1},\ldots,e_{j_{r_*},k_{r_*}}$ with $1 \leq j_l \leq s-1$, $1 \leq k_l \leq D_l$ for $l=1,\ldots,r_*$ and $j_1+\ldots+j_{r_*} = s-1$.  Suppose that at least one of the generators, say $e_{j_1,k_1}$, is an ``independent'' generator in the sense that $D_{*,j_1} < k_1 \leq D_{*,j_1} + D'_{\ind,j_1}$.  Then $(\phi_{(1)}(w),\id,\id,\id,\id) \in G_P$.
\end{corollary}

\begin{proof} We may assume that $k_l \leq D_{*,j_l}$ for all $2 \leq l \leq r_*$, as the claim would follow from Corollary \ref{fw} otherwise.

For $e_{j_1,k_1}$, we may invoke Lemma \ref{gp-comp2} and find an element $g_{j_1,k_1}$ of $G_P \cap G_{j_1}$ for which the first $4$ coordinates are  equal (modulo projection to $\Horiz_{j_1}(G_{(1)}) \times \Horiz_{j_1}(G_{(2)}) \times \Horiz_{j_1}(G_{(3)}) \times \Horiz_{j_1}(G_{(4)})$) is equal to $(\phi_{(1)}(e_{j_1,k_1}),\id,\id,\id)$.  For the other $e_{j,k}$, we can find  elements $g'_{j,k}$ of $G_P \cap G_{j}$ for which the first coordinate is equal (modulo projection to $\Horiz_{j}(G_{(1)})$) to $\phi_{(1)}(e_{j,k})$.  Taking commutators of $g_{j_1,k_1}$ and $g'_{j,k}$ in the order indicated by $w$, we obtain the claim.
\end{proof}

Combining Corollary \ref{fw2} with Lemma \ref{gapp-vanish} we obtain the second part of Theorem \ref{slang-petal}.

\section{Building a nilobject}\label{multi-sec}

The aim of this section is to at last build an object coming from an $s$-step nilmanifold. Recall from the discussion in \S \ref{overview-sec} that this object will be a multidegree $(1,s-1)$-nilcharacter $\chi'(h,n)$, and that this completes the proof of Theorem \ref{linear-induct}. This in turn was used iteratively to prove Theorem \ref{linear-thm}, the heart of our whole paper. It will then remain to supply the \emph{symmetry argument}, which will take us from a 2-dimensional nilsequence to a 1-dimensional one; this will be accomplished in the next section.

Let $f,H,\chi,(\chi_h)_{h \in H}$ be as in Theorem \ref{linear-induct}.  If we apply Lemma \ref{sunflower}, we obtain the following objects:
\begin{itemize}
\item A dense subset $H'$ of $H$;
\item Dimension vectors $\vec D_* = \vec D_{*,\ind} + \vec D_{*,\rat}$ and $\vec D' = \vec D'_\lin + \vec D'_\ind + \vec D'_\sml$, which we write as $\vec D_* = (D_{*,i})_{i=1}^{s-1}$, $\vec D_{*,\ind} = (D_{*,\ind,i})_{i=1}^{s-1}$, etc.;
\item A core horizontal frequency vector $\F_* = (\xi_{*,i,j})_{1 \leq i \leq s-1; 1 \leq j \leq D_{*,i}}$, which is partitioned as $\F_* = \F_{*,\ind} \uplus \F_{*,\rat}$, with the indicated dimension vectors $\vec D'_\ind, \vec D'_\rat$;
\item A petal horizontal frequency vector $\F'_h = (\xi'_{h,i,j})_{1 \leq i \leq s-1; 1 \leq j \leq D'_i}$, which is partitioned as $\F'_h = \F'_{h,\lin} \uplus \F'_{h,\ind} \uplus \F'_{h,\sml}$, which is a limit function of $h$ and with the indicated dimension vectors $\vec D'_\lin, \vec D'_\ind, \vec D'_\sml$;
\item Nilmanifolds $G_h/\Gamma_h$ and $G_{0,h}/\Gamma_{0,h}$ of degree-rank $\leq (s-1,r_*)$ and $\leq (s-1,r_*-1)$ respectively for each $h \in H'$, depending in a limit fashion on $h$;
\item Polynomial sequences $g_h, g_{0,h} \in \ultra \poly(\Z_\N \to (G_h)_\N)$ for each $h \in H'$, depending in a limit fashion on $h$;
\item Lipschitz functions $F_h \in \Lip(\ultra(G_h/\Gamma_h \times G_{0,h}/\Gamma_{0,h})\to \overline{S^{\omega}})$ for each $h \in H'$, depending in a limit fashion on $h$;
\item a filtered $\phi_h: G^{\vec D_* + \vec D'} \to G_h$ for each $h \in H'$, depending in a limit fashion on $h$; and
\item a character $\eta_h: G^{\vec D_* + \vec D'}_{(s-1,r_*)} \to \R$ for each $h \in H'$, depending in a limit fashion on $h$
\end{itemize}
that obey the following properties:
\begin{itemize}
\item For every $1 \leq i \leq d$ and $1 \leq j \leq D'_{i,\lin}$, there exists $\alpha_{i,j} \in \ultra\T$ and $\beta_{i,j} \in \ultra \R$ such that \eqref{xih-def} holds, and furthermore that the map $h \mapsto \xi'_{h,i,j}$ is a Freiman homomorphism on $H'$.
\item For almost all additive quadruples $(h_1,h_2,h_3,h_4)$ in $H$, 
$$\F_{*,\ind} \uplus \biguplus_{i=1}^4 \F'_{h_i,\ind} \uplus \biguplus_{i=1}^3 \F'_{h_i,\lin}$$ 
is independent.
\item We have the representation
$$ \chi_h(n) = F_h( g_h(n) \ultra \Gamma_h, g_{0,h}(n) \ultra \Gamma_{0,h} )$$
for every $h \in H'$.
\item $\phi_h: G^{\vec D_* + \vec D'} \to G_h$ is a filtered homomorphism such that
\begin{equation}\label{phil}
F_h( \phi_h(t) x, x_0 ) = e( \eta_h(t) ) F_h(x,x_0)
\end{equation}
for all $t \in G^{\vec D_* + \vec D'}_{(s-1,r_*)}$, $x \in G_h/\Gamma_h$, and $x_0 \in G_{0,h}/\Gamma_{0,h}$;
\item One has the Taylor coefficients
\begin{equation}\label{thune}
 \Taylor_i(g_h\Gamma_h) =  \pi_{\Horiz_i(G_h/\Gamma_h)}(\phi_h( \prod_{j=1}^{D_{*,i}+D'_i} e_{i,j}^{\xi_{h,i,j}} )) 
\end{equation}
for all $1 \leq i \leq s-1$.
\end{itemize}

There are only countably many nilmanifolds $G/\Gamma$ up to isomorphism, so by passing from $H'$ to a dense subset using Lemma \ref{dense-dich} we may assume that
$$ G_h/\Gamma_h = G/\Gamma \quad \mbox{and} \quad G_{0,h}/\Gamma_{0,h} = G_0/\Gamma_0$$
are independent of $h$.  Similarly we may take $\eta_h = \eta$ and $\phi_h = \phi$ to be independent of $h$.  From the Arzel\`a-Ascoli theorem, the space of possible $F_h$ is totally bounded, and so (shrinking $\eps$ slightly if necessary) we may also assume that $F_h = F$ is independent of $h$. 

For $j$ with $1 \leq j \leq D_{*,i}$, since $\xi_{h,i,j}$ is independent of $h$, we can ensure that $\xi_{h,i,j} =\gamma_{i,j}$ is also independent of $h$. Meanwhile, for $D_{*,i} < j \leq D_{*,i}+D'_{i,\lin}$, from \eqref{xih-def} we may assume that $\xi_{h,i,j}$ takes the form
$$ \xi_{h,i,j} = \{ \alpha_{i,j} h \} \beta_{i,j} \mod 1$$
for some $\alpha_{i,j} \in \ultra\T$ and $\beta_{i,j} \in \ultra \R$.  By passing to a dense subset of $H'$ using the pigeonhole principle, we may assume for each $i,j$, that $\{ \alpha_{i,j} h \}$ is contained in a subinterval $\ultra I_{i,j}$ around $\ultra 0$ of length at most $1/10$ (say).

We now wish to apply Theorem \ref{slang-petal} to obtain more convenient equivalent representatives (in $\Xi_{\DR}^{(s-1,r_*)}([N])$ ) $\tilde \chi_h$  for the nilcharacters $\chi_h$. Let $\tilde G$ be the free Lie group generated by the generators $\tilde e_{i,j}$ for $1 \leq i \leq s-1$ and $1 \leq j \leq D_{*_i} + D'_{\lin,i}$ subject to the following relations:
\begin{itemize}
\item Any $(r-1)$-fold iterated commutator of $\tilde e_{i_1,j_1},\ldots,\tilde e_{i_r,j_r}$ with $i_1+\ldots+i_r > s-1$ vanishes;
\item Any $(r-1)$-fold iterated commutator of $\tilde e_{i_1,j_1},\ldots,\tilde e_{i_r,j_r}$ with $i_1+\ldots+i_r = s-1$ and $r > r_*$ vanishes;
\item Any $(r-1)$-fold iterated commutator of $\tilde e_{i_1,j_1},\ldots,\tilde e_{i_r,j_r}$ in which $j_l > D_{*,i_l}$ for at least two values of $l$ vanishes.
\end{itemize}
We give this group a $\DR$-filtration $\tilde G_\DR$ by defining $\tilde G_{(d,r)}$ to be the group generated by the $(r'-1)$-fold iterated commutators of $\tilde e_{i_1,j_1},\ldots,\tilde e_{i_{r'},j_{r'}}$ with $i_1+\ldots+i_{r'} \geq d$ and $r' \geq r$.  We then let $\tilde \Gamma$ be the discrete group generated by the $\tilde e_{i,j}$; $\tilde G/\tilde \Gamma$ is then a nilmanifold of degree-rank $\leq (s-1,r_*)$.

Let $G^*$ be the subgroup of $G^{\vec D_* + \vec D'}$ generated by $(r-1)$-fold iterated commutators $\tilde e_{i_1,j_1},\ldots,\tilde e_{i_r,j_r}$ with $i_1+\ldots+i_r = s-1$ in which $j_l > D_{*,i_l}$ for at least two values of $l$, or $j_l > D_{*,i_l} + D'_{\lin,i_l}$ for at least one value of $l$.  Then $G^*$ is a subgroup of the central group $G^{\vec D_* + \vec D'}_{(s-1,r_*)}$ of $G^{\vec D_* + \vec D'}$, and $\tilde G$ is isomorphic to the quotient of $G^{\vec D_* + \vec D'}$ by $G^*$. We let $\tilde \phi: G^{\vec D_* + \vec D'} \to \tilde G$ denote the quotient map.  From Theorem \ref{slang-petal}, the character $\eta: G^{\vec D_* + \vec D'}_{(s-1,r_*)} \to \R$ annihilates $G^*$, and thus descends to a vertical character $\tilde \eta: \tilde G_{(s-1,r_*)} \to \R$.

We select a function $\tilde F \in \Lip( \tilde G/\tilde \Gamma \to S^\omega)$ with vertical frequency $\tilde \eta$; such a function can be built using the construction \eqref{fkts}.  

We then define the polynomial sequences $ g_0, \tilde g_h \in \ultra \poly(\Z_\N \to \tilde G_\N)$ by the formulae
\begin{align}
 g_0(n) &:= \prod_{i=1}^{s-1} \prod_{j=1}^{D_{*,i}} \tilde e_{i,j}^{\gamma_{i,j} \binom{n}{i}}\label{g0-def}\\
\tilde g_h(n) &:= \prod_{i=1}^{s-1} \prod_{j=D_{*,i}+1}^{D_{*,i}+D'_{\lin,i}} \tilde e_{i,j}^{\{\alpha_{i,j} h\} \beta_{i,j} \binom{n}{i}}\label{gh-def}
\end{align}
and consider the nilcharacter
\begin{equation}\label{chok}
 \tilde \chi_h(n) := \tilde F( g_0(n) \tilde g_h(n) \ultra \tilde \Gamma ).
\end{equation}

These nilcharacters are equivalent to $\chi_h$ in $\Symb_{\DR}^{(s-1,r_*)}([N])$, as the following lemma shows.

\begin{lemma}  For each $h \in H'$, $\chi_h$ and $\tilde \chi_h$ are equivalent \textup{(}as nilcharacters of degree-rank $(s-1,r_*)$\textup{)} on $[N]$.
\end{lemma}

\begin{proof}  Fix $h$.  It suffices to show that $\chi_h \otimes \overline{\tilde \chi_h}$ is a nilsequence of degree $<s-1$.  We can write this sequence as
\begin{equation}\label{fhg}
n \mapsto F'_h( g'_h(n) \ultra \Gamma'),
\end{equation}
where $G' := G \times G_0 \times \tilde G$, $\Gamma := \Gamma \times \Gamma_0 \times \tilde \Gamma$,
 $g'_h \in \ultra \poly( \Z_\N \to G'_\N )$ is the sequence
\[
g'_h(n) := ( g_h(n), g_{0,h}(n), g_0(n) \tilde g_h(n) )
\]
and $F'_h \in \Lip(\ultra(G'/\Gamma'))$ is the function
\[
F'_h(x,x_0,y) :=  F_h(x,x_0) \otimes \overline{\tilde F(y)}.
\]
We define a $\DR$-filtration $G'_\DR$ on $G'$ by defining $G'_{(d,r)}$ for $(d,r) \in \DR$ with $r \geq 1$ to be the Lie group generated by the following sets:
\begin{enumerate}
\item $G_{(d,r+1)} \times (G_0)_{(d,r)} \times \tilde G_{(d,r+1)}$;
\item $\{ (\phi(g), \id, \tilde \phi (g) ): g \in G^{\vec D_*+\vec D'}_{(d,r)} \}$,
\end{enumerate}
with the convention that $(d,d+1) = (d+1,0)$.
We also set $G'_{(d,0)} := G'_{(d,1)}$ for $d \geq 1$.  One easily verifies that this is a filtration.

We claim that $g'$ is polynomial with respect to this filtration.  Indeed, the sequence $n \mapsto (\id,g_{0,h}(n),\id)$ is already polynomial in this filtration, so by Corollary \ref{laz} it suffices to verify that the sequence
\begin{equation}\label{gig}
n \mapsto (g_h(n), \id, g_0(n) \tilde g_h(n))
\end{equation}
is polynomial.  We use Lemma \ref{taylo} to Taylor expand $g_h(n) = \prod_{i=0}^{s-1} g_{h,i}^{\binom{n}{i}}$
where $g_{h,i} \in G_{(i,0)}$.  From \eqref{thune}, one has
\[
g_{h,i} = \phi\big( \prod_{j=1}^{D_{*,i}+D'_i} e_{i,j}^{\xi_{h,i,j}} \big) \mod G_{(i,2)}.
\]
By construction of the filtration of $G'$, this implies that
\[
 \big( g_{h,i}, \id, \prod_{j=1}^{D_{*,i}+D_i'} e_{i,j}^{ \xi_{h,i,j}} \mod G^* \big) \in G'_{(i,1)}.
 \]
Applying Corollary \ref{laz}, we conclude that the sequence
\[
 n \mapsto \big(g_h(n), \id, \prod_{i=0}^{s-1} 
 ( \prod_{j=1}^{D_{*,i}+D_i'} e_{i,j}^{ \xi_{h,i,j}})^{\binom{n}{i}} \mod G^* \big)
 \]
 is polynomial with respect to the $G'$ filtration.
 Applying the Baker-Campbell-Hausdorff formula repeatedly, and using \eqref{g0-def}, \eqref{gh-def}, we see that
\[
 n \mapsto
\prod_{i=0}^{s-1}  (\prod_{j=1}^{D_{*,i}+D_i'} e_{i,j}^{ \xi_{h,i,j}})^{\binom{n}{i}} \mod G^*
 \]
differs from the sequence $n \mapsto g_0(n) \tilde g_h(n)$ by a sequence which is polynomial in the shifted filtration $(\tilde G_{(d,r+1)})_{(d,r) \in \DR}$.  We conclude that \eqref{gig} is polynomial as required.

Next, we claim that $F'_h$ is invariant with respect to the action of the central group
$$
G'_{(s-1,r_*)} = \{ (\phi(g), \id, \tilde \phi (g) ): g \in G^{\vec D}_{(s-1,r_*)} \}.
$$
It suffices to check this for generators $(\phi(w),\id,w \mod G^*)$, where $w$ is an $(r_*-1)$-fold commutator of 
$e_{i_1,j_1},\ldots,e_{i_{r_*},j_{r_*}}$ in $G^{\vec D}$ with $i_1+\ldots+i_r = s-1$.   There are two cases.  If one has $j_l > D_{*,i_l} + D'_{\lin,i_l}$ for some $l$, then $w$ lies in $G^*$ and is also annihilated by $\eta$, and the claim follows from \eqref{phil}.  If instead one has $j_l \leq D_{*,i_l} + D'_{\lin,i_l}$ for all $l$, then the claim again follows from \eqref{phil} together with the construction of $\tilde \eta$ and $\tilde F$.

We may now quotient out $G'_{(0,0)}$ by $G'_{(s-1,r_*)}$ and obtain a representation of
\eqref{fhg} as a nilsequence of degree-rank $<(s-1,r_*)$, as desired.
\end{proof}

From this lemma and Lemma \ref{symbolic}(ii) we can express $\chi_h$ as a bounded linear combination of $\tilde \chi_h \otimes \psi_h$ for some nilsequence $\psi_h$ of degree-rank $\leq (s-1,r_*-1)$.  Thus, to prove Theorem \ref{linear-induct} it suffices to show that there is a nilcharacter $\tilde \chi \in \Xi^{(1,s-1)}(\ultra \Z^2)$, such that $\tilde \chi_h(n) = \tilde \chi(h,n)$ for many $h \in H'$ and all $n \in [N]$.

We illustrate the construction with an example. Let 
$$G:= G^{(2,0)} = \{ e_1^{t_1} e_2^{t_2} [e_1,e_2]^{t_{12}}: t_1,t_2,t_{12} \in \R \}$$ 
be the universal degree $2$ nilpotent group \eqref{heisen} generated by $e_1,e_2$.  Let $F$ be the Lipschitz function in equation (\ref{fkts}). Suppose
\[ \chi_h(n) := F(g_h(n)\ultra \Gamma)
\] 
with $g_h(n) := e_2^{\beta n} e_1^{\alpha_h n} $, where $\alpha_h:=\{\delta h\} \gamma$, and $\alpha,\beta,\gamma \in \ultra \R$.  As computed in \S \ref{nilcharacters}, we have
\[F_k(g_h(n)\ultra \Gamma)= \phi_k(\alpha_hn \mod 1,\beta n \mod 1)e(\{\alpha_h n \} \beta n)\]
for some Lipschitz function $\phi_k: \T^2 \to \C$.
We would like to interpret the function $(h,n) \mapsto \chi_h(n)$ as a nilcharacter in $ \Xi_{\MD}^{(1,2)}(\ultra \Z^2)$.
The first task is to identify a subgroup $G_{\petal}$ of the  group $G$ representing that part of $G$ that is ``influenced by'' the petal frequency $\alpha_h$; more specifically, we take $G_{\petal}$
to be the subgroup of $G$ generated by $e_1$ and $[e_1, e_2]$, that is to say
$$ G_\petal = \langle e_1, [e_1,e_2] \rangle_\R = \{ e_1^{t_1} [e_1,e_2]^{t_{12}}: t_1,t_{12} \in \R \}.$$
Note that $G_{\petal}$ is abelian and normal in $G$. In particular $G$ acts on $G_{\petal}$ by conjugation, and we may form the semidirect product 
$$G \ltimes G_{\petal} := \{ (g,g_1): g \in G, g_1 \in G_\petal \},$$
defining multiplication by
\[ 
(g, g_1)\cdot (g', g'_1) = (gg', g_1^{g'} g'_1),
\] 
where $a^b := b^{-1} a b$ denotes conjugation.

Now consider the action $\rho$ of $\R$ on $G \ltimes G_{\petal}$ defined by 
\[
 \rho(t)(g, g_1) := (g g_1^t, g_1).
 \] 
We may form a further semidirect product
\[ G' := \R \ltimes_{\rho} (G \ltimes G_{\petal}),\] in which the product operation is defined by
\[ 
(t, (g, g_1)) \cdot (t', (g', g'_1)) = (t + t', \rho(t')(g, g_1) \cdot (g', g'_1)).
\]
$G'$ is a Lie group; indeed, one easily verifies that it is $3$-step nilpotent. We give $G'$ a $\N^2$-filtration: 
\begin{align*}
G'_{(0,0)}&:= G' \\
G'_{(1,0)}&:=\{(t,(g,\id)): t \in \R, g \in G_\petal \} \\
G'_{(1,1)}&:=\{(0,(g,\id)): g \in G_\petal\},\\
G'_{(1,2)}&:=\{(0,(g,\id)): g \in [G,G]\}, \\
G'_{(0,1)}&:=\{(0,(g,g_1)): g \in G_\petal; g_1 \in G_{\petal}\},\\
G'_{(0,2)}&:=\{(0,(g,g_1)): g, g_1\in [G,G]\}, 
\end{align*}
with $G'_{i,j}:=\{\id\}$ for all other $(i,j) \in \N^2$.  One easily verifies that this is a filtration.
  Inside $G'$ we take the lattice 
\[ 
 \Gamma' := \Z \ltimes_{\rho} (\Gamma \ltimes \Gamma_{\petal}),
\]
where $\Gamma_{\petal} := \Gamma \cap G_{\petal}$. Now consider the polynomial $g':\Z^2 \to G'$ defined by
\[
 g'(h, n) := (0, (e_2^{\beta n}, e_1^{\gamma n})) \cdot (\delta h, (\id, \id))
 \]
and observe that
\begin{align*} g'(h,n)\Gamma' & = (0, (e_2^{\beta n}, e_1^{\gamma n})) \cdot (\{\delta h\} , (\id, \id))  \Gamma' \\ & = (\{\delta h\}, (e_2^{\beta n} e_1^{\{\delta h\}\gamma n}, e_1^{\gamma n}))\Gamma'.
\end{align*}
For a dense subset $H''$, $\{\delta h\}$ is in a small interval $I$, and let $\psi$ be a smooth cutoff function supported on $2I$. 
Take $ F' :  G'/ \Gamma' \rightarrow \C^D$ to be the function defined by
\[ F'((t, (g, g'))\Gamma') := \psi(t) F(g\Gamma)\] whenever $ t \in I$ and $0$ otherwise. Then we have for $h \in H''$
\[  F'(g'(h,n)\tilde\Gamma) = F(e_2^{\beta n} e_1^{\{\delta h\}\gamma n}\Gamma)=\chi_h(n),\]
giving the desired representation of $(h,n) \mapsto \chi_h(n)$ as an (almost) degree $(1,2)$ nilcharacter.\vspace{11pt}

We now turn to the general case.  Our construction shall proceed by an abstract algebraic construction.
Let $\tilde G_{\petal}$ be the subgroup of $\tilde G$ generated by $(r-1)$-fold ($r \ge 1$) iterated commutators of $\tilde e_{i_1,j_1},\ldots,\tilde e_{i_r,j_r}$ in which $j_l > D_{*,i_l}$ for exactly one value of $l$.  Then $\tilde G_{\petal}$ is a rational abelian normal subgroup of $\tilde G$. To see that $\tilde G_{\petal}$ is normal, ones uses the equalities
\[
\tilde e^{-1}_{i,j}[g,h] \tilde e_{i,j}=[\tilde e^{-1}_{i,j}g\tilde e_{i,j},\tilde e^{-1}_{i,j}h \tilde e_{i,j}]  \quad \mbox{and} \quad \tilde e^{-1}_{i,j}g \tilde e_{i,j}= g[g,\tilde e_{i,j}],
\]
the commutator identities in equation (\ref{com-ident}), and the fact that any iterated commutators of $\tilde e_{i_1,j_1},\ldots,\tilde e_{i_r,j_r}$ in which $j_l > D_{*,i_l}$ for more than one value of $l$ is trivial in $\tilde G$. 

  In particular, $\tilde G$ acts on $\tilde G_{\petal}$ by conjugation, leading to the semidirect product $\tilde G \ltimes \tilde G_{\petal}$ of pairs $(g,g_1)$ with the product
$$ (g,g_1) (g',g'_1) := (gg', g_1^{g'} g'_1).$$
Next, let $R$ be the commutative ring of tuples $t = (t_{i,j})_{1 \leq i \leq s-1; D_{*,i} < j \leq D_{*,i}+D'_{\lin,i}}$ with $t_{i,j} \in \R$, which we endow with the pointwise product.  For each $t \in R$, we can define an homomorphism $g \mapsto g^t$ on $\tilde G$, which we define on generators by mapping $\tilde e_{i,j}$ to $\tilde e_{i,j}^t$ for $D_{*,i} < j \leq D_{*,i}+D'_{\lin,i}$, but preserving $\tilde e_{i,j}$ for $j \leq D_{*,i}$.  Such a homomorphism is well-defined as it preserves the defining relations of $\tilde G$.  We observe the composition law
$$ (g^t)^{t'} = g^{tt'}$$
for $g \in \tilde G$ and $t,t' \in R$.  Also, on the abelian subgroup $\tilde G_{\petal}$ on $\tilde G$, we see that
\begin{equation}\label{g0g}
g^t g^{t'} = g^{t+t'}
\end{equation}
as can be seen from the Baker-Campbell-Hausdorff formula \eqref{bch}.  We can thus express
\begin{equation}\label{chok2}
\tilde g_h(n) = g_1(n)^{\{ \alpha h \}}
\end{equation}
where $g_1 \in \ultra \poly(\Z_\N \to (\tilde G_{\petal})_\N)$ is the polynomial sequence
\[
g_1(n) := \prod_{i=1}^{s-1} \prod_{j=D_{*,i}+1}^{D_{*,i}+D'_{\lin,i}} \tilde e_{i,j}^{ \beta_{i,j} \binom{n}{i}}
\]
and $\{ \alpha h \} \in R$ is the element
$$ \{ \alpha h \} := ( \{ \alpha_{i,j} h \} )_{1 \leq i \leq s-1; D_{*,i} < j \leq D_{*,i}+D'_{\lin,i}}.$$
The homomorphism $g \mapsto g^t$ preserves $\tilde G_{\petal}$, and is the identity once $\tilde G_{\petal}$ is quotiented out.  As a consequence we see that
\begin{equation}\label{g1g}
 (g g_1 g^{-1})^t = g g_1^t g^{-1}
 \end{equation}
for any $g \in \tilde G$ and $g_1 \in \tilde G_{\petal}$.

We can now define an action $\rho$ of $R$ (viewed now as an additive group) on $\tilde G \ltimes \tilde G_{\petal}$ by defining
$$ \rho(t)( g, g_1 ) := (g g_1^t, g_1);$$
the properties \eqref{g0g}, \eqref{g1g} ensure that this is indeed an action.  We can then define the semi-direct product $G' := R \ltimes_\rho (\tilde G \ltimes \tilde G_{\petal})$ to be the set of pairs $(t, (g,g_1) )$ with the product
$$ (t, (g,g_1)) (t', (g',g'_1)) = (t+t', \rho(t')(g,g_1) (g',g'_1)).$$
This is a Lie group.  We can give it a $\N^2$-filtration $(G'_{(d_1,d_2)})_{(d_1,d_2) \in \N^2}$ as follows:
\begin{enumerate}
\item If $d_1 > 1$, then $G'_{(d_1,d_2)} := \{\id\}$.
\item If $d_1=1$ and $d_2 > 0$, then $G'_{(1,d_2)}$ consists of the elements $(0,(g,\id))$ with $g \in \tilde G_{d_2} \cap \tilde G_\petal$.
\item If $d_1=1$ and $d_2 = 0$, then $G'_{(1,0)}$ consists of the elements $(t,(g,\id))$ with $t \in R$ and $g \in \tilde G_\petal$.
\item If $d_1=0$ and $d_2 > 0$, then $G'_{(0,d_2)}$ consists of the elements $(0,(g,g_1))$ with $g \in \tilde G_{d_2}$ and $g_1 \in \tilde G_{\petal} \cap \tilde G_{d_2}$.
\item $G'_{(0,0)} = G'$.
\end{enumerate}
One easily verifies that this is a filtration of degree $\leq (1,s-1)$ with $G'_{(0,0)} = G'$.  

We let $\Gamma'$ be the subgroup of $\tilde G$ consisting of pairs $(t,(g,g_1))$ with $g \in \tilde \Gamma$, $g_1 \in \tilde \Gamma_{\petal}$, and with all coefficients of $t$ integers.  One easily verifies that $\Gamma'$ is a cocompact subgroup of $G'$, and that the above $\N^2$-filtration of $G'$ is rational with respect to $\Gamma'$, so that $G'/\Gamma'$ has the structure of a filtered nilmanifold. 

We consider the orbit $\orbit' \in \ultra \poly(\Z^2_{\N^2} \to (G'/\Gamma')_{\N^2})$ defined by
$$ \orbit'(h,n) := (0,(g_0(n),g_1(n))) (\alpha h, (\id,\id)) \ultra \Gamma',$$
where
$$ \alpha h := ( \alpha_{i,j} h )_{1 \leq i \leq s-1; D_{*,i} < j \leq D_{*,i}+D'_{\lin,i}}.$$
As $g_0$, $g_1$ were already known to be polynomial maps, and the linear map $h \mapsto\alpha h$ is clearly polynomial also, we see from Corollary \ref{laz} and the choice of filtration on $G'$ that $\orbit'$ is a polynomial orbit.

Now we simplify the orbit.  Working on the abelian group $R$, we see that
$$ (\alpha h, (\id,\id)) \ultra \Gamma' = (\{\alpha h\}, (\id,\id)) \ultra \Gamma',$$
and then commuting this with $(0,(g_0(n),g_1(n)))$, we obtain
\begin{equation}\label{orb}
\orbit'(h,n) = (\{\alpha h\},  (g_0(n) g_1(n)^{\{\alpha h\}}, g_1(n) ) ) \ultra \Gamma'.
\end{equation}
Recall that for many $h \in H$ that each component $\{ \alpha_{i,j} h\}$ of $\{\alpha h \}$ lies in an interval $I_{i,j}$ of length at most $1/10$.  Let $2I_{i,j}$ be the interval of twice the length and with the same centre as $I_{i,j}$, and let $\varphi_{i,j}: \R \to \R$ be a smooth cutoff function supported on $I_{i,j}$.  We then define a function $F': G'/\Gamma' \to \C^\omega$ by setting
$$ F'( ((t_{i,j})_{1 \leq i \leq s-1; D_{*,i} < j \leq D_{*,i}+D'_{\lin,i}}, (g, g_1)) \ultra \Gamma' ) := \big(\prod_{i=1}^{s-1} \prod_{j=D_{*,i}+1}^{D_{*,i}+D'_{\lin,i}} \varphi_{i,j}(t_{i,j})\big) \tilde F(g \ultra \tilde \Gamma)$$
whenever $(g,g_1) \in \tilde G \ltimes \tilde G_{\petal}$ and $t_{i,j} \in 2I_{i,j}$ for all $1 \leq i \leq s-1$ and $D_{*,i} < j \leq D_{*,i}+D'_{\lin,i}$, with $F'$ set equal to zero whenever no representation of the above form exists.  One can easily verify that $F'$ is well-defined and Lipschitz.  Since $\tilde F$ has vertical frequency $\tilde \eta$, $F'$ has vertical frequency $\eta': G'_{(1,s-1)} \to \R$, defined by the formula
$$ \eta'( (0, (g, \id) ) := \tilde \eta(g)$$
for all $g \in \tilde G_{s-1}$.  From \eqref{chok}, \eqref{chok2} and \eqref{orb}, we see that for many $h \in H'$ we have
$$ \tilde \chi_h(n) = F' \circ \orbit'(h,n)$$
for all $n \in [N]$.  By construction, $F' \circ \orbit' \in \Xi_{\MD}^{(1,s-1)}(\ultra \Z^2)$, and Theorem \ref{linear-thm} follows.\\

\section{The symmetry argument}\label{symsec}

In this, the last section of the main part of the paper, we supply the symmetry argument, Theorem \ref{aderiv}; we recall that statement now.

\begin{theorem74-repeat} Let $f \in L^\infty[N]$, let $H$ be a dense subset of $[[N]]$, and let $\chi \in \Xi^{(1,s-1)}(\ultra \Z^2)$ be such that $\Delta_h f$ $<(s-2)$-correlates with $\chi(h,\cdot)$ for all $h \in H$.  Then there exists a nilcharacter $\Theta \in \Xi^{s}(\ultra \Z)$ \textup{(}with the degree filtration\textup{)} and a nilsequence $\Psi \in \Nil^{\subset J}(\ultra \Z^2)$ \textup{(}with the multidegree filtration\textup{)}, with $J$ given by the downset
\begin{equation}\label{lower-again}
J := \{ (i,j) \in \N^2: i+j \leq s-1 \} \cup \{ (i,s-i): 2 \leq i \leq s \},
\end{equation}
such that $\chi(h,n)$ is a bounded linear combination of $\Theta(n+h) \otimes \overline{\Theta(n)} \otimes \Psi(h,n)$.
\end{theorem74-repeat}

\begin{example} Suppose that $s=2$, $\chi(h,n) = e(P(h,n))$, and $P(h,n): \ultra \Z^2 \to \ultra \R$ is a symmetric bilinear form in $n,h$.  Then observe that
\begin{equation}\label{chan-sym}
 \chi(h,n) = \Theta(n+h) \overline{\Theta(n)} \Psi(h,n)
\end{equation}
where $\Theta(n) := e( \frac{1}{2} P(n,n) )$ and $\Psi(h,n) := e( - \frac{1}{2} P(h,h) )$, which illustrates a special case of Theorem \ref{aderiv}.  More generally, if $s \geq 2$ and $\chi(h,n) = e(P(h,n,\ldots,n))$ with $P(h,n_1,\ldots,n_{s-1}): \ultra \Z^s \to \ultra \R$ a symmetric multilinear form, then we have \eqref{chan-sym} with $\Theta(n) := e( \frac{1}{s} P(n,\ldots,n) )$, and $\Psi(h,n)$ a polynomial phase involving terms of multidegree $(i,s-i)$ in $h,n$ with $2 \leq i \leq s$.  Thus we again obtain a special case of Theorem \ref{aderiv}.  Note how the symmetry of $P$ is crucial in order to make these examples work, which explains why we refer to Theorem \ref{aderiv} as a symmetrisation result.  Morally speaking, this type of symmetry property ultimately stems from the identity $\Delta_h \Delta_k f = \Delta_k \Delta_h f$.  We remark that an analogous symmetrisation result was crucial to the analogous proof of $\GI(2)$ in \cite{green-tao-u3inverse} (see also \cite{sam}), although our arguments here are slightly different.
\end{example}

From the inclusions at the end of \S \ref{nilcharacters}, $\chi(h,n)$ is a nilcharacter on $\Z^2$ (with the degree filtration) of degree $\leq s$.  For similar reasons, any nilsequence $\Psi(h,n)$ of degree $\leq s-1$ (using the degree filtration on $\Z^2$) will automatically be of the form required for Theorem \ref{aderiv}.  In view of this and Lemma \ref{symbolic}, we see that it will suffice to obtain a factorisation of the form
$$ [\chi]_{\Xi^s([[N]] \times [N])} = [\Theta(n+h)]_{\Xi^s([[N]] \times [N])} - [\Theta(n)]_{\Xi^s([[N]] \times [N])}
+ [\Psi(h,n)]_{\Xi^s([[N]] \times [N])}$$
where $\Theta \in \Xi^s(\ultra \N)$ is a one-dimensional nilcharacter of degree $\leq s$ (which automatically makes $(h,n) \mapsto \Theta(n)$ and $(h,n) \mapsto \Theta(n+h)$ two-dimensional nilcharacters of degree $\leq s$, by Lemma \ref{symbolic}(vi)), and $\Psi \in \Xi^s(\ultra \N^2)$ is a two-dimensional nilcharacter of multidegree
\begin{equation}\label{slosh}
\subset \{ (i,j) \in \N^2: i+j \leq s; j \leq s-2 \}.
\end{equation}
The set of classes $[\Psi(h,n)]_{\Xi^s([[N]] \times [N])}$, with $\Psi$ of the above form, is a subgroup of the space $\Symb^s([[N]] \times [N])$ of all symbols of degree $s$ nilcharacters on $[[N]] \times [N]$.  Denoting the equivalence relation induced by these classes as $\equiv$, our task is thus to show that
$$ [\chi]_{\Xi^s([[N]] \times [N])} \equiv [\Theta(n+h)]_{\Xi^s([[N]] \times [N])} - [\Theta(n)]_{\Xi^s([[N]] \times [N])}.$$

In view of Theorem \ref{multilinearisation} and Lemma \ref{symbolic} (vii), there is a nilcharacter $\tilde \chi$ on  $\ultra \Z^s$ of degree $(1,\ldots,1)$ which is symmetric in the last $s-1$ variables, and such that 
\begin{equation}\label{change}
[ \chi(h,n) ]_{\Xi^s(\ultra \Z^2)} = s [ \tilde \chi(h,n,\ldots,n) ]_{\Xi^s(\ultra \Z^2)}.
\end{equation}
Inspired by the polynomial identity
$$ s h n^{s-1} = (n+h)^s - n^s - \ldots$$
where the terms in $\ldots$ are of degree $s$ in $h,n$ but of degree at most $s-2$ in $n$, we now choose
$$ \Theta(n) := \tilde \chi(n,\ldots,n).$$
From Lemma \ref{symbolic} (vi) we see that $\Theta$ is a nilcharacter of degree $\leq s$.  Our task is now to show that
\begin{align}\nonumber
 [\tilde \chi(n+h,\ldots,n+h)]_{\Xi^s([[N]] \times [N])} -&
[\tilde \chi(n,\ldots,n)]_{\Xi^s([[N]] \times [N])} - \\ & - 
s[\tilde \chi(h,n\ldots,n)]_{\Xi^s([[N]] \times [N])} \equiv 0.\label{tilch}
\end{align}

To manipulate this, we use the following lemma.

\begin{lemma}[Multilinearity]\label{multil}  Let $\tilde \chi$ be a nilcharacter on $\Z^s$ \textup{(}with the multidegree filtration\textup{)} of degree $(1,\ldots,1)$.  Let $m \geq 1$ be standard, and let $L_1,\ldots,L_s: \Z^m \to \Z$ and $L'_1: \Z^m \to \Z$ be homomorphisms.  Then we have linearity in the first variable, in the sense that
\begin{align*}
 [\tilde \chi(L_1(\vec n)+L'_1(\vec n),L_2(\vec n),\ldots,L_s(\vec n))]_{\Xi^s(\ultra \Z^m)}
&= [\tilde \chi(L_1(\vec n),L_2(\vec n),\ldots,L_s(\vec n))]_{\Xi^s(\ultra \Z^m)}\\
&\quad + [\tilde \chi(L'_1(\vec n),L_2(\vec n),\ldots,L_s(\vec n)]_{\Xi^s(\ultra \Z^m)},
\end{align*}
where $\vec n = (n_1,\ldots,n_m)$ are the $m$ independent variables of $\ultra \Z^m$, and $\Z^m$ is given the degree filtration.  We similarly have linearity in the other $s-1$ variables.
\end{lemma}

\begin{proof} We prove the claim for the first variable, as the other cases follow from symmetry.  From Lemma \ref{baby-calculus} and Lemma \ref{symbolic}(vi), it will suffice to show that the expression
\begin{equation}\label{touch}
\tilde \chi(h_1+h'_1,h_2,\ldots,h_s) \otimes \overline{\tilde \chi}(h_1,h_2,\ldots,h_s) \otimes \overline{\tilde \chi}(h'_1,h_2,\ldots,h_s)
\end{equation}
is a degree $<s$ nilsequence in $h_1,h'_1,h_2,\ldots,h_s$ (using the degree filtration).

Write $\tilde \chi(h_1,\ldots,h_s) = F( g(h_1,\ldots,h_s) \ultra \Gamma)$, where $G/\Gamma$ is a $\N^s$-filtered nilmanifold of degree $\leq (1,\ldots,1)$, $F \in \Lip(\ultra(G/\Gamma))$ has a vertical frequency, and $g \in \ultra \poly(\Z^s_{\N^s} \to G_{\N^s})$.  Then the expression \eqref{touch} takes the form
$$ \tilde F( \tilde g(h_1,h'_1,h_2,\ldots,h_s) \ultra \Gamma^3 )$$
where $\tilde g: \ultra \Z^{s+1} \to G^3$ is the map
$$ \tilde g(h_1,h'_1,h_2,\ldots,h_s) := ( g( h_1+h'_1,h_2,\ldots,h_s), g( h_1,h_2,\ldots,h_s), g(h'_1,h_2,\ldots,h_s) )$$
and $\tilde F \in \Lip(\ultra(G/\Gamma)^3)$ is the map
$$ \tilde F( x_1,x_2,x_3) = F(x_1) \otimes \overline{F(x_2)} \otimes \overline{F(x_3)}.$$
By Lemma \ref{taylo}, we can expand
$$ g(h_1,\ldots,h_s) = \prod_{i_1,\ldots,i_s = \{0,1\}} g_{i_1,\ldots,i_s}^{\binom{h_1}{i_1} \ldots \binom{h_s}{i_s}}$$
for some $g_{i_1,\ldots,i_s} \in G_{(i_1,\ldots,i_s)}$, where we order $\{0,1\}^s$ lexicographically (say).  

We now $\N$-filter $G^3$ by defining\footnote{In the published version of the paper there is an error in the definition of this filtration, which was pointed out by Jan Fornal. We thank James Leng for indicating the appropriate correction.} $(G^3)_{i}$ to be the group generated by $G_{(i_1,\ldots,i_s)}^{\Delta} = \{ (g,g,g) : g \in G_{(i_1,\dots, i_s)}\}$ for all $i_1,\ldots,i_s \in \N$ with $i_1+\ldots+i_s = i$ and $i_1 = 0$, together with the groups $\{ (g_1g_2,g_1,g_2): g_1,g_2 \in G_{(i_1,\ldots,i_s)} \}$ for $i_1+\ldots+i_s = i$ and $i_1 = 1$.  From the Baker-Campbell-Hausdorff formula \eqref{bch} one verifies that this is a rational filtration of $G^3$.  From the Taylor expansion we also see that $\tilde g$ is polynomial with respect to this filtration (giving $\Z^{s+1}$ the degree filtration).  Finally, as $F$ has a vertical character, we see that $\tilde F$ is invariant with respect to the action of $(G^3)_{s} = \{ (g_1g_2,g_1,g_2): g_1,g_2 \in G_{(1,\ldots,1)}\}$.  Restricting $G^3$ to $(G^3)_{0}$ and quotienting out by $(G^3)_{s}$ we obtain the claim.

\end{proof}

Using this lemma repeatedly, together with the symmetry of $\tilde \chi$ in the final $s-1$ variables, we see that we can expand
\[\begin{split}
& [\tilde \chi(n+h,\ldots,n+h)]_{\Xi^s(\ultra \Z^2)} =\\
& \sum_{j=0}^{s-1} \binom{s-1}{j} 
\left( [\tilde \chi(n,h,\ldots,h,n,\ldots,n)]_{\Xi^s(\ultra \Z^2)} + [\tilde \chi(h,h,\ldots,h,n,\ldots,n)]_{\Xi^s(\ultra \Z^2)} \right),
\end{split}\]
where in the terms on the right-hand side, the final $j$ coefficients are equal to $n$, the first coefficient is either $n$ or $h$, and the remaining coefficients are $h$.  Note that a term with $j$ $h$ factors and $(s-j)$ $n$ factors will have degree \eqref{slosh} and thus be negligible as long as $j \geq 2$.  Neglecting these terms, we obtain the simpler expression
\[ \begin{split}
[\tilde \chi(n+h,\ldots,n+h)]_{\Xi^s(\ultra \Z^2)}
 \equiv &
[\tilde \chi(n,\ldots,n)]_{\Xi^s(\ultra \Z^2)}
+
[\tilde \chi(h,n,\ldots,n)]_{\Xi^s(\ultra \Z^2)} \\ 
& + (s-1) [\tilde \chi(n,h,n,\ldots,n)]_{\Xi^s(\ultra \Z^2)}.
\end{split}
\]
Comparing this with \eqref{slosh}, we will be done as soon as we can show the symmetry property
\begin{equation}\label{total-sym}
(s-1) [\tilde \chi(h,n,\ldots,n)]_{\Xi^s([[N]] \times [N])} = (s-1) [\tilde \chi(n,h,n,\ldots,n)]_{\Xi^s([[N]] \times [N])}.
\end{equation}

This property does not automatically follow from the construction of $\tilde \chi$.  Instead, we must use the correlation properties of $\chi$, as follows.

By hypothesis and Lemma \ref{limone}, we have that for all $h$ in 
a dense subset $H$ of $[[N]]$, we can find a degree $\leq s-2$ nilcharacter $\varphi_h$ such that $f_1(\cdot+h)f_2(\cdot)$ correlates with $\chi(h,\cdot,\ldots,\cdot) \otimes \varphi_h$.  By Corollary \ref{mes-select}, we may assume that the map $h \mapsto \varphi_h$ is a limit map.  We set $\varphi_h=0$ for $h \not \in H$.  

To use this information, we return\footnote{Here is a key place where we use the hypothesis $s \geq 3$ (the other is Lemma \ref{discard}).  For $s=2$ the lower order terms in Proposition \ref{cs} are useless; however a variant of the argument below still works, see \cite{green-tao-u3inverse}.} to Proposition \ref{cs}.  Invoking that proposition, we see that for many additive quadruples $(h_1,h_2,h_3,h_4)$ in $[[N]]$, the sequence
\begin{align*}
n &\mapsto \chi(h_1,n) \otimes \chi(h_2,n+h_1-h_4) \otimes \overline{\chi(h_3,n)} \otimes \overline{\chi(h_4,n+h_1-h_4)}\\
&\quad \otimes \varphi_{h_1}(n) \otimes \varphi_{h_2}(n + h_1 - h_4) \otimes \overline{\varphi_{h_3}(n)} \otimes \overline{\varphi_{h_4}(n + h_1 - h_4)}
\end{align*}
is biased.

We make the change of variables $(h_1,h_2,h_3,h_4) = (h+a,h+b,h+a+b,h)$ and then pigeonhole in $h$, to conclude the existence of an $h_0$ for which
$$ n \mapsto \tau(a,b,n) \otimes \varphi_{h_0+a}(n) \otimes \varphi_{h_0+b}(n+a) \otimes \overline{\varphi}_{h_0+a+b}(n) \otimes \overline{\varphi_{h_0}(n+a)}$$
is biased for many pairs $a,b \in [[2N]]$, where $\tau = \tau_{h_0}$ is the expression
\begin{equation}\label{tabn}
 \tau(a,b,n) := \chi(h_0+a,n) \otimes \chi(h_0+b,n+a) \otimes \overline{\chi(h_0+a+b,n)} \otimes \overline{\chi(h_0,n+a)}.
\end{equation}
Henceforth $h_0$ is fixed, and we will suppress the dependence of various functions on this parameter.  From Lemma \ref{baby-calculus}, $\tau$ is a degree $\leq 3$ nilcharacter on $\ultra \Z^3$ (with the degree filtration).  We record its top order symbol:

\begin{lemma}\label{calc-1}  We have
$$ [\tau(a,b,n)]_{\Xi^s(\ultra \Z^3)} \equiv s(s-1) [\tilde \chi(b,a,n,\ldots,n)]_{\Xi^s(\ultra \Z^3)}$$
where by $\equiv$ we are quotienting by all symbols of degree $\leq s-3$ in $n$.
\end{lemma}

\begin{proof} From \eqref{change}, \eqref{tabn}, Lemma \ref{baby-calculus} and Lemma \ref{symbolic} one has
\begin{align*} [\tau(a,b,n)]_{\Xi^s(\ultra \Z^3)} = & s( [\tilde \chi(a,n,\ldots,n)]_{\Xi^s(\ultra \Z^3)} + [\tilde \chi(b,n+a,\ldots,n+a)]_{\Xi^s(\ultra \Z^3)} - \\ & - [\tilde \chi(a+b,n,\ldots,n)]_{\Xi^s(\ultra \Z^3)}).
\end{align*}
Applying Lemma \ref{multil} in the first variable we simplify this as
$$ s 
( [\tilde \chi(b,n+a,\ldots,n+a)]_{\Xi^s(\ultra \Z^3)} - [\tilde \chi(a,n,\ldots,n)]_{\Xi^s(\ultra \Z^3)}).$$
Applying Lemma \ref{multil} in all the other variables and gathering terms using the symmetry of $\tilde \chi$ in those variables, we arrive at
$$ \sum_{j=0}^{s-2} s \binom{s-1}{j} [\tilde \chi(b,a,\ldots,a,n,\ldots,n)]_{\Xi^s(\ultra \Z^3)},$$
where there are $j$ occurrences of $n$ and $s-1-j$ occurrences of $a$.  All the terms with $j<s-2$ are of degree $\leq s-2$ in $n$, and the claim follows.
\end{proof}

From Lemma \ref{symbolic}, we know that $\varphi_{h_0+b}(n+a)$ is a bounded linear combination of $\varphi_{h_0+b}(n) \otimes \psi_{a,b}(n)$ for some degree $\leq s-3$ nilsequence $\psi_{a,b}$.  Similarly for $\varphi_{h_0}(n+a)$.  We conclude that
$$ n \mapsto \tau(a,b,n) \otimes \varphi_{h_0+a}(n) \otimes \varphi_{h_0+b}(n) \otimes \overline{\varphi}_{h_0+a+b}(n) \otimes \overline{\varphi_{h_0}(n)}$$
is $\leq (s-3)$-biased for many $a,b \in [[2N]]$.

We will now eliminate the $\varphi_h$ terms in order to focus attention on $\tau$.  Applying Corollary \ref{mes-select}, we may thus find a scalar degree $\leq s-3$ nilsequence $\psi_{a,b}$ depending in a limit fashion on $a, b \in [[2N]]$, such that
\begin{align*} |\E_{a,b \in [[2N]]; n \in [N]}
\tau(a,b,n) \otimes \varphi_{h_0+a}(n) \otimes \varphi_{h_0+b}(n) \otimes & \overline{\varphi_{h_0+a+b}(n)} \otimes \\ & \otimes \overline{\varphi_{h_0,k'}(n+a)} \psi_{a,b}(n)| \gg 1.\end{align*}
We pull out the $b$-independent factors $\varphi_{h_0+a}(n) \otimes \overline{\varphi}_{h_0}(n)$ and Cauchy-Schwarz in $a,n$ to conclude that
\begin{align*}
|\E_{a,b,b' \in [[2N]]; n \in [N]}
\tau(a,b,n) &\otimes \overline{\tau(a,b',n)} \otimes \varphi_{h_0+b}(n) \otimes \overline{\varphi_{h_0+b'}(n)} \\
&\otimes \overline{\varphi_{h_0+a+b}(n)} \otimes \varphi_{h_0+a+b'}(n) \psi_{a,b,b'}(n)| \gg 1,
\end{align*}
where $(a,b,b') \mapsto \psi_{a,b,b'}$ is a limit map assigning a scalar degree $\leq s-3$ nilsequence to each $a,b,b'$.
Next, we make the substitution $c := a+b+b'$ and conclude that
\begin{align*}
|\E_{c,b,b' \in [[3N]]; n \in [N]}&
\tau(c-b-b',b,n) \otimes \overline{\tau(c-b-b',b',n)} \\
&\otimes \varphi_{h_0+b}(n) \otimes \overline{\varphi_{h_0+b'}}(n) \otimes \overline{\varphi_{h_0+c-b'}(n)} \varphi_{h_0+c-b}(n) \psi'_{c,b,b'}(n)| \gg 1
\end{align*}
where $(c,b,b') \mapsto \psi'_{c,b,b'}$ is a limit map assigning a scalar degree $\leq s-3$ nilsequence to each $c,b,b'$.
By the pigeonhole principle, we can thus find a $c_0$ such that
\begin{equation}\label{retour}
 |\E_{b,b' \in [[3N]]; n \in [N]}
\alpha(b,b',n) \otimes \varphi'_b(n) \otimes \overline{\varphi'_{b'}(n)} \psi'_{c_0,b,b'}(n)| \gg 1
\end{equation}
where $\alpha = \alpha_{c_0}$ is the form
\begin{equation}\label{abab}
 \alpha(b,b',n) := \tau(c_0-b-b',b,n) \otimes \overline{\tau(c_0-b-b',b',n)}
\end{equation}
and $\varphi'_b = \varphi'_{b,c_0}$ is the quantity
$$ \varphi'_b(n) := \varphi_{h_0+b,k}(n) \otimes \overline{\varphi_{h_0+c_0-b}(n)}.$$
We fix this $c_0$.  Again by Lemma \ref{baby-calculus}, $\alpha$ is a degree $\leq s$ nilcharacter on $\ultra \Z^3$, and we pause to record its symbol in the following lemma.

\begin{lemma}\label{calc-2}  We have
$$ [\alpha(b,b',n)]_{\Xi^s(\ultra \Z^3)} \equiv -s(s-1) [\tilde \chi(b+b',b-b',n,\ldots,n)]_{\Xi^s(\ultra \Z^3)}$$
where by $\equiv$ we are quotienting by all symbols of degree $\leq s-3$ in $n$.
\end{lemma}

\begin{proof} From \eqref{abab} and Lemma \ref{symbolic} we can write the left-hand side as
$$
[\tau(-b-b',b,n)]_{\Xi^s(\ultra \Z^3)} -  [\tau(-b-b',b',n)]_{\Xi^s(\ultra \Z^3)}.$$
Applying \eqref{calc-1}, we can write this as
$$
s(s-1)
( [\tilde \chi(-b-b',b,n,\ldots,n)]_{\Xi^s(\ultra \Z^3)} - [\tilde \chi(-b-b',b',n,\ldots,n)]_{\Xi^s(\ultra \Z^3)} ).$$
The claim then follows from some applications of Lemma \ref{multil}.
\end{proof}

We return now to \eqref{retour}, and Cauchy-Schwarz in $b',n$ to eliminate the $\varphi'_{b'}(n)$ factor, yielding
$$ |\E_{b_1,b_2,b' \in [[3N]]; n \in [N]}
\alpha(b_1,b',n) \otimes \overline{\alpha(b_2,b',n)} \otimes \varphi'_{b_1}(n) \otimes \overline{\varphi'_{b_2}(n)} \psi''_{b_1,b_2,b'}(n)| \gg 1$$
where $(b_1,b_2,b') \mapsto \psi''_{b_1,b_2,b'}$ is a limit map assigning a scalar degree $\leq s-3$ nilsequence to each $b_1,b_2,b'$.  Finally, we Cauchy-Schwarz in $b_1,b_2,n$ to eliminate the $\varphi'_{b_1}(n) \overline{\varphi'_{b_2}(n)}$ factor, yielding
\begin{align*} |\E_{b_1,b_2,b'_1,b'_2 \in [[3N]]; n \in [N]}
\alpha(b_1,b'_1,n) \otimes & \overline{\alpha(b_2,b'_1,n)} \otimes \overline{\alpha(b_1,b'_2,n)} \otimes \\ & \otimes \alpha(b_2,b'_2,n) \psi''_{b_1,b_2,b'_1,b'_2}(n)| \gg 1.\end{align*}
Note how the $\varphi$ terms have now been completely eliminated.  To eliminate the $\psi''$ terms, we first use the pigeonhole principle to find $b_0,b'_0$ such that
\begin{equation}\label{ebony}
 |\E_{b,b' \in [[3N]]; n \in [N]}
\alpha'(b,b',n) \psi''_{b,b_0,b',b'_0}(n)| \gg 1
\end{equation}
where $\alpha' = \alpha'_{b_0,b'_0}$ is the expression 
\begin{equation}\label{abab2}
 \alpha'(b,b',n) := \alpha(b,b',n) \otimes \overline{\alpha(b_0,b',n)} \otimes \overline{\alpha(b,b'_0,n)} \otimes \alpha(b_0,b'_0,n).
\end{equation}
We fix this $b_0,b'_0$.  Again, $\alpha'$ is a degree $\leq s$ nilcharacter on $\ultra \Z^3$.  From Lemma \ref{calc-2} and Lemma \ref{multil} (and using Lemma \ref{symbolic} to eliminate shifts by $b_0$) we conclude
\begin{equation}\label{calc-3}
[\alpha'(b,b',n)]_{\Xi^s(\ultra \Z^3)} \equiv s(s-1) 
([\tilde \chi(b,b',n,\ldots,n)]_{\Xi^s(\ultra \Z^3)} - [\tilde \chi(b',b,n,\ldots,n)]_{\Xi^s(\ultra \Z^3)}).
\end{equation}
Note the similarity here with \eqref{total-sym}.

From \eqref{ebony}, we conclude that the sequence $n \mapsto \alpha'(b,b',n)$ is $\leq s-3$-biased for many $b,b' \in [[3N]]$.  Applying Proposition \ref{inv-nec-nonst}, we conclude that
$$ \| \alpha'(b,b',n) \|_{U^{s-2}[N]} \gg 1$$
for many $b,b' \in [[3N]]$.  We conclude (using Corollary \ref{auton-2} to obtain the needed uniformity) that
$$ \E_{b,b' \in [[3N]]} \| \alpha'(b,b',n) \|_{U^{s-2}[N]}^{2^{s-2}} \gg 1.$$ 
By definition of the Gowers norm, this implies that
\begin{equation}\label{sorba}
 |\E_{b,b',h_1,\ldots,h_{s-2} \in [[3N]]; n \in [N]} 
\sigma( b, b', h_1, \ldots, h_{s-2}, n )
1_\Omega(h_1,\ldots,h_{s-2},n) | \gg 1,
\end{equation}
where $\Omega$ is the polytope
$$ \Omega := \{ (h_1,\ldots,h_{s-2},n): n+\sum_{j=1}^{s-2} \omega_j h_{s-2} \in [N] \hbox{ for all } \omega \in \{0,1\}^{s-2} \}$$
and $\sigma$ is the expression
\begin{equation}\label{sdef}
\sigma( b, b', h_1, \ldots, h_{s-2}, n) :=
\bigotimes_{\omega \in \{0,1\}^{s-2}} {\mathcal C}^{|\omega|} \alpha'(b,b',n+\sum_{j=1}^{s-2} \omega_j h_{s-2}),
\end{equation}
with ${\mathcal C}$ being the conjugation map.

From Lemma \ref{baby-calculus}, $\sigma$ is a nilcharacter of degree $s$ on $\ultra \Z^{s+1}$.    In the following lemma we compute its symbol.

\begin{lemma}  We have
\begin{equation}\label{sorba-2}
\begin{split}
[\sigma(b,b',h_1,\ldots,h_{s-2},n)]_{\Xi^s(\ultra \Z^{s+1})} = &s! 
([\tilde \chi(b,b',h_1,\ldots,h_{s-2})]_{\Xi^s(\ultra \Z^{s+1})} \\
&\quad - [\tilde \chi(b',b,h_1,\ldots,h_{s-2})]_{\Xi^s(\ultra \Z^{s+1})}).
\end{split}
\end{equation}
\end{lemma}

\begin{proof}  From \eqref{sdef} and Lemma \ref{symbolic} we can write the left-hand side as
\begin{equation}\label{flip}
\sum_{\omega \in \{0,1\}^{s-2}} (-1)^{|\omega|} [\alpha'(b,b',n+\sum_{j=1}^{s-2} \omega_j h_{s-2})]_{\Xi^s(\ultra \Z^{s+1})};
\end{equation}
one should think of this as an $s-2$-fold ``derivative'' of $[\alpha'(b,b',n)]_{\Xi^s(\ultra \Z^3)}$ in the $n$ variable.

From \eqref{calc-3} we can write
\begin{align*}
[\alpha'(b,b',n)]_{\Xi^s(\ultra \Z^3)} &= s(s-1) 
([\tilde \chi(b,b',n,\ldots,n)]_{\Xi^s(\ultra \Z^3)} - [\tilde \chi(b',b,n,\ldots,n)]_{\Xi^s(\ultra \Z^3)}) \\
&\quad + [\beta(b,b',n)]_{\Xi^s(\ultra \Z^3)}
\end{align*}
where $\beta$ is of degree at most $s-3$ in $n$.  In fact, by inspection of the derivation of $\beta$, and heavy use of Lemma \ref{multil}, one can express 
$[\beta(b,b',n)]_{\Xi^s(\ultra \Z^3)}$ as a linear combination of classes of the form
$$ [\tilde \chi(n_1,\ldots,n_s)]_{\Xi^s(\ultra \Z^3)}$$
where each of $n_1,\ldots,n_s$ is equal to either $b$, $b'$, or $n$, with at most $s-3$ copies of $n$ occurring.  If one then substitutes this expansion into \eqref{flip} and applies Lemma \ref{multil} repeatedly, one obtains the claim.
\end{proof}

On the other hand, from \eqref{sorba} and Lemma \ref{bias}, we see that on $[[3N]]^{s+1}$, $\sigma$ is equal to a nilsequence of degree $\leq s-1$, and thus by Lemma \ref{symbolic}
$$
[\sigma(b,b',h_1,\ldots,h_{s-2},n)]_{\Xi^s([[3N]]^{s+1})} = 0$$
and thus by Lemma \eqref{sorba-2}
$$
s! 
([\tilde \chi(b,b',h_1,\ldots,h_{s-2})]_{\Xi^s([[3N]]^{s+1})} - [\tilde \chi(b',b,h_1,\ldots,h_{s-2})]_{\Xi^s([[3N]]^{s+1})}) = 0.$$
Applying Lemma \ref{baby-calculus} we conclude that
$$
s! 
([\tilde \chi(h,n,\ldots,n)]_{\Xi^s([[N]] \times [N])} - [\tilde \chi(n,h,n,\ldots,n)]_{\Xi^s([[N]] \times [N])}) = 0.$$
The claim \eqref{total-sym} now follows from Lemma \ref{torsion}.  The proof of Theorem \ref{aderiv} is now complete.

\appendix

\section{Basic theory of ultralimits}\label{nsa-app}

In this appendix we review the machinery of ultralimits.  

We will assume the existence of a \emph{standard universe} ${\mathfrak U}$ which contains all the objects and spaces of interest for Theorem \ref{mainthm}, such as real numbers, subsets of real numbers, functions from $[N]$ to $\C$ for finite $N \in \N$, nilmanifolds (or more precisely, a representative from each equivalence class of nilmanifolds), and so forth.  The precise construction of this universe is not important, so long as it forms a set.  We refer to objects and spaces inside the standard universe as \emph{standard objects} and \emph{standard spaces}, with the latter being sets whose elements are in the former category.  Thus for instance, elements of $\N$ are standard natural numbers, the Heisenberg nilmanifold $\left(\begin{smallmatrix} 1 & \R & \R \\ 0 & 1 & \R \\ 0 & 0 & 1 \end{smallmatrix} \right)/ \left(\begin{smallmatrix} 1 & \Z & \Z \\ 0 & 1 & \Z \\ 0 & 0 & 1 \end{smallmatrix}\right)$ is a standard nilmanifold (consisting entirely of standard points), and so forth.

The one technical ingredient we need is the following:

\begin{lemma}[Ultrafilter lemma]  There exists a collection $p$ of subsets of the natural numbers $\N$ with the following properties:
\begin{enumerate}
\item \textup{(Monotonicity)} If $A \in p$ and $B \supset A$, then $B \in p$.
\item \textup{(Closure under intersection)} If $A,B \in p$, then $A \cap B \in p$.
\item \textup{(Maximality)} If $A \subset \N$, then either $A \in p$ or $\N \backslash A \in p$, but not both.
\item \textup{(Non-principality)} If $A \in p$, and $A'$ is formed from $A$ by adding or deleting finitely many elements to or from $A$, then $A' \in p$.
\end{enumerate}
\end{lemma}

\begin{proof}  The collection of subsets of $\N$ which are cofinite (i.e. whose complement is finite) already obeys the monotonicity, closure under intersection, and non-principality properties.  Using Zorn's lemma\footnote{By using this lemma, our results thus rely on the axiom of choice, which we will of course assume throughout this paper.  On the other hand, it is tedious but straightforward to rephrase the inverse conjecture (Conjecture \ref{gis-conj}) in the language of Peano arithmetic (e.g. using Mal'cev bases \cite{malcev} to represent a nilmanifold, and approximating a Lipschitz function by a piecewise linear one).  Applying a famous theorem of G\"odel \cite{godel}, we then conclude that Conjecture \ref{gis-conj} is provable in ZFC if and only if it is provable in ZF.  In fact, it is possible (with some effort) to directly translate these ultrafilter arguments to a (lengthier) argument in which ultrafilters or the axiom of choice is not used.  We will not do so here, though, as the translation is quite tedious.}, one can enlarge this collection to a maximal collection, which then obeys all the required properties.
\end{proof}

Throughout the paper, we fix a non-principal ultrafilter $p$.  A property $P(\n)$ depending on a natural number $\n$ is said to hold \emph{for $\n$ sufficiently close to $p$} if the set of $\n$ for which $P(\n)$ holds lies in $p$.

Once we have fixed this ultrafilter we can define limit objects and spaces as follows.

\begin{definition}[Limit objects]
Given a sequence $(x_\n)_{\n \in \N}$ of standard objects in ${\mathfrak U}$, we define their \emph{ultralimit} $\lim_{\n \to p} x_\n$ to be the equivalence class of all sequences $(y_\n)_{\n \in \N}$ of standard objects in ${\mathfrak U}$ such that $x_\n = y_\n$ for $\n$ sufficiently close to $p$.  Note that the ultralimit $\lim_{\n \to p} x_\n$ can also be defined even if $x_\n$ is only defined for $\n$ sufficiently close to $p$.

An ultralimit of standard natural numbers is known as a \emph{limit natural number}, an ultralimit of standard real numbers is known as a \emph{limit real number}, etc.

For any standard object $x$, we identify $x$ with its own ultralimit $\lim_{\n \to p} x$.  Thus, every standard natural number is a limit natural number, etc.

Any operation or relation on standard objects can be extended to limit objects in the obvious manner.  For instance, the sum of two limit real numbers $\lim_{\n \to p} x_\n$, $\lim_{\n \to p} y_\n$ is the limit real number
$$\lim_{\n \to p} x_\n + \lim_{\n \to p} y_\n = \lim_{\n \to p} x_\n + y_\n,$$
and the statement $\lim_{\n \to p} x_\n < \lim_{\n \to p} y_\n$ means that $x_\n < y_\n$ for all $\n$ sufficiently close to $p$.
\end{definition}

A famous theorem of {\L}o\'s asserts that any statement in first-order logic which is true for standard objects is automatically true for limit objects as well.  For instance, the standard real numbers form an ordered field, and so the limit real numbers do also, because the axioms of an ordered field can be phrased in first-order logic.  We will use this theorem in the sequel without further comment.

\begin{definition}[Limit spaces and functions]  Let $(X_\n)_{\n \in \N}$ be a sequence of standard spaces $X_\n$ in ${\mathfrak U}$ indexed by the natural numbers.  The \emph{ultrapower} $\prod_{\n \to p} X_\n$ of the $X_\n$ is defined to be the space of all ultralimits $\lim_{\n \to p} x_\n$, where $x_\n \in X_\n$ for all $\n$.  Note $X_\n$ only needs to be well-defined for $\n$ sufficiently close to $p$ in order for the ultraproduct to be well-defined.  If $X$ is a set, the set $\prod_{\n \to p} X$ is known as the \emph{ultrapower} of $X$ and is denoted $\ultra X$.  Thus for instance $\ultra \N$ is the space of all limit natural numbers, $\ultra \R$ is the space of all limit reals, etc.

We define a \emph{limit set} to be an ultraproduct of sets, a \emph{limit group} to be an ultraproduct of groups, a \emph{limit finite set} to be an ultraproduct of finite sets, and so forth.  A \emph{limit subset} of a limit set $X = \prod_{\n \to p} X_\n$ is a limit set of the form $Y = \prod_{\n \to p} Y_\n$, where $Y_\n$ is a standard subset of $X_\n$ for all $\n$ sufficiently close to $p$.

Given a sequence of standard functions $f_\n: X_\n \to Y_\n$ between standard sets $X_\n, Y_\n$, we can form the \emph{ultralimit} $f = \lim_{\n \to p} f_\n$ to be the function $f: \prod_{\n \to p} X_\n \to \prod_{\n \to p} Y_\n$ defined by the formula
$$ f( \lim_{\n \to p} x_\n ) := \lim_{\n \to p} f_\n(x_\n).$$
We refer to $f$ as a \emph{limit function} or \emph{limit map}, and say that $f(x)$ depends \emph{in a limit fashion} on $x$.
\end{definition}

\emph{Remark.} In the nonstandard analysis literature, limit natural numbers are known as \emph{nonstandard natural numbers}, limit sets are known as \emph{internal sets}, and limit functions are known as \emph{internal functions}. We have chosen the limit terminology instead as we believe that it is less confusing and emphasises the role of ultralimits in the subject.

It is important to note that not every subset of a limit set is again a limit set, for instance $\N$ is not a limit subset of $\ultra \N$ (this fact is known as the \emph{overspill principle}).  Indeed, one can think of the limit subsets of a limit set as being analogous to the measurable subsets of a measure space.  In a similar vein, not every function between two limit sets is a limit function; in this regard, limit functions are analogous to measurable functions.

\emph{Example.} (Pigeonhole principle) If $X$ is finite, then $\ultra X = X$.  This is ultimately because if the natural numbers is partitioned into finitely many classes, then exactly one of those classes lies in $p$.  In particular, we see that every standard finite set is a limit finite set.  However, the converse is not true.  For instance, if $N$ is the limit natural number $N := \lim_{\n \to p} \n$, then the limit set 
$$[N] := \{ n \in \ultra \N: 1 \leq n \leq N \} = \prod_{\n \to p} [\n]$$
is a limit finite set, but not a finite set.

\emph{Example.} One has the identifications $\ultra \T = (\ultra \R)/(\ultra \Z)$ and $\ultra(\R^k) = (\ultra \R)^k$ for any standard $k$, so one can talk about the limit unit circle $\ultra \T$ or the limit vector space $\ultra \R^k$ without ambiguity.  We will refer to elements of $\ultra \T$ as \emph{frequencies}.

\emph{Example.} Every standard function $f: X \to Y$ can be identified with its ultralimit $f: \ultra X \to \ultra Y$, thus for instance the fundamental character $e$ is a limit function from $\ultra \R$ (or $\ultra \T$) to $\ultra \C$, and the fractional part function $\{\}$ is an limit function from $\ultra \R$ to $\ultra I_0$.

\emph{Remark.} A limit finite set $A = \lim_{\n \to p} A_\n$ has an \emph{limit cardinality} $|A|$, defined by the formula
$$ |A| := \lim_{\n \to p} |A_\n|.$$
Of course, $|A|$ is a limit natural number, and not a natural number in general.  Thus for instance, if $N$ is a limit natural number, then the limit finite set $[N]$ has a limit cardinality of $N$ (despite being uncountable in the standard sense).

\textsc{Asymptotic notation.} By taking ultralimits, one can formalise asymptotic notation, such as the $O()$ notation, in a manner that requires no additional quantifiers.

\begin{definition}[Asymptotic notation]    A limit complex number $X$ is said to be \emph{bounded} if one has $|X| \leq C$ for some standard real number $C$, in which case we also write $X=O(1)$ or $|X| \ll 1$.  More generally, given a limit complex number $X$ and limit non-negative number $Y$, we write $|X| \ll Y$, $Y \gg |X|$, or $X = O(Y)$ if one has $|X| \leq CY$ for some standard real number $C$.  We write $X = o(Y)$ if one has $|X| \leq \eps Y$ for every standard $\eps > 0$.  Observe that for any $X, Y$ with $Y$ positive, one has either $|X| \gg Y$ or $X = o(Y)$.  We say that $X$ is \emph{infinitesimal} if $X=o(1)$, and \emph{unbounded} if $1/X = o(1)$.  Thus for instance any limit complex number $X$ will either be bounded or unbounded.

In a similar spirit, if $x \in \ultra V$ is a limit element of a standard topological space $V$, we say that $x$ is \emph{bounded} if $x$ is a limit element of standard compact subset $K$ of $V$ (i.e. $x \in \ultra K$), and \emph{unbounded} otherwise.  The set of all bounded elements of $\ultra V$ will be denoted $\overline{V}$.  
\end{definition}

\emph{Example.} The limit real $\lim_{\n \to p} 1/\n$ defines an infinitesimal, but non-zero, limit real number $x$; its reciprocal $\lim_{\n \to p} \n$ is an unbounded limit real.  

\emph{Example.}  Any bounded element of a discrete standard space is standard, by our example on the pigeonhole principle.  In particular, bounded integers are automatically standard: $\overline{\Z} = \Z$.  On the other hand, bounded elements in a continuous space need not be standard, as the example $\lim_{\n \to p} 1/\n$ shows.

From the Bolzano-Weierstrass theorem, every bounded limit real number can be expressed uniquely as the sum of a standard real number and an infinitesimal, which may help explain the notation $\overline{\R}$.  Note that $\overline{\R}$ contains the limit fundamental domain $\ultra I_0$.  Similarly, $\overline{\C}$ contains the limit unit circle $\ultra S^1 = \overline{S^1} = \{ z \in \overline{\C}: |z|=1\}$, where $S^1 := \{ z \in \C: |z|=1\}$.

\emph{Example.} For any standard $D \in \N^+$, we endow $\C^D$ with the Euclidean norm
$$ |(z_1,\ldots,z_D)| := (|z_1|^2+\ldots+|z_D|^2)^{1/2}.$$
Then we have $\overline{\C^D} = \overline{\C}^D$: an element $(z_1,\ldots,z_D) \in \ultra \C^D$ is bounded if and only if each component is bounded.

One modest advantage of the ultralimit framework is that one can rigorously work with such equivalence relations as ``$x$ and $y$ differ by $O(1)$'', for instance by quotienting $\ultra \R$ by the subring $\overline{\R}$; in the finitary setting, this relation is only ``morally'' an equivalence relation (because of the need to quantify the constants in the $O()$ notation).

Suppose one has a limit function $f: \Omega \to \ultra \C$ on a limit set $\Omega$.  If one asserts that $f(x) = O(1)$ for each $x \in \Omega$, one may be concerned that this statement provides no uniformity in $x$.  However, it turns out such uniformity is automatic for limit functions.

\begin{lemma}[Automatic uniformity]\label{auton}  Let $D \in \N^+$, and let $f: \Omega \to \ultra \C^D$ be a limit function on a limit set $\Omega$.  Then the following statements are equivalent:
\begin{itemize}
\item \textup{(Pointwise boundedness)}  For each $x \in \Omega$, one has $f(x) \in \overline{\C}^D$ \textup{(}i.e. $f(x) = O(1)$ for all $x \in \Omega$\textup{)}.
\item \textup{(Uniform boundedness)} There exists a standard real $C$ such that $|f(x)| \leq C$ for all $x \in \Omega$.
\end{itemize}
\end{lemma}

Intuitively, this lemma is asserting that the only types of functions that always map unbounded sequences to bounded sequences (but with a bound possibly depending on the initial sequence) are those functions that are uniformly bounded.  The lemma can clearly fail if one considers functions $f$ that are not limit functions; thus it will be important to establish the limit nature of various functions in the arguments below.  This lemma is also closely related to the \emph{overspill principle} in nonstandard analysis, or the model-theoretic fact that ultraproducts are countably saturated.

\begin{proof}  Clearly uniform boundedness implies pointwise boundedness, so we show the converse.  Suppose for contradiction that $f$ was pointwise bounded but not uniformly bounded.  Then for every standard integer $M$ there exists an element $x_M$ in $\Omega$ such that $|f(x_M)| > M$.

Write $\Omega$ as the ultralimit of standard sets $\Omega_\n$, write $f$ as an ultralimit of a sequence $f_\n: \Omega_\n \to \C^D$, and write $x_M = x_{M,\n} \in \Omega_\n$.  Thus for each standard $M$, the statement $|f_\n( x_{M,\n} )| > M$ is true for $\n$ sufficiently close to $p$.

Now we diagonalise.  Set $y = \lim_{\n \to p} y_\n$, where $y_\n := x_{\n,\n}$.  Then $y \in X$ and one sees that for every standard $M$, the statement $|f_\n(y_\n)| > M$ holds for $\n$ sufficiently close to $p$, thus $f(y)$ is unbounded.  But this contradicts pointwise boundedness.
\end{proof}

We observe a useful corollary to Lemma \ref{auton}.

\begin{corollary}[Automatic uniform lower bounds]\label{auton-2} Let $D \in \N^+$, and let $f: \Omega \to \ultra \C^D$ be a limit function on a limit set $\Omega$ such that $|f(x)| \gg 1$ for all $x \in \Omega$.  Then there exists a standard $c>0$ such that $|f(x)| \geq c$ for all $x \in \Omega$.
\end{corollary}

\begin{proof} Apply Lemma \ref{auton} to $1/|f|$.
\end{proof}

Inspired by Lemma \ref{auton}, we shall simply call an limit function $f: \Omega \to \ultra \C^D$ \emph{bounded} if it is either pointwise bounded or uniformly bounded.  The space of all bounded limit functions from $\Omega$ to $\ultra \C^D$ will be denoted $L^\infty(\Omega \to \overline{\C}^D)$, and we also write
\begin{equation}\label{sigma-bounded}
 L^{\infty}(\Omega)=L^\infty(\Omega \to \overline{\C}^\omega) := \bigcup_{D \in \N^+} L^\infty(\Omega \to \overline{\C}^D).
 \end{equation}

When $D=1$, $L^\infty(\Omega \to \overline{\C})$ is a $*$-algebra over the bounded complex numbers $\overline{\C}$ (i.e. it is closed under addition, pointwise multiplication, complex conjugation, and multiplication by bounded complex numbers).  It is not, however, a limit set.

For higher dimensions $D > 1$, we still have the operations of addition, complex conjugation (conjugating each coefficient of $\C^D$ separately), and multiplication by bounded complex numbers.  However, we do not have a natural product on $\C^D$.  Instead, we will use the \emph{tensor product} $\otimes: \C^D \times \C^{D'} \to \C^{DD'}$, defined in \S \ref{notation-sec}.  This induces a tensor product 
\[
 \otimes: L^\infty(\Omega \to \overline{\C}^D) \times L^\infty(\Omega \to \overline{\C}^{D'})
\to L^\infty(\Omega \to \overline{\C}^{DD'})
\]
for any $\Omega$, which is then a bilinear operation on $L^\infty(\Omega \to \overline{\C}^\omega)$.  Strictly speaking, this tensor product is neither commutative nor associative.  However, it is ``essentially'' commutative and associative in the following sense.  Let us say that a function $f \in L^\infty(\Omega \to \overline{\C}^D)$ is a \emph{bounded linear combination} of another function $f' \in L^\infty(\Omega \to \overline{\C}^{D'})$ if there exists a linear transformation $T: \ultra \C^{D'} \to \ultra \C^D$ with bounded coefficients such that $f = T \circ f'$.  Then it is clear that for any $f_1,f_2,f_3 \in L^\infty(\Omega \to \overline{\C}^\omega)$, we have that $f_2 \otimes f_1$ is a bounded linear combination of $f_1 \otimes f_2$, and that $f_1 \otimes (f_2 \otimes f_3)$ is a bounded linear combinastion of $(f_1 \otimes f_2) \otimes f_3$.  This will be a satisfactory substitute for commutativity and associativity for our purposes.

We define the spheres
$$ \overline{S^{2D-1}} := \{ z \in \overline{\C}^D: |z| = 1 \}$$
and
$$ \overline{S^\omega} := \bigcup_{D \in \N^+} \overline{S^{2D-1}} = \{ z \in \overline{\C}^\omega: |z|=1\}$$
and observe that $\overline{S^\omega}$ is closed under complex conjugation and tensor product, and so $L^\infty(\Omega \to \overline{S^\omega})$ is also.  Also, observe that for any $f \in L^\infty(\Omega \to \overline{S^\omega})$, $1$ is a bounded linear combination of $f \otimes \overline{f}$.

When $\Omega$ is a non-empty limit finite set (e.g. $\Omega = [N]$ or $\Omega = [N]^k$ for some positive limit integer $N$ and some standard $k \geq 1$), we have some additional structures.

\begin{definition}[Bias and correlation]\label{linfty}  Let $\Omega$ be a non-empty limit finite set. Given two functions $f \in L^\infty(\Omega \to \overline{\C}^\omega)$, $g \in L^\infty(\Omega \to \overline{\C}^{\omega})$, we say that $f$ and $g$ \emph{correlate} if one has
$$ |\E_{n \in \Omega} f(n) \otimes \overline{g(n)}| \gg 1,$$
and that $f$ is \emph{biased} if one has
$$ |\E_{n \in \Omega} f(n)| \gg 1,$$
i.e. if $f$ correlates with $1$.  We say that $f$ is \emph{unbiased} if it is not biased.  We define the $L^p$ norms
$$ \| f \|_{L^p(\Omega)} := (\E_{n \in \Omega} |f(n)|^p)^{1/p}$$
for $1 \leq p < \infty$, with the usual convention
$$ \|f\|_{L^\infty(\Omega)} := \sup_{n \in \Omega} |f(n)|;$$
these are bounded limit non-negative numbers.
\end{definition}

We will also find the following notation useful.

\begin{definition}[Density]\label{dense-def}  We say that a limit subset $H$ of a limit finite set $X$ is \emph{dense} if $|H| \gg |X|$, and that a statement $P(x)$ is true for \emph{many} $x \in X$ if it is true for all $x$ in a dense subset $H$ of $X$.  If instead $|H| = o(|X|)$, we say that $H$ is a \emph{sparse} subset of $X$, and if $P(x)$ only holds true for $x$ in a sparse set, we say that $P(x)$ only holds for \emph{few} $x \in X$.  If the complement of $H$ in $X$ is sparse, we say that $H$ is a \emph{co-sparse} subset of $X$, and if $P(x)$ holds for all $x$ in a co-sparse subset, we say that $P(x)$ holds for \emph{almost all} $x \in X$.

A function $f: X \to \ultra \C^D$ is said to be \emph{almost bounded} if $f(x) \in \overline{\C}^D$ for almost all $x \in X$. (For instance, for an unbounded limit natural number $N$, the function $n \mapsto \frac{N}{n+1}$ is almost bounded on $[N]$.)
\end{definition}

\emph{Remarks.}  Note that the statement $P$ does not need to be a limit statement (i.e. the set $\{ x \in X: P(x) \hbox{ true} \}$ need not be a limit set) for these definitions to make sense; for instance, for $P$ to hold for many $x$, it suffices that $\{ x \in X: P(x) \hbox{ true} \}$ contain an dense limit subset of $X$, but need not be a limit set itself.
If one property $P(x)$ holds for almost all $x \in X$, and another property $Q(x)$ holds for many $x \in X$, then $P(x)$ and $Q(x)$ simultaneously hold for many $x \in X$.  However, if $P$ only holds for many $x$ rather than for almost all $x$, then it need not be the case that $P(x)$ and $Q(x)$ simultaneously hold for any $x$.

From the pigeonhole principle we see that if an limit set is partitioned into a bounded number of limit pieces, then at least one of the pieces is dense.  We can strengthen this principle as follows.

\begin{lemma}[Pigeonhole principle]\label{dense-dich}  Let $X$ be a limit finite set, and let $f$ be an almost bounded limit function from $X$ to $\ultra \N$. Then there exists a dense subset of $X$ on which $f$ is constant and equal to a standard natural.
\end{lemma}

\begin{proof}  By hypothesis, $f$ is bounded on almost all of $X$, and hence uniformly bounded on almost all of $X$ by Lemma \ref{auton}.  The claim now follows from the pigeonhole principle.
\end{proof}

We also record here a technical lemma regarding correlation.

\begin{definition}[$\sigma$-limit]\label{separ}  A subset $S$ of an limit set $X$ is said to be a \emph{$\sigma$-limit} set if there is a limit sequence $n \mapsto S_n$ from limit natural numbers $n \in \ultra \N$ of limit subsets $S_n$ of $X$, such that $S$ is the union of the $S_n$ over all \emph{standard} natural numbers.
\end{definition}

\emph{Example.} If $\Omega$ is a limit set and $D \in \N^+$, then the space $L^\infty(\Omega \to \overline{\C}^D)$, which is an external (i.e. non-limit) subset of the limit space of all limit functions from $\Omega$ to $\ultra \C^D$, is a $\sigma$-limit space, since one can express this space as the union, over all standard $M$, of the functions bounded uniformly in magnitude by $M$.  Similarly, $L^\infty(\Omega \to \overline{\C}^\omega)$ is also a $\sigma$-limit set.

\begin{lemma}[Limit selection lemma]\label{int-select}  Let $X, Y$ be limit sets, let $R \subset X \times Y$ be a an limit relation between $X$ and $Y$, and let $S$ be a $\sigma$-limit subset of $Y$.  Suppose that for every $x \in X$ there exists $s_x \in S$ such that $(x,s_x) \in R$.  Then there exists a \emph{limit} function $x \mapsto s_x$ from $X$ to $S$ such that $(x,s_x) \in R$ for all $x \in X$.
\end{lemma}

\emph{Remark.}
The key point here is the limit nature of the assignment $x \mapsto s_x$; for external (i.e. non-limit) assignments, the claim is immediate from the axiom of choice.  There is a similar need for such ``measurable selection lemmas'' in the ergodic theory analogue of the inverse conjectures for the Gowers norms, see e.g. \cite[Appendix A]{host-kra} or \cite[Lemma C.4]{bergelson-tao-ziegler}.

\begin{proof}
We may assume that the sets $S_n$ in Definition \ref{separ} are increasing in $n$.  

For each $x \in X$, let $n_x$ be the first limit natural number such that $(x,s) \in R$ for some $s \in S_{n_x}$.  By construction, $x \mapsto n_x$ is a limit map from $X$ to $\ultra \N$ which is pointwise bounded.  Thus, by Lemma \ref{auton}, $n_x$ is uniformly bounded by some standard natural number $n_*$, thus for every $x \in R$ the set $\{ s \in S_{n_*}: (x,s) \in R \}$ is non-empty.  Applying a limit choice function, we may thus find a limit map $x \mapsto s_x$ with the stated properties.
\end{proof}

We isolate a special case of this lemma.

\begin{corollary}\label{mes-select}  Let $\Omega$ be a non-empty limit-finite set, Let $S$ be a $\sigma$-limit subset of $L^\infty(\Omega \to \overline{\C}^\omega)$, and let $(f_h)_{h \in H}$ be a limit family of limit functions $f_h \in L^\infty(\Omega \to \overline{\C}^\omega)$ indexed by an limit set $H$, and suppose that for each $h \in H$, $f_h$ correlates with an element of $S$.  Then one can find an \emph{limit} family $(\phi_h)_{h \in H}$ of functions $\phi_h \in S$ such that $f_h$ correlates with $\phi_h$ for all $h \in H$.
\end{corollary}

\begin{proof} Write $S$ as the union of limit sets $S_n$ for standard $n$, and let $S' := \bigcup_{n \in \N} S_n \cup \{n\}$.  Note that this is a $\sigma$-limit subset of $L^\infty(\Omega \to \overline{\C}^\omega) \times \ultra \N$.  Defining a relation $R$ between $H$ and $S'$ by declaring $(h,(\phi,n)) \in R$ if $|\E_{n \in \Omega} f_h(n) \otimes \overline{\phi(n)}| \geq 1/n$, and applying Lemma \ref{int-select}, we obtain the claim.
\end{proof}

\section{Polynomial algebra}\label{poly-app}

In section \S \ref{nilcharacters} we introduced the notion of a polynomial map between $I$-filtered groups $H$ and $G$ when the group $H$ was abelian (Definition \ref{poly-map-def}). In this appendix we study the more general notion of
a polynomial map, no longer restricting to the case $H$ abelian. The concept of a polynomial  map between groups was introduced by Leibman in \cite{leibman-group-1,leibman-group-2}, and here we adapt it to filtered groups.

Recall the definitions of an ordering $I$ and of an $I$-filtration of a group $G$ in Definitions \ref{order-def} and \ref{filtered-group}. 

\begin{definition}[Polynomial map] \label{poly-def} Let $G,H$ be groups with $I$-filtrations $G_I, H_I$.  If $g : H \rightarrow G$ is a map then we define the \emph{derivative} $\partial_h g : H \rightarrow G$ by the formula
$$ \partial_h g(n) := g(hn) g(n)^{-1}$$
for all $n \in H$. We say that map $g: H \to G$ is \emph{polynomial} if one has
\[ \partial_{h_1} \ldots \partial_{h_m} g(n) \in G_{i_1+\ldots+i_m}\]
whenever $m \geq 0$, $i_1,\ldots,i_m \in I$, $h_j \in H_{ i_j}$ for $j=1,\ldots,m$ and $n \in H_{0}$.  The space of all polynomial maps is denoted $\poly(H_I  \to G_I )$.
\end{definition}

\emph{Remark.} As mentioned in \S \ref{notation-sec}, if $G$ or $H$ are written as additive groups instead of multiplicative ones, the definition of partial derivative is adjusted appropriately.

\emph{Example 1.} If $I=\N$, and $H$ is abelian and is given the filtration $H_i = H$ for $i = 0,1$ and $H_i=\{0\}$ for $i>0$, then a map $g : H \rightarrow G$ lies in $\poly(H_\N, G_\N)$ if and only if 
\[ \partial_{h_1} \ldots \partial_{h_m} g(n) \in G_m\]
for all $m \geq 0$ and $h_1,\dots,h_m \in H$. This coincides with the definition given in \cite[Definition 6.1]{green-tao-nilratner}. Definition \ref{poly-def} may be considered as a generalisation of this in which the domain group $H$ is allowed to have nontrivial filtrations.

\emph{Example 2.} Any map $\phi: G \to H$ between two $I$-filtered groups which is constant and takes values in $H_{0}$ is polynomial.

\emph{Example 3.} If $\phi: H_I \to G_I$ is a homomorphism of $I$-filtered groups that maps $H_{i}$ to $G_{i}$ for each $i \in I$, then $\phi$ is a polynomial map since, for each $h \in H$, $\partial_h \phi$ is the constant map $n \mapsto \phi(h)$.  We will call such homomorphisms \emph{$I$-filtered homomorphisms} from the $I$-filtered group $H_I$ to the $I$-filtered group $G_I$.

\emph{Example 4.} If $G$ is an $I$-filtered group, and $g \in G$, then the left translation maps $x \mapsto gx$ lie in $\poly(G_I \to G_I)$.  Indeed, the derivative of this map in any direction $h \in G_{ i}$ is simply the constant map $ghg^{-1}$, which lies in $G_{i}$, and any further derivative of this map is trivial.  This example is a special case of the Lazard-Leibman theorem (Corollary \ref{laz} below), since the translation map is the product of a constant map and the identity homomorphism.

\emph{Example 5.} Given three $I$-filtered groups $H, G, G'$, 
a map $g: H \to G \times G'$ is polynomial ($G \times G'$ is given the product filtration) if and only if its projections to $G$ and $G'$ are polynomial.  In other words, we have a canonical isomorphism
$$ \poly(H_I \to (G \times G')_I) \equiv \poly(H_I  \to G_I ) \times \poly(H_I  \to G'_I ).$$

\textsc{Host-Kra cube groups.} There is an important alternative characterisation of polynomial maps in terms of \emph{Host-Kra cube groups}, which we now define.  The material in this section is a generalisation of \cite{green-tao-nilratner}, and particularly \cite[Proposition 6.5]{green-tao-nilratner}, to the context of polynomial maps $\poly(H_I \to G_I)$ (there matters were discussed only in the case $\poly(H \to G_I)$). 
The Host-Kra groups are the group-theoretic analogue of the Host-Kra spaces $X^{[k]}$ of a dynamical system $X$ introduced in \cite{host-kra}.

If $m$ is a natural number, we let $2^{[m]}$ be the power set of $[m] := \{1,\dots,m\}$.

\begin{definition}\label{hkg}  Let $G$ be an $I$-filtered group, and let $i_1,\ldots,i_m \in I$.  We define the \emph{Host-Kra cube group} $\HK^{i_1,\ldots,i_m}(G_I)$ to be the subgroup of $G^{2^{[m]}}$ generated by the elements of the form
$$ \iota_{\omega_0}(g_{\omega_0}) := ( g_\omega )_{\omega \subset [m]},$$
where $\omega_0 \subset [m]$, $g_{\omega_0} \in G_{\sum_{j \in \omega_0} i_j}$, and $g_\omega$ equals $g_{\omega_0}$ when $\omega \supseteq \omega_0$ and is the identity otherwise.  Thus we see that the $\iota_{\omega_0}$ are embeddings of $G_{\sum_{j \in \omega_0} i_j}$ into $\HK^{i_1,\ldots,i_m}(G)$.  We refer to $m$ as the \emph{order} of the Host-Kra cube groups, and refer to elements of $\HK^{i_1,\ldots,i_m}(G)$ as \emph{cubes} of \emph{dimension} $m$ and \emph{degrees} $i_1,\ldots,i_m$.
\end{definition}

\emph{Example.} Let $G$ be a $k$-step nilpotent group, and let $G_i=[G,G_{i-1}]$ be the lower central series filtration.
Then $\HK^{1,\ldots,1}(G)$ is the subgroup of $G^{2^{[m]}}$ generated by the ``side'' elements
$(g^i_\omega)_{\omega \subset [m]}$ where $g^i_\omega=g$ if $i \in \omega$ and $g^i_{\omega} = \id$ otherwise, for $i=1,\ldots,m$, and
by the diagonal elements $(g,\ldots,g)$.

\begin{theorem}\label{hk}  Let $G, H$ be $I$-filtered groups, and let $g: H \to G$ be a map.  Then $g$ is a polynomial map if and only if it preserves cubes, in the sense for every $m \geq 0$ and $i_1,\ldots,i_m \in I$, the homomorphism $g^{2^{[m]}}: H^{2^{[m]}} \to G^{2^{[m]}}$ defined by
$$ g^{2^{[m]}}( (h_\omega)_{\omega \subset [m]} )  := (g(h_\omega))_{\omega \subset [m]}$$
maps $\HK^{i_1,\ldots,i_m}(H_I)$ to $\HK^{i_1,\ldots,i_m}(G_I)$.
\end{theorem}
\begin{proof}  For inductive reasons it is convenient to establish the following slightly stronger result.  For any $m_0 \geq 0$, we say that a map $g: H \to G$ is \emph{polynomial to order $m_0$} if we have
$$ \partial_{h_1} \ldots \partial_{h_m} g(n) \in G_{i_1+\ldots+i_m}$$
for all $m$ with $0 \leq m \leq m_0$, all $i_1,\ldots,i_m \in I$, all $h_j \in H_{ i_j}$ for $j=1,\ldots,m$, and all $n \in H_{ 0}$.  It will suffice to show that a map $g: H \to G$ is polynomial to order $m_0$ if and only if it preserves the cubes of dimension up to $m_0$.

We establish this by induction on $m_0$.  The case $m_0=0$ is easy: $g$ is polynomial to order $0$ if it maps $H_{0}$ to $G_{0}$, but these are also essentially the Host-Kra groups of order $0$, and the claim follows.  Now suppose inductively that $m_0 \geq 1$ and that the claim has already been shown for all smaller values of $m_0$.

Suppose first that $g: H \to G$ preserves all cubes of dimension up to $m_0$.  Then by the preceding discussion, $g$ maps $H_{0}$ to $G_{0}$.  To show that $g$ is polynomial to order $m_0$, it thus suffices to show that for every $i \in I$ and $h \in H_{i}$, $\partial_h g$ is polynomial to order $m_0-1$ in the shifted $I$-filtration $G_I^{+i}$ defined by 
\begin{equation}\label{shifted}
G_I^{+i} := (G_{ j+i})_{j \in I}.
\end{equation}
By the induction hypothesis, it suffices to show that $\partial_h g$ preserves cubes of dimension $m_0-1$.  Accordingly, let $\vec h = (h_\omega)_{\omega \subset [m_0-1]}$ be an element of  $\HK^{i_1,\ldots,i_{m_0-1}}(H)$.  We may view $(\vec h, h \cdot \vec h)$ as an element of  $\HK^{i_1,\ldots,i_{m_0-1},i}(H)$ of one higher order, where $h \cdot \vec h := (h h_\omega)_{\omega \subset [m_0-1]}$.  By hypothesis on $g$, we have
$$ (g^{2^{[m_0-1]}}(\vec h), g^{2^{[m_0-1]}}(h \cdot \vec h)) \in \HK^{i_1,\ldots,i_{m_0-1},i}(G).$$
An inspection of Definition \ref{hkg} reveals that $(\vec g_1, \vec g_2)$ lies in $\HK^{i_1,\ldots,i_{m_0-1},i}(G)$ if and only if $\vec g_1$ lies in $\HK^{i_1,\ldots,i_{m_0-1}}(G)$ and $\vec g_2 (\vec g_1)^{-1}$ lies in $\HK^{i_1,\ldots,i_{m_0-1}}(G,G_I^{+i})$ (which is easily seen to be a normal subgroup of $\HK^{i_1,\ldots,i_{m_0-1}}(G)$).  We conclude that
$$ g^{2^{[m_0-1]}}(h \cdot \vec h) \cdot g^{2^{[m_0-1]}}(\vec h)^{-1} \in \HK^{i_1,\ldots,i_{m_0-1}}(G,G_I^{+i}).$$
But 
$$ g^{2^{[m_0-1]}}(h \cdot \vec h) \cdot g^{2^{[m_0-1]}}(\vec h)^{-1} = (\partial_h g)^{2^{[m_0-1]}}(\vec h),$$
and the claim follows.

Next, suppose conversely that $g: H \to G$ is a polynomial map of order up to $m_0$; by the inductive hypothesis, it suffices to show that $g$ preserves all the cubes of dimension exactly $m_0$.  Accordingly, let $\vec h$ be an element of  $\HK^{i_1,\ldots,i_{m_0}}(H)$ of this dimension.  Arguing as before, we may write
$$ \vec h = (\vec h_1, \vec h_2 \vec h_1)$$
where $\vec h_1 \in \HK^{i_1,\ldots,i_{m_0-1}}(H)$ and $\vec h_2 \in \HK^{i_1,\ldots,i_{m_0-1}}(H, H_I^{+i_{m_0}})$.  Our objective is then to show that
$$ g^{2^{[m_0]}}(\vec h) = (g^{2^{[m_0-1]}}(\vec h_1), g^{2^{[m_0-1]}}(\vec h_2 \vec h_1))$$
lies in $\HK^{i_1,\ldots,i_{m_0}}(G)$.  By the decomposition of $\HK^{i_1,\ldots,i_{m_0}}(G)$, it thus suffices to show that
\begin{equation}\label{lemon}
 g^{2^{[m_0-1]}}(\vec h_2 \vec h_1) g^{2^{[m_0-1]}}(\vec h_1)^{-1} \in \HK^{i_1,\ldots,i_{m_0-1}}(G,G_I^{+i_{m_0}}).
\end{equation}

Recall that $\HK^{i_1,\ldots,i_{m_0-1}}(H, H_I^{+i_{m_0}})$ is generated by elements of the form $\iota_{\omega_0}(h_{\omega_0})$, where $\omega_0 \subset [m_0-1]$ and $h_{\omega_0} \in H_{\sum_{j \in \omega_0} i_j + i_{m_0}}$.  By telescoping series, we thus see that to establish the above claim it suffices to do so under the additional assumption that $\vec h_2$ is a generator 
$$ \vec h_2 = \iota_{\omega_0}(h_{\omega_0})$$
for some $\omega_0 \subset [m-1]$ and $h_{\omega_0} \in H_{\sum_{j \in \omega_0} i_j + i_{m_0}}$.

By relabeling we may assume that $\omega_0 = \{m'+1,\ldots,m_0-1\}$ for some $0 \leq m' \leq m_0-1$.  The left-hand side of \eqref{lemon} then simplifies to
\begin{equation}\label{salami}
 (\partial_{h_{\omega_0}} g)^{2^{[m']}}(\vec h'_1),
\end{equation}
where $\vec h'_1$ is the restriction of $\vec h_1$ to $2^{[m']}$, and we embed $G^{2^{[m']}}$ into $G^{2^{[m_0-1]}}$ by identifying $(g_\omega)_{\omega \subset [m']}$ with the tuple $(\tilde g_\omega)_{\omega \subset [m_0-1]}$, where $\tilde g_\omega$ is equal to $g_{\omega \cap [m']}$ when $\omega$ contains $\Omega$, and is equal to the identity otherwise.

But by induction hypothesis, \eqref{salami} lies in $\HK^{i_1,\ldots,i_{m_0-1}}(G,G_I^{+\sum_{j \in \omega_0} i_j +i_{m_0}})$.  By Definition \ref{hkg}, this embeds into $\HK^{i_1,\ldots,i_{m_0-1}}(G,G_I^{+i_{m_0}})$, giving \eqref{lemon} as desired, and the claim follows.
\end{proof}

Theorem \ref{hk} has two immediate corollaries.

\begin{corollary}[Lazard-Leibman theorem]\label{laz}  Let $G, H$ be $I$-filtered groups.  Then $\poly(H_I \to G_I )$ is also a group \textup{(}using pointwise multiplication as a group operation\textup{)}.
\end{corollary}

\begin{corollary}[Composition]\label{chain}  Let $G,H,K$ be $I$-filtered groups.  If $g \in \poly(H_I \to G_I )$ and $h \in \poly(K_I  \to H_I )$, then $g \circ h \in \poly(H_I \to K_I )$.
\end{corollary}

In other words, for any fixed $I$, the class of $I$-filtered groups together with their polynomial maps form a category.
It is remarkably difficult to establish Corollary \ref{chain} in full generality without the machinery of Host-Kra cube groups.

\emph{Example.}  If $G, H$ are $I$-filtered groups with $H = (H,+)$ abelian, and $g$ is a polynomial map from $H$ to $G$, then the translates $g(\cdot+h)$ and dilates $g(q\cdot)$ for $h \in H$ and $q \in \Z$ are also polynomial maps from $H$ to $G$, thanks to Corollary \ref{chain} and Examples 3 and 4 following Definition \ref{poly-def}.  More generally, if $\phi: H' \to H$ is a filtered homomorphism and $g \in \poly(H_I  \to G_I )$, then $g \circ \phi \in \poly(H'_I \to G_I )$.

\emph{Example.} Using Corollary \ref{laz} we can establish that any algebraic word $w$ on $k$ generators defines a polynomial map from $H^k$ to $H$ for any $I$-filtered group $H$.  For instance, the map $(g,h) \to g^2 h^{-3} g h$ is a polynomial map from $H \times H$ to $H$.

We can strengthen Corollary \ref{laz} slightly, by giving $\poly(H_I \to G_I)$ the structure of an $I$-filtered group:

\begin{proposition}[Filtered Lazard-Leibman theorem]\label{filter-laz}  Let $(G,G_I), (H,H_I)$ be $I$-filtered groups.  Then $\poly(H_I \to G_I)$ is also an $I$-filtered group, with filtration $(\poly(H_I \to G_I^{+i}))_{i \in I}$, where the shifted filtration $G_I^{+i}$ was defined in \eqref{shifted}.  In particular, the $\poly(H_I \to G_I^{+i})$ are normal subgroups of $\poly(H_I \to G_I)$.
\end{proposition}

\begin{proof}  The only non-trivial claim to show is that if $g_i \in \poly(H_I \to G_I^{+i})$ and
$g_j \in \poly(H_I \to G_I^{+j})$ for some $i,j \in I$, then $[g_i,g_j] \in \poly(H_I \to G_I^{+i+j})$.  It suffices to show for each $m_0 \geq 0$ that if $g_i$, $g_j$ are polynomial maps up to order $m_0$ from $(H,H_I)$ to $(G,G_I^{+i})$, $(G,G_I^{+j})$ respectively, then $[g_i,g_j]$ is a polynomial map up to order $m_0$ from $(H,H_I)$ to $(G,G_I^{+i+j})$.

Again we induct on $m_0$.  The case $m_0=0$ is trivial, so suppose that $m_0 \geq 1$ and that the claim has already been proven for smaller values of $m_0$.  

As $g_i, g_j$ map $H_0$ to $G_i, G_{j}$ respectively, $[g_i,g_j]$ maps $H_0$ to $G_{i+j}$.  It thus suffices to show that for each $k \in I$ and $h \in H_k$, that $\partial_h [g_i,g_j]$ is a polynomial map up to order $m_0-1$ from $(H,H_I)$ to $(G,G_I^{+i+j+k})$.  But a brief calculation shows that
\begin{equation}\label{phg}
 \partial_h [g_i,g_j] = g_i^{-1} (\partial_h g_i)^{-1} g_j^{-1} (\partial_h g_j)^{-1} (\partial_h g_i) g_i (\partial_h g_j) g_i^{-1} g_j g_i.
 \end{equation}
By induction hypothesis (and Corollary \ref{laz}), the maps that are polynomial up to order $m_0-1$ from $(H,H_I)$ to $(G,G_I^{+i+j+k})$ form a normal subgroup of the maps that are polynomial up to order $m_0-1$ from $(H,H_I)$ to $(G,G_I)$.  If we quotient out by this normal subgroup, then a further application of the induction hypothesis shows that $\partial_h g_i$ commutes with $g_j$ and $\partial_h g_j$, and that $g_i$ commutes with $\partial_h g_j$.  An inspection of \eqref{phg} then shows that the right-hand side vanishes once one quotients out by this normal subgroup, and the claim follows.
\end{proof}

Proposition \ref{filter-laz} has some useful corollaries:

\begin{corollary}[Approximate linearity and commutativity]\label{collox}  Let $G,H$ be $I$-filtered groups, let $i,j,k,l \in I$, and let $g_i \in \poly(H_I \to G_I^{+i})$, $g_j \in \poly(H_I \to G_I^{+j})$, $h_k \in H_k$, and $h_l \in H_{l}$.  Then we have
\begin{equation}\label{glin}
\partial_{h_k}( g_i g_j ) = (\partial_{h_k} g_i) (\partial_{h_k} g_j) \mod \poly(H_I \to G_I^{+i+j+k})
\end{equation}
and
\begin{equation}\label{hlin}
\partial_{h_k h_l}( g_i) = (\partial_{h_k} g_i) (\partial_{h_l} g_i) \mod \poly(H_I \to G_I^{+i+k+l}).
\end{equation}
If $H$ is abelian, we also have
\begin{equation}\label{clairaut}
(\partial_{h_l} g_i) (\partial_{h_k} g_i) = (\partial_{h_k} g_i) (\partial_{h_l} g_i) \mod \poly(H_I \to G_I^{+i+k+l}).
\end{equation}
\end{corollary}

\begin{proof} The conclusions \eqref{glin}, \eqref{hlin} follow from Proposition \ref{filter-laz} and the identities
$$
\partial_{h_k}( g_i g_j ) = (\partial_{h_k} g_i) (\partial_{h_k} g_j) [\partial_{h_k} g_j, g_i^{-1}]$$
and
\begin{equation}\label{hlin-base}
 \partial_{h_k h_l}( g_i) = (\partial_{h_l} \partial_{h_k} g_i) (\partial_{h_k} g_i) (\partial_{h_l} g_i).
\end{equation}
The identity \eqref{clairaut} then follows by swapping the roles of $h_k$ and $h_l$ in \eqref{hlin}.
\end{proof}

Next, we make the useful observation that in order to check polynomiality of a map, it suffices to do so on generators.

\begin{proposition}[Checking polynomiality on generators]\label{check}  Let $G,H$ be $I$-filtered groups.  For each $i \in I$, let $E_i$ be a set of generators for $H_i$.  Then a map $g: H \to G$ is polynomial if and only if one has
\begin{equation}\label{ham}
 \partial_{h_1} \ldots \partial_{h_m} g(n) \in G_{i_1+\ldots+i_m}
\end{equation}
for all $m \geq 0$, all $i_1,\ldots,i_m \in I$, and all $h_j \in E_{i_j}$ for $j=1,\ldots,m$, and all $n \in H_0$.  
\end{proposition}

\begin{proof}  The ``only if'' part is trivial, so it suffices to prove the ``if'' part.  For inductive reasons, we shall prove the following more general statement: if $l,m_0 \geq 0$, and $g: H \to G$ is such that
$\partial_{h_1} \ldots \partial_{h_m} g$ is a polynomial map up to order $l$ from $(H,H_I)$ to $(G,G_I^{+i_1+\ldots+i_m})$ whenever $0 \leq m \leq m_0$, $i_1,\ldots,i_m \in I$ and $h_j \in E_{i_j}$ for $j=1,\ldots,m$, then $g$ is a polynomial map from $H$ to $G$ up to order $m_0+l$.  Indeed, by setting $l=0$ and sending $m_0\to\infty$ we obtain the claim.

We establish the claim by induction on $m$.  The case $m_0=0$ is trivial, so suppose that $m_0 \geq 1$ and that the claim has already been proven for smaller values of $m_0$.

Fix $l$.  Let $1 \leq m \leq m_0$ and $i_1,\ldots,i_m \in I$, and suppose that $h_j \in E_{i_j}$ for $j=2,\ldots,m$, and write $\tilde g := \partial_{h_2} \ldots \partial_{h_m} g$.  By hypothesis, we have that $\partial_{h_1} \tilde g$ is a polynomial map of order $l$ from $(H,H_I)$ to $(G,G_I^{+i_1+\ldots+i_m})$ whenever $h_1$ lies in $E_{i_1}$.  Using \eqref{hlin-base} and Corollary \ref{laz}, we conclude the same statement holds when $h_1$ lies in $H_{i_1}$.  Also, by induction hypothesis $\tilde g$ is also known to be a polynomial map of order $l$ from $(H,H_I)$ to $(G,G_I^{+i_2+\ldots+i_m})$.  We conclude that $\tilde g$ is in fact a polynomial map of order $l+1$ from $(H,H_I)$ to $(G,G_I^{+i_2+\ldots+i_m})$.  Applying the induction hypothesis again, we conclude that $g$ is a polynomial map of order $l+m$ from $H$ to $G$, as required.
\end{proof}

\emph{Example.} Let $G_1, G_2, G$ be $I$-filtered groups, and let $B: G_1 \times G_2 \to G$ be a map which is ``bilinear'' in the sense that the maps $g_1 \mapsto B(g_1,g_2)$ for fixed $g_2 \in G_2$ and $g_2 \mapsto B(g_1,g_2)$ for fixed $g_1 \in G_1$ are homomorphisms, and such that $B$ maps $G_{1,\geq i} \times G_{2,\geq j}$ to $G_{i+j}$ for any $i,j \in I$.  Then $B$ is a polynomial map, as can be seen by using Proposition \ref{check} with $G_{1,\geq i} \times \{\id\} \cup \{\id\} \times G_{2,\geq i}$ as the generating set for $(G_1 \times G_2)_i = G_{1,\geq i} \times G_{2,\geq i}$.  Combining this with Corollary \ref{chain}, we conclude in particular that if $H$ is an $I$-filtered group and $g_1 \in \poly(H_I \to (G_1)_I)$, $g_2 \in \poly(H_I \to (G_2)_I)$, then $B(g_1,g_2) \in \poly(H_I \to G_I)$; informally, this is asserting that the product of polynomials is again a polynomial.

\emph{Example.} Let $G$ be an $\N^k$-filtered group, and let $g \in \poly(\Z^k_{\N^k} \to G_{\N^k})$ be a polynomial sequence, in which $\Z^k$ is given the multidegree filtration.  We can collapse the $\N^k$-filtration on $G$ to an $\N$-filtration by defining $G_i$ to be the group generated by $G_{(i_1,\ldots,i_k)}$ for all $(i_1,\ldots,i_k) \in \N^k$ with $i_1+\ldots+i_k = i$.  From Proposition \ref{check} we thus conclude that $g$ remains a polynomial map from $\Z^k$ to $G$ if we now give $\Z^k$ the degree filtration, and give $G$ the $\N$-filtration indicated above.\vspace{11pt}

The next lemma describes a useful type of Taylor expansion for polynomial sequences.

\begin{lemma}[Taylor expansion]\label{taylo}  Let $d \geq 1$ be a natural number, let $G$ be an $\N^d$-filtered group of degree $\subset J$ for some finite downset $J$, and let $g \in \poly(\Z^d_{\N^d} \to G_{\N^d})$, where $\Z^d$ is given the multidegree filtration.  We complete the partial ordering on $J$ to a total ordering in some arbitrary fashion.  Then there exist unique \emph{Taylor coefficients} $g_j \in G_{j}$ for each $j \in J$ such that
$$ g(n) = \prod_{j \in J} g_j^{\binom{n}{j}}.$$
Here we adopt the notational convention
$$ \binom{(n_1,\ldots,n_d)}{(j_1,\ldots,j_d)} := \binom{n_1}{j_1} \ldots \binom{n_d}{j_d}.$$
%Equivalently, we have $\binom{n}{j} = \frac{n!}{j!(n-j)!}$ where $(n_1,\ldots,n_d)! := n_1! \ldots n_d!$.
\end{lemma}

\begin{proof}  We first show uniqueness.  Suppose that we have two Taylor expansions that agree everywhere, that is to say
$$ \prod_{j \in J} g_j^{\binom{n}{j}} = \prod_{j \in J} (g'_j)^{\binom{n}{j}}$$
for all $n \in \Z^d$.  Setting $n=0$ we see that $g_0 = g'_0$.  Cancelling this, we see that
$$ \prod_{j \in J: j > 0} g_j^{\binom{n}{j}} = \prod_{j \in J: j > 0} (g'_j)^{\binom{n}{j}}.$$
More generally, suppose inductively that we have shown that $g_j=g'_j$ for all $j \leq j_0$ and
$$ \prod_{j \in J: j > j_0} g_j^{\binom{n}{j}} = \prod_{j \in J: j > j_0} (g'_j)^{\binom{n}{j}}$$
for all $n \in \Z^d$ some $j_0 \in J$. If $j_0$ is the maximal element of $J$ then we are done.  Otherwise, let $j_1$ be the next element after $j_0$ in the total ordering of $J$.  Setting $n=j_1$ we conclude that $g_{j_1}=g_{j'_1}$, and then we can continue the induction and establish uniqueness.

Now we show existence by inducting on the cardinality of $J$.  The claim is trivial for $J$ empty, so suppose that $J$ is non-empty, and let $j_*$ be the maximal element of $J$.  The group $G_{j_*}$ is a central subgroup of $G$; if we quotient $G$ by $G_{j_*}$, we obtain an $\N^d$-filtered group $G/G_{j_*}$ of degree $\subset J \backslash \{j_*\}$.  Let $\pi: G \to G_{j_*}$ be the quotient map.  Applying the induction hypothesis, we have a Taylor expansion
$$ \pi(g(n)) = \prod_{j \in J: j \neq j_*} h_j^{\binom{n}{j}}$$
for some $h_j \in \pi(G_{j})$.  Writing $h_j = \pi(g_j)$ for some $g_j \in G_{j}$, and using the central nature of $G_{j_*}$, we conclude that
$$ g(n) = (\prod_{j \in J: j \neq j_*} g_j^{\binom{n}{j}}) g'(n)$$
for some $g'(n)$ taking values in $G_{j_*}$.  By Corollary \ref{laz}, $g'$ is a polynomial sequence, and therefore
$$ \partial_{e_1}^{j_1} \ldots \partial_{e_k}^{j_k} g'(n) = \id$$
whenever $(j_1,\ldots,j_k) \not \leq j_*$, with $e_1,\ldots,e_k$ being the basis of $\Z^k$.  We can ``integrate'' this difference equation repeatedly using the abelian nature of $G_{j_*}$ (and the Pascal's triangle relation $\partial_{e_i} \binom{n}{j+e_i} = \binom{n}{j}$) and conclude that
$$ g'(n) = \prod_{j \leq j_*} (g'_j)^{\binom{n}{j}}$$
for some $g'_j \in G_{j_*}$.  Using the central nature of $G_{j_*}$, we conclude that
$$ g(n) = \prod_{j \in J} (g_j g'_j)^{\binom{n}{j}}$$
(with the convention that $g_{j_*} = \id$) and the claim follows.
\end{proof}

\begin{corollary}[Pullback]\label{pull}  Let $d \geq 1$ be a natural number, let $G$ be an $\N^d$-filtered group of degree $\subset J$ for some finite $J$, and let $g \in \poly(\Z^d_{\N^d} \to G_{\N^d})$.  Let $\phi: G' \to G$ be a $\N^d$-filtered homomorphism of $\N^d$-filtered groups such that $\phi: G'_{j} \to G_{j}$ is surjective for every $j$.  Then there exists $g' \in \poly(\Z^d_{\N^d} \to G'_{\N^d})$ such that $g = g' \circ \phi$.
\end{corollary}

\begin{proof} Apply Lemma \ref{taylo} and then pull back each of the resulting Taylor coefficients $g_j$.
\end{proof}

\section{Lifting linear nilsequences to polynomial ones}\label{lift-app}

The purpose of this appendix is to demonstrate the equivalence of the linear inverse conjecture, Conjecture \ref{gis-conj}, with the polynomial inverse conjecture, Conjecture \ref{gis-poly}.  We remind the reader that this is not strictly speaking necessary to establish the results in \cite{green-tao-linearprimes}, but the latter paper was written before the more general notion of a polynomial nilsequence came to the fore.

The key observation here is that every polynomial nilsequence of degree $\leq s$ can be ``lifted'' to an $s$-step linear nilsequence in a certain sense.  

We begin by recording a useful lemma.

\begin{lemma}[Discrete polynomials are cocompact]\label{cocompact}  Let $G/\Gamma$ be an $\N$-filtered nilmanifold.  Then $\poly(\Z_\N \to \Gamma_\N)$ is a lattice \textup{(}i.e. a discrete cocompact subgroup\textup{)} of $\poly(\Z_\N \to G_\N)$ \textup{(}where we give $\Z$ the degree filtration\textup{)}.
\end{lemma}

%\begin{remark} The same claim is also true with $\Z$ replaced by $\Z^k$, but we will not need this fact here.
%\end{remark}

\begin{proof}  We may assume that $G/\Gamma$ has degree-rank $\leq d$.
It will suffice to show that any polynomial sequence $g \in \poly(\Z_\N \to G_\N)$ can be factorised as $g = \gamma g'$ where $\gamma \in \poly(\Z_\N \to \Gamma_\N)$ and $g'$ ranges in a compact subset of $\poly(\Z_\N \to G_\N)$.  It is enough to show by induction on $i$ that for every $0 \leq i \leq d+1$, there exists a factorisation $g = \gamma_i h_i g'_i$ where $\gamma' \in \poly(\Z_\N \to \Gamma_\N)$, $g'_i$ lies in a compact subset of $\poly(\Z \to G)$, and $h_i \in \poly(\Z_\N \to G_\N)$ is such that $h(0)=\ldots=h(i-1)=\id$, since for $i=d+1$ this forces $h$ to be trivial.

This inductive claim is trivial for $i=0$ (setting $\gamma_0 = g'_0$ to be trivial).  Now suppose inductively that one has a factorisation $g = \gamma_i h_i g'_i$ for some $0 \leq i \leq d$.  Since $h(0)=\ldots=h(i-1)=\id$, we see from Taylor expansion that $h(i) \in G_{i}$.  Since $\Gamma_{i} := \Gamma \cap G_{i}$ is cocompact in $G_{i}$, we may factorise $h(i) = \tilde \gamma_{i+1}(i) \tilde g'_{i+1}(i)$ for some $\tilde \gamma_{i+1}(i) \in \Gamma_{i}$ and $\tilde g'_{i+1}(i)$ in a cocompact subset of $G_{i}$.  By Taylor expansion we may extend $\tilde \gamma_{i+1}$, $\tilde g'_{i+1}$ to elements of $\poly(\Z_\N \to \Gamma_\N)$ and of a compact subset of $\poly(\Z_\N \to G_\N)$ respectively which are trivial on $0,\ldots,i-1$.  Writing $\gamma_{i+1} := \gamma_i \tilde \gamma_{i+1}$, $h_{i+1} := \tilde \gamma_{i+1}^{-1} h_i (\tilde g'_{i+1})^{-1}$, and $g'_{i+1} := \tilde g'_{i+1} g'_i$ we obtain the claim.
\end{proof}

Now we establish the key lifting proposition.

\begin{proposition}[Polynomial nilsequences can be lifted to linear ones]\label{lift}  Let $G/\Gamma$ be a filtered nilmanifold of degree $\leq s$.  Then there exists a standard $s$-step nilmanifold $\tilde G/\tilde \Gamma$, a standard compact subset $K$ of $\tilde G/\tilde \Gamma$, and a standard Lipschitz map $\pi: K \to G/\Gamma$, such that for every \textup{(}standard\textup{)} polynomial sequence $g: \Z \to G$, there exists $\tilde g \in \tilde G$ and $\tilde x \in \tilde G/\tilde \Gamma$ such that $\tilde g^n \tilde x \in K$ and $g(n) \ultra \Gamma = \pi( \tilde g^n \tilde x )$ for all $n \in \Z$.
\end{proposition}

Indeed, with this proposition, any degree $\leq s$  nilsequence $n \mapsto F( g(n) \ultra \Gamma )$ can then be lifted to an $s$-step linear nilsequence $n \mapsto (F \circ \pi)( \tilde g^n \tilde x )$ with $\tilde g \in \ultra \tilde G$ and $\tilde x \in \ultra (\tilde G/\tilde \Gamma)$, where $F \circ \pi$ is extended from a Lipschitz function on $K$ to a Lipschitz function on $\tilde G/\tilde \Gamma$ in some arbitrary fashion.   From this one easily concludes that Conjecture \ref{gis-conj} follows from Conjecture \ref{gis-poly}.  (The converse implication is trivial, because every linear nilsequence is a polynomial nilsequence.)

To motivate Proposition \ref{lift} let us present an illustrative example.  We take $s=2$ and $G/\Gamma$ to be the unit circle $\R/\Z$ with the quadratic filtration (thus $G_{i}$ equals $\R$ for $i \leq 2$ and $\{0\}$ for $i>2$).  By Remark \ref{taylor}, a polynomial sequence $g: \Z \to G$ then takes the form $g(n) = \alpha_0 + \alpha_1 \binom{n}{1} + \alpha_2 \binom{n}{2}$ for some frequencies $\alpha_0,\alpha_1,\alpha_2$ (i.e. a non-standard classical quadratic polynomial).  To lift this quadratic sequence to a linear one, we introduce use the Heisenberg nilmanifold $\tilde G/\tilde \Gamma$ (Example \ref{heisen}), and place inside it the \emph{skew torus}
$$ K := \{ g_1^{t_1} [g_1,g_2]^{t_{[1,2]}} \Gamma: t_1, t_{[1,2]} \in \R \}.$$
This is easily seen to be compact (indeed, it is topologically equivalent to $\T^2$).
We define the map $\pi: K \to \T$ by the formula
$$ \pi( g_1^{t_1} [g_1,g_2]^{t_{[1,2]}} ) := t_{[1,2]} \mod 1;$$
it is easy to see that $\pi$ is well-defined and smooth.  If we set
$$ \tilde g := g_1^{\alpha} g_2 [g_1,g_2]^{\beta}; \quad \tilde x := [g_1,g_2]^{\gamma} \tilde \Gamma $$
for some frequencies $\alpha,\beta,\gamma \in \R$, then a brief calculation shows that for any integer $n$, $\tilde g^n \tilde x$ lies in $K$ and
$$ \pi( \tilde g^n \tilde x ) = \frac{n(n+1)}{2} \alpha + n \beta + \gamma \mod 1,$$
and so one can arrange for $\pi( \tilde g^n \tilde x ) = g(n)$ by choosing $\alpha,\beta,\gamma$ appropriately in terms of $\alpha_0, \alpha_1, \alpha_2$.

The above construction was \emph{ad hoc} in nature, requiring one to conjure up the Heisenberg group out of thin air.  However, it is possible to canonically construct a lifted nilmanifold $\tilde G/\tilde \Gamma$ in the general case. Fix $G/\Gamma$.  By Remark \ref{taylor}, $\poly(\Z_\N \to G_\N)$ is a Lie group topologically isomorphic to $\prod_{i \geq 0} G_{i}$, but with a different group structure.  Since $G$ has degree $<s+1$, we see that $G$ is $\leq s$-step nilpotent, which implies that $\poly(\Z_\N \to G_\N)$ is $\leq s$-step nilpotent also.

Let $\Gamma_\N$ be the restriction of the filtration $G_\N$ to $\Gamma$ (Example \ref{pushpull}), thus $\Gamma$ is now a filtered group.  By Lemma \ref{cocompact}, $\poly(\Z_\N \to \Gamma_\N)$ has the structure of an $s$-step nilmanifold.  This is not yet the nilmanifold $\tilde G/\tilde \Gamma$ needed for Proposition \ref{lift}, but we can modify it as follows.  We observe that there is a shift automorphism $T$ acting on both $\poly(\Z_\N \to G_\N)$ and $\poly(\Z_\N \to \Gamma_\N)$ by the formula $Tg(n) := g(n+1)$.  It also acts on the Lie algebra $\log \poly(\Z_\N \to G_\N)$ of $\poly(\Z_\N \to G_\N)$, which by abuse of notation we shall call $\poly(\Z_\N \to \log G_\N)$.  This action is unipotent; indeed, $T-1$ maps $\poly(\Z_\N \to \log G^{+i}_\N)$ to $\poly(\Z_\N \to \log G^{+(i+1)}_\N)$ for all $i \geq 0$, where $G^{+i}$ is $G$ with the shifted filtration $G^{+i}_{d} := G_{d+i}$.  The conjugation action of $\poly(\Z_\N \to G_\N)$ on $\poly(\Z_\N \to \log G_\N)$ has the same unipotence property by the filtered nature of $G$.  Because of this, we see that the conjugation action of semi-direct product\footnote{Note that $\Z$ is viewed as an additive group, while $\poly(\Z_\N \to G_\N)$ is viewed as a multiplicative group; we hope that this will not cause confusion.} $\poly(\Z_\N \to G_\N) \rtimes_T \Z$ on $\poly(\Z_\N \to \log G_\N)$ is $s$-step unipotent, which implies that $\poly(\Z_\N \to G_\N) \rtimes_T \Z$ is $s$-step nilpotent.  

Unfortunately, the group $\poly(\Z_\N \to G_\N) \rtimes_T \Z$ is not connected, so it is not directly suitable for the purposes of establishing Proposition \ref{lift}.  But this can be easily remedied by using the unipotent nature of the action of $T$ on $\poly(\Z_\N \to \log G_\N)$ to express\footnote{This can also be done by the machinery of Mal'cev bases for both discrete and continuous nilpotent groups, see \cite{leibman-poly}.} $T = T^1$ for some smooth unipotent group action $t \mapsto T^t$ of the real line $\R$ on $\poly(\Z_\N \to \log G_\N)$, which can then be exponentiated to provide a unipotent group action (which we will also call $t \mapsto T^t$) on $\poly(\Z_\N \to G_\N)$.  The action of 
the group $\tilde G := \poly(\Z_\N \to G_\N) \rtimes_T \R$ on $\poly(\Z_\N \to \log G_\N)$ is then $s$-step unipotent, which implies that $\tilde G$ is $s$-step nilpotent.

The group $\tilde G$ is an $s$-step nilpotent Lie group which is both connected and simply connected.  It contains the discrete subgroup $\tilde \Gamma := \poly(\Z_\N \to \Gamma_\N) \rtimes_T \Z$.  Since $\poly(\Z_\N \to \Gamma_\N)$ is cocompact in $\poly(\Z_\N \to G_\N)$ (and $\Z$ is cocompact in $\R$), $\tilde G$ is cocompact in $\tilde G$; thus $\tilde G/\tilde \Gamma$ has the structure of a nilmanifold.

There is a canonical map $\theta$ from $\tilde G/\tilde \Gamma$ to $\T$ induced by the projections of $\tilde G, \tilde \Gamma$ to $\R$ and $\Z$ respectively.  We denote the kernel $\theta^{-1}(\{0\})$ of this map by $K$, thus $K$ is a compact subset of $\tilde G, \tilde \Gamma$.  Observe that every element of $K$ can be represented as $(g, 0) \tilde \Gamma$ for some $g \in \poly(\Z_\N \to G_\N)$, which is unique up to multiplication on the right by $\poly(\Z_\N \to \Gamma_\N)$.  We then define the map $\pi: K \to G/\Gamma$ by the formula $\pi(g) := g(0) \Gamma$; it is clear that $\pi$ is a Lipschitz continuous map.

We are now ready to establish Proposition \ref{lift}.  Let $g \in \poly(\Z_\N \to G_\N)$, then we set $\tilde x := (g,0)\tilde \Gamma \in K$ and $\tilde g := (\id,1) \in \tilde G$.  One easily verifies that for any integer $n$, $\tilde g^n \tilde x = (T^n g, 0) \tilde \Gamma \in K$, and so $\pi( \tilde g^n \tilde x) = g(n)$.  Proposition \ref{lift} follows.

\section{Equidistribution theory}\label{equiapp}

The purpose of this appendix develop the quantative Ratner-type equidistribution theory for nilmanifolds, which will help us determine when averages such as
\begin{equation}\label{fon}
 \E_{n \in [N]} F(\orbit(n)) 
\end{equation}
are large, for various nilsequences $n \mapsto F(\orbit(n))$.  We will also need a multidimensional version\footnote{On the other hand, we will however only need to work with the degree filtration, although it is certain that the theory here would extend to $I$-filtered nilsequences for other orderings.} of this theory, in which $[N]$ is replaced with $[N]^k$, or more generally by the Cartesian product of $k$ arithmetic progressions.  

This theory is based on the results \cite{green-tao-nilratner} on equidistribution in nilmanifolds, translated to the language of ultralimits.  The results in this appendix will be needed in two places.  Firstly, Theorem \ref{ratt} below, which gives a criterion for when averages such as \eqref{fon} are large, will be used in \S \ref{linear-sec} to analyse the correlation property arising from Proposition \ref{gcs-prop}.  Secondly, Theorem \ref{factor2}, which (locally) factorises an arbitrary multidimensional polynomial orbit into equidistributed and smooth pieces, will be used to give an important criterion for when a nilcharacter is biased (see Lemma \ref{bias}).

We begin with some basic definitions.

\begin{definition}[Equidistribution]\label{equid-def} Let $G/\Gamma$ be a standard nilmanifold, which then admits a canonical Haar probability measure $\mu$.  Let $\Omega$ be a non-empty limit finite set, and let $\orbit: \Omega \to \ultra (G/\Gamma)$ be a limit function.  We say that $\orbit$ is \emph{equidistributed} in $G/\Gamma$ if, for every $F \in \Lip(G/\Gamma)$, one has
\begin{equation}\label{equid-def-eq}
\E_{n \in \Omega} F( \orbit(n) ) = \int_{G/\Gamma} F\ d\mu + o(1),
\end{equation}
or equivalently if $n \mapsto F(\orbit(n))$ is unbiased on $\Omega$ whenever $\int_{G/\Gamma} F\ d\mu = 0$.

Now we specialise to the case $\Omega=[N]^k$.  We say that $\orbit$ is \emph{totally equidistributed} on $[N]^k$ if it is equidistributed on every product $P_1 \times \ldots \times P_k$ of dense arithmetic progressions $P_1,\ldots,P_k$ in $[N]$, thus
\begin{equation}\label{total-equid-def}
\E_{n \in P_1 \times \ldots \times P_k} F( \orbit(n) ) = \int_{G/\Gamma} F\ d\mu + o(1)
\end{equation}
for every standard Lipschitz function $F: G/\Gamma \to \C$.  
\end{definition}

\emph{Remark.} We defined equidistribution using standard Lipschitz functions $F \in \Lip(G/\Gamma)$, but the statement \eqref{equid-def-eq} for $F \in \Lip(G/\Gamma)$ automatically implies the same claim for $F \in \Lip(\ultra(G/\Gamma))$.

This notion of equidistribution on $[N]$ is closely related to, but not identical with, the more classical notion of equidistribution involving an \emph{infinite} sequence $g: \Z \to G$, in which one takes a limit as $N \to \infty$; we refer to this latter concept as \emph{asymptotic equidistribution} in order to distinguish it from the ``single-scale'' equidistribution considered here, in which one is working with a fixed (but unbounded) $N$.  While there is a close analogy between the theory of asymptotic equidistribution and single-scale equidistribution, there does not seem to be a soft way to automatically transfer results from the former to the latter.  Single-scale equidistribution is in fact much closer to the notion of \emph{$\delta$-equidistribution} studied for instance in \cite{green-tao-nilratner}; we refer readers to that paper for further discussion of the distinction between the different types of equidistribution.

\emph{Example.}  We consider the case when $G/\Gamma = \T^d$ is a torus.  Weyl's equidistribution criterion, in our notation, then asserts that an limit map $\orbit: [N]^k \to \T^d$ is equidistributed if and only if one has
$$ \E_{n \in [N]^k} e( \xi \cdot \orbit(n) ) = o(1)$$
for all standard $\xi \in \Z^d \backslash \{0\}$.  One can also show (using some Fourier analysis) that $\orbit$ will be totally equidistributed if and only if
$$ \E_{n \in [N]^k} e( \xi \cdot \orbit(n) ) e( \eta \cdot n ) = o(1)$$
for all standard $\xi \in \Z^d \backslash \{0\}$ and $\eta \in \Z^k$.  As a consequence of this and some further Fourier analysis, we see that a one-dimensional linear orbit $\orbit: [N] \to \T^d$ defined by $\orbit(n) := \alpha n + \beta$ for some $\alpha,\beta \in \T^d$ will be equidistributed or totally equidistributed in $\T^d$ if and only if $\alpha$ is not of the form $q + O(N^{-1}) \mod 1$ for some standard rational $q \in \Q$.

Given a standard filtered nilmanifold $G/\Gamma$, a \emph{horizontal character} is a continuous standard homomorphism $\xi: G \to \T$ which vanishes on $\Gamma$.  We say that the character is \emph{non-trivial} if it is not identically zero.

We have the following basic equidistribution criterion, generalising the torus example above.

\begin{theorem}[Leibman theorem]\label{leib-nil}  Let $k \in \N^+$, let $N$ be an unbounded natural number, let $G/\Gamma$ be an $\N$-filtered nilmanifold, and let $\orbit \in \ultra \poly(\Z^k_\N \to (G/\Gamma)_\N)$ be a $k$-dimensional polynomial orbit, where $\Z^k$ is given the degree filtration.  Then on $[N]^k$, the following statements are equivalent:
\begin{itemize}
\item[(i)] $\orbit$ is totally equidistributed in the nilmanifold $G/\Gamma$;
\item[(ii)] $\orbit$ is equidistributed in the nilmanifold $G/\Gamma$;
\item[(iii)] $\orbit$ is equidistributed in in the torus $G/([G,G] \Gamma)$, and
\item[(iv)] There does not exist any non-trivial horizontal character $\xi$ such that $\xi \circ g$ is Lipschitz with constant $O(1/N)$.
\end{itemize}
\end{theorem}

\begin{proof} See \cite[Theorems 1.19, 2.9, 8.6]{green-tao-nilratner} (where in fact a more quantitative strengthening of this equivalence is established).  The analogue of this result for asymptotic equidistribution was established previously in \cite{leibman} (and the result is classical in the case of linear sequences).  The main difficulty is to show that (iv) implies (ii), which is the main content of \cite[Theorem 2.9]{green-tao-nilratner}, which relies primarily on a certain van der Corput type equidistribution lemma for nilmanifolds.
\end{proof}

Theorem \ref{leib-nil} implies the following weak factorisation theorem.

\begin{theorem}[Weak factorisation theorem]\label{wfat} Let $k \in \N^+$, let $N$ be an unbounded natural number, let $G/\Gamma$ be an $\N$-filtered nilmanifold and let $g \in \ultra \poly(\Z^k_\N \to G_\N)$.  Suppose that $g$ is not totally equidistributed on $[N]$ in $G/\Gamma$.  Then one can factorise $g = \eps g' \gamma$, where $\eps, g', \gamma \in \ultra \poly(\Z^k_\N \to G_\N)$ have the following properties:
\begin{itemize}
\item $\eps$ is a bounded sequence on $[N]^k$ with the $i^\th$ Taylor coefficient of size $O(N^{-|i|})$ for each $i \in \N^k$;
\item $g'$ takes values in a standard proper rational subgroup $G'$ of $G$ \textup{(}i.e. $G'$ is a connected proper Lie subgroup of $G$, and $\Gamma' := G' \cap \Gamma$ is cocompact in $G$\textup{)}.
\item $\gamma$ is periodic modulo $\Gamma$ with a standard period $q \in \N^+$, thus $\gamma(n+qv) = \gamma(n) \mod \Gamma$ for all $n, v \in \ultra \Z^k$.  Furthermore, $\gamma$ takes values in a standard subgroup $\tilde \Gamma$ of $G$ which contains $\Gamma$ as a subgroup.
\end{itemize}
\end{theorem}

\begin{proof} See \cite[Proposition 9.2]{green-tao-nilratner}.  The basic idea is to use the non-trivial horizontal character $\xi$ generated by Theorem \ref{leib-nil} to cut out the subgroup $G'$.  In order to keep $G'$ connected, one needs to first factorise $\xi = m \xi'$ where $m$ is a standard positive integer and $\xi'$ is an irreducible horizontal nilcharacter; this integer $m$ is responsible for the periodic term $\gamma$.
\end{proof}

One can iterate this to obtain a ``Ratner-type'' theorem.

\begin{theorem}[Factorisation theorem]\label{factor} Let $k \in \N^+$, let $N$ be an unbounded natural number, let $G/\Gamma$ be a \textup{(}filtered\textup{)} nilmanifold, and let $g \in \ultra \poly(\Z^k_\N \to G_\N)$.  Then there exists a standard rational subgroup $G'$ of $G$ \textup{(}i.e. $G'$ is connected and $G' \cap \Gamma$ is cocompact in $G$\textup{)} and a factorisation
$$ g(n) = \eps(n) g'(n) \gamma(n)$$
where $\eps, g', \gamma \in \ultra \poly(\Z^k_\N \to G_\N)$ have the following additional properties:
\begin{itemize}
\item $\eps$ is a bounded sequence with the $i^\th$ Taylor coefficient of size $O(N^{-|i|})$ for each $i \in \N^k$, and has Lipschitz constant $O(1/N)$;
\item $g'$ takes values in a standard proper rational subgroup $G'$ of $G$, and is totally equidistributed in $G'/\Gamma'$ whenever $\Gamma'$ is any standard subgroup of $G' \cap \Gamma$ of finite \textup{(}standard\textup{)} index.
\item $\gamma$ is periodic modulo $\Gamma$ with a standard period, and takes values in a standard discrete subgroup $\tilde \Gamma$ of $G$ which contains $\Gamma$.
\end{itemize}
\end{theorem}

This theorem is a close relative of \cite[Theorem 1.19]{green-tao-nilratner}, and can be proven by the same methods; for the convenience of the reader we sketch a proof here.

\begin{proof}  Let us say that $g$ can be \emph{represented} using a standard rational subgroup $G'$ of $G$ if one has a factorisation $g = \eps g' \gamma$ which obeys all the conclusions of the theorem except for the total equidistribution of $g'$.  Clearly, $g$ can be represented using $G$ itself, by setting $\eps$ and $\gamma$ to be the identity and $g' := g$.  By the principle of infinite descent\footnote{The ability to use this principle is an advantage of the ultralimit setting.  In the finitary setting, in which one needs to quantify such concepts as total equidistribution, periodicity, etc., one has to instead perform an iterative ``dimension reduction argument'' which requires one to manage many more parameters; see \cite{green-tao-nilratner} for an example of this.  See also the beginning of \S \ref{reg-sec} for a related discussion.} (using the fact that $G$ has a finite standard dimension), we may thus find a standard rational subgroup $G'$ which represents $g$, and is \emph{minimal} in the sense that no proper standard rational subgroup of $G'$ represents $g$.  Let $g = \eps g' \gamma$ be the associated factorisation.  It then suffices to show that $g'$ is totally equidistributed in $G'/\Gamma'$ for every standard finite index subgroup $\Gamma'$ of $\Gamma \cap G'$.

Suppose for contradiction that this is not the case.  Applying Theorem \ref{wfat}, one can factorise $g' = \eps'' g'' \gamma''$ where $\eps''$ is a bounded sequence with Lipschitz constant $O(1/N)$, $\gamma''$ is periodic with a standard period, and takes values in a standard discrete subgroup $\tilde \Gamma'$ that contains $\Gamma'$, and $g''$ takes values in a proper rational subgroup $G''$ of $G'$.  One can enlarge $\tilde \Gamma'$ to contain $\Gamma$, and this is easily verified to still be discrete. One can then show that the factorisation $g = (\eps \eps'') g'' (\gamma'' \gamma)$ is a representation of $g$ using $G''$ (see \cite[\S 10]{green-tao-nilratner} for details), contradicting the minimality of $G''$.
\end{proof}

It will be convenient to convert the factorisation in Theorem \ref{factor} into a more convenient form, eliminating the periodic factor $\gamma$ and the slowly varying factor $\eps$ by passing to subprogressions.

\begin{theorem}[Factorisation theorem, II]\label{factor2} Let $k \in \N^+$, let $N$ be an unbounded natural number and let $\orbit \in \ultra \poly(\Z^k_\N \to (G/\Gamma)_\N)$.  Then one can partition $[N]^k$ into a bounded number of products $P= P_1 \times \ldots P_k$ of dense arithmetic subprogressions of $[N]$, such that for each $P$ one has a polynomial $\eps \in \ultra \poly(\Z^k_\N \to G_\N)$ which is bounded with Lipschitz constant $O(1/N)$ on $P$ and with the $i^{\th}$ Taylor coefficient of size $O(N^{-|i|})$ for each $i$, a standard rational subgroup $G_P$ of $G$, and a polynomial sequence $g_P \in \ultra \poly(\Z^k_\N \to (G_P)_\N)$ totally equidistributed on $G_P/\Gamma_P$ where $(G_P)_\N := (G_P \cap G_i)_{i \in \N}$ and $\Gamma_P := G_P \cap \Gamma$, such that
$$ \orbit(n) = \eps_P(n) g_P(n)\ultra \Gamma$$
for all $n \in P$.  Furthermore, for each $i \in \N^k$, the horizontal Taylor coefficients $\Taylor_i(g)$ and $\Taylor_i(g_P)$ differ by $O(N^{-|i|})$.  Finally, for two different products $P,P'$ of progressions in this partition of $[N]^k$, the sequences $g_P$ and $g_{P'}$ are conjugate, with $g_{P'} = \gamma_{P,P'}^{-1} g_P \gamma_{P,P'}$ for some $\gamma_{P,P'} \in G$ which is \emph{rational} in the sense that $\gamma_{P,P'}^m \in \Gamma$ for some bounded positive integer $m$.
\end{theorem}

\begin{proof}  Write $\orbit(n) = g(n) \ultra \Gamma$ for some $g \in \ultra \poly(\Z^k_\N \to G_\N)$.  We apply Theorem \ref{factor} to obtain a rational standard subgroup $G'$ and a factorisation $g = \eps g' \gamma$ with the stated properties.  The sequence $\gamma$ is periodic with a standard period, so we may partition $[N]^k$ into a bounded number of products $P=P_1\times\ldots\times P_k$ of dense arithmetic subprogressions of $[N]$ on which $\gamma = \gamma_P$ is constant.  As $\Gamma$ is cocompact, we may thus find $\gamma'_P \in \gamma_P \Gamma$ which is bounded, thus $\gamma'_P = O(1)$.  Note that $\gamma'_P$ lives in a discrete group $\tilde \Gamma$ and is thus standard.  Since $\Gamma$ is cocompact, it has finite index in $\tilde \Gamma$, which implies that $\gamma'_P$ is rational, or equivalently that $\gamma'_P$ has rational coefficients with respect to a Mal'cev basis \cite{malcev} of $G/\Gamma$.

For $n \in P$, we can write
$$ \orbit(n) = \eps(n) g'(n) \gamma_P \Gamma = \eps(n) \gamma'_P g_P(n) \ultra \Gamma$$
where $g_P(n) := (\gamma'_P)^{-1} g'(n) \gamma'_P$ is the conjugate of $g'(n)$ by $\gamma'_P$.  Note that this gives the claim about the conjugate nature of $g_P$ and $g_{P'}$.

As $\gamma'_P$ is rational, the conjugate $\gamma'_P \Gamma (\gamma'_P)^{-1}$ intersects $\Gamma$ in a subgroup $\Gamma'$ of finite index, which then has the property that $\Gamma' \gamma'_P \subset \gamma'_P \Gamma$.  From this, we see that the conjugation operation $g \mapsto (\gamma'_P)^{-1} g \gamma'_P$ on $G$ descends to a continuous projection of $G/\Gamma'$ to $G/\Gamma$, which maps $g(n) \ultra \Gamma'$ to $g_P(n) \ultra \Gamma$.  Since $g(n)$ is totally equidistributed on $G'/(G' \cap \Gamma')$ by construction, we conclude that $g_P$ is totally equidistributed on $G_P/(G_P \cap \Gamma)$, where $G_P := (\gamma'_P)^{-1} G' \gamma'_P$ is the conjugate of $G'$.  Note that $G_P$ is also a standard rational subgroup of $G$.  If we now set $\eps_P := \eps \gamma'_P$, we obtain all the claims except for the one about horizontalTaylor coefficients.  But from the remarks following Definition \ref{horton} and the factorisations $g = \eps g' \gamma$, $g_P = (\gamma'_P)^{-1} g' \gamma_P$ we have
$$ \Taylor_i(g) = \Taylor_i(\eps) \Taylor_i(g') \Taylor_i(\gamma)$$
and
$$ \Taylor_i(g_P) = \Taylor_i(g').$$
Since $\gamma$ takes values in $\Gamma$, $\Taylor_i(\gamma)$ vanishes.  Finally, by construction we have $\Taylor_i(\eps) = O(N^{-|i|})$.  The claim follows.
\end{proof}

We can now give a criterion for when  an average of the form $\E_{n \in [N]} F(\orbit(n))$ is large.

\begin{theorem}[Ratner-type theorem]\label{ratt}  Let $G/\Gamma$ be $\N$-filtered nilmanifold of some degree $d$, let $\orbit \in \ultra \poly(\Z_\N \to (G/\Gamma)_\N)$ be a polynomial orbit, and let \[ F \in \Lip(\ultra(G/\Gamma) \to \overline{\C^\omega})\] be such that
$$ |\E_{n \in [N]} F(\orbit(n))| \gg 1.$$
Then one has
$$ |\int_{G_P / \Gamma_P} F(\eps x)\ d\mu(x)| \gg 1$$
for some bounded $\eps \in G$ and some rational subgroup $G_P$ of $G$, with the property that
$$ \pi_{\Horiz_i(G)}(G_P \cap G_{i}) \geq \Xi_i^\perp $$
where the horizontal space $\Horiz_i(G)$ and the projection map $\pi_{\Horiz_i(G)}: G_{i} \to \Horiz_i(G)$ was defined in Definition \ref{horton},
$$ \Xi_i^\perp := \{ x \in \Horiz_i(G) : \xi_i(x) = 0 \hbox{ for all } \xi_i \in \Xi_i \}$$
and $\Xi_i \leq \widehat{\Horiz_i(G/\Gamma)}$ is the group of all \textup{(}standard\textup{)} continuous homomorphisms $\xi_i: \Horiz_i(G/\Gamma) \to \T$ such that 
$$ \xi_i( \Taylor_i(\orbit) ) = O( N^{-i} ).$$
\end{theorem}

One could also generalise this theorem to multidimensional orbits, but we will not need to do so in this paper.  We will motivate this theorem with some examples after the proof.

\begin{proof} By taking components we may assume that $F$ is scalar-valued.
Write $\orbit(n) = g(n) \ultra \Gamma$ for some $g \in \ultra \poly(\Z_\N \to G_\N)$.  We partition $[N]$ into dense arithmetic progressions $P$ induced from the partition of $[N]$ coming from Theorem \ref{factor2} (using the Chinese remainder theorem and passing to dense subprogressions as necessary).  By the pigeonhole principle, for at least one of these progressions $P$ one has
$$ |\E_{n \in P} F(g(n) \ultra \Gamma)| \gg 1.$$
Now let $\delta > 0$ be a small standard number to be chosen later.  By further partitioning of $P$ and the pigeonhole principle one can assume that $P$ has diameter at most $\delta N$ (note that the implied constant in the $\gg$ notation remains independent of $\delta$ when doing so).  Then for any $n_0 \in P$, $\eps_P(n)$ and $\eps_P(n_0)$ differ by $O(\delta)$, and so (by the Lipschitz nature of $F$) $F(g(n) \ultra \Gamma)$ differs from $F( \eps_P(n_0) g_P(n) \ultra \Gamma )$ by $O(\delta)$.  Thus, for $\delta$ sufficiently small, and setting $\eps := \eps_P(n_0)$, one has
$$ |\E_{n \in P} F(\eps g_P(n) \ultra \Gamma)| \gg 1.$$
Using the total equidistribution of $g_P$, we have
$$ \E_{n \in P} F(\eps g_P(n) \ultra \Gamma) = \int_{G_P / \Gamma_P} F(\eps x)\ d\mu(x) + o(1)$$
and so
$$ \int_{G_P / \Gamma_P} F(\eps x)\ d\mu(x) \gg 1.$$
To finish the proof of Theorem \ref{ratt}, we need to show that
\begin{equation}\label{pigp}
\pi_{\Horiz_i(G)}(G_P \cap G_{i}) \geq \Xi_i^\perp
\end{equation}
for all positive standard integers $i$, with $\Xi_i$ as in Theorem \ref{ratt}.  

Fix $i$. To show the above claim, observe that $g_P$ takes values in $G_P$, and so  $\Taylor_{i}(g_P) \in \pi_{\Horiz_i(G/\Gamma)}(G_P \cap G_{i})$.  On the other hand, $\Taylor_{i}(g_P)$ differs from $\Taylor_{i}(g)$ by $O(N^{-i})$, and so 
\begin{equation}\label{dist}
\dist( \Taylor_{i}(g), \pi_{\Horiz_i(G/\Gamma)}(G_P \cap G_{i}) ) = O(N^{-i} ).
\end{equation}
Suppose the inclusion \eqref{pigp} failed.  Then by duality (and the rational nature of $G_P$), there exists a standard continuous homomorphism $\xi_i: \Horiz_i(G/\Gamma) \to \T$ outside of $\Xi_i$ which annihilates $\pi_i(G_P \cap G_{i})$.  From \eqref{dist}, This implies that 
$$ \xi_i( \Taylor_{i}(g) ) = O(N^{-i}),$$
and thus $\xi_i \in \Xi_i$ by definition of $\Xi_i$, contradiction.  The claim follows.
\end{proof}

To get a feel for this proposition, let us first examine a simple special case, when $G/\Gamma$ is just a two-dimensional torus $\T^2$, and $\orbit$ is a linear orbit $\orbit(n) := (\alpha n, \beta n)$ for some $\alpha,\beta \in \ultra \T$.  We take $F$ to be a standard Lipschitz function from $\T^2$ to $\C$.
Our hypothesis is then the assertion that
$$ |\E_{n \in [N]} F( \alpha n, \beta n )| \gg 1. $$
The conclusion is then that
$$ |\int_{T} F(\eps + x)\ d\mu_T(x)| \gg 1$$
for some subtorus $T := G_P / (G_P \cap \Z^2)$ of $\T^2$, where $\eps \in \T^2$ and $G_P$ is a rational subgroup of $\R^2$.  Furthermore, $G_P$ contains the subgroup
$$\Xi_1^\perp := \{ x \in \R^2: \xi(x)=0 \hbox{ for all } \xi \in \Xi_1 \}$$
and $\Xi_1$ is the subgroup of $\Z^2$ defined by
$$\Xi_1 := \{ \xi \in\leq \Z^2: \xi \cdot (\alpha,\beta) = O(N^{-1}) \}.$$
We investigate some subcases of this result.  First consider the case when $\alpha, \beta$ are both within $O(N^{-1})$ of standard rationals.  Then $\Xi_1$ is a finite index subgroup of $\Z^2$, and so $\Xi_1^\perp$ is trivial.  The conclusion is then simply the trivial conclusion that $|F(\eps)| \gg 1$ for some $\eps \in \T^2$, which was of course obvious from the pigeonhole principle.

Now suppose that $\beta$ is within $O(N^{-1})$ of a standard rational $p/q$ with $p,q$ coprime, but that $\alpha$ does not lie within $O(N^{-1})$ of a standard rational.  Then $\Xi_1 = \{ (0,qa): a \in \Z \}$, and so $\Xi_1^\perp = \R \times \{0\}$.  The conclusion is now that $|\int_\T F( x, \eps )\ dx| \gg 1$ for some $\eps \in \T$.  This can also be seen directly by observing that on any subprogression of $[N]$ of spacing $q$ and length $\delta N$ for some small $\delta > 0$, the orbit $\orbit(n)$ is within $O(\delta)$ of being equidistributed on a coset $\T \times \{\eps\}$ of $\T \times \{0\}$ for some $\eps \in \T$, with the implied constant in the $O(\delta)$ notation independent of $\delta$.  The claim then follows from the pigeonhole principle (choosing $\delta$ sufficiently small, but still standard) and the Lipschitz nature of $F$.

Finally, suppose that $\alpha, \beta$ are incommensurate in the sense that there does not exist any non-zero $\xi \in \Z^2$ for which $\xi \cdot (a,b) = O(N^{-1})$.  Then $\Xi_1$ is trivial and so $\Xi_1^\perp = \R^2$.  The claim is then that $|\int_{\T^2} F(x,y) \ dx dy| \gg 1$, which is also apparent from the equidistribution of $\orbit$ in $\T^2$ in this case.

One can also repeat the above example with the linear orbit $n \mapsto (\alpha n, \beta n)$ replaced by a polynomial orbit such as $n \mapsto (\alpha n^D, \beta n^D)$ for some standard $D \geq 1$.  The discussion is identical, except that the $O(N^{-1})$ errors must now be replaced by $O(N^{-D})$.

Now we consider the more general non-abelian setting, in which $G/\Gamma$ is not necessarily a torus (i.e. we allow $d$ to exceed $1$).  We first remark upon the ``incommensurate'', ``generic'', or ``equidistributed'' case when all the $\Xi_i$ are trivial, i.e. there are no non-trivial relations of the form
$$ \xi_i( \Taylor_{\vec i}(\orbit) ) = O( N^{-i} ).$$
In this case, $\Xi_i^\perp = \Horiz_i(G)$ and so the maps $\pi_i: G_P \cap G_i \to \Horiz(G_i)$ are all surjective.  This implies that all the horizontal spaces of the quotient group $G/G_P$ are trivial, which one easily sees to imply that $G/G_P$ itself must be trivial, i.e. that $G_P=G$.  We conclude that $|\int_{G/\Gamma} F\ d\mu| \gg 1$.  Indeed, in this case it turns out that $\orbit$ is totally equidistributed and
$$ \E_{n \in P_0} F(\orbit(n))=\int_{G/\Gamma} F\ d\mu+o(1).$$
This fact can also be deduced from the \emph{arithmetic counting lemma} \cite[Theorem 1.11]{green-tao-arithmetic-regularity}.

Finally, to illustrate how we actually use Theorem \ref{ratt} in practice, we consider a model problem in which we are given frequencies $\alpha,\beta,\alpha',\beta' \in \ultra \T$ obeying the correlation property
\begin{equation}\label{enn}
 |\E_{n \in [N]} e( \{ \alpha n \} \beta n ) e( \{ \alpha' n \} \beta' n )| \gg 1,
\end{equation}
and we wish to conclude some constraints between these four frequencies; informally, the problem here is to determine for which frequencies $\alpha,\beta,\alpha',\beta'$ can one have a non-trivial relationship between $\{ \alpha n \} \beta n$ and $\{\alpha' n\} \beta'n$ (cf. \eqref{brackalg}).  Strictly speaking, for the analysis that we are about to give to apply, we must first replace the bracket polynomial expressions above by suitable vector-valued smoothings (or else develop analogues of the above equidistribution theory for piecewise Lipschitz nilsequences, as was done in the $d=2$ case in \cite{u4-inverse}), but to simplify the exposition we shall completely ignore this technical issue here.

Ignoring the technical issue alluded to above, we can express the left-hand side of \eqref{enn} in the form $|\E_{n \in [N]} F(\orbit(n))|$, where $G/\Gamma$ is the product Heisenberg nilmanifold of degree $\leq 2$, generated by four generators $e_1,e_2,e'_1,e'_2$ with $[e_1,e_2]$, $[e'_1,e'_2]$ central (and with $e_1, e_2$ commuting with $e'_1, e'_2$),
$$ \orbit(n) := e_2^{\beta n} e_1^{\alpha n} (e'_2)^{\beta' n} (e'_1)^{\alpha' n} \Gamma,$$
and $F$ is a (piecewise) Lipschitz function on $G/\Gamma$ obeying the vertical frequency property
\begin{equation}\label{fee}
 F( [e_1,e_2]^{t_{12}} [e'_1,e'_2]^{t'_{12}} x ) = e( - t_{12} - t'_{12} ) F(x)
 \end{equation}
for all $x \in G/\Gamma$ and $t_{12}, t'_{12} \in \R$.  Note that $\Horiz_1(G)$ is isomorphic to $\R^4$ (being generated by the projections of $e_1,e_2,e'_1,e'_2$ via $\pi_{\Horiz_1(G)}$), while $\Horiz_2(G)$ is trivial.  Applying Theorem \ref{ratt}, we conclude that
$$ |\int_{G_P / \Gamma_P} F(\eps x)\ d\mu(x)| \gg 1$$
for some bounded $\eps \in G$ and some rational subgroup $G_P$ of $G$, with the property that
\begin{equation}\label{pizza}
 \pi_{\Horiz_1(G)}(G_P) \geq \Xi_1^\perp
\end{equation}
where 
$$ \Xi_1 := \{ \xi \in \Z^4: \xi\cdot (\alpha,\beta,\alpha',\beta') = O(N^{-1}) \}.$$
If the vertical group $G_P \cap G_2$ contains any element $[e_1,e_2]^{t_{12}} [e'_1,e'_2]^{t'_{12}}$ with $-t_{12}-t'_{12} \neq 0$, then from \eqref{fee} we see that $\int_{G/P} F(\eps x)\ d\mu(x) = 0$, a contradiction.  We conclude that
\begin{equation}\label{gp2}
 G_P \cap G_2 \subset \langle [e_1,e_2] [e'_1,e'_2]^{-1} \rangle_\R.
\end{equation}
This gives us some information concerning the group $\Xi_1$, and hence on the frequencies $\alpha,\beta,\alpha',\beta'$.  Indeed, suppose that we are given two elements $(a,b,a',b')$ and $(c,d,c',d')$ in $\Xi^\perp$.  By \eqref{pizza}, we conclude that $G_P$ contains two elements $g, h$ such that
$$ g = e_1^a e_2^b (e'_1)^{a'} (e'_2)^{b'} \mod G_2 $$
and
$$ h = e_1^c e_2^d (e'_1)^{c'} (e'_2)^{d'} \mod G_2.$$
Since $g$ and $h$ lie in $G_P$, the commutator
$$ [g,h] = [e_1,e_2]^{ad-bc} [e'_1,e'_2]^{a'd'-b'c'}$$
must also lie in $G_P$.  Comparing this with \eqref{gp2} we obtain an algebraic constraint on $\Xi$ that prevents it from being too small, namely that
\begin{equation}\label{adbc}
 (ad-bc) + (a'd'-b'c') = 0
\end{equation}
whenever $(a,b,a',b'), (c,d,c',d') \in \Z^4$ are both orthogonal to $\Xi$; thus the symplectic form \eqref{adbc} must vanish when restricted to $\Xi^\perp$.

For instance, suppose that $(\alpha',\beta') = (\beta,\alpha)$, but that $\alpha, \beta$ are otherwise in general position (cf. \eqref{brackalg}).  Then $\Xi$ is generated by $(1,0,0,1)$ and $(0,1,1,0)$, so $\Xi^\perp$ is generated by $(1,0,0,-1)$ and $(0,1,-1,0)$, and one easily verifies the property.  It is in principle possible to work out what other quadruples $\alpha,\beta,\alpha',\beta'$ are permitted by Theorem \ref{ratt}, but we will not compute this here.

\section{Some basic properties of nilcharacters and symbols}\label{basic-sec}

In this appendix we establish some basic properties of nilcharacters and symbols; this material is broadly comparable to \cite[\S 3]{u4-inverse}.  

Throughout this appendix, $I$ is understood to be an ordering (see Definition \ref{order-def}).

We first begin with some basic closure properties of nilsequences.  

\begin{lemma}[Nilsequences are preserved by Lipschitz operations]\label{lip}  Let $H$ be an $I$-filtered group, let $J$ be a finite downset in $I$, and let $\Omega$ be a limit subset of $\ultra H$.  If $\psi_i \in \Nil^{\subset J}(\Omega \to \overline{\C}^{D_i})$ and $D_i\in \N^+$ for $i=1,\ldots,m$, and $F: \C^{D_1} \times \ldots \times  \C^{D_m} \to \C^D$ is a locally Lipschitz standard function, then $F(\psi_1,\ldots,\psi_m) \in \Nil^{\subset J}(\Omega \to \C^D)$.
\end{lemma}

\begin{proof} This follows immediately from Definition \ref{nilch-def-gen} and Example \ref{prodeq}.
\end{proof}

As an immediate corollary we have the following.

\begin{corollary}[Algebra property]\label{alg}  Let $H$ be an $I$-filtered group, let $J$ be a finite downset of $I$, and let $\Omega$ be a limit subset of $H$.  Then $\Nil^{\subset J}(\Omega \to \overline{\C})$ is a sub-$*$-algebra of $L^\infty(\Omega \to \overline{\C})$, that is to say it is closed under pointwise multiplication, scalar multiplication by bounded constants, addition, and complex conjugation.  Similarly, $\Nil^{\subset J}(\Omega \to \overline{\C}^\omega)$ is closed under complex conjugation, tensor product, and bounded linear combinations.
\end{corollary}

\emph{Remark.} From the example after Corollary \ref{chain} we also see that if $\psi \in L^\infty(\ultra H \to \overline{\C}^\omega)$ is a nilsequence of degree $\subset J$, then so is any translate $\psi(\cdot+h)$ or dilate $\psi(q\cdot)$ of $\psi$ for $h \in \ultra H$ and $q \in \Z$.

\begin{lemma}[Basic facts about nilcharacters]\label{baby-calculus}  Let $H = (H,+)$ be an $I$-filtered abelian group for some $I$, let $d \in I$, and let $\chi, \chi'$ be nilcharacters in $\Xi^{d}(\ultra H)$.
Then $\chi \otimes \chi'$, $\chi(\cdot+h)$, $\chi(q\cdot)$, and $\overline{\chi}$ are also nilcharacters of degree $\leq d$ for every $h \in \ultra H$, and $q \in \Z$.

More generally, if $T: H' \to H$ is a \textup{(}standard\textup{)} filtered homomorphism from another $I$-filtered abelian group $H' = (H',+)$ to $H$, then $\chi \circ T$ is a nilcharacter in $\Xi^d(\ultra H')$.

Finally, one has $\Xi^{d'}(\ultra H) \subset \Xi^d(\ultra H')$ whenever $d' < d$.
\end{lemma}

\begin{proof}  This follows from Corollary \ref{chain} (cf. the example after that corollary, and Corollary \ref{alg}).
\end{proof}

From \eqref{multipoly} it is trivial that a multidimensional polynomial of multidegree $\subset J \cup J'$ can be decomposed as the sum of a multidimensional polynomial of multidegree $\subset J$, and a multidimensional polynomial of multidegree $\subset J'$.  There is an analogous decomposition for nilcharacters.

\begin{lemma}[Splitting lemma]\label{approx}  Let $k \in \N^+$, and let $J, J'$ be finite downsets of $\N^k$.  Let $\psi \in \Nil^{\subset J \cup J'}(\ultra Z^k \to \overline{\C})$ be a nilsequence, and let $\eps > 0$ be standard.   Then
$$ \| \psi(n) - \sum_{k=1}^{K} \psi_k(n) \psi'_k(n) \|_{L^\infty(\ultra Z^k)} \leq \eps$$
where $K$ is standard and for each $1 \leq k \leq K$, $\psi_{k} \in \Nil^{\subset J}(\ultra \Z^k \to \overline{\C})$ and $\psi'_{k} \in \Nil^{\subset J'}(\ultra \Z^k \to \overline{\C})$.
\end{lemma}

\begin{proof} We can write $\psi = F \circ \orbit$, where $F \in \Lip(\ultra(G/\Gamma) \to \overline{\C})$, $\orbit\in \ultra \poly(\Z^k \to G/\Gamma)$, and $G/\Gamma$ is an $\Z^k$-filtered nilmanifold with degree $\subset J \cup J'$.  

For each $j \in J \cup J'$, let $e_{j,1},\ldots,e_{j,d_j}$ be a basis of generators for $\Gamma_{j}$.  We may then lift $G$ to the \emph{universal nilpotent Lie group} that is formally generated by the $e_{j,i}$, subject to the constraint that any $r-1$-fold iterated commutator of the $e_{j_1,i_1},\ldots,e_{j_r,i_r}$ with $j_1+\dots+j_r \not \in J \cup J'$ vanishes, and similarly lift $\Gamma$, $F$, $\orbit$ (using Corollary \ref{pull} for the latter).  Thus we may assume without loss of generality that $G$ is universal.

The degree $\subset J \cup J'$ nilmanifold $G/\Gamma$ projects down to the degree $\subset J$ nilmanifold $G/G_{>J}\Gamma$, where $G_{>J}$ is the group generated by the $G_j$ for all $j \in J' \backslash J$.  Similarly we have a projection from $G/\Gamma$ to the degree $\subset J'$ nilmanifold $G/G_{>J'}\Gamma$.  The algebras
$\Lip(\ultra(G/G_{>J}\Gamma) \to \overline{\C})$, $\Lip(\ultra(G/G_{>J'}\Gamma) \to \overline{\C})$ then pull back to subalgebras of $\Lip(\ultra(G/\Gamma) \to \overline{\C})$.  By universality of $G$, $G_{>J}$ and $G_{>J'}$ are disjoint.  Thus, the union of these two algebras separate points in $G/\Gamma$.  By the Stone-Weierstrass theorem, one can thus approximate $F$ to arbitrary accuracy by products of elements from these algebras, and the claim follows.
\end{proof}

Next, we show that nilsequences can be decomposed into nilcharacters.

\begin{lemma}[Fourier decomposition]\label{limone}  Let $H$ be an $I$-filtered group, and let $d \in I$. If $\psi \in \Nil^{\leq d}(\ultra H)$ and $\eps>0$ is standard, then one can find a standard natural number $m$, and nilcharacters $\chi_j \in \Xi^d(\ultra H)$, scalar nilsequences $\psi_j \in \Nil^{<d}(\ultra H)$, and bounded linear transformations $T: \overline{\C}^{D_j} \to \overline{\C}^D$ for suitable dimensions $D_j, D$ for each $1 \leq j \leq m$ such that
$$ \| \psi - \sum_{j=1}^m T_j( \psi_j  \otimes \chi_j) \|_{L^\infty(\ultra H)} \leq \eps.$$
\end{lemma}

\begin{proof}  It suffices to show this for scalar nilsequences $\psi$. Let $G/\Gamma$ be an $I$-filtered nilmanifold  of degree $\leq d$, let $F \in \Lip(\ultra(G/\Gamma) \to \overline{\C})$, and let  $\eps > 0$. We need  to show that one can approximate $F$ to uniform error at most $\eps$ by $\sum_{j=1}^m T_j(F_j \otimes f_j)$, where each $F_j \in \Lip(G/\Gamma \to S^{2D_j-1})$ has a vertical frequency, $f_j \in \Lip(G/\Gamma \to \C)$ is invariant with respect to the $G_{d}$ action (so that $f_j$ descends to the quotient nilmanifold $G/G_{d}\Gamma$, which has degree $<d$), and the $T_j: \C^{D_j} \to \C$ are linear transformations.

Observe that the class of functions of the form $\sum_{j=1}^m T_j(F_j \otimes f_j)$ form a complex algebra that are closed under conjugations.  Thus by the Stone-Weierstrass theorem, it suffices to show functions of the form $F \otimes f$, where $F \in \Lip(G/\Gamma \to S^{2D-1})$ has a vertical frequency and $f \in \Lip(G/\Gamma \to \C)$ and is invariant under $G_d$, separate points.  This is trivial for two points which descend to distinct points on $G/G_{d}\Gamma$, so it suffices to do so for two points on a common $G_{d}$ fibre.  For this, it suffices by the definition of vertical frequency to show that for each $g \in G_{d}$ with $g \not \in \Gamma_{d}$, there exists a function $F \in \Lip(G/\Gamma \to S^{2D-1})$ has a vertical frequency $\eta$ with $\eta(g) \not \in \Z$.

The existence of a character $\eta: G_{d} \to \R$ with $\eta(g) \not \in \Z$ is guaranteed by Pontryagin duality.  Fixing such an $\eta$, we now perform the same construction used at the start of \S \ref{nilcharacters} (i.e. smoothly partition the base space $G/G_{d}\Gamma$ into balls of small radius) to generate the desired function $F$.
\end{proof}

\begin{corollary}[Correlation]\label{limone-cor}  Let $H$ be an $I$-filtered group, let $d \in I$, and let $\Omega$ be a limit finite subset of $\ultra H$.  If $f \in L^\infty(\Omega)$ is $\leq d$-biased, then $f$ correlates with a nilcharacter in $\Xi^d(\Omega)$.
\end{corollary}

\begin{proof}  We assume inductively that the claim has already been proven for all smaller values of $d$.
We may assume that $f$ is scalar.
Applying Lemma \ref{limone} for $\eps$ small enough, we see that $f$ correlates with an expression of the form $\sum_{j=1}^m T_j( \psi_j \otimes \chi_j)$, and thus by the pigeonhole principle, $f$ correlates with one of the $\psi_j \otimes \chi_j$, and thus $f \overline{\chi_j}$ correlates with $\psi_j$.  We can express the downset $\{ i \in I: i < d \}$ as the finite union of downsets $\{ i \in I: i \leq d'\}$ for various $d' < d$.  Applying Lemma \ref{approx} repeatedly for sufficiently small $\eps$, we thus see that $f \overline{\chi_j}$ correlates with $\prod_{d' \leq d} \psi_{d'}$, where each $\psi_{d'}$ is a nilsequence of degree $\leq d'$.  Applying the inductive hypothesis repeatedly, we thus see that $f \overline{\chi_j}$ correlates with $\bigotimes_{d' < d} \chi_{d'}$ for some nilcharacters $\chi_{d'}$ of degree $\leq d'$, and so $f$ correlates with $\chi_j \otimes \bigotimes_{d' < d}\chi_{d'}$.  The claim now follows from Lemma \ref{baby-calculus}.
\end{proof}

We turn now to a discussion of the basic properties of symbols.  We begin by clearing up a small issue left over from \S \ref{nilcharacters}: that of proving that the notion of equivalence we introduced in Definition \ref{equiv-def} is indeed an equivalence relation. Recall that nilcharacters $\chi$ and $\chi'$ were said to be equivalent if $\chi \otimes \overline{\chi'}$ is a nilsequence of degree strictly less than $d$.

\begin{lemma}\label{equiv-lemma}
Equivalence of nilcharacters, thus defined, is an equivalence relation.
\end{lemma}

\begin{proof} The symmetry is obvious.  For transitivity, suppose that $\chi_1 \sim \chi_2$ and that $\chi_2 \sim \chi_3$. Then each component of 
\[ (\chi_1 \otimes \overline{\chi_2}) \otimes (\chi_2 \otimes \overline{\chi_3}) = \chi_1 \otimes (\overline{\chi_2} \otimes \chi_2) \otimes \overline{\chi_3}\] is a nilsequence of degree strictly less than $d$. However the trace of $\overline{\chi_2} \otimes \chi_2$ is $1$, and so $\chi_1 \otimes \overline{\chi_3}$ is a combination of the components of $\chi_1 \otimes (\overline{\chi_2} \otimes \chi_2) \otimes \overline{\chi_3}$. In particular, it is a nilsequence of degree strictly less than $d$. 

To show reflexivity, we must confirm that $\chi \otimes \overline{\chi}$ is a nilsequence of degree $\prec d$ for any nilcharacter $\chi \in \Xi^d(\Omega)$.  If we write $\chi(n) = F(g(n)\ultra \Gamma)$, where $F \in \Lip(\ultra(G/\Gamma) \to \overline{S^\omega})$ has a vertical frequency $\eta$, we have
$$ \chi \otimes \overline{\chi}(n) = (F \otimes \overline{F})(g(n)\ultra \Gamma).$$
Noting that $F \otimes \overline{F}$ is invariant with respect to the $G_{d}$ action, we may quotient out by this central group and represent $\chi \otimes \overline{\chi}$ using a nilmanifold of degree $ \prec d$.\end{proof}

%{\bf insert brief discussion of bracket quadratic identities here}

The space $\Symb^d(\Omega)$ has many nice properties.

\begin{lemma}[Symbol calculus]\label{symbolic} Let $H = (H,+)$ be an abelian $I$-filtered group, let $d \in I$, and let $\Omega$ be a limit subset of $\ultra H$.  
\begin{enumerate}
\item[(i)] If $\chi, \chi' \in \Xi^d(\Omega)$ and $\psi \in \Nil^{<d}(\Omega)$, and the components of $\chi$ are bounded linear combinations of those of $\chi' \otimes \psi$, then $\chi, \chi'$ are equivalent on $\Omega$ and thus $[\chi]_{\Symb^d(\Omega)} = [\chi']_{\Symb^d(\Omega)}$.
\item[(ii)] Conversely, if $\chi, \chi' \in \Xi^d(\Omega)$ are equivalent, then $\chi$ is a bounded linear combination of $\chi' \otimes \psi$ for some $\psi \in \Nil^{<d}(\Omega)$.
\item[(iii)] $\Symb^d(\Omega)$ is an abelian group with the group operation induced from tensor product.
\item[(iv)] If $\chi \in \Xi^d(\ultra H)$ and $h \in \ultra H_{i}$ for some $i > 0$, then $\chi$ and $\chi(\cdot+h)$ are equivalent on $\ultra H$ (and thus on $\Omega$ also).  Thus, $[\chi(\cdot+h)]_{\Symb^d(\Omega)} = [\chi]_{\Symb^d(\Omega)}$.  
\item[(v)] If $H = \Z^k$ with either the multidegree or degree filtration, $\chi \in \Xi^d(\ultra H)$ and $q \in \Z$, then $\chi^{\otimes q^{|d|}}$ and $\chi(q\cdot)$ are equivalent on $\ultra H$ \textup{(}and thus on $\Omega$ also\textup{)}, thus $[\chi(q\cdot)]_{\Symb^d(\Omega)} = q^{|d|} [\chi]_{\Symb^d(\Omega)}$. %Here we write $|(d_1,\ldots,d_k)| := d_1+\ldots+d_k$.
\item[(vi)] \textup{(Pullback)} If $T: \ultra \Z^k \to \ultra \Z^{k'}$ is a linear transformation, and $\chi$ is a nilcharacter of degree $d$ on $\ultra \Z^{k'}$, then $\chi \circ T$ is a nilcharacter of degree $d$ on $\Z^k$.  Moreover, if $\chi'$ is another nilcharacter of degree $d$ on $\ultra \Z^k$ with $[\chi]_{\Symb^d(\ultra \Z^{k'})} = [\chi']_{\Symb^d(\ultra \Z^{k'})}$, then $[\chi \circ T]_{\Symb^d(\ultra \Z^{k})} = [\chi' \circ T]_{\Symb^d(\ultra \Z^{k})}$.
\item[(vii)] \textup{(Divisibility)} If $H = \Z^k$ with either the multidegree or degree filtration, $d \neq 0$, $\chi \in \Xi^d(\Omega)$ and $q \in \N^+$, then there exists $\tilde \chi$ such that $[\chi]_{\Symb^d(\Omega)} = q [\tilde \chi]_{\Symb^d(\Omega)}$.
\end{enumerate}
\end{lemma}

%\emph{Remark.}  Readers are encouraged to keep the heuristic that a polynomial phase $n \mapsto e(\alpha_0 + \ldots + \alpha_d n^d)$ has a symbol that is essentially equal to $\alpha_d$.

\begin{proof}  The claim (i) follows from the same argument used to prove reflexivity in Lemma \ref{equiv-lemma}.  For (ii), we proceed much as in the proof of transitivity in Lemma \ref{equiv-lemma}: write $\phi := \chi \otimes \overline{\chi'}$ and consider $\chi' \otimes\phi = (\chi' \otimes \overline{\chi'}) \otimes \chi$. Since $1$ may be written as a linear combination of the components of $\chi' \otimes \overline{\chi'}$, the claim follows.

The claim (iii) follows easily from (i) and (ii). Part (iv) is more substantial. It should be compared to some of the consequences of the ``bracket quadratic identities'' developed in \cite[Lemma 5.5]{u4-inverse}.

From Definition \ref{equiv-def}, it suffices to show that the derivative $\Delta_h \chi(n) := \chi(n+h) \otimes \overline{\chi(n)}$ lies in $\Nil^{<d}(\ultra H)$.  We write $\chi(n) = F( g(n) \ultra \Gamma)$, where $G/\Gamma$ is an $I$-filtered nilmanifold of degree $\leq d$, $g \in \ultra \poly(H_I \to G_I)$, and $F \in \Lip(\ultra(G/\Gamma))$ has a vertical frequency $\eta: G_{d} \to \R$, then we have
$$ \Delta_h \chi(n) = F( (\partial_h g(n)) g(n) \ultra \Gamma ) \otimes \overline{F}(g(n) \ultra \Gamma).$$
As $g \in \ultra \poly(H \to G)$ and $h \in \ultra H_{i}$, we have $\partial_h g \in \ultra \poly(H_I \to G^{+i}_I)$, where $G^{+i}_I = (G_{j+i})_{j \in I}$ is the shifted filtration.

We now give $G^2$ an $I$-filtration by defining $(G^2)_{j}$ to be the group generated by $G_{j+i} \times \id$ and by the diagonal group $\{ (g,g): g \in G_{j} \}$.  One easily verifies that this is a filtration, which is rational with respect to $\Gamma^2$.  
In particular, if we set $G^\Box := (G^2)_{0}$ and $\Gamma^\Box := \Gamma^2 \cap G^\Box$, we have that $G^\Box/\Gamma^\Box$ is an $I$-filtered nilmanifold of degree $\leq d$.  Furthermore, from Corollary \ref{laz} we see that the map
$$ \orbit: n \mapsto (\partial_h g(n) g(n), g(n)) \ultra \Gamma^\Box$$
lies in $\ultra \poly(H_I \to G^\Box_I/\Gamma^\Box_I)$.  We can thus wrote $\Delta_h \chi = \tilde F \circ \orbit$, where $\tilde F \in \Lip(\ultra(G^\Box/\Gamma^\Box))$ is the function
$$ \tilde F( x, y ) := F(x) \otimes \overline{F}(y).$$
This is still a degree $\leq d$ representation.  But observe from the vertical character nature of $F$ that $\tilde F$ is invariant with respect to the action of the group $G^\Box_{d} = \{ (g,g): g \in G_{d}\}$.  Thus we may quotient by this map and descend to a degree $<d$ nilmanifold, and the claim follows.

Now we turn to (v), which is a similar claim to (iv).  Writing $\chi = F(g(n) \ultra \Gamma)$ as before, we reduce to showing that
\begin{equation}\label{nagle}
 n \mapsto F( g(qn) \ultra \Gamma ) \otimes \overline{F}(g(n) \ultra \Gamma)^{\otimes q^{|d|}}
 \end{equation}
can be represented as a nilsequence of degree $<d$, with the convention that $F^{\otimes -q} = \overline{F}^{\otimes q}$ to deal with the case of negative exponents.

We give $G^2$ a $\N^k$-filtration by declaring $G^2_{i}$ to be the group generated by $G_{j} \times G_{j}$ for all $j>i$, together with the set $\{ (g^{q^|i|}, g): g \in G_{i} \}$.  From the Baker-Campbell-Hausdorff formula one easily sees that this is a filtration, which is rational with respect to $\Gamma^2$; and so $G^2_{0}/\Gamma^2_{0}$ is a degree $\leq d$ nilmanifold.  Also, from Taylor expansion (Lemma \ref{taylo}) and Corollary \ref{laz} we see that the map
$$ \orbit: n \mapsto (g(qn), g(n)) \ultra \Gamma^2_{0}$$
lies in $\ultra \poly(H \to G^2_{0}/\Gamma^2_{0})$.  We then write \eqref{nagle} as $n \mapsto \tilde F(\orbit(n))$, where $\tilde F \in \Lip(\ultra(G^2_{0}/\Gamma^2_{0}))$ is the function
$$ \tilde F(x,y) := F(x) \otimes \overline{F}^{\otimes q^{|d|}}.$$
From the vertical character nature of $F$, we see that $\tilde F$ is invariant with the action of $G^2_{2} = \{ (g^{q^{|d|}},g): g \in G_{d}\}$.  Quotienting out by this group as in the proof of (iv) we obtain the claim.

The claim (vi) follows easily from Corollary \ref{chain}, so we now turn to (vii).  We will prove this for the multidegree filtration, as the degree filtration is similar.  As usual we write $\chi = F(g(n) \ultra \Gamma)$.  Applying Taylor expansion (Lemma \ref{taylo}) and the Baker-Campbell-Hausdorff formula, we may factorise
$$ g(n) = \prod_{j \leq d} g_j^{n^j}$$
for some $g_j \in G_j$, where the product is over all multiindices $j \leq d$ (arranged in some arbitrary fashion).  Taking roots of each of the $g_j$, we may write $g_j = (g'_j)^{q^{|j|}}$ for each $j$.  We then have $g(n) = g'(qn)$, where $g'$ is the polynomial sequence
$$ g'(n) := \prod_{j \leq d} (g'_j)^{n^j}.$$
If we write $\chi'(n) := F(g'(n) \ultra \Gamma)$, we thus see that $\chi' \in \Xi^d(\Omega)$ and $\chi(n) = \chi'(qn)$, so by (v), 
$[\chi]_{\Symb^d(\Omega)} = q^{|d|} [\chi']_{\Symb^d(\Omega)}$.  The claim now follows by setting $\tilde \chi := (\chi')^{\otimes q^{|d|-1}}$.
\end{proof}

If $P(n) = \alpha_0 + \ldots + \alpha_d n^d$ is a polynomial of one variable $n$ of degree $d$, then $P$ is equal (up to degree $<d$ errors) to the multilinear form $Q(n,\ldots,n)$, where $Q(n_1,\ldots,n_d) := \alpha_d n_1 \ldots n_d$.  A bit more generally, if $P(n_1,\ldots,n_k)$ is a polynomial of $k$ variables $n_1,\ldots,n_k$ of multidegree $d = (d_1,\ldots,d_k)$, then $P$ is equal (up to degree $<d$ errors) to a degree $(1,\ldots,1)$ form $Q(n_1,\ldots,n_1,\ldots,n_k,\ldots,n_k)$, where $1$ is repeated $|d|$ times and each $n_i$ is repeated $d_i$ terms.  We may generalise this observaton to nilcharacters.  We begin with the simpler $k=1$ case.

\begin{proposition}[Multilinearisation, $k=1$ case]\label{multilinearisation-1}  Let $d \in \N$ and $\chi \in \Xi^d(\ultra \Z)$.  Then there exists $\tilde \chi \in \Xi^{(1,\ldots,1)}(\ultra \Z^d)$ \textup{(}where $1$ is repeated $d$ times\textup{)} such that the nilcharacter
$$ \chi': n \mapsto \tilde \chi( n, \ldots, n )$$
\textup{(}where $n$ is repeated $d$ times\textup{)} is equivalent to $\chi$ in $\Xi^d(\ultra \Z)$ \textup{(}thus $[\chi]_{\Xi^d(\ultra \Z)} = [\chi']_{\Xi^d(\ultra \Z)}$\textup{)}.  Furthermore, one can select $\tilde \chi(n_1,\ldots,n_d)$ to be symmetric with respect to permutations of $n_1,\ldots,n_d$.
\end{proposition}

To motivate this proposition, we present an ``almost-example'' of this proposition in action: if $d=2$ and $\chi$ is the degree $2$ almost-nilcharacter 
$$ \chi(n) := e( \{ \alpha n \} \beta n ),$$
(where the ``almost'' arises because the relevant function $F$ is only piecewise Lipschitz rather than Lipschitz, as discussed at the start of \S \ref{nilcharacters}) then one can take
\begin{equation}\label{tchan}
 \tilde \chi(n_1,n_2) := e( \frac{1}{2} \{ \alpha n_1 \} \beta n_2 + \frac{1}{2} \{ \alpha n_2 \} \beta n_1 )
\end{equation}
which is a multidegree $(1,1)$ almost-nilcharacter, with $\tilde \chi(n,n)$ equivalent (and in fact exactly equal, in this case) to $\chi(n)$.  More generally, if we are able to represent a nilcharacter in terms of bracket polynomials of the correct degree and rank, then the above proposition becomes obvious by inspection.  Such a representation is in fact possible (by extending the theory in \cite{leibman}), but we will proceed here instead by using abstract algebraic constructions.

\begin{proof}  This will be a more complicated version of the argument used to establish claims (iv), (v) of Lemma \ref{symbolic}.   It will be convenient for technical reasons to construct $\tilde \chi$ so that $\chi'$ is equivalent to $\chi^{\otimes d!}$ rather than to $\chi$ itself; to recover the original claim in the proposition, one simply appeals to Lemma \ref{symbolic}(vii).

We have $\chi(n) = F( g(n) \ultra \Gamma )$ for some degree $d$ nilmanifold $G/\Gamma$, some polynomial sequence $g \in \ultra \poly( \Z_\N \to (G/\Gamma)_\N )$, and some $F \in \Lip( \ultra(G/\Gamma) \to \overline{S^\omega} )$, obeying the vertical frequency property
$$ F( g_d x ) = e(\eta(g_d)) F(x)$$
for all $x \in G/\Gamma$ and $g_d \in G_d$, where $\eta: G_d \to \R$ is a continuous homomorphism that maps $\Gamma_d$ to the integers.

We now build the various components $\tilde G, \tilde \eta, \tilde g, \tilde F$ required to construct $\tilde \chi$.  (A simple example of this construction will be given after the end of this proof.)

The first step is build the multidegree $(1,\ldots,1)$ nilpotent group $\tilde G$.  We will construct this group via its nilpotent Lie algebra $\log \tilde G$.  As a (real) vector space, this Lie algebra will be given as a direct sum
$$ \log \tilde G := \oplus_{J \subset \{1,\ldots,d\}} \log G_{|J|}.$$
For each $J \subset \{1,\ldots,d\}$, let $\iota_J: \log G_{|J|} \to \log \tilde G$ be the vector space embedding indicated by this direct sum, thus every element of $\log \tilde G$ can be uniquely expressed in the form $\sum_{J \subset \{1,\ldots,d\}} \iota_J( x_J )$ for some $x_J \in \log G_{|J|}$.

Next, we endow $\log \tilde G$ with a Lie bracket structure by declaring
$$ [\iota_J(x_J), \iota_K( y_K ) ] = 0$$
whenever $J, K \subset \{1,\ldots,d\}$ intersect and $x_J \in \log G_{|J|}$, $y_K \in \log G_{|K|}$, and
$$ [\iota_J(x_J), \iota_K( y_K ) ] = \iota_{J \cup K}( [x_J, y_K ] )$$
whenever $J, K \subset \{1,\ldots,d\}$ are disjoint and $x_J \in \log G_{|J|}$, $y_K \in \log G_{|K|}$.  One easily verifies that this operation obeys the axioms of a Lie bracket (i.e. it is bilinear, antisymmetric, and obeys the Jacobi identity), and so $\log \tilde G$ is a Lie algebra.

We now give $\log \tilde G$ a multidegree filtration.  For any $(a_1,\ldots,a_d) \in \N^d$, let $\log \tilde G_{(a_1,\ldots,a_d)}$ be the sub-Lie-algebra of $\log \tilde G$ generated by the $\iota_J(x_J)$ for which $1_J(j) \geq a_j$ for each $j=1,\ldots,d$, and $x_J \in G_{|J|}$.  One easily verifies that this is a multidegree filtration of multidegree $(1,\ldots,1)$, and so one can exponentiate to create a multidegree-filtered Lie group $\tilde G$ of multidegree $(1,\ldots,1)$ also.

We define a lattice $\tilde \Gamma$ in $\tilde G$ to be the group generated by $\exp( M! \iota_J( \log \gamma_j ) )$ for all $J \subset \{1,\ldots,d\}$ and $\gamma_j \in \Gamma_{|J|}$, where $M$ is a fixed natural number (depending only on $d$) which we will assume to be sufficiently large.  From the Baker-Campbell-Hausdorff formula we see that this is indeed a lattice, and so $\tilde G/\tilde \Gamma$ is a nilmanifold.  For $M$ large enough, we see from further application of the Baker-Campbell-Hausdorff formula that $\tilde \Gamma_{(1,\ldots,1)}$ is contained in $\iota_{(1,\ldots,1)}( \log \Gamma_d )$.

Next, we define a vertical frequency $\tilde \eta$ on $\tilde G_{(1,\ldots,1)}$ by setting
$$ \tilde \eta( \iota_{(1,\ldots,1)}( \log g_d ) ) := \eta( g_d ).$$
One easily verifies that $\tilde \eta$ is a vertical frequency (here we use the inclusion $\tilde \Gamma_{(1,\ldots,1)} \subset \iota_{(1,\ldots,1)}( \log \Gamma_d )$ and the central nature of $G_{(1,\ldots,1)}$).  

Now let $\tilde F \in \Lip( \ultra(\tilde G/\tilde \Gamma) \to \overline{S^\omega} )$ be a function with vertical frequency $\tilde \eta$; such a function can be constructed using partitions of unity as in \eqref{fkts}.  

The next step is to define $\tilde g$.  From Lemma \ref{taylo} and many applications of the Baker-Campbell-Hausdorff formula, we may write
$$ g(n) = \prod_{j=0}^d g_j^{n^j}$$
for some coefficients $g_j \in G_j$.  We then write
$$ \tilde g(n_1,\ldots,n_d) := \prod_{j=0}^d \exp( j! \sum_{J \subset \{1,\ldots,d\}: |J| = j} (\prod_{i \in J} n_i ) \iota_J( \log g_j ) ).$$
Observe that each individual monomial
$$ (n_1,\ldots,n_d) \mapsto \exp( j! (\prod_{i \in J} n_i ) \iota_J( \log g_j ) )$$
with $0 \leq j \leq d$ and $|J|=j$ is a polynomial map in $\ultra \poly( \Z^d_{\N^d} \to \tilde G_{\N^d})$, so by Corollary \ref{laz} and the Baker-Campbell-Hausdorff formula we see that the same is true for $\tilde g$.

Finally, we set
$$ \tilde \chi(n_1,\ldots,n_d) := \tilde F( \tilde g( n_1,\ldots,n_d) \ultra \tilde \Gamma ).$$
By construction, $\tilde \chi \in \Xi^{(1,\ldots,1)}(\ultra \Z^d)$, which by Lemma \ref{symbolic}(vi) (and the embeddings in Example \ref{inclusions}) implies that $\chi' \in \Xi^d(\ultra \Z)$.  It is also clear that $\tilde \chi$ is symmetric with respect to permutations of the $n_1,\ldots,n_d$.  It remains to show that $\chi'$ is equivalent to $\chi^{\otimes d!}$ in $\Xi^d(\ultra \Z)$, or in other words that the sequence
$$ n \mapsto \chi(n)^{\otimes d!} \otimes \overline{\tilde \chi}(n,\ldots,n)$$
is a nilsequence of degree $<d$.  We expand this sequence as
$$ (F^{\otimes d!} \otimes \overline{\tilde F})\left( \prod_{j=0}^d (g_j, \exp( j! \sum_{J \subset \{1,\ldots,d\}: |J| = j} \iota_J( \log g_j ) ))^{n^j}  \ultra(\Gamma \times \tilde \Gamma) \right).$$
The function $F \otimes \overline{\tilde F}$ is a Lipschitz function on the nilmanifold $(G \times \tilde G)/(\Gamma \times \tilde \Gamma)$.  Let $G^*$ be the subgroup of $G \times \tilde G$ defined as
$$ G^* := \{ (g_d,  \exp(d! \iota_{(1,\ldots,1)}(\log g_d) ): g_d \in G_d \} \leq G_d \times \tilde G_{(1,\ldots,1)}.$$
This is a rational central subgroup.  As $F$ and $\tilde F$ have vertical frequencies $\eta$ and $\tilde \eta$ respectively, we see that $F \otimes \overline{\tilde F}$ is invariant in the $G^*$ direction, and thus descends to a Lipschitz function $F'$ on the nilmanifold $G'/\Gamma'$, where $G' := (G \times \tilde G)/G^*$ and $\Gamma'$ is the projection of $\Gamma \times \tilde \Gamma$ to $G'$.  We thus have
\begin{equation}\label{cheese}
 \chi(n) \otimes \overline{\tilde \chi}(n,\ldots,n) = F'( \prod_{j=1}^d (g'_j)^{n^j} \ultra \Gamma' )
\end{equation}
where $g'_j$ is the projection of $(g_j, \exp( j! \sum_{J \subset \{1,\ldots,d\}: |J| = j} \iota_J( \log g_j ) ) )$ to $G'$.

We now give $G'$ a degree filtration by defining $G'_j$ to be the group generated by elements of the form
$$(h_j, \exp( j! \sum_{J \subset \{1,\ldots,d\}: |J| = j} \iota_J( \log h_j ) ) ) \mod G^*$$
for $h_j \in G_j$, together with elements of the form
$$ (h_{j+1}, \id), \left(\id, \exp( \iota_J( \log h_{j+1} ) ) \right) \mod G^*$$
for $h_{j+1} \in G_{j+1}$ and $J \subset \{1,\ldots,d\}$ with $|J| = j+1$.  By a tedious number of applications of the Baker-Campbell-Hausdorff formula, we see that this is a filtration of degree $<d$ (here we use the fact that every set of cardinality $j+k$ has $\frac{(j+k)!}{j! k!}$ partitions into a set $J$ of cardinality $j$ and a set $K$ of cardinality $k$, which cancels the $j!$ prefactors appearing in the definition of $G'_j$).  By construction, $g'_j \in G'_j$.  Thus the right-hand side of \eqref{cheese} is a nilsequence of degree $<d$, and the claim follows.
\end{proof}

\emph{Example.}  We illustrate the above proposition with the simple $d=2$ example mentioned before the proof.  We consider a nilcharacter $\chi$ that is a vector-valued smoothing of the sequence $n \mapsto e(\{\alpha n \} \beta n)$ for some fixed frequencies $\alpha,\beta \in \ultra \T$, which we will write schematically as
$$ \chi(n) \sim e(\{\alpha n \} \beta n).$$
As discussed in \S \ref{nilcharacters}, such a nilcharacter arises from the Heisenberg nilmanifold \eqref{heisen} with the polynomial sequence
$$ g(n) = e_2^{\beta n} e_1^{\alpha n}$$
and vertical character $\eta( [e_1,e_2]^{t_{12}} ) := -t_{12}$.  We may Taylor expand $g$ as
$$ g(n) = g_1^n g_2^{n^2}$$
where $g_1 := \exp( \alpha \log e_1 + \beta \log e_2 ) = e_1^\alpha e_2^\beta [e_1,e_2]^{-\alpha\beta/2}$ and $g_2 := [e_1,e_2]^{-\alpha\beta/2}$.

The nilpotent Lie algebra $\log \tilde G$ is the seven-dimensional vector space
$$ \log \tilde G = \log G \oplus \log G \oplus \log G_{12}$$
with a basis of this space given by
\begin{equation}\label{basis}
 \iota_1( \log e_1 ), \iota_1( \log e_2 ), \iota_1( \log [e_1,e_2] ), \iota_2( \log e_1 ), \iota_2( \log e_2 ), \iota_2( \log [e_1,e_2] ), \iota_{12}( \log [e_1,e_2] ).
\end{equation}
The Lie algebra commutation relations on basis elements are given by the formulae
\begin{align*}
[\iota_1(\log e_1), \iota_2(\log e_2)] &= \iota_{12}(\log [ e_1, e_2 ]) \\
[\iota_1(\log e_2), \iota_2(\log e_1)] &= -\iota_{12}(\log [ e_1, e_2 ])
\end{align*}
with all other pairs of basis elements commuting.  This gives a nilpotent Lie group $\tilde G$ generated (as a Lie group) by the exponentials of \eqref{basis}, which we will label as
\begin{align*}
a_1 &:= \exp( \iota_1( \log e_1 ) ) \\
a_2 &:= \exp( \iota_1( \log e_2 ) ) \\
a_{12} &:= \exp( \iota_1( \log e_{12} ) ) \\
b_1 &:= \exp( \iota_2( \log e_1 ) ) \\
b_2 &:= \exp( \iota_2( \log e_2 ) ) \\
b_{12} &:= \exp( \iota_2( \log e_{12} ) ) \\
c_{12} &:= \exp( \iota_{12}( \log e_{12} ) ),
\end{align*}
thus one has the group commutation relations
$$ [a_1,b_2] = c_{12}; \quad [a_2,b_1] = c_{12}^{-1}$$
with all other pairs of generators commuting.  The generators $a_{12},b_{12}$ will play no essential role in the analysis that follows and may be ignored by the reader.

The group $\tilde G$ is a multidegree $(1,1)$ filtered Lie group with filtration 
\begin{align*}
\tilde G_{(0,0)} &:= \tilde G;\\
\tilde G_{(1,0)} &:= \langle a_1, a_2, a_{12}, c_{12} \rangle_\R;\\
\tilde G_{(0,1)} &:= \langle b_1, b_2, b_{12}, c_{12} \rangle_\R;\\
\tilde G_{(1,1)} &:= \langle c_{12} \rangle_\R.
\end{align*}
To construct $\tilde \Gamma$, we may take $M=1$, so that
$$
\tilde \Gamma := \langle a_1,a_2,a_{12},b_1,b_2,b_{12},c_{12} \rangle.$$
From the Baker-Campbell-Hausdorff formula one sees that
$$ \tilde \Gamma_{(1,1)} := \tilde \Gamma \cap \tilde G_{(1,1)} = \langle c_{12} \rangle.$$
A typical element of $\tilde G/\tilde \Gamma$ can be parameterised as
$$ a_1^{r_1} a_2^{r_2} a_{12}^{r_{12}} b_1^{s_1} b_2^{s_2} b_{12}^{s_{12}} c_{12}^{t_{12}} \tilde \Gamma$$
for $r_1,r_2,r_{12},s_1,s_2,s_{12},t_{12} \in I_0$.

The polynomial sequence $\tilde g$ is given as
\begin{align*}
 \tilde g(n_1,n_2) &:= \exp( n_1 \iota_1( \log g_1 ) + n_2 \iota_2( \log g_1 ) ) \exp( 2 n_1 n_2 \iota_{12}( \log g_2 ) )\\
&= \exp( \alpha n_1 \log a_1 + \beta n_1 \log a_2 +  \alpha n_2 \log b_1 + \beta n_2 \log b_2 )  \times \\ & \qquad\qquad \times \exp( -\alpha\beta n_1 n_2 \log c_{12} )
\end{align*}  
which by the Baker-Campbell-Hausdorff formula expands to
$$ \tilde g(n_1,n_2) = a_1^{\alpha n_1} a_2^{\beta n_1} b_1^{\alpha n_2} b_2^{\beta n_2} c_{12}^{-\alpha\beta n_1 n_2}.$$
This is clearly a polynomial sequence.  If we then let $\tilde \eta: \tilde G_{(1,1)} \to \R$ be the vertical character
$$ \tilde \eta( \exp( t_{12} \iota_{12}( \log [e_1,e_2] ) ) ) := -t_{12}$$
and let $\tilde F: \tilde G/\tilde \Gamma \to S^1$ be the (piecewise) Lipschitz function
$$ \tilde F( a_1^{r_1} a_2^{r_2} a_{12}^{r_{12}} b_1^{s_1} b_2^{s_2} b_{12}^{s_{12}} c_{12}^{t_{12}} \tilde \Gamma ) := e(-t_{12})$$
for $r_1,r_2,r_{12},s_1,s_2,s_{12},t_{12} \in I_0$, then the sequence
$$ \tilde \chi(n_1,n_2) := \tilde F( \tilde g(n_1,n_2) \ultra \tilde \Gamma)$$
is almost a nilcharacter of multidegree $(1,1)$, if we make the usual cheat of ignoring the fact that $\tilde F$ is only piecewise Lipschitz rather than Lipschitz.

Now let us look at the diagonal sequence
$$ \tilde \chi(n,n) = \tilde F( a_1^{\alpha n} a_2^{\beta n} b_1^{\alpha n} b_2^{\beta n} c_{12}^{-\alpha \beta n^2} \ultra \tilde \Gamma).$$
A brief computation using the Baker-Campbell-Hausdorff formula shows that one can rewrite
$$ a_1^{\alpha n} a_2^{\beta n} b_1^{\alpha n} b_2^{\beta n} c_{12}^{-\alpha \beta n^2} \ultra \tilde \Gamma $$
as
$$ a_1^{\{\alpha n\}} a_2^{\{\beta n\}} b_1^{\{\alpha n\}} b_2^{\{\beta n\}} c_{12}^{(\alpha n-\{\alpha n\}) \{ \beta n\} - (\beta n-\{\beta n\})\{ \alpha n\} -\alpha \beta n^2} \ultra \tilde \Gamma.$$
Noting that $(\alpha n - \{\alpha n\})(\beta n - \{\beta n\})$ is an integer (cf. \eqref{brackalg}), we can write the $c_{12}$ exponent modulo $1$ as
$$ \{ \alpha n \} \{ \beta n \} - 2 \{ \alpha n \} \beta n \mod 1$$
and thus
$$ \tilde \chi(n,n) = e( 2 \{ \alpha n \} \beta n ) e( -\{ \alpha n \} \{ \beta n \} ).$$
The second factor $e(- \{ \alpha n \} \{ \beta n \} )$ is a piecewise Lipschitz function of $(\alpha n \mod 1$, $\beta n \mod 1)$ and is thus almost a $1$-step nilsequence.  We thus see that $\tilde\chi(n,n)$ is almost equivalent (as a degree $2$ almost nilcharacter) to $\chi(n)^2$.  To eliminate the exponent of $2$, one can go back to the start of the argument and replace $\beta$ (for instance) by $\beta/2$.  The reader may verify that once one does so, the almost nilcharacter $\tilde \chi$ is essentially equal to \eqref{tchan}.

Finally, we mention that with the above example, the group $G^*$ takes the form
$$ G^* := \{ ( [e_1,e_2]^{t_{12}}, c_{12}^{2t_{12}}): t_{12} \in \R \}$$
and the group $G' := (G \times \tilde G)/G^*$ has the degree $1$ filtration
\begin{align*}
G'_0 &:= G' \\
G'_1 &:= \{ ( e_1^{t_1} e_2^{t_2} [e_1,e_2]^{t_{12}}, a_1^{t_1} a_2^{t_2} a_{12}^{t'_{12}} b_1^{t_1} b_2^{t_2} b_{12}^{t'_{12}} c_{12}^{t''_{12}}: t_1, t_2, t_{12}, t'_{12}, t''_{12} \in \R  \} \mod G^*.
\end{align*}
One can verify by hand that this is indeed a degree $1$ filtration on $G'$, which explains why $\chi(n)^2 \overline{\tilde \chi}(n,n)$ is a degree $1$ (almost) nilsequence.\vspace{11pt}

This concludes the discussion of the example. Now we generalise Proposition \ref{multilinearisation} to higher $k$.

\begin{theorem}[Multilinearisation]\label{multilinearisation}   Let $\Omega$ be a limit subset of $\Z^k$, which we give the multidegree filtration.  Let $d = (d_1,\ldots,d_k) \in \N^k$, and $\chi \in \Xi^d(\Omega)$.  Then there exists $\tilde \chi \in \Xi^{(1,\ldots,1)}(\ultra \Z^{|d|})$ \textup{(}where $1$ is repeated $|d|$ times\textup{)} such that the nilcharacter
$$ \chi': (n_1,\ldots,n_k) \mapsto \tilde \chi( n_1, \ldots, n_1, n_2, \ldots, n_2, \ldots, n_k, \ldots, n_k )$$
\textup{(}where each $n_i$ is repeated $d_i$ times\textup{)} is equivalent to $\chi$ in $\Xi^d(\Omega)$ (thus $[\chi]_{\Xi^d(\Omega)} = [\chi']_{\Xi^d(\Omega)}$).  Furthermore, one can select 
$$\tilde \chi(n_{1,1},\ldots,n_{1,d_1}, n_{2,1},\ldots,n_{2,d_2},\ldots,n_{k,1},\ldots,n_{k,d_k})$$
to be symmetric with respect to the permutation of $n_{i,1},\ldots,n_{i,d_i}$ for each $i=1,\ldots,k$.
\end{theorem}

\begin{proof}  Without loss of generality we may take $\Omega = \ultra \Z^k$.  The argument is exactly the same as that used to establish Proposition \ref{multilinearisation-1} except that the notation is more complicated.  Accordingly, we will focus primarily on the notational setup in this proof.

As before, it will suffice to make $\chi'$ equivalent to $\chi^{\otimes d!}$ rather than $\chi$, where $d! := d_1! \ldots d_k!$.
We have $\chi(n) = F( g(n) \ultra \Gamma )$ for some multidegree $d$ nilmanifold $G/\Gamma$, some polynomial sequence $g \in \ultra \poly( \Z^k_{\N^k} \to (G/\Gamma)_{\N^k} )$, and some $F \in \Lip( \ultra(G/\Gamma) \to \overline{S^\omega} )$, obeying the vertical frequency property
$$ F( g_d x ) = e(\eta(g_s)) F(x)$$
for all $x \in G/\Gamma$ and $g_d \in G_d$, where $\eta: G_d \to \R$ is a vertical frequency.

As before, we begin by bulding the nilpotent Lie algebra $\log \tilde G$.  As a (real) vector space, this Lie algebra will be given as a direct sum
$$ \log \tilde G := \oplus_{J \subset \{1,\ldots,|d|\}} \log G_{\| J\|}$$
where $\|J\| \in \N^k$ is the vector
$$ \|J\| := ( |J \cap \{ d_1+\ldots+d_{i-1}+1,\ldots,d_1+\ldots+d_i\}| )_{1 \leq i \leq k}.$$
For each $J \subset \{1,\ldots,d\}$, let $\iota_J: \log G_{\|J\|} \to \log \tilde G$ be the vector space embedding indicated by this direct sum.
Next, we endow $\log \tilde G$ with a Lie bracket structure by declaring
$$ [\iota_J(x_J), \iota_K( y_K ) ] = 0$$
whenever $J, K \subset \{1,\ldots,d\}$ intersect and $x_J \in \log G_{\|J\|}$, $y_K \in \log G_{\|K\|}$, and
$$ [\iota_J(x_J), \iota_K( y_K ) ] = \iota_{J \cup K}( [x_J, y_K ] )$$
whenever $J, K \subset \{1,\ldots,d\}$ are disjoint and $x_J \in \log G_{\|J\|}$, $y_K \in \log G_{\|K\|}$.  As before, one easily verifies the Lie bracket axioms.

We now give $\log \tilde G$ a multidegree filtration.  For any $(a_1,\ldots,a_{|d|}) \in \N^{|d|}$, let $\log \tilde G_{(a_1,\ldots,a_{|d|})}$ be the sub-Lie-algebra of $\log \tilde G$ generated by the $\iota_J(x_J)$ for which $1_J(j) \geq a_j$ for each $j=1,\ldots,|d|$, and $x_J \in G_{\|J\|}$.  As before, this is a multidegree filtration of multidegree $(1,\ldots,1)$, and exponentiates to create a multidegree-filtered Lie group $\tilde G$ of multidegree $(1,\ldots,1)$ also.

We define a lattice $\tilde \Gamma$ in $\tilde G$ to be the group generated by $\exp( M! \iota_J( \log \gamma_j ) )$ for all $J \subset \{1,\ldots,|d|\}$ and $\gamma_j \in \Gamma_{\|J\|}$.  Again, this creates a nilmanifold $\tilde G/\tilde \Gamma$ is a nilmanifold, and for $M$ large enough, $\tilde \Gamma_{(1,\ldots,1)}$ is contained in $\iota_{(1,\ldots,1)}( \log \Gamma_d )$.

As before, we define a vertical frequency $\tilde \eta$ on $\tilde G_{(1,\ldots,1)}$ by the exact same formula:
$$ \tilde \eta( \iota_{(1,\ldots,1)}( \log g_d ) ) := \eta( g_d ),$$
and then construct $\tilde F \in \Lip( \ultra(\tilde G/\tilde \Gamma) \to \overline{S^\omega} )$ with vertical frequency $\tilde \eta$.

The next step is to define $\tilde g$.  As before, we have the Taylor expansion
$$ g(n) = \prod_{j \leq d} g_j^{n^j}$$
for some coefficients $g_j \in G_j$, where $j = (j_1,\ldots,j_k)$ now ranges over multi-indices less than or equal to $d$, arranged in some arbitrary order (e.g. lexicographical will suffice).  We then write
$$ \tilde g(n_1,\ldots,n_{|d|}) := \prod_{j \leq d} \exp( j! \sum_{J \subset \{1,\ldots,|d|\}: \|J\| = j} (\prod_{i \in J} n_i ) \iota_J( \log g_j ) ),$$
recalling that $j! := j_1! \ldots j_k!$.  As before, one verifies that $\tilde g$ is a polynomial map.

Finally, we set
$$ \tilde \chi(n_1,\ldots,n_{|d|}) := \tilde F( \tilde g( n_1,\ldots,n_{|d|}) \ultra \tilde \Gamma ).$$
The rest of the argument proceeds exactly as in Proposition \ref{multilinearisation-1}, the main difference being that $d$ is replaced with $|d|$, and $|J|$ with $\|J\|$, whenever necessary; we omit the details.
\end{proof}

Now we show how nilcharacters interact with the concept of bias.  

\begin{lemma}[Bias lemma]\label{bias}  Let $k,d \in \N^+$ with $d \geq 2$, let $\chi$ be a degree $d$ nilcharacter on $\ultra \Z^k$ \textup{(}with the degree filtration\textup{)}, and let $N$ be an unbounded limit natural number.  Let $\Omega$ be a convex polytope in $[[N]]^k$, let $P_1,\ldots,P_k$ be dense subprogressions of $[N]$.  Suppose that $1_\Omega(n) 1_{P_1 \times\ldots \times P_k} \chi(n)$ is $<d$-biased on $[[N]]^k$.
Then on $[[N]]^k$, $\chi$ is equal to a nilsequence of degree $<d$.
\end{lemma}

\emph{Remark.} Note that the claim fails for $d=1$, even when $k=1$; if $q > 1$ is a bounded integer, then the degree $1$ nilcharacter $n \mapsto e(n/q)$ is of course biased on progression of spacing $q$, but not on the original interval $[[N]]$.  However, this is a purely ``degree $1$'' obstruction and vanishes for higher degree.

\begin{proof}  Write $P := P_1 \times \ldots P_k$.  By Corollary \ref{limone-cor}, $1_\Omega(n) 1_P \chi(n)$ correlates with a nilcharacter of degree $d-1$; we may absorb this nilcharacter into $\chi$, and assume that $1_\Omega(n) 1_P \chi(n)$ is in fact biased.

By partitioning $\Omega \cap P$ into the product $P' = P'_1 \times \ldots \times P'_k$ of dense progressions of $[[N]]$ (using \cite[Corollary A.2]{green-tao-linearprimes} to control the error), we see that there exists such a product $P' = P'_1 \times \ldots \times P'_k$ for which
$$ |\E_{n \in P'} \chi(n)| \gg 1$$
Write $\chi = F \circ \orbit$ for some degree $\leq d$ nilmanifold $G/\Gamma$, some $F \in \Lip(\ultra(G/\Gamma))$ with a vertical frequency $\eta$, and some $\orbit \in \ultra \poly(\Z^k \to G/\Gamma)$.  Applying Theorem \ref{factor2} and using the pigeonhole principle to refine the progressions $P'_1,\ldots,P'_k$ if necessary, we may assume without loss of generality that we can factorise
$$ \orbit(n) = \eps_{P'}(n) g_{P'}(n)\ultra \Gamma$$
for all $n \in P'$, where $g_{P'} \in \ultra \poly(\Z^k \to G_{P'})$ is totally equidistributed on $G_{P'}/\Gamma_{P'}$ for some standard rational subgroup $G_{P'}$ of $G$, and $\eps_{P'} \in \ultra \poly(\Z^k \to G)$ being bounded and having Lipschitz constant $O(1/N)$ on $P$, and with the $i^{\operatorname{th}}$ Taylor coefficients of size $O(N^{-|i|})$ for each $i \in \N^k$.

For any $n, n_{P'} \in {P'}$, we have from the Lipschitz nature of $\eps_{P'}$ that
$$ F(\orbit(n)) = F(\eps_{P'}(n_{P'}) g_{P'}(n) \ultra \Gamma ) + O( |n-n_0|/N ),$$
and thus by dividing ${P'}$ into sufficiently small (but still dense) sub-products, we may assume that
$$ |\E_{n \in {P'}} F( \eps_{P'}(n_{P'}) g_{P'}(n) \ultra \Gamma)|\gg 1$$
for some $n_{P'} \in {P'}$, which by the total equidistribution of $g_{P'}$ implies that
$$ \left|\int_{G_{P'}/\Gamma_{P'}} F( \eps_{P'}(n_{P'}) x )\ d\mu_{G_{P'}/\Gamma_{P'}}\right| \gg 1.$$
As $F$ has vertical frequency $\eta$, this implies that $\eta$ must annihilate $G_{{P'},\geq d}$, and so $F$ is invariant with respect to the action of this group.  By quotienting out by this central group we may thus assume that $G_{{P'},\geq d}$ is trivial, thus $G_{P'}/\Gamma_{P'}$ now has degree $<d$.  We can then write
$$ \chi(n) = \tilde F( g_{P'}(n) \ultra \Gamma_{P'}, \frac{n}{5N} \mod 1)$$
for all $n \in [[N]]$, where $\tilde F: \ultra(G_{P'}/\Gamma_{P'} \times \T)$ is defined so that
$$ \tilde F( x, \frac{n}{5N} ) = F( \eps_{P'}( n ) x )$$
for $n \in [N]$ and $x \in \ultra(G_{P'}/\Gamma_{P'})$, and extended in a Lipschitz function to all of $\ultra (G_{P'}/\Gamma_{P'} \times \T)$.  This represents $\chi$ as a nilsequence of degree $<d$ on $P'$.  Using the conjugate nature of the various sequences $g_{P}$ in Theorem \ref{factor2}, we conclude that $\chi$ can also be represented as a nilsequence of degree $<d$ on all translates $P'+h$ of $P'$ also.  On the other hand, since $P'$ is dense in $[[N]]^k$, one can partition $1 = \sum_{j=1}^J \psi_j$ on $[[N]]^k$, where $J$ is bounded and the $\psi_j$ are degree $\leq 1$ nilsequences, each of which is supported on a translate $P'+h_j$ of $P'$.  This implies that $\chi = \sum_{j=1}^J \psi_j \chi$.  As $d \geq 2$, the $\psi_j$ have degree $<d$, and the claim now follows from Corollary \ref{alg}.
\end{proof}

We have the following useful consequence of Lemma \ref{bias}.

\begin{corollary}[Extrapolation lemma]\label{extrap} Let $k,d \in \N^+$ with $d \geq 2$, let $\chi$ be a degree $d$ nilcharacter on $\ultra \Z^k$ \textup{(}with the degree filtration\textup{)}, and let $N$ be an unbounded limit natural number.   Let $P_1,\ldots,P_k$ be dense subprogressions of $[[N]]$, and let $P := P_1 \times \ldots \times P_k$. Then the following are equivalent:
\begin{itemize}
\item $\chi$ is $<d$-biased on $[[N]]^k$.
\item $\chi$ is $<d$-biased on $P$.
\item $[\chi]_{\Xi^d([[N]]^k)} = 0$.
\item $[\chi]_{\Xi^d(P)} = 0$.
\end{itemize}
\end{corollary}

\begin{proof}  We trivially have that that (iii) implies (iv).  Since $\chi$ correlates with itself, we see that (iii) implies (i) and (iv) implies (ii).  Lemma \ref{bias} gives that (i) or (ii) both imply (iii), and the claim follows.
\end{proof}

The Pontragyin dual $\T$ of the integers $\Z$ of course contains plenty of torsion.  It turns out however that this torsion is a purely degree $1$ phenomenon, and disappears in higher degree.

\begin{lemma}[Torsion-free lemma]\label{torsion}  Let $k \in \N^+$, let $N$ be an unbounded integer, and let $d \geq 2$ be standard.  Then the abelian group $\Symb^d([[N]]^k)$ \textup{(}with the degree filtration\textup{)} is torsion-free.
\end{lemma}

\begin{proof}  Our goal is to show that if $q \geq 1$ is bounded and $\chi$ is a degree $\leq s$ nilcharacter such that $\chi^{\otimes q}$ is equal to a degree $<s$ nilsequence on $[N]^k$, then $\chi$ is also equal to a degree $<s$ nilsequence.

We modify the arguments used to prove Lemma \ref{bias}.  We write $\chi = F\circ \orbit$ where $G/\Gamma$ is a degree $\leq s$ nilmanifold, $\orbit \in \ultra \poly(\Z^k \to G/\Gamma)$, and $F \in \Lip(\ultra(G/\Gamma))$ has a vertical frequency $\eta$, then we have
$$ |\E_{n \in [N]^k} F(\orbit(n))^{\otimes q} F_0(\orbit_0(n))| \gg 1$$
for some degree $<s$ nilmanifold $G_0/\Gamma_0$, some $\orbit_0 \in \ultra \poly(\Z^k \to G_0/\Gamma_0)$, and $F_0 \in \Lip(\ultra(G_0/\Gamma_0))$.  Using Theorem \ref{factor2}, we may thus find a product $P = P_1 \times \ldots \times P_k$ of progressions in $[[N]]^k$ and a factorisation
$$ (\orbit(n), \orbit_0(n)) = 
(\eps_{P}(n) g_{P}(n)\ultra \Gamma, \eps_{P,0}(n) g_{P,0}(n)\ultra \Gamma_0)$$
where $\eps_P \in \ultra \poly(\Z^k \to G)$, $\eps_{P,0} \in \ultra \poly(\Z^k \to G_0)$ are bounded and Lipschitz on $[[N]]^k$ with Lipschitz constant $O(1/N)$, and $(g_P,g_{P,0}) \in \ultra \poly(\Z^k \to \tilde G_P)$ is totally equidistributed in $\tilde G_P/\tilde \Gamma_P$ for some rational subgroup $\tilde G_P$ of $G \times G_0$.  Shrinking $P$ if necessary as in the proof of Lemma \ref{bias}, we may assume that
$$ |\int_{\tilde G_P/\tilde \Gamma_P} F( \eps_P(n_P) x)^{\otimes q} F_0(\eps_{P,0}(n_P) x_0 )\ d\mu_{\tilde G_P/\tilde \Gamma/P}(x,x_0)| \gg 1$$
for any $n_P \in P$.  From the vertical character nature of $F$, this implies that $\eta^q$ annihilates $(\tilde G_P)_{s}$.  But $\eta$ is a continuous homomorphism on the connected abelian Lie group $(\tilde G_P)_{s}$, and so $\eta$ itself must also annihilate $(\tilde G_P)_{s}$.  If we then quotient by this space, we can represent $\chi$ by a degree $<s$ nilsequence on $P$, and the claim now follows from Corollary \ref{extrap}.
\end{proof}

\section{A linearisation result from additive combinatorics}\label{app-f}

In this appendix, we record a lemma from additive combinatorics (essentially in \cite{gowers-4aps} or \cite{green-tao-u3inverse}, and in the spirit of Freiman's inverse sumset theorem) which asserts that functions from a large subset of $[-N,N]$ to $\T$ with a large amount of additive structure are essentially bracket-linear in nature.

\begin{lemma}[Linearisation lemma]\label{lin}  Let $\eps > 0$ be a limit real, let $N$ be a limit natural number, let $H$ be a dense subset of $[[N]]$, let $\alpha \in \ultra\T$ be a frequency, and let $\xi_1,\xi_2,\xi_3,\xi_4: H \to \ultra \T$ be limit functions such that
\begin{equation}\label{so}
 \xi_1(h_1) + \xi_2(h_2) + \xi_3(h_3) + \xi_4(h_4) = \alpha + O(\eps) 
\end{equation}
for many additive quadruples $(h_1,h_2,h_3,h_4) \in H$.  Then there exists a standard $k \geq 0$, a frequency $\delta \in \ultra\T$, a dense subset $H'$ of $H$, and a Freiman homomorphism $\xi: H' \to \ultra\T$ of the form
$$ \xi(h) = \sum_{k=1}^K \{ \alpha_k h \} \beta_k \mod 1$$
for all $h \in H'$ and some $\alpha_k \in \ultra\T$ and $\beta_k \in \ultra \R$ and some standard $K$, such that
\begin{equation}\label{xam}
\xi_1(h) = \xi(h) + \delta + O(\eps) 
\end{equation}
for many $h \in H$.
\end{lemma}

\begin{proof}  We may replace $\eps$ by $1/M$ for some limit integer $M$.  By rounding each $\xi_i(h)$ to the nearest multiple of $1/M$, we may assume that $\xi_i(h)$ is a multiple of $1/M$ for all $h \in H$ and $i=1,2,3,4$.  There are now only a bounded number of possibilities for the right-hand side $\alpha+O(\eps)$, so by the pigeonhole principle (and by redefining $\alpha$ if necessary) we may assume that
$$ \xi_1(h_1)+\xi_2(h_2)+\xi_3(h_3)+\xi_4(h_4) = \alpha $$
for many additive quadruples $(h_1,h_2,h_3,h_4)$ in $H$.

For each $i=1,2,3,4$, let $\Gamma \subset \ultra \Z \times \ultra \T$ be the (limit) graph $\Gamma_i := \{ (h,\xi_i(h) \mod 1): h \in H \}$.  Then by the preceding discussion, we see that $(0,\alpha)$ has $\gg N^3$ representations of the form $\gamma_1+\gamma_2+\gamma_3+\gamma_4$, where $\gamma_i \in \Gamma_i$ for $i=1,2,3,4$.  On the other hand, from several applications of the Cauchy-Schwarz inequality, the number of such quadruples is bounded by $\prod_{i=1}^4 E(\Gamma_i)^{1/4}$, where $E(\Gamma_i)$ is the number of additive quadruples in $\Gamma_i$ (i.e. the \emph{additive energy} of $\Gamma_i$).  Since we have the trivial upper bound $E(\Gamma_i) \ll N^3$ for all $i$, we conclude that
$$ E(\Gamma_1) \gg N^3.$$

At this point we invoke some standard additive combinatorial machinery from \cite{green-tao-u3inverse} (see also \cite{gowers-4aps,sam}).  Applying the Balog-Szemer\'edi-Gowers lemma followed by the Pl\"unnecke-Ruzsa inequalities exactly as in \cite[Proposition 5.4]{green-tao-u3inverse}, we can find a dense subset $\Gamma'$ of $\Gamma_1$ such that $|9\Gamma' - 8\Gamma'| \ll N$.
Applying \cite[Lemma 9.2]{green-tao-u3inverse}, we can refine to a further dense subset $\Gamma'' := \{ (h,\xi(h) \mod 1): h \in H'' \}$ such that $4\Gamma''-4\Gamma''$ is a graph; thus there exists a Freiman homomorphism\footnote{The notion of a Freiman homomorphism was defined in \S \ref{notation-sec}.} $\zeta: 2H''-2H'' \to \T$ such that
\begin{equation}\label{etah}
 \xi_1(h_1) + \xi_1(h_2) - \xi_1(h_3) - \xi_1(h_4) = \zeta(h_1+h_2-h_3-h_4) 
\end{equation}
for all $h_1,h_2,h_3,h_4 \in H''$.  By the Bogulybov lemma (see \cite[Lemma 6.3]{green-tao-u3inverse}), $2H''-2H''$ contains a dense regular Bohr set $B$ of bounded rank (see \cite{green-tao-u3inverse} for definitions; strictly speaking, one has to identify an interval such as $[[10N]]$ with $\Z/20N\Z$ in order to apply these tools, but this is not difficult to do).  Arguing as\footnote{This proposition involved a quadratic function on a Bohr set, rather than a linear one, but it is clear that the argument specialises to the linear case.} in \cite[Proposition 10.8]{green-tao-u3inverse}, we see that we may write
$$ \zeta(h) = \sum_{j=1}^k \{ \alpha_j h \} \beta_j \mod  1$$
for $h \in B$ for some standard $k$ and frequencies $\alpha_j, \beta_j$.  
Applying \eqref{etah} and the pigeonhole principle, we obtain the claim, except possibly for the claim that $\xi$ is a Freiman homomorphism.  But observe that if we restrict the fractional part of $\{\alpha_j h\}$ to a sub-interval of $I_0$ of length at most $1/10$ (say) then we obtain the Freiman homomorphism property automatically; so the claim follows from one final application of the pigeonhole principle.
\end{proof}

\providecommand{\bysame}{\leavevmode\hbox to3em{\hrulefill}\thinspace}

\end{document}